\documentclass[a4paper,10pt,leqno]{amsart}
\title{Algebraic K-theory of reductive $p$-adic groups}
\author{Bartels, A.}
\address{WWU M\"unster\\
  Mathematisches Institut\\
  Einsteinstr.~62, D-48149 M\"unster, Germany} \email{a.bartels@wwu.de}
\urladdr{http://www.math.uni-muenster.de/u/bartelsa} \author{L\"uck, W.}
\address{Mathematisches Institut der Universit\"at Bonn\\
  Endenicher Allee 60\\
  53115 Bonn, Germany} \email{wolfgang.lueck@him.uni-bonn.de}
\urladdr{http://him-lueck.uni-bonn.de} \date{December 2023}
\keywords{Hecke algebras, algebraic K-theory, Farrell--Jones Conjecture, reductive $p$-adic groups}
          \makeatletter
     \@namedef{subjclassname@2020}{%
  \textup{2020} Mathematics Subject Classification}
\makeatother
\subjclass[2020]{20C08,19D50}

  \usepackage[utf8]{inputenc}
    \usepackage{comment}
  \usepackage{calc}
  \usepackage{enumerate,amssymb,bm}
  \usepackage[arrow,curve,matrix,tips,2cell]{xy}
    \SelectTips{eu}{10} \UseTips
    \UseAllTwocells
  \usepackage{tikz}
  \usetikzlibrary{decorations.pathreplacing}
  \usepackage{color}
  \usepackage{scalerel}
  \usepackage{braket}
  \usepackage{mathtools}
  \usepackage{bbm}
  \usepackage{enumitem}
  \usepackage{float}
  \DeclareMathAlphabet{\matheurm}{U}{eur}{m}{n}
  \usepackage{hyperref}

\DeclareMathAlphabet{\matheurm}{U}{eur}{m}{n}

\newcommand{\contstrCtilde}{\widetilde{\mfC}}
\newcommand{\contrCatUcoef}[3]{#3_{#1}(#2)}
\newcommand{\contrCatUcoefover}[3]{\overline{#3}_{#1}(#2)}
\newcommand{\contrCatNUcoef}[3]{#3^{\nounit}_{#1}(#2)}

\newcommand{\Hecke}[1]{\calh_{#1}}

\newcommand{\Int}{\matheurm{Int}}

\newcommand{\MODcat}[1]{#1\text{-}\matheurm{Mod}}

\newcommand{\Or}{\matheurm{Or}}

\newcommand{\OrGF}[2]{\matheurm{Or}_{#2}(#1)}

\newcommand{\Rep}{\matheurm{Rep}}

\newcommand{\Sets}{\matheurm{Sets}}

\newcommand{\Spaces}{\matheurm{Spaces}}

\newcommand{\SpacesC}[1]{#1\text{-}{\matheurm{Spaces}}}

\newcommand{\Spectra}{\matheurm{Spectra}}

\newcommand{\Sub}{\matheurm{Sub}}

\newcommand{\underlinetupel}[1]{{\underline{#1}}}

\newcommand{\ZerodimR}{\bfM}
\newcommand{\ZerodimRunderlyingSet}{M}
\newcommand{\ZerodimRunderlyingSetElement}{m}

\newcommand{\contstraPW}{\mfE^{\text{pw}}}

\newcommand{\contc}{\EC}
\newcommand{\contC}{\widehat{\EC}}

\newcommand{\contstrc}{\mfC}
\newcommand{\contstrC}{\widehat{\mfC}}

\newcommand{\contd}{\ED}

\newcommand{\contstrd}{\mfD}

\newcommand{\regularPOrGSC}{\calr}
\newcommand{\regularPOrGSCzero}{\calr^0}
\newcommand{\regularPOrGSCzerocalf}[1]{\calr^0_{#1}}

\newcommand{\twovec}[2]{{\begin{psmallmatrix}{#1}\\{#2}\end{psmallmatrix}}}

\newcommand{\bbone}[1]{\mathbbm{1}\hspace{-.5ex}_{#1}}

\DeclareMathOperator{\aut}{aut}

\DeclareMathOperator{\bary}{bary}

\DeclareMathOperator{\cent}{cent}
\DeclareMathOperator{\ch}{ch}

\DeclareMathOperator*{\colimunder}{colim}
\DeclareMathOperator{\comp}{comp}

\DeclareMathOperator{\diam}{diam}
\DeclareMathOperator{\dg}{dg}
\DeclareMathOperator{\homendo}{end}

\DeclareMathOperator{\GL}{GL}

\DeclareMathOperator*{\hocolimunder}{hocolim}

\DeclareMathOperator{\holim}{holim}

\DeclareMathOperator{\id}{id}

\DeclareMathOperator{\Image}{im}

\DeclareMathOperator{\invlim}{invlim}
\DeclareMathOperator{\ind}{ind}

\DeclareMathOperator{\Idem}{Idem}

\DeclareMathOperator{\Kgroup}{K}

\DeclareMathOperator{\mor}{mor}

\DeclareMathOperator{\ob}{ob}

\DeclareMathOperator{\PGL}{PGL}

\DeclareMathOperator{\pr}{pr}
\DeclareMathOperator{\res}{res}

\DeclareMathOperator{\sh}{sh}

\DeclareMathOperator{\simp}{simp}

\DeclareMathOperator{\SL}{SL}

\DeclareMathOperator{\supp}{supp}

\DeclareMathOperator{\vertices}{vert}

\newcommand{\fold}[1]{d_{#1\text{-}\mathrm{fol}}}

\newcommand{\All}{{\cala\hspace{-.3pt}\mathrm{ll}}}
\newcommand{\COM}{{\calc\hspace{-1pt}\mathrm{om}}}

\newcommand{\COP}{{\calc\hspace{-1pt}\mathrm{op}}}

\newcommand{\CVCYC}{{\calc\hspace{-1pt}\mathrm{vcy}}}

\newcommand{\FIN}{{{\mathcal F}\mathrm{in}}}
\newcommand{\OPEN}{{\calo\mathrm{p}}}
\newcommand{\VCYC}{{{\mathcal V}\mathrm{cyc}}}

  \newcommand{\IC}{\mathbb{C}}

  \newcommand{\IN}{\mathbb{N}}

  \newcommand{\IQ}{\mathbb{Q}}
  \newcommand{\IR}{\mathbb{R}}

  \newcommand{\IZ}{\mathbb{Z}}

  \newcommand{\EC}{\matheurm{C}}  
  \newcommand{\ED}{\matheurm{D}}

  \newcommand{\EP}{\matheurm{P}}

  \newcommand{\cala}{\mathcal{A}}
  \newcommand{\calb}{\mathcal{B}}
  \newcommand{\calc}{\mathcal{C}}
  
  \newcommand{\cald}{\mathcal{D}}
  \newcommand{\cale}{\mathcal{E}}
  \newcommand{\calf}{\mathcal{F}}

  \newcommand{\calh}{\mathcal{H}}

  \newcommand{\calk}{\mathcal{K}}
    \newcommand{\call}{\mathcal{L}}

  \newcommand{\calo}{\mathcal{O}}
  
  \newcommand{\calp}{\mathcal{P}}
  \newcommand{\calq}{\mathcal{Q}}
  \newcommand{\calr}{\mathcal{R}}
  \newcommand{\cals}{\mathcal{S}}

    \newcommand{\calu}{\mathcal{U}}
  \newcommand{\calv}{\mathcal{V}}
  
  \newcommand{\calw}{\mathcal{W}}
  
  \newcommand{\caly}{\mathcal{Y}}

  \newcommand{\bfA}{\mathbf{A}}
  
  \newcommand{\bfB}{\mathbf{B}}
  
   \newcommand{\bfC}{\mathbf{C}}
   
  \newcommand{\bfD}{\mathbf{D}}
  
  \newcommand{\bfDelta}{{\bf \Delta}}
  \newcommand{\bfE}{\mathbf{E}}

  \newcommand{\bff}{\mathbf{f}}

  \newcommand{\bfH}{\mathbf{H}}

  \newcommand{\bfi}{\mathbf{i}}
  \newcommand{\bfJ}{\mathbf{J}}
  
  \newcommand{\bfK}{\mathbf{K}}

  \newcommand{\bfL}{\mathbf{L}}

  \newcommand{\bfM}{\mathbf{M}}

  \newcommand{\bfp}{\mathbf{p}}

  \newcommand{\bfS}{\mathbf{S}}
  
  \newcommand{\bfSigma}{{\bf \Sigma}}

  \newcommand{\bftr}{\mathbf{tr}}
  
  \newcommand{\bfU}{\mathbf{U}}
  
  \newcommand{\bfV}{\mathbf{V}}

  \newcommand{\bfX}{\mathbf{X}}

 \newcommand{\bfKinfty}{\mathbf{K}^{\infty}}

\newcommand{\mfC}{\mathfrak{C}}
\newcommand{\mfD}{\mathfrak{D}}
\newcommand{\mfE}{\mathfrak{E}}

\newcommand{\dis}{{\mathrm{dis}}}
\newcommand{\fin}{{\mathrm{fin}}}
\newcommand{\fol}{{\mathrm{fol}}}
\newcommand{\op}{{\mathrm{op}}}
 
 \newcommand{\sw}{{\mathrm{sw}}}

\newcommand{\EPplus}{\EP\hspace{-2.5pt}_+\hspace{-1.2pt}}
\newcommand{\EGF}[2]{E_{#2}(#1)}

\newcommand{\prodprimedisplay}{\sideset{}{'}\prod}
\newcommand{\prodprimeinline}{{\sideset{}{'}{\textstyle\prod}}}

\newcommand{\suppX}{\supp_2}
\newcommand{\suppG}{\supp_G}
\newcommand{\suppobj}{\supp_1}

\newcounter{commentcounter}

\renewcommand{\thesubsection}{\arabic{section}.{\sc\alph{subsection}}} 

\theoremstyle{plain}
\newtheorem{theorem}{Theorem}[section]

\newtheorem{lemma}[theorem]{Lemma}
\newtheorem{corollary}[theorem]{Corollary}
\newtheorem{proposition}[theorem]{Proposition}
\newtheorem{conjecture}[theorem]{Conjecture}
\newtheorem{assumption}[theorem]{Assumption}

\newtheorem*{theorem*}{Theorem}
\newtheorem*{theoremA*}{Theorem A}
\newtheorem*{theoremB*}{Theorem B}

\theoremstyle{definition}
\newtheorem{definition}[theorem]{Definition}

\newtheorem{example}[theorem]{Example}

\newtheorem{remark}[theorem]{Remark}

\newtheorem*{definition*}{Definition}

\theoremstyle{remark}

\makeatletter\let\c@equation=\c@theorem\makeatother

\theoremstyle{definition}
  
\newtheorem{construction}[theorem]{Construction} 

\newcounter{othercommentcounter}

\newcommand{\ball}{D} 
\newcommand{\newD}{L^+} 

\newcommand{\nounit}{\text{\tiny nu}}

\newcommand{\contcover}{\overline{\EC}}

\newcommand{\allG}{+}
\newcommand{\nowedge}{\sharp}

\newcommand{\dd}{{\partial}}
\newcommand{\e}{{\varepsilon}}
\newcommand{\CAT}{\operatorname{CAT}}
\newcommand{\FS}{\mathit{FS}}

\newcommand{\version}[1]              
{\begin{center} last edited on #1\\
last compiled on \today\\
name of tex-file: \jobname
\end{center}
}

\begin{document}

\begin{abstract}
  Motivated by the Farrell--Jones Conjecture for group rings, we formulate the
  $\COP$-Farrell--Jones Conjecture for the K-theory of Hecke algebras of td-groups.  We
  prove this conjecture for (closed subgroups of) reductive p-adic  groups $G$.  In
  particular, the projective class group $\Kgroup_0(\calh(G))$ for a (closed subgroup) of
  a reductive p-adic  group $G$ can be computed as a colimit of projective class groups
  $\Kgroup_0(\calh(U))$ where $U$ varies over the compact open subgroups of $G$.
   This implies that all finitely generated smooth complex representations of a reductive p-adic  $G$ admit
   finite projective resolutions by compactly induced representations.  For $\SL_n(F)$
   we translate the colimit formula for $\Kgroup_0(\calh(G))$ to a more
  concrete cokernel description in terms of stabilizers for the action on the Bruhat-Tits
  building.
	
  For negative K-theory we obtain vanishing results, while we identify the higher
  K-groups $\Kgroup_n(\calh(G))$ with the value of $G$-homology theory on the extended
  Bruhat-Tits building.  Our considerations apply to general Hecke algebras of the form
  $\calh(G;R,\rho,\omega)$, where we allow a central character $\omega$ and a twist by an
  action $\rho$ of $G$ on $R$.  For the $\COP$-Farrell--Jones Conjecture we need to assume
  $\IQ \subseteq R$ and a regularity assumption.  As a key intermediate step we introduce
  the $\CVCYC$-Farrell--Jones conjecture.  For the latter no regularity assumptions on $R$
  are needed. 
\end{abstract}

\maketitle


\section{Introduction}

The Farrell--Jones conjecture~\cite{Farrell-Jones(1993a)} originated in surgery theory and has applications to the classification of manifolds, notably it implies (in dimension $\geq 5$) Borel's conjecture on the topological rigidity of aspherical manifolds.
  The conjecture concerns the K- and L-groups of group rings and expresses these in terms of an equivariant homology theory.
  It can be viewed as reducing computations to the case of group rings for virtually cyclic groups. 
  Under regularity assumptions there are often further reductions, typically to group rings of finite groups. 
  Further information on the conjecture can be found for instance in~\cite{Lueck(2022book),Lueck-Reich(2005)}.  
  Farrell and Jones used the geodesic flow on non-positively curved manifolds as a tool to confirm their conjecture for fundamental groups of such manifolds~\cite{Farrell-Jones(1993a)}.
  
 In this paper we study the K-theory of Hecke algebras of td-groups and  transfer the Farrell--Jones conjecture and the geodesic flow method to smooth representation theory.
 We obtain formulas for the K-theory of Hecke algebras $\calh(G;R)$ where $G$ is a closed subgroup of a reductive p-adic group and $R$ is a field of characteristic $0$\footnote{It suffices for $R$ to contain $\IQ$ and to satisfy a regularity assumption.}. These express the K-theory of $\calh(G;R)$ as $G$-homology groups of the associated Bruhat-Tits building, see Corollary~\ref{cor:K_n-no-fuzz}.
 On the level of $\Kgroup_0$ this yields isomorphisms 
 \begin{equation*}
  \colimunder_{U \in \Sub_\COP(G)} \Kgroup_0 (\calh(U;R)) \xrightarrow{\cong} \Kgroup_0 (\calh(G;R))
 \end{equation*}
 where the colimit is taken over a category of compact open subgroups of $G$, see Corollary~\ref{cor:K_0-no-fuzz}.
 This confirms in particular a conjecture of Dat~\cite[Conj.~1.11]{Dat(2003)}.
 For finitely generated smooth representations it implies the existence of finite length resolutions by compactly induced representations, generalizing a result of Schneider--Stuhler for admissable representations, see Subsection~\ref{subsec:Resolutions_of_smooth_representations}.  
  A long standing conjecture in smooth representation theory asks whether all irreducible cuspidal representations of reductive $p$-adic groups $G$ are compactly induced\footnote{If $G$ has non-trivial center, then one needs to consider open subgroups that are compact modulo center. There are versions of our results in this situation as well, see Corollary~\ref{cor:main-for-twisted}.}.
  Our results imply that the $\Kgroup_0$-classes of finitely generated representations can be expressed in terms of compact induction. 
  
  We proceed to explain our results in more detail.


\subsection{Hecke algebras and $p$-adic groups}
Let $G$ be a \emph{td-group}, i.e., a locally compact second countable totally disconnected
topological Hausdorff group.  In such a group the neutral element $e$ has a countable
neighborhood basis consisting of compact open subgroups.  Let $R$ be a not necessarily
commutative ring with unit containing $\IQ$.  The \emph{Hecke algebra} $\calh(G;R)$ of $G$
over $R$ is the algebra of locally constant compactly supported $R$-valued functions on
$G$.  Its multiplication is given by
convolution\footnote{$\varphi \ast \varphi' (g) = \int_G \varphi(gx) \varphi'(x^{-1}) dx$.}
relative to a $\IQ$-valued left-invariant Haar
measure on $G$\footnote{If $\mu$ an $\IR$-valued Haar measure and $K$ is compact open in
  $G$, then $\frac{\mu}{\mu(K)}$ is $\IQ$-valued; the choice of Haar measure changes the
  Hecke algebra only by canonical isomorphism.}.  There are more general Hecke algebras
$\calh(G;R,\rho,\omega)$ allowing for twists $\rho$ by an action of $G$ on $R$ and a
central character $\omega$.  Hecke algebras are in general not unital.  A module $M$ over
the Hecke algebra is \emph{non-degenerate}, if $\calh(G;R) \cdot M = M$.  A representation of $G$
on an $R$-module $V$ is said to be \emph{smooth}, if all isotropy groups of the action of $G$ on
$V$ are open.  The category of non-degenerated $\calh(G;R)$-modules is equivalent to the
category of smooth representations on $R$-modules, see~\cite[Sec.~9]{Garrett(2012)}.  By a
\emph{reductive $p$-adic group} we will mean the $F$-points of an algebraic group over $F$,
whose component of the identity is reductive, where $F$ is a non-Archimedian local field,
i.e., a finite extension of the field of $p$-adic numbers or the field of formal Laurent
series $k((t))$ over a finite field $k$.  Reductive $p$-adic groups are td-groups.

Associated to a reductive $p$-adic group is its extended \emph{Bruhat-Tits building}
$X$~\cite{Bruhat-Tits(1972), Bruhat-Tits(1984), Tits(1979)}.  This is a $\CAT(0)$-space
with a cocompact proper isometric $G$-action.  The
building can also be given the structure of a simplicial complex such that the action of
$X$ is simplicial and smooth.  For a short review of the Bruhat-Tits building, emphasizing
the aspects we need, see~\cite[Appendix~A]{Bartels-Lueck(2023almost)}.


\subsection{Compact induction}
The \emph{compact induction} of a smooth representation $V$ of a compact open subgroup $U$ of $G$
is the $G$-representation consisting of compactly supported $U$-equivariant maps
$G \to V$\footnote{The formula for the $G$-action is $(gf)(x) := f(xg)$ for $f \colon G \to V$, $g \in G$.}.  On the level of Hecke algebras compact induction is induced by the inclusion
$\calh(U;R) \subseteq \calh(G;R)$.  This inclusion exists for open subgroups $U$ of $G$;
locally constant functions on open subgroups can be extended by zero.  Smooth
representations of a reductive $p$-adic group $G$ are often studied through compact
induction.  For example, type theory, introduced by
Bushnell-Kutzko~\cite{Bushnell-Kutzko(1998smooth_reps)}, aims at describing Bernstein
blocks in the representation category as modules over endomorphism rings of
representations that are compactly induced.  Conjecturally, all irreducible cuspidal
representations are induced from compact modulo center open subgroups.  See
Fintzen~\cite{Fintzen(2021types)} for recent far reaching results concerning these
conjectures.

Following Dat~\cite{Dat(2000)} we study the K-theory of Hecke algebras (equivalently, of
smooth representations) via compact induction.  While this leads to less explicit results
about smooth representations, it allows for very general results.  Ultimately we will
describe the K-theory of Hecke algebras of reductive $p$-adic groups  in terms
of the K-theory of Hecke algebras of compact open subgroups.  We hope that the connection
to the above mentioned conjectures can be explored in the future.


\subsection{$\Kgroup_0$ of Hecke algebras}
Let $\calh(G;R)$ be the Hecke algebra of a td-group $G$ with coefficients in $R$.  The
projective class group $\Kgroup_0 (\calh(G;R))$ is the abelian group with a generator
$[P]$ for each finitely generated projective $\calh(G;R)$-module subject to
the relation $[P \oplus P'] = [P] \oplus [P']$.  Compact induction preserves finitely
generated projective modules and induces a map on $\Kgroup_0$.  Combining these maps for
all compact open subgroups of $G$ we obtain
\begin{equation}\label{eq:K_0-assembly-hecke}
  \colimunder_{U \in \Sub_\COP(G)} \Kgroup_0 (\calh(U;R)) \to \Kgroup_0 (\calh(G;R))
\end{equation}
where $\Sub_\COP(G)$ is the following category.  Objects are compact open subgroups of
$G$.  Morphisms $U \to U'$ are equivalence classes of group homomorphisms of the form
$x \mapsto gxg^{-1}$ with $g \in G$.  Two such group homomorphisms are identified if they
differ by an inner automorphism of $U'$\footnote{In other words, $\mor_{\Sub_\COP}(U,U')$
  is the double coset $U' \backslash \{ g \in G \mid gUg^{-1} \subseteq U' \} / C_G(U)$
  where $C_G(U)$ is the centralizer of $U$ in $G$.}.  To study surjectivity the colimit
in~\eqref{eq:K_0-assembly-hecke} can of course be replaced with the sum of the groups
$\Kgroup_0 (\calh(U;R))$.  Dat~\cite{Dat(2000)} has shown
that~\eqref{eq:K_0-assembly-hecke} is rational surjective for $G$ a reductive $p$-adic
group and $R=\IC$.  In particular, the cokernel of~\eqref{eq:K_0-assembly-hecke} is a
torsion group.  Dat~\cite[Conj.~1.11]{Dat(2003)} conjectured that this cokernel is
$\widetilde w_G$-torsion.  Here $\widetilde w_G$ is a certain multiple of the order of the
Weyl group of $G$.  Dat proved this conjecture for
$G = \GL_n(F)$~\cite[Prop.~1.13]{Dat(2003)} and asked about integral surjectivity, see the
comment following~\cite[Prop.~1.10]{Dat(2003)}.  The following will be a consequence of
our main result.

\begin{corollary}\label{cor:K_0-no-fuzz}
  Assume that $G $ is a modulo a compact subgroup isomorphic to a closed subgroup of a
  reductive $p$-adic group. Let $R$ be a ring containing $\IQ$.  Assume that $R$ is uniformly 
    regular, i.e., $R$ is noetherian and there
    is $l$ such that every $R$-module admits a projective resolution
     of length at most $l$.
    Then~\eqref{eq:K_0-assembly-hecke} is an isomorphism.
\end{corollary}

This is a special case of Corollary~\ref{cor:main-for-twisted}~\ref{cor:main-for-twisted:K_0},
where we consider more general Hecke
algebras, allowing for twists by actions of $G$ on $R$ and central characters.


\subsection{Resolutions of smooth representations}\label{subsec:Resolutions_of_smooth_representations}

Let $G$ be a reductive p-adic group $G$. Bernstein~\cite{Bernstein(1992)} showed that the
category of smooth complex representations is noetherian and has finite cohomological
dimension.  Consequently, any finitely generated smooth complex representation has a
finite resolution
\begin{equation}\label{eq:resolution-of-V}
  P_n \to P_{n-1} \to \dots \to P_0 \to V 
\end{equation}
where the $P_i$ are finitely generated projective.  A smooth $G$-representation is said to be 
\emph{admissible}, if for every compact open subgroup $U$ of $G$ the subspace $V^U$ of
$U$-fixed vectors is finite dimensional.  
  It is called \emph{compactly induced},
  if it is for some compact open subgroup $U \subseteq G$
  the compact induction of a finitely generated projective $U$-representation.

  Schneider and
Stuhler~\cite{Schneider-Stuhler(1997)} showed that for finitely generated
admissible $V$ the $P_i$ in the
above resolution can be chosen to be finite direct sums of compactly induced representations.
From Corollary~\ref{cor:K_0-no-fuzz} we obtain a generalization to arbitrary
finitely generated $V$.

\begin{corollary}\label{cor:resolutions}
  Every finitely generated smooth complex representation $V$ of $G$ admits a finite
  resolution~\eqref{eq:resolution-of-V} where the $P_i$ are direct sums of compactly
  induced representations.
\end{corollary}

\begin{proof}
  Under the equivalence of categories between smooth representations and (non-degenerated)
  Hecke modules the compactly induced representations correspond to the modules in the
  image of the induction map 
  \begin{equation*}
    \MODcat{\calh(U;\IC)} \to \MODcat{\calh(G;\IC)}, \quad M \mapsto \calh(G;\IC) \otimes_{\calh(U;\IC)} M
  \end{equation*}
  for some compact open subgroup $U \subseteq G$.
  Let $P$ be a finitely generated
  projective $\calh(G;\IC)$-module.  Corollary~\ref{cor:K_0-no-fuzz} implies
  that in $\Kgroup_0$ we have $[P] = [W] - [W']$ where both $W$ and $W'$ are sums of
  compactly induced modules.  This means that there is an isomorphism
  $P \oplus W \oplus Q \cong W' \oplus Q$ for some finitely generated projective
  $\calh(G;\IC)$-module $Q$.  As any finitely generated projective module is a direct
  summand of a compactly induced modules,\footnote{If $v_1,\dots,v_n$ generates $P$ and $U$
    fixes the $v_i$, then $P$ is a direct summand of
    $\calh(G;\IC) \otimes_{\calh(U;\IC)} \IC^n$.} we can stabilize further and then absorb
  $Q$ into $W$ and $W'$, i.e., we obtain $P \oplus W \cong W'$ with $W$ and $W'$ finite direct sums of
  compactly induced modules.

  We can applying this to the $P_i$ in~\eqref{eq:resolution-of-V}.  Thus by adding
  appropriate elementary chain complexes on compactly induced modules
  to~\eqref{eq:resolution-of-V} we obtain the desired resolution of $V$.
   \end{proof}


\subsection{Smooth $G$-homology theories}%
\label{subsec:Smooth_G-homology_theories}

The \emph{orbit category} has as objects homogeneous $G$-sets $G/V$ with $V$ closed in $G$
and as morphisms $G$-maps.  The \emph{smooth orbit category} $\Or_{\OPEN}(G)$ is the full
subcategory on all $G/U$ with $U$ open in $G$.  Let $\Spectra$ be the category of (not
necessarily connective) spectra.  Associated to a covariant functor
$\bfE \colon \OrGF{G}{\OPEN} \to \Spectra$ there is a smooth $G$-homology theory
\begin{equation}
  H_*^G(-;\bfE)
  \label{H_(ast)_upper_G(-,bfE)}
\end{equation}
such that $H_n^G(G/H;\bfE) = \pi_n(\bfE(G/H))$ for $n \in \IZ$.  Here a \emph{smooth
  $G$-homology theory} is to be understood in the obvious way: It digests (pairs of)
smooth $G$-$CW$-complexes, yields an abelian group $H_n(X;\bfE)$ for every $n \in \IZ$,
and satisfies the expected axioms, namely, functoriality in $G$-maps, $G$-homotopy
invariance, the long exact sequence of a smooth $G$-$CW$-pair, and $G$-excision.  All this
is explained in~\cite{Davis-Lueck(1998)}.  The point here is that smooth $G$-CW-complexes
are contravariant free $\calc$-CW-complexes in the sense
of~\cite[Def.~3.2]{Davis-Lueck(1998)} for $\calc = \Or_{\OPEN}(G)$.  See also the discussions
in~\cite[Sec.~2.C]{Bartels-Lueck(2023foundations)}.

   
\subsection{The K-theory spectrum of Hecke algebras}\label{subsec:K-theory-spectrum-Hecke-algebras}
To generalize~\eqref{eq:K_0-assembly-hecke} to the K-theory spectrum, we introduce some
notation.  A \emph{category with $G$-support} is a $\IZ$-linear category $\calb$ together with
maps $\supp_G$ that associate to objects and morphisms compact subsets of $G$ subject to a
natural list of axioms, see Definition~\ref{def:category_with_G-support}.  Given a $G$-set $X$
and such a $\calb$, we naturally obtain a $\IZ$-linear category $\calb[X]$, see
Definition~\ref{def:calbX}.   The key example associated to
the Hecke algebra $\calh(G;R)$ is the category $\calb(G;R)$, see
Example~\ref{ex:calb(G;R)}.  
 Its objects are compact open subgroups $U \subseteq G$. Morphisms $U \to U'$ are elements $f$
  of $\calh(G;R)$ satisfying $f = e_{U'} f e_U$, where $e_U$ is the idempotent
  in $\calh(G;R)$
associated to the compact open subgroup $U$.
Here $\supp_G(f) = \{g \in G \mid f(g) \not= 0\}$,
which is automatically compact as $f$ is compactly supported and locally constant. 
We define
\begin{equation}\label{eq:bfK_R}
  \bfK_R \colon \Or_\OPEN(G) \to \Spectra, \quad G/U \mapsto \bfK \big(\calb({G};{R})[G/U]\big),
\end{equation} 
where $\bfK$ is the K-theory functor for $\IZ$-linear categories, see
Subsection~\ref{subsec:K-theory}.  The homotopy groups of $\bfK_R (G/U)$ are the K-groups
of the Hecke algebra $\calh(U;R)$, see~\cite[(6.8)]{Bartels-Lueck(2023foundations)}.
As discussed in Subsection~\ref{subsec:Smooth_G-homology_theories} we can
apply~\cite{Davis-Lueck(1998)} and obtain a smooth $G$-homology theory $H_n^G(\--;\bfK_R)$
with $H_n^G(G/U;\bfK_R) \cong \Kgroup_n(\calh(U;R))$.

Associated to the family $\COP$ of compact open subgroups, there is a $G$-$CW$-complex
$\EGF{G}{\COP}$ that  is uniquely determined up to $G$-homotopy by the property that all
its isotropy groups belong to $\COP$ and $\EGF{G}{\COP}^H$ is weakly contractible for
$H \in \COP$. In particular $\EGF{G}{\COP}$ is a proper smooth $G$-$CW$-complex.  Every
$G$-$CW$-complex $X$, whose isotropy belongs to $\COP$, has up to $G$-homotopy precisely
one $G$-map to $\EGF{G}{\COP}$, see~\cite[Subsec.~1.2]{Lueck(2005s)}.  Analogously one can
define for the family $\COM$ of all compact subgroups its classifying space
$\EGF{G}{\COM}$.  It turns out that the canonical $G$-map
$\EGF{G}{\COP} \to \EGF{G}{\COM}$ is a $G$-homotopy equivalence for a td-group $G$,
see~\cite[Lemma~3.5]{Lueck(2005s)}.

The projection $\EGF{G}{\COP} \to G/G$ induces a map
\begin{equation}\label{eq:COP-assembly-homolgy-theory-R}
  H_n^G(\EGF{G}{\COP};\bfK_R) \to H_n^G(G/G;\bfK_R) = K_n \big(\calh(G;R)\big).
\end{equation} 
If $R$ is a regular ring containing $\IQ$, then there is an isomorphism
$H_0^G(\EGF{G}{\COP};\bfK_R) \cong \colimunder_{U \in \Sub_\COP(G)} \Kgroup_0
(\calh(U;R))$.  Using this isomorphism~\eqref{eq:K_0-assembly-hecke} can be identified
with~\eqref{eq:COP-assembly-homolgy-theory-R} for $n=0$,
see~\cite[Thm~1.1~(iii))]{Bartels-Lueck(2023recipes)}.

We note that if $G$ is a reductive $p$-adic group, then we can take for $\EGF{G}{\COP}$ the
extended Bruhat-Tits building associated to $G$~\cite[Thm.~4.13]{Lueck(2005s)}\footnote{More
  general, if $G$ is a closed subgroup of a reductive $p$-adic group $\widehat G$, then we
  can use the extended Bruhat-Tits building associated to $\widehat G$ with the restricted
  action.}.  The following will be a consequence of our main result.

\begin{corollary}\label{cor:K_n-no-fuzz}
  Assume that $G $ is a modulo a compact subgroup isomorphic to a closed subgroup of a
  reductive $p$-adic group. Let $R$ be a ring containing $\IQ$.  Assume that $R$ is uniformly\footnote{It is
    plausible that the result is also true if $R$ is only assumed to be regular, but our
    proof certainly uses uniform regularity.} regular, i.e., $R$ is noetherian and there
  is $l$ such that every $R$-module admits a projective resolution of length at most $l$.
  Then~\eqref{eq:COP-assembly-homolgy-theory-R} is an isomorphism.
\end{corollary}

This is a special case of
Corollary~\ref{cor:main-for-twisted}~\ref{cor:main-for-twisted:ass}, where we consider
more general Hecke algebras, allowing for twists by actions of $G$ on $R$ and central
characters.  Conjecture~\ref{cor:K_n-no-fuzz} was stated in~\cite[Conjecture~119 on
page~773]{Lueck-Reich(2005)} for $R = \IC$.


\subsection{Vanishing of negative K-theory}
Bernstein's results from~\cite{Bernstein(1992)} which we briefly recalled in
Subsection~\ref{subsec:Resolutions_of_smooth_representations}, also imply for a reductive
$p$-adic group $G$ that $\Kgroup_n(\calh(G,\IC)) = 0$ holds for $n \le -1$.  Under the more general 
assumptions on $G$ and $R$ from Corollary~\ref{cor:K_n-no-fuzz}
we get $\Kgroup_n(\calh(G,R)) = 0$ for $n \le -1$, see
Corollary~\ref{cor:main-for-twisted}~\ref{cor:main-for-twisted:negative}.


\subsection{The $\COP$-Farrell--Jones Conjecture}

To formulate our main result we generalize coefficients. 
For a category $\calb$ with $G$-support we obtain
\begin{equation*}
  \bfK_\calb \colon \Or_\OPEN(G) \to \Spectra, \quad G/U \mapsto \bfK \big(\calb[G/U]\big).
\end{equation*} 
As discussed in Subsection~\ref{subsec:Smooth_G-homology_theories} we can
apply~\cite{Davis-Lueck(1998)} and obtain a smooth $G$-homology theory
$H_n^G(\--;\bfK_\calb)$.   
The projection
$\EGF{G}{\COP} \to G/G$ induces the $\COP$-assembly map 
\begin{equation}\label{eq:COP-assembly-homolgy-theory-cala}
  H_n^G(\EGF{G}{\COP};\bfK_\calb) \to  H_n^G(G/G;\bfK_\calb) = \Kgroup_n (\calb).
\end{equation}
We define \emph{Hecke categories with $G$-support} in Definition~\ref{def:Hecke-category}.
Essentially, these are categories with $G$-support satisfying axioms that are modeled on
$\calb(G;R)$, i.e., on Hecke algebras.  In particular, $\calb[G/U]$ is then equivalent to
the subcategory $\res_G^U\calb$ of $\calb$ on objects and morphisms with support in $U$.

\begin{conjecture}[$\COP$-Farrell--Jones Conjecture]\label{conj:COP-FJ-for-Hecke-Cat}
  Let $G$ be a td-group and let $\calb$ be a Hecke category with $G$-support.
  Assume that $\calb$ satisfies (Reg) from Definition~\ref{def:regularity-for-COP}.
  Then~\eqref{eq:COP-assembly-homolgy-theory-cala} is an isomorphism for all $n$.
\end{conjecture}

The following is our main result.

\begin{theorem}\label{thm:main-COP-FJ-for-reductive} 
  Conjecture~\ref{conj:COP-FJ-for-Hecke-Cat} holds for reductive $p$-adic groups.
\end{theorem}

Theorem~\ref{thm:main-COP-FJ-for-reductive} is a direct consequence of the
$\CVCYC$-Farrell--Jones Conjecture~\ref{conj:Farrell-Jones-Conj-for-td} for reductive
$p$-adic groups from Theorem~\ref{thm:Farrell-Jones-Conjecture-for-reductive-p-adic} and
the Reduction Theorem~\ref{thm:reduction} that reduces the $\COP$-Farrell--Jones
Conjecture~\ref{conj:COP-FJ-for-Hecke-Cat} to the $\CVCYC$-Farrell--Jones
Conjecture~\ref{conj:Farrell-Jones-Conj-for-td}.  The proof of
Theorem~\ref{thm:main-COP-FJ-for-reductive} seems not to simplify if we only consider
$\Kgroup_0$; it uses for example localization sequences that combine all $\Kgroup_n$.

\begin{remark}\label{rem:uses-mostly-CAT0}
  The proof of the $\CVCYC$-Farrell--Jones Conjecture~\ref{conj:Farrell-Jones-Conj-for-td}
  for reductive $p$-adic groups uses only their action on its extended Bruhat-Tits
  building.
	
  Let $M$ be a Coxeter matrix over a finite set $I$.  Let $C$ be a building of type $M$, in the sense of~\cite[\S 3]{Davis(1998buildings)}. 
  Its realization $|C|$ is a $\CAT(0)$-space, see~\cite[Thm.~11.1]{Davis(1998buildings)}.  Let $G$ be
  a td-group with a cofinite smooth proper action
  on $C$.  We obtain an induced cocompact smooth proper isometric action on $|C|$.  It
  seems to be reasonable to expect that the $\CVCYC$-Farrell--Jones
  Conjecture~\ref{conj:Farrell-Jones-Conj-for-td} (and therefore the $\COP$-Farrell--Jones
  Conjecture~\ref{conj:COP-FJ-for-Hecke-Cat}) holds also in this situation.  The only
  input to the proof of the $\CVCYC$-Farrell--Jones
  Conjecture~\ref{conj:Farrell-Jones-Conj-for-td} for reductive p-adic groups that does
  not directly generalize to this situation is Theorem~\ref{thm:X-to-J}.  This result
  relies on a technical assumption for the action of $G$ on the flow space associated to
  $|C|$, this is~\cite[Assumption~2.7]{Bartels-Lueck(2023almost)}.  Under this assumption
  the proof of the $\CVCYC$-Farrell--Jones Conjecture~\ref{conj:Farrell-Jones-Conj-for-td}
  (and therefore also of the $\COP$-Farrell--Jones
  Conjecture~\ref{conj:COP-FJ-for-Hecke-Cat}) for reductive $p$-adic groups generalizes
  directly to $G$.
\end{remark}

  \begin{remark}[Novikov Conjecture]\label{rem:Novikov_Conjecture}
    The  Novikov Conjecture about the homotopy invariance of higher signatures of
    closed oriented manifolds with fundamental group $\Gamma$ is equivalent to the rational injectivity of the $L$-theoretic
    assembly map $H_n(B\Gamma;\bfL_{\IZ}) \to L_n(\IZ \Gamma)$.     
    B\"okstedt-Hsiang-Madsen~\cite{Boekstedt-Hsiang-Madsen(1993)} proved using cyclotomic
    traces that the $K$-theoretic analogue $H_n(B\Gamma;\bfK_{\IZ}) \to K_n(\IZ \Gamma)$ is
    rationally split injective, if $\Gamma$ satisfies some homological finiteness conditions,
    which are automatically satisfied, provided that $B\Gamma$ has a model of finite
    type. 
    Mostad shows in his PhD-thesis~\cite{Mostad(1991)} using the descent method of
    Carlsson-Pedersen~\cite{Carlsson-Pedersen(1995a)} that the assembly map
    $H_n(B\Gamma;\bfK_{R}) \to K_n(R\Gamma)$ is split injective, if $R$ is a ring and $\Gamma$ is a
    torsionfree cocompact discrete subgroup of $\SL_n(\IQ_p)$. 
    Moreover, the descent method of
    Carlsson-Pedersen~\cite{Carlsson-Pedersen(1995a)} has been used to show the split
    injectivity of the assembly map $H_n^\Gamma(\EGF{\Gamma}{\FIN};\bfK_{R}) \to K_n(R\Gamma)$ for a
    large class of groups and any ring $R$, see for instance~\cite{Kasprowski(2015findeccom), Kasprowski(2016linear), Ramras-Tessera-Yu(2014)}, and also~\cite[Section~15.6]{Lueck(2022book)}, whereas the rational injectivity of the assembly
    map $H_n^\Gamma(\EGF{\Gamma}{\FIN};\bfK_{\IZ}) \to K_n(\IZ \Gamma)$ has been studied using cyclotomic
    traces in~\cite{Lueck-Reich-Rognes-Varisco(2017)}.  
    It would be interesting to see whether
  the descent method of Carlsson-Pedersen~\cite{Carlsson-Pedersen(1995a)} leads to proofs
    of the split injectivity of the assembly~\eqref{eq:COP-assembly-homolgy-theory-R} for classes of td-groups.  
\end{remark}

\begin{remark}
   The Baum--Connes Conjecture for reductive $p$-adic groups	 has been proven by Lafforgue~\cite{Lafforgue(1998)}.
   See also Baum--Higson--Plymen~\cite{Baum-Higson-Plymen(1997)} for $p$-adic $\GL_n$.
   This yields an Atiyah-Hirzebruch spectral sequence that computes the topological K-theory of the reduced group C$^*$-algebra, compare Subsection~\ref{subsec:AHSS} below.
   We note that the spectral sequence in this case is not a first quadrant spectral sequence (because the negative topological K-theory does in general not vanish by Bott periodicity).
   Hence for topological K-theory one does not get formulas such as~\eqref{eq:K_0-is-H_0} or Corollary~\ref{cor:K_0-no-fuzz}, where the relevant $\Kgroup_0$-group is expressed in terms of $\Kgroup_0$-group of compact open subgroups.
\end{remark}


\subsection{Inheritance}%
\label{subsec:Inheritance}
An advantage of the generalization from $\calb(G;R)$ to Hecke categories with $G$-support
is the following result, proven in~\cite[Theorem~1.5]{Bartels-Lueck(2023foundations)}.

\begin{theorem}\label{the:inheritance}
  If Conjecture~\ref{conj:COP-FJ-for-Hecke-Cat} holds for a td-group $G$, then it also
  holds for all td-groups $G'$ which are modulo a normal compact subgroup isomorph to a
  closed subgroup of $G$.
\end{theorem}

Thus Conjecture~\ref{conj:COP-FJ-for-Hecke-Cat} holds for groups that are modulo a normal
compact subgroup isomorphic to a closed subgroup of a reductive $p$-adic group.  This
applies in particular to parabolic subgroups that appear for example in parabolic
induction and restriction.  


\subsection{The Atiyah-Hirzebruch spectral sequence}\label{subsec:AHSS}

Given any smooth $G$-homo\-lo\-gy theory there is a (strongly convergent)
\emph{equivariant Atiyah-Hirzebruch spectral sequence}, see~\cite[Thm~4.7 and
Sec.~7]{Davis-Lueck(1998)} and~\cite[Thm~2.1]{Bartels-Lueck(2023recipes)}.  For
$H_n^G(-;\bfK_R)$ it takes the form
\begin{equation}
  E_{p,q}^2 = B\!H^G_{p}(X;H_q^G(G/-;\bfK_R)) \implies H^G_{p+q}(X;\bfK_R).
 \label{equi-AHSS}
\end{equation}
The $E^2$-page is given by Borel homology.  If $R$ is regular and contains $\IQ$, then the
spectral sequence~\eqref{equi-AHSS} is a first quadrant spectral sequence.  In particular,
$H_0^G(X;\bfK_R) = B\!H^G_{0}(X;H_0^G(G/-;\bfK_R))$.  
Thus, if $R$ is uniformly regular and if $G$
satisfies the $\COP$-Farrell--Jones conjecture, then we obtain 
\begin{equation} \label{eq:K_0-is-H_0}
	K_0(\calh(G;R)) = B\!H^G_{0}(\EGF{G}{\COP};H_0^G(G/-;\bfK_R)).
\end{equation}   
This homology group can then be
described as the cokernel of a map between sums of $\Kgroup_0$ of Hecke algebras of
compact open subgroups of $G$.  For a non-Archimedian local field $F$ and
$G = \SL_n(F)$, $\PGL_n(F)$ or $\GL_n(F)$ this is worked out
in~\cite[Section~6]{Bartels-Lueck(2023recipes)}.


\subsection{Central characters and actions on the coefficients}%
\label{subsec:central_characters_and_actions}

Consider an exact sequence of td-groups $1 \to N \to G \xrightarrow{p} Q \to 1$, a unital
ring $R$ with $\IQ \subseteq R$, a locally constant group homomorphism
$\rho \colon Q \to \aut(R)$ to the group of ring automorphisms of $R$ and a so called
\emph{normal character} $\omega \colon N \to \cent(R)^{\times}$, which is a locally
constant group homomorphism to the multiplicative group of units of the center of $R$.  We
assume that that $N$ is locally central\footnote{I.e., the centralizer of $N$ in $G$ is an
open subgroup of $G$.} and that $\omega$ is $G$-conjugation invariant\footnote{I.e.,
$\omega(gng^{-1})=\omega(n)$ for all $g \in G$, $n \in N$.}.  For example, $N$ could be
a closed subgroup of the center of $G$.  We also assume that the $Q$-action on $R$ fixes
the image of $\omega$.  For example, $Q$ could fix the center of $R$.  In this situation
we obtain a Hecke algebra $\calh(G;R,\rho,\omega)$,
see~\cite[Sec.~2.B]{Bartels-Lueck(2022KtheoryHecke)}.  
Its elements are locally constant
functions $s \colon G \to R$ with support compact modulo $N$ satisfying
$s(ng) = \omega(n) \cdot s(g)$ for all $n \in N$ and $g \in G$.  In the special case that
$\rho$ is trivial and $\omega \colon N \to R^{\times}$ is a central character, i.e.,
$N \subseteq G$ is central and $\omega \colon N \to \cent(R)^{\times}$ is a locally
constant homomorphism, this is the usual Hecke algebra of $G$ with coefficients in $R$
associated to the central character $\omega$.

Similar to $\calb({G};{R})$ we obtain a category $\calb(G;R,\rho,\omega)$, see~\cite[Section~6.C.]{Bartels-Lueck(2023foundations)}.
The support of elements of $\calh(G;R,\rho,\omega)$ in $G$ is compact modulo $N$.  Projecting we
obtain compact subsets of $Q$.  In this way $\calb(G;R,\rho,\omega)$ can be viewed as
a category with $Q$-support and we obtain
\begin{equation*}
  \bfK_{R,\rho,\omega} \colon \Or_\OPEN(Q) \to \Spectra, \quad Q/U
  \mapsto \bfK \big(\calb(G;R,\rho,\omega)[Q/U]\big).
\end{equation*}
For $U$ open in $Q$ the homotopy groups of $\bfK_{R,\rho,\omega}(Q/U)$ are the K-groups of
the Hecke algebra associated to $p^{-1}(U)$, i.e., of
$\calh(p^{-1}(U);R,\rho,\omega)$\footnote{Strictly speaking we should write
  $ \calh(p^{-1}(U);R,\rho_{p^{-1}(U)},\omega)$.}.   
We can again
apply~\cite{Davis-Lueck(1998)}, see Subsection~\ref{subsec:Smooth_G-homology_theories},
and obtain a smooth $Q$-homology theory $H_n^Q(\--;\bfK_{R,\rho,\omega})$ with
$H_n^G(Q/U;\bfK_{R,\rho,\omega}) \cong \Kgroup_n (\calh(p^{-1}(U);R,\rho,\omega))$.

The formulation of the $\COP$-Farrell--Jones conjecture can be applied in this situation as
well and we have the following corollary to Theorem~\ref{thm:main-COP-FJ-for-reductive}.

\begin{corollary}\label{cor:main-for-twisted}
  Assume that $Q$ is a modulo a compact subgroup isomorphic to a closed subgroup of a
  reductive $p$-adic group and that $R$ is a uniformly regular ring containing $\IQ$.
  Then
  \begin{enumerate}
  \item\label{cor:main-for-twisted:ass} The assembly map induced by the projection
    $\EGF{Q}{\COP} \to Q/Q$
    \begin{equation*}
      H_n^{Q}(\EGF{Q}{\COP}; \bfK_{R,\rho,\omega})
      \to  H_n^{Q}(Q/Q; \bfK_{R,\rho,\omega}) = K_n(\calh(G;R,\rho,\omega))
    \end{equation*}
    is an isomorphism for all $n$;
  \item\label{cor:main-for-twisted:K_0} The various inclusions $U \subseteq Q$ induce an
    isomorphism
    \begin{equation*}
      \colimunder_{U \in \Sub_\COP(Q)} \Kgroup_0 (\calh(p^{-1}(U);R,\rho, \omega))
      \xrightarrow{\cong}  \Kgroup_0 (\calh(G;R,\rho,\omega));
    \end{equation*}
  \item\label{cor:main-for-twisted:negative} We have $K_{n}(\calh(G;R,\rho,\omega)) = 0$
    for $n \le -1$.
  \end{enumerate}
\end{corollary}
\begin{proof}
  See~\cite[Theorem~1.1]{Bartels-Lueck(2023recipes)}.
\end{proof}


\subsection{Homotopy colimits}\label{subsec:homotopy-colimit-in-intro}
We write $\Or_\COP(G)$ for the full subcategory of $\Or(G)$ on the $G/U$ with $U$ compact
open.  The projections $G/U \to G/G$ for $U$ compact open in $G$ induce a map
\begin{equation}\label{eq:COP-assembly-map-hocolim-cala}
  \hocolimunder_{G/U \in \Or_\COP(G)} \bfK_\calb(G/U) \to \bfK_\calb(G/G) \simeq \bfK(\calb).
\end{equation}
This map can be identified with the map
$\bfH^G(\EGF{G}{\COP};\bfK_\calb) \to \bfH^G(G/G;\bfK_\calb)$, see~\cite[Section~6]{Davis-Lueck(1998)}.  
Applying $\pi_n$ to~\eqref{eq:COP-assembly-map-hocolim-cala}
therefore recovers~\eqref{eq:COP-assembly-homolgy-theory-cala}.

Often the homotopy colimit in~\eqref{eq:COP-assembly-map-hocolim-cala} can be replaced
with a homotopy colimit over a smaller category than $\Or_\COP(G)$.  Let $X$ be a
simplicial complex with a smooth proper cellular simplicial  action of $G$.  \emph{Cellular}
means that, if $g \in G$ sends a simplex to itself, then $g$ fixes the simplex
pointwise.  We also assume that $X$ is a model for $\EGF{G}{\COP}$, i.e., for $U \subseteq G$
compact open $X^U$ is contractible.  For example, if $G$ is a $p$-adic group, then we can
take for $X$ (a subdivision of) the associated extended
Bruhat-Tits building.

Let $C$ be a collection of simplices of $X$ that contains at least one simplex from each
orbit of the action of $G$ on the set of simplices of $X$.  Define a category $\calc(C)$
as follows.  Its objects are the simplices from $C$.    
A morphism
$gG_{\sigma} \colon \sigma \to \tau$ is an element $gG_{\sigma} \in G/G_{\sigma}$
satisfying $g\sigma \subseteq \tau$, where we view a simplex of $X$ as subspace of $X$ in
the obvious way.  The composite of $gG_{\sigma} \colon \sigma \to \tau$ with
$hG_{\tau} \colon \tau \to \rho$ is $hgG_{\sigma} \colon \sigma \to \rho$. Define a
functor $\iota_C \colon \calc(C)^{\op} \to \Or_\COP(G)$ by sending an object
$\sigma$ to $G/G_\sigma$ and a morphism $gG_{\sigma} \colon \sigma \to \tau$ to
$G/G_{\tau} \to G/G_{\sigma}, \; g'G_{\tau}\mapsto g'gG_{\sigma}$. 
\begin{lemma}\label{lem:iota_C-is-cofinal}
  The functor $\iota_C \colon \calc(C) \to \Or_\COP(G)$ is cofinal.
\end{lemma}

\begin{proof}
  For $C \subseteq C'$ it is not difficult to check that the inclusion
  $\calc(C) \to \calc(C')$ is an equivalence.  Thus we can assume that $C$ contains
  exactly one simplex from each orbit of the $G$-action.  For $G/U \in \Or_{\COP}(G)$ the
  category $G/U \downarrow \calc$ can then be identified with the poset of simplices in
  $X^U$.  By assumption $X^U$ is contractible.
\end{proof}

Lemma~\ref{lem:iota_C-is-cofinal} in combination with the cofinality
Lemma~\ref{lem:cofinality} for homotopy colimits imply that the canonical map
\begin{equation}\label{eq:hocolim-calc-not-Or}
  \hocolimunder_{c \in \calc(C)} \bfK_\calb(G/G_c) \xrightarrow{\sim} 
  \hocolimunder_{G/U \in \Or_\COP(G)} \bfK_\calb(G/U) 
\end{equation}
is an equivalence.

If $X$ admits a strict fundamental domain $X_0$, i.e., a subcomplex $X_0$ that contains
exactly one simplex from each orbit for the $G$-action on the set of simplices of $X$,
then we can take for $C$ the simplices from $X_0$.  In this case $\calc(C)$ can be
identified with the poset (viewed as a category) $\simp(X_0)$ of simplices of $X_0$.  
If $\calb$ is a Hecke category with $G$-support, then the
inclusions $\res_G^{G_\sigma}\calb \to \calb[G/G_\sigma]$ are
equivalences and induce an equivalence between $\bfK_\calb \circ \iota_C$ and
\begin{equation*}
  \simp(X_0)  \to  \Spectra, \qquad
  \sigma  \mapsto  \bfK \big(\res_G^{G_\sigma}\calb\big).
\end{equation*}
Thus, in this situation,~\eqref{eq:hocolim-calc-not-Or} can be simplified further
and~\eqref{eq:COP-assembly-map-hocolim-cala} can be identified with the canonical map
\begin{equation}\label{eq:COP-assembly-map-hocolim-cala-Delta}
  \hocolimunder_{\sigma \in \simp(X_0)} \bfK \big(\res_G^{G_\sigma}\calb\big)
  \to \bfK (\calb).
\end{equation}
By Theorem~\ref{thm:main-COP-FJ-for-reductive}, this map is an equivalence, if $G$ is a
reductive $p$-adic group and $\calb$ satisfies (Reg) from Definition~\ref{def:regularity-for-COP}.
In particular, for $R$ uniformly regular with $\IQ \subseteq R$, the canonical map
\begin{equation}\label{eq:COP-assembly-map-hocolim-hecke-alg-Delta}
  \hocolimunder_{\sigma \in \simp(X_0)} \bfK \big(\calh(G_\sigma;R)\big) \to \bfK \big(\calh(G;R)\big)
\end{equation}
is an equivalence, see Corollary~\ref{cor:K_n-no-fuzz}.
 
\begin{example}[$\SL_n(F)$]\label{ex:SL}
  Let $R$ be  a uniformly regular ring containing $\IQ$.
  Any chamber of the Bruhat-Tits building for $\SL_n(F)$ is a strict fundamental domain
  and we obtain a homotopy pushout diagram
  from~\eqref{eq:COP-assembly-map-hocolim-hecke-alg-Delta}.  We will illustrate this for
  $n=2,3$.  Let $v \colon F \to \IZ \cup \{\infty\}$ the valuation of $F$.  Let
  $\calo = \{ v \geq 0 \}$ be the ring of integers in $F$.  Choose $\mu \in \calo$ with
  $v(\mu) = 1$.  Put
  \begin{equation*}
    h := \left(
      \begin{array}{cccc}  & 1 & &  
        \\
                           &  & \ddots  \\ & & & 1 \\ \mu
      \end{array} \right) \in \GL_n(F).
  \end{equation*}
  For $n=2$ the homotopy pushout diagram is
  \begin{equation*}
    \xymatrix{\bfK(\calh(I;R)) \ar[r] \ar[d] & \bfK(\calh(U_{1};R)) \ar[d]
      \\ 
      \bfK(\calh(U_{0};R)) \ar[r] & \bfK(\calh(\SL_2(F);R)).}
  \end{equation*}  
  Here $U_0 = \SL_2(\calo)$, $U_1 = hU_0h^{-1}$, and $I = U_0 \cap U_1$ is the Iwahori
  subgroup.
  For the K-groups this yields a Mayer-Vietoris sequence, infinite to the left,
  \begin{multline*}
    \cdots \to K_n(\calh(I;R))  \to K_n(\calh(U_{1};R)) \oplus  K_n(\calh(U_{0};R))  \to
    K_n(\calh(\SL_2(F);R))
    \\
    \to  K_{n-1}(\calh(I;R))  \to K_{n-1}(\calh(U_{1};R)) \oplus  K_{n-1}(\calh(U_{0};R))  \to \cdots
    \\
    \cdots \to K_0(\calh(I;R))  \to K_0(\calh(U_{1};R)) \oplus  K_0(\calh(U_{0};R))  \to
    K_0(\calh(\SL_2(F);R)) \to 0,
      \end{multline*}
  and  $K_n(\calh(\SL_2(F);R)) = 0$ for $n \le -1$.
  
  For $n=3$ we obtain the homotopy pushout diagram
  \begin{equation*}
    \xymatrix@1@=8pt{&  \bfK(\calh(U_{12};R)) \ar[rr] \ar[dd] |!{[d]}\hole
      & & \bfK(\calh(U_{2};R)) \ar[dd]
      \\
      \bfK(\calh(I;R))  \ar[ur]\ar[rr]\ar[dd]
      & & \bfK(\calh(U_{02};R)) \ar[ur]\ar[dd]
      \\
      & \bfK(\calh(U_{1};R)) \ar[rr] |!{[r]}\hole
      & & \bfK(\calh(\SL_3(F);R)) \\
      \bfK(\calh(U_{01};R)) \ar[rr]\ar[ur]
      & & \bfK(\calh(U_{0};R)) \ar[ur]
    }
  \end{equation*} 
  where $U_0 = \SL_2(\calo)$, $U_1 = h\SL_2(\calo)h^{-1}$, $U_2 = h^2\SL_2(\calo)h^{-2}$,
  $U_{ij} = U_i \cap U_j$ and $I = U_0 \cap U_1 \cap U_2$ is the Iwahori subgroup.
  
  In general, for $\SL_n(F)$ we obtain a homotopy pushout diagram whose shape is an
  n-cube.
\end{example}
  

\subsection{Comparison with the discrete case}\label{subsec:comparison-discrete}

Let $\Gamma$ be a discrete group.  Let $\FIN$ be the collection of finite groups of
$\Gamma$ and $\VCYC$ be the collection of the virtually cyclic subgroups of
$\Gamma$\footnote{Alternatively one can work with the family of subgroups that are finite
  or admit a surjection onto $\IZ$ with finite kernel.}.  We write $\Or_\FIN(\Gamma)$ and
$\Or_\VCYC(\Gamma)$ for the corresponding subcategories of the orbit category
$\Or(\Gamma)$ of $\Gamma$.  For a ring $R$ there is a functor $\Or(\Gamma) \to \Spectra$
whose value on $\Gamma/H$ is equivalent to the K-theory spectrum of $R[H]$,
see~\cite[Sec.~2]{Davis-Lueck(1998)}.  To distinguish it from~\eqref{eq:bfK_R} we will
denote it here as $\bfK_R^{\text{dis}}$.  We obtain a commutative diagram
\begin{equation}\label{eq:discret-assembly-maps}
  \xymatrix@C=-5ex{\displaystyle \hocolimunder_{\Gamma/F \in \Or_\FIN(\Gamma)}
    \bfK^{\text{dis}}_R(\Gamma/F) \ar[rr]^{\alpha^{\text{dis}}_{\FIN}} \ar[dr]_{\alpha^{\text{dis}}_{\text{rel}}} 
    & & \bfK^{\text{dis}}_R(\Gamma/\Gamma) \simeq \bfK ( R[\Gamma] ) \\
    & \displaystyle \hocolimunder_{\Gamma/V \in \Or_\VCYC(\Gamma)} \bfK^{\text{dis}}_R(\Gamma/V)
    \,.\ar[ur]_{\alpha^{\text{dis}}_{\VCYC}}	  
  }  
\end{equation} 
The map $\alpha^{\text{dis}}_\FIN$ is the analog
of~\eqref{eq:COP-assembly-homolgy-theory-cala} in the
formulation~\eqref{eq:COP-assembly-map-hocolim-cala}.  The (K-theoretic) Farrell--Jones
Conjecture for discrete groups asserts that $\alpha^{\text{dis}}_{\VCYC}$ is an
equivalence.  If $R$ is regular and contains $\IQ$, then the relative assembly map
$\alpha^{\text{dis}}_{\text{rel}}$ is an equivalence (for all $\Gamma$),
see~\cite[Prop.~2.14]{Lueck-Reich(2005)}.  Therefore, for such $R$,
$\alpha^{\text{dis}}_{\FIN}$ is an equivalence if and only if
$\alpha^{\text{dis}}_{\VCYC}$ is an equivalence.  The Farrell--Jones Conjecture holds for a
large class of groups, including all groups $\Gamma$ that admit a cocompact isometric
action on a finite-dimensional $\CAT(0)$-space $X$,
see~\cite[Thm.~B]{Bartels-Lueck(2012annals)} and~\cite[Thm.~1.1]{Wegner(2012)}.  The proof
uses geodesic flows as pioneered by Farrell and Jones, e.g.~\cite{Farrell-Jones(1986a),
  Farrell-Jones(1993a)}.  Here virtually cyclic subgroups appear as follows.  For a
bi-infinite geodesic $c \colon \IR \to X$ let $V_c$ be the subgroup of $\Gamma$ consisting
of all $g \in \Gamma$ for which there is $t_g \in \IR$ such that $gc(t) = c(t+t_g)$ for
all $t$.  This is a virtually cyclic subgroup.  More precisely, the homomorphism
$V_c \to \IR$, $g \mapsto t_g$ has discrete and therefore infinite cyclic or trivial
image.  The kernel of this homomorphism is finite as the action of $\Gamma$ on $X$ is
proper.  Thus $V_c$ is either finite or admits a surjection onto $\IZ$ with finite kernel.

To prove Theorem~\ref{thm:main-COP-FJ-for-reductive} we will use the action of a reductive
$p$-adic group $G$ on its associated extended Bruhat-Tits building $X$.  The building $X$ is a
$\CAT(0)$-space and our general strategy is to apply the geodesic flow method and argue
along a variation of the diagram~\eqref{eq:discret-assembly-maps}.
For bi-infinite geodesics $c \colon \colon \IR \to X$ we obtain the subgroups $V_c$ of $G$
as above.  The $V_c$ are now either compact (as pointwise stabilizers of bi-infinite
geodesics) or admit a surjection onto $\IZ$ with compact kernel.

\begin{definition}
  For a td-group $G$ we write $\CVCYC$ for the family of all closed subgroups $V$ that are
  either compact or admit a surjection onto the infinite cyclic group with compact
  kernel\footnote{As $\IZ$ is discrete the kernel is automatically open in $V$.}.
  We write $\Or_{\CVCYC}(G)$ for the full subcategory of $\Or(G)$ on all $G/V$ with $V \in \CVCYC$.
\end{definition}

Our general strategy will be to replace $\Gamma$ with $G$, $\FIN$ with $\COP$, and $\VCYC$
with $\CVCYC$ in~\eqref{eq:discret-assembly-maps}.  
However, two problems arise because
the $V \in \CVCYC$ are in general not open in $G$.  
The first problem is that, if $V$
not open in $G$, then $\calh(V;R)$ is not a subalgebra of $\calh(G;R)$; extending by zero
does not produce locally compact functions on $G$ from locally constant functions on $V$.
Thus there is no induction map from $\bfK(\calh(V;R))$ to $\bfK(\calh(G;R))$ and it is not
clear how $\bfK_{R}$ can be extended from $\Or_{\COP}(G)$ to $\Or_{\CVCYC}(G)$.  The
second problem is less clear at this point, but it comes from the fact that (unlike the
discrete case) a product of orbits $G/V \times G/V'$ cannot be written as a coproducts of
orbits\footnote{The precise place where this comes up is
  Theorem~\ref{thm:FS-to-J}. Locally there are maps on the flow space of the form
  $U \to G/V$, but if one patches them together over the flow space one ends up with maps
  to products of orbits.}.  For this reason orbits are not necessarily the correct
building blocks for topological groups and we will work with a category of formal products
of orbits instead of the orbit category.  As an added bonus this will allow us to
disregard many morphisms in $\Or(G)$ and we arrive at the category $\EP\All(G)$, see
Subsection~\ref{subsec:product-cats}.  A technical point is that $G/G$ is no longer
terminal in $\EP\All(G)$, but this can be remedied by allowing the empty product $\ast$
and we obtain the category $\EPplus\All(G)$ as our replacement for $\Or(G)$.  For a
\emph{family}\footnote{A family of closed subgroups is always assumed to be closed under
  conjugation and taking finite intersections.}
$\calf$ of closed subgroups of $G$ we obtain the subcategory $\EP\calf(G)$ of
$\EPplus\All(G)$ on all (non-empty) formal products of the form
$G/F_1 \times \dots \times G/F_n$ with $F_i \in \calf$. 
 We will construct a functor
$\EPplus\All(G) \to \Spectra$, $P \mapsto \bfK (\contc_G(P))$ in
Subsection~\ref{subsec:category-contc}\footnote{There is a functor for each
  category with $G$-support $\calb$.  To simplify the discussion here we tacitly assume
  $\calb = \calb(G;R)$ for a ring $R$ containing $\IQ$.}.  On orbits $G/U$ with $U$ open
in $G$, the K-theory of $\contc_G(P)$ will (up to a degree shift) be the K-theory of
$\calh(G;R)$, compare Proposition~\ref{prop:c-is-b-for-smooth-P}.  Our replacement
for~\eqref{eq:discret-assembly-maps} is then
\begin{equation}\label{eq:discret-assembly-maps-now-td}
  \xymatrix@C=-5ex{\displaystyle \hocolimunder_{P \in \EP\COP(G)}
    \bfK (\contc_G(P))\ar[rr]^{\alpha_{\COP}} \ar[dr]_{\alpha_{\text{rel}}} 
    & & \bfK (\contc_G(\ast)) \simeq \Sigma \bfK ( \calh(G,R) )
    \\
    & \displaystyle \hocolimunder_{P \in \EP\CVCYC(G)} \bfK (\contc_G(P)) \,.\ar[ur]_{\alpha_{\CVCYC}}	  
  }  
\end{equation} 
It is not difficult to identify $\alpha_{\COP}$
with~\eqref{eq:COP-assembly-map-hocolim-cala}, compare
Proposition~\ref{prop:cop-iso-via-PSub}.  
Thus the task is to show that $\alpha_{\CVCYC}$
and $\alpha_{\text{rel}}$ are equivalences.  For $\alpha_{\CVCYC}$ this means that $G$
satisfies the $\CVCYC$-Farrell--Jones Conjecture~\ref{conj:Farrell-Jones-Conj-for-td}, see
Theorem~\ref{thm:Farrell-Jones-Conjecture-for-reductive-p-adic}.  
As in the discrete case
Theorem~\ref{thm:Farrell-Jones-Conjecture-for-reductive-p-adic} does not depend on any
assumptions on the coefficients (here $R$).  For $\alpha_{\text{rel}}$ it is the content
of the Reduction Theorem~\ref{thm:reduction}; this depends on a regularity assumptions for
the coefficients (here $R$).

The functor $\EPplus\All(G) \to \Spectra$, $P \mapsto \bfK (\contc_G(P))$ is not
determined by its restriction to $\EPplus\OPEN(G)$.  There are many variations of the
category $\contc_G(P)$ such that the K-theory is unchanged for $P \in \EPplus\OPEN(G)$.
The specific choices from Subsection~\ref{subsec:category-contc} may seem overly
complicated at first, but are made in order for the proof of
Theorem~\ref{thm:Farrell-Jones-Conjecture-for-reductive-p-adic} to work.  For example, the
foliated distance from Subsections~\ref{subsec:foliated-distance-V}
and~\ref{subsec:foliated-distance-P} is modeled on foliated distance on flow spaces, see
Subsection~\ref{subsec:V-fol-and-FS}.  This in turn makes the proof of the reduction
theorem more complicated than its discrete counterpart.  In fact, we only know that
$\alpha_\text{rel}$ is an equivalence under the assumption that $\alpha_{\CVCYC}$ is an
equivalence.  The reason is that we are not able to prove the reduction theorem directly
for our functor $P \mapsto \bfK (\contc_G(P))$, but only for a variation
$P \mapsto \bfK (\contC_G(P))$ thereof, see Subsection~\ref{subsec:contC}.  There is a map
$\bfK (\contc_G(P)) \to \bfK (\contC_G(P))$ and, under the assumption that
$\alpha_{\CVCYC}$ is an equivalence, a simple diagram chase proves then the reduction
theorem for $\bfK (\contc_G(P))$.

While the construction of $\contc_G(P)$ is in many ways the key ingredient to the proof of
Theorem~\ref{thm:main-COP-FJ-for-reductive}, a better understanding of it would still be
desirable.  The value of the functor $\EPplus\All(G) \to \Spectra$,
$P \mapsto \bfK (\contc_G(P))$ does not only depend on the groups $H_i$ occurring in
$P = (G/H_1, \ldots , G/H_n)$ but also on how $H_i$ sits in $G$ unless each $H_i$ is
open. This is illustrated in Remark~\ref{rem:value_on_compact_subgroups}. For discrete
groups the Farrell--Jones Conjecture (with appropriate coefficients) passes to subgroups.
Similarly, the $\COP$-Farrell-Jones Conjecture~\ref{conj:COP-FJ-for-Hecke-Cat} passes to
closed subgroups.  It is natural to expect the same for the $\CVCYC$-Farrell-Jones
Conjecture~\ref{conj:Farrell-Jones-Conj-for-td}, but this remains open and seems to
require a better understanding of $\contc_G(P)$.  This would imply the (K-theoretic)
Farrell--Jones Conjecture for all discrete subgroups of reductive $p$-adic groups, because
the $\CVCYC$-Farrell--Jones Conjecture~\ref{conj:Farrell-Jones-Conj-for-td} reduces for
discrete groups to the original K-theoretic Farrell--Jones Conjecture, see
Remark~\ref{rem:CVCYC-FJ-implies-discrete-FJ}.


\subsection{Open problems}\label{subsec:Open_Problems}
There is an interesting instance where the $\CVCYC$-Farrell--Jones
Conjecture~\ref{conj:Farrell-Jones-Conj-for-td} applies, but we do not know to what extend
the $\COP$-Farrell--Jones Conjecture~\ref{conj:COP-FJ-for-Hecke-Cat} applies.  This arises
as follows.  Let $G$ be a reductive $p$-adic group and $R$ be a ring containing $1/p$.  In
this case $G$ admits a pro-p-group $U_0$ as a compact open subgroup $U_0$.  There is still
a Hecke algebra $\calh(G;R)$, see for example~\cite[Sec.I.3]{Vigneras(1996)}, which can be
used to define a variant $\calb'(G;R)$ of $\calb(G;R)$, see
Example~\ref{ex:calbprime(G;R)}.  Now
Theorem~\ref{thm:Farrell-Jones-Conjecture-for-reductive-p-adic} applies and we obtain a
homotopy colimit description of the K-theory of $\calh(G;R)$
from~\eqref{eq:FJ-assembly-map}.  In general\footnote{If $[U:U_0]$ is invertible for all
  compact open subgroups $U$ of $G$ containing the pro-p-group $U_0$, then we expect that
  $\calb'(G;R)$ satisfies (Reg) from Definition~\ref{def:regularity-for-COP}.  Thus, under
  this additional assumption the Reduction Theorem~\ref{thm:reduction} should apply and
  lead for example to a version of Corollary~\ref{cor:K_n-no-fuzz} for the K-theory of
  $\calh(G;R)$.}, for example if $R=\IZ[1/p]$, we do not expect that $\calb'(G;R)$
satisfies (Reg) from Definition~\ref{def:regularity-for-COP}.  In particular, we do not expect that~\eqref{eq:COP-assembly-homolgy-theory-R} is an isomorphism in this
situation.  Nevertheless, it seems interesting to evaluate this homotopy colimit
from~\eqref{eq:FJ-assembly-map} in this situation further.
For example, it is conceivable that~\eqref{eq:COP-assembly-homolgy-theory-R} is an isomorphism modulo $p$-torsion. 

For a reductive p-adic group $G$ Bernstein decomposed the category of finitely generated
non-degenerated $\calh(G;\IC)$-modules as a direct sum of subcategories, now called
Bernstein blocks, see~\cite{Bernstein(1984),Bernstein(1992)}. In particular,
there is a corresponding direct sum decomposition of $K_n(\calh(G;\IC))$.  By
Corollary~\ref{cor:K_n-no-fuzz} there must then exist a corresponding decomposition of
$H_n^G(\EGF{G}{\COP};\bfK_\IC)$.  It seems interesting to give a direct description of the
summands in $H_n^G(\EGF{G}{\COP};\bfK_\IC)$ corresponding to the Bernstein blocks in
$K_n(\calh(G;\IC))$.  Let $G = \GL_n(F)$ and let $I$ be the Iwahori subgroup.  The
Iwahori-Hecke algebra $\calh(G,I)$ is the (unital) subalgebra of $I$-bi-invariant
functions of $\calh(G;\IC)$.  The Iwahori block in the category of finitely generated
non-degenerated $\calh(G;\IC)$-modules can be identified with the category of finitely
generated $\calh(G,I)$-modules.  Even for this block it is not quite clear what the
correct analog of the assembly map~\eqref{eq:COP-assembly-homolgy-theory-R} should be.


\subsection{Overview}
Section~\ref{sec:prelim} fixes some conventions and notations.

Section~\ref{subsec:COP-FJ-Conj} contains the details of the formulation of the
$\COP$-Farrell--Jones Conjecture and a reformulation using products of orbits as
building blocks.

Controlled algebra is a key tool for proofs of the Farrell--Jones Conjecture for discrete
groups  sets up a variant of this theory also
suitable for td-groups.  In the usual theory controlled objects over a space $X$ have as
support a subset of $X$, while morphisms have as support subsets of $X \times X$.  In our
version of the theory objects also have a support in $X \times X$\footnote{If one thinks
  of objects as idempotents this is quite natural.}.    One can think of the controlled categories, that we  introduce in Section~\ref{sec:controlled-algebra}, as
  generalizations of Hecke algebras. Thus it is quite natural that these categories also
come with a notion of support in $G$.

Section~\ref{sec:Farrell-Jones-Conj-for-td} contains the formulation of the
$\CVCYC$-Farrell--Jones Conjecture.  Central is the construction of the categories
$\contc_G(P)$ already discussed in Subsection~\ref{subsec:comparison-discrete}.  Their
construction uses the language of controlled algebra.  (This in contrast to the discrete
case, where controlled algebra only enters proofs of the Farrell--Jones Conjecture but not
its formulation.)

Section~\ref{sec:formal-framwork} contains the formal framework of the proof of the
$\CVCYC$-Farrell--Jones Conjecture of reductive $p$-adic groups.  This framework is formally
different from the one used for example in~\cite{Bartels-Lueck(2012annals)} for discrete
groups, but also centers around a $G$-homology theory (here $\bfD_G$) and a transfer map,
see Theorem~\ref{thm:transfer-bfD0}.  For technical reasons we also introduce a variant
$\bfD^0_G$ of $\bfD_G$.  The domain of the $G$-homology theory $\bfD_G$ is a category
$\regularPOrGSC$ of combinatorial $G$-space, whose building blocks are products of orbits
and simplices.  This category contains an analog $\bfJ_\CVCYC(G)$ of the numerable
classifying space for $\CVCYC$.  The transfer map realizes the functor
$\EPplus \All(G) \to \Spectra$, $P \mapsto \bfK(\contc_G(P))$ as a retract of
$P \mapsto \bfD_G(\bfJ_\CVCYC(G) \times P)$\footnote{Really, the transfer uses the close
  relative $\bfD^0_G$ of $\bfD_G$, but this is a technical point.}.  This allows then the
application of excision and homotopy invariance results of $\bfD_G$ in the variable
$\bfJ_\CVCYC(G)$.

Section~\ref{sec:cats-D_and_D0} contains the construction of the functors $\bfD_G$ and
$\bfD^0_G$ as the K-theory of certain categories.  These categories are constructed
using controlled algebra.  Their construction builds on that of $\contc_G(P)$ by adding a
second control direction for what is called an $\epsilon$-control condition.  The precise
formulation is tailored in order for $\bfD_G$ and $\bfD^0_G$ to satisfy the properties
formulated in Section~\ref{sec:formal-framwork}.  With the exception of the transfer these
properties are then verified in Section~\ref{sec:properties-cats-D_upper_D0}, following
similar results in the discrete case.
   
The construction of the transfer is outlined in Section~\ref{sec:outline-transfer} and
carried out in
Sections~\ref{sec:cals},~\ref{sec:tensor-products},~\ref{sec:support-estimates}
and~\ref{sec:constr-transfer}.  This depends on the construction of certain \emph{almost
  equivariant} maps from the building $X$ to a space $|\bfJ_\CVCYC(G)|^\wedge$ associated
to $\bfJ_\CVCYC(G)$.  This is the point where the dynamics of the geodesic flow on a flow
space associated to $X$ is exploited.  The details of this construction is outsourced
to~\cite{Bartels-Lueck(2023almost)}, but we give an overview in Appendix~\ref{app:X-to-J}.

Section~\ref{sec:reduction} contains the proof of the reduction theorem.  The difficulty
here is that it is not clear that the regularity of the coefficients induces a regularity
property for $\contc_G(P)$.  Roughly, the controlled algebra nature of $\contc_G(P)$ makes
it too big to satisfy a regularity property.  In a number of steps we reduce the
problem to certain categories associated to infinite product categories (the limit
category from Subsection~\ref{subsec:sequence-cat}) and use a K-theory computation
from~\cite{Bartels-Lueck(2020additive)}.

Appendix~\ref{app:homotopy-colimits} reviews some results on homotopy colimits that are
used throughout the paper.  Appendix~\ref{app:K-theory-dg-cat} reviews K-theory for
dg-categories.  This formalism is applied in Appendix~\ref{app:homotopy-coherent} to
homotopy coherent functors and ultimately used in the construction of the transfer in
Section~\ref{sec:tensor-products}.


\subsection*{Acknowledgements}  
We thank Tyrone Crisp, Jessica Fintzen, Eugen Hellmann, Nigel Higson,
Linus Kramer, Achim Krause, Thomas Nikolaus, Peter Schneider, Peter
Scholze, and Stefan Witzel for many helpful comments and discussions.

This paper is funded by the ERC Advanced Grant ``KL2MG-interactions'' (no.  662400) of the
second author granted by the European Research Council, by the Deutsche
Forschungsgemeinschaft (DFG, German Research Foundation) \-– Project-ID 427320536 \-- SFB
1442, as well as under Germany's Excellence Strategy \-- GZ 2047/1, Projekt-ID 390685813,
Hausdorff Center for Mathematics at Bonn, and EXC 2044 \-- 390685587, Mathematics
M\"unster: Dynamics \-- Geometry \-- Structure.

\bigskip

The paper is organized as follows:\\[2mm]

\tableofcontents


\section{Preliminaries}\label{sec:prelim}


\subsection{Convention on units}\label{subsec:conv-units}
Categories and rings will always be assumed to be unital, unless we explicitly allow
non-unital categories or rings.  Of course, Hecke algebras are typically not unital.


\subsection{Formally adding finite sums}\label{subsec:adding-sums}
For a $\IZ$-linear category $\cala$ we obtain an additive category $\cala_\oplus$ by
formally adding finite sums.  A concrete model for $\cala_\oplus$ has as objects finite
sequences $(A_1,\dots,A_n)$ of objects in $\cala$ and as morphisms
$\varphi \colon (A_1,\dots,A_n) \to (A'_1,\dots,A'_{n'})$ matrices
$\varphi = (\varphi_{i}^{i'} \colon A_i \to A'_{i'})_{i,i'}$ of morphisms in $\cala$, see
for example~\cite[p.214]{Davis-Lueck(1998)}.


\subsection{Idempotent completion}\label{subsec:idem}
The idempotent completion $\Idem \cala$ of a category $\cala$ has as objects pairs $(A,p)$,
where $p$ is an idempotent on $A$ in $\cala$.  Morphisms
$\varphi \colon (A,p) \to (A',p')$ are morphisms in $\cala$ satisfying
$\varphi = p' \circ \varphi \circ p$.  For a category without units the idempotent
completion makes still sense and produces a category with units: $\id_{(A,p)} = p$.


\subsection{K-theory}\label{subsec:K-theory}
 A construction of the \emph{non-connective K-theory spectrum} $\bfKinfty(\cala)$ of a
 unital additive category $\cala$ can be found for instance
 in~\cite{Lueck-Steimle(2014delooping)} or~\cite{Pedersen-Weibel(1985)}.  The K-theory of
 a $\IZ$-linear category $\cala$ is defined as the K-theory of the additive category
 $\cala_\oplus$.  The canonical embedding $\cala \to \Idem \cala$ induces an equivalence
 in K-theory, see for instance~\cite[Lemma~3.3~(ii)]{Bartels-Lueck(2020additive)}.

 A key tool for us will be a fiber sequence in K-theory that goes back to Karoubi~\cite{Karoubi(1970)} and
 Carlsson-Pedersen~\cite{Carlsson-Pedersen(1995a)}.  
 To state it we need a definition.

 \begin{definition}\label{def:U-filtered} Let $\calu$ be a full additive subcategory of
   an additive category $\cala$.
   \begin{enumerate}[
                 label=(\thetheorem\alph*),
                 align=parleft, 
                 leftmargin=*,
                 itemsep=1pt
                 ] 
   \item\label{def:U-filtered:quotient} The \emph{quotient category} $\cala/\calu$ has
     the same objects as $\cala$.  Morphisms in $\cala / \calu$ are equivalence classes of
     morphisms in $\cala$, where morphisms from $\cala$ are identified in $\cala/\calu$
     whenever their difference factors over an object from $\calu$.
   \item\label{def:U-filtered:filtered} We say that $\cala$ is \emph{$\calu$-filtered} if
     the following condition is satisfied.  Let $A \in \cala$, $U_-, U_+ \in \calu$ and
     let $U_- \xrightarrow{\varphi_-} A \xrightarrow{\varphi_+} U_+$ be morphisms in
     $\cala$.  We require that there is a direct summand $U$ of $A$ with $U \in \calu$
     such that $\varphi_-$ and $\varphi_+$ factor over $U$, i.e., if we write
     $p \colon A \to A$ for the projection associated to the direct summand $U$, then
     $\varphi_- = p \circ \varphi_-$ and $\varphi_+ = \varphi_+ \circ p$.
   \end{enumerate}
 \end{definition}

 Definition~\ref{def:U-filtered} is originally due to Karoubi~\cite{Karoubi(1970)}.
 In~\ref{def:U-filtered:filtered} we used Kasprowski's reformulation~\cite[Def.~5.4,
 Rem.~5.7~(1)]{Kasprowski(2015findeccom)}.

 \begin{theorem}[Karoubi sequence]\label{thm:Karoubi-filtration} Let $\calu$ be a Karoubi
   filtration of $\cala$.  Write $i \colon \calu \to \cala$ and
   $p \colon \cala \to \cala / \calu$ for the associated inclusion and projection.  Then
   \begin{equation*}
     \bfK (\calu) \xrightarrow{i_*} \bfK (\cala) \xrightarrow{p_*} \bfK (\cala / \calu)
   \end{equation*}
   is a fibration sequence of spectra\footnote{The precise statement is as follows. The
     composition $p_* \circ i_0$ has a canonical null homotopy as all objects in $\calu$
     are isomorphic to $0$ in $\cala / \calu$. The induced map from $\bfK (\calu)$ to the
     homotopy fiber of $p_*$ is an equivalence. }.
 \end{theorem}

\begin{proof}
  This is~\cite[Thm.~1.28]{Carlsson-Pedersen(1995a)}.
\end{proof}


\subsection{Small compact open subgroups}

\begin{lemma}\label{lem:small-compact-open}
   Let $G$ be a td-group, $K$ be a compact subgroup of $G$, and $W$ be an open neighborhood of $K$ in $G$.
   Then there exists a compact open subgroup $U$ with $K \subseteq U \subseteq W$.	
\end{lemma}

\begin{proof}
	For each $k$ there is an open subgroup $V_k$ with $kV_k \subseteq W$.
	As $K$ is compact there is $S \subset K$ finite with $K \subseteq \bigcup_{s \in S} sV_s$.
	Then $V := \bigcap_{s \in S} V_s$ is a compact open subgroup for which $kV \subseteq W$ for all $k \in K$.
	As $K \cap V$ has finite index in $K$, $N := \bigcap_{g \in K} gVg^{-1}$ is still compact open.
	Now $K$ normalizes $N$ and so $U := KN$ is a compact open subgroup containing $K$.
	Also $U = KN \subseteq KV \subseteq W$. 
\end{proof}


\section{The $\COP$-Farrell--Jones Conjecture}\label{subsec:COP-FJ-Conj}


\subsection{Categories with $G$-support}

\begin{definition}\label{def:category_with_G-support} Let $G$ be a td-group.  A \emph{category with
  $G$-support} is a $\IZ$-linear category $\calb$ together with a map that assigns to every
  morphism $\varphi$ in $\calb$ a compact subset $\supp \varphi$ of $G$.  
  We require the following
  \begin{enumerate}[
                 label=(\thetheorem\alph*),
                 align=parleft, 
                 leftmargin=*,
                 ] 
  \item\label{def:category_with_G-support:(emptyset)}
    $\supp \varphi = \emptyset \Longleftrightarrow \varphi =
    0$;
  \item\label{def:category_with_G-support:(v_circ_u)}
    $\supp(\varphi' \circ \varphi) \subseteq \suppG \varphi' \cdot \suppG \varphi$;
  \item\label{def:category_with_G-support:(plus)}
    $\supp(\varphi+\varphi') \subseteq \supp \varphi \cup \supp \varphi'$,
    $\supp(-\varphi) = \supp \varphi$.
  \end{enumerate}
  We abbreviate $\supp B := \supp \id_B$. 
\end{definition}

\begin{definition}\label{def:Hecke-category}
	A \emph{Hecke category with $G$-support} is a category $\calb$ with $G$-support such that the following holds.			
	\begin{enumerate}[
                 label=(\thetheorem\alph*),
                 align=parleft, 
                 leftmargin=*,
                 itemsep=2pt
                 ] 
               \item\label{def:Hecke-category:subgr} \emph{Subgroups}\\
                 $\supp B$ is a
                 compact subgroup of $G$ for all objects $B$. For morphisms
                 $\varphi \colon B \to B'$ we have
                 $\supp \varphi = \supp B' \cdot \supp \varphi \cdot \supp B'$. Moreover,
                 the sets $\supp B' \backslash \supp \varphi$ and
                 $\supp \varphi / \supp B$ are both finite;
               \item\label{def:Hecke-category:trans} \emph{Translations}\\ 
                 For every $B \in \calb$ and $g \in G$ there is an isomorphism
                 $\varphi \colon B \xrightarrow{\cong} B'$ satisfying  $\supp B' = g \supp B g^{-1}$,
                 $\supp \varphi = g \supp B$, and $\supp \varphi^{-1} = g \supp B$;
               \item\label{def:Hecke-category:add}\emph{Morphism additivity}\\
                 Let
                 $\varphi \colon B \to  B'$ be a  morphism.  Suppose $\supp \varphi = L_1 \sqcup L_2$ is
                 a disjoint union,  where $\supp (B') \cdot L_i \cdot \supp (B) = L_i$.
                 We require the existence of morphisms
                 $\varphi_i \colon B \to  B'$ for  $i = 1,2$ satisfying 
                 $\varphi = \varphi_1 + \varphi_2$ and $\supp \varphi_i = L_i$                   
                   for $i =1,2$;
               \item\label{def:Hecke-category:cofinal}\emph{Support cofinality}\\
                 For every
                 object $B \in \calb$ and every subgroup $L$ of finite index in $\supp B$
                 there are morphisms
                 \begin{equation*} B \xrightarrow{i_{B,L}} B|_L \xrightarrow{r_{B,L}}
                   B \end{equation*} such that $\supp B|_L = L$,
                 $\supp i_{B,L} = \supp r_{B,L} = \supp B$ and
                 $r_{B,L} \circ i_{B,L} = \id_{B}$.  Moreover, for $L'$ a subgroup of
                 finite index in $L$ we require $B|_{L'} = (B|_L)|_{L'}$,
                 $i_{B,L'} = i_{B|_L,L'} \circ i_{B,L}$ and
                 $r_{B,L'} = r_{B,L} \circ r_{B|_L,L'}$.
               \end{enumerate}
             \end{definition}

             We note that~\ref{def:Hecke-category:cofinal} means that $\calb$ is equipped
             with a choice of $B|_L,i_{B,L},r_{B,L}$ for all $B$ and $L$.

\begin{example}\label{ex:calb(G;R)} 
  Let $R$ be a ring containing $\IQ$ and $G$ be a td-group.  Consider the Hecke algebra
  $\calh(G;R)$ associated to a $\IQ$-valued (left-invariant) Haar measure $\mu$ on $G$.
  Associated to $\calh(G;R)$ is the category with $G$-supports $\calb(G;R)$.  Objects of
  $\calb(G;R)$ are compact open subgroups of $G$.  Morphisms $\varphi \colon U \to U'$ are
  elements of $\calh(G;R)$ satisfying
  \begin{equation*} 
    \varphi(u'gu) = \varphi(g) \quad \text{for all $u' \in U'$, $u \in U$.}
  \end{equation*}
  The support of $\varphi$ is
  $\supp \varphi = \{ g \in G \mid \varphi(g) \neq 0 \}$\footnote{As $\varphi$ is locally
    constant and compactly supported this is a compact subset of $G$.}.  The identity of
  $U$ is the idempotent $\frac{\chi_U}{\mu(U)} \in \calh(G;R)$ where $\chi_U$ is the
  characteristic function of $U$.  The category $\Idem(\calb(G;R)_\oplus)$ is equivalent
  to the category of finitely generated projective  $\calh(G;R)$-modules,
  compare~\cite[Lem.~6.6]{Bartels-Lueck(2023foundations)}
  In particular $\bfK \calb(G;R) \simeq \bfK \calh(G;R)$.
  
  It is not difficult to check that $\calb(G;R)$ is a Hecke category with $G$-support.
  The subgroup property and morphism additivity are clear from the definitions.  For $U$
  compact open in $G$ and $g \in G$ we have an isomorphism
  $\frac{\chi_{gU}}{\mu(gU)} \colon U \to gUg^{-1}$ in $\calb(G;R)$; its inverse is
  $\frac{\chi_{Ug^{-1}}}{\mu(Ug^{-1})}$.  This proves the translation property.  For support
  cofinality we can set $U|_L := L$, $i_{U,L} = r_{U,L} := \frac{\chi_U}{\mu(U)}$.  For
  more details see~\cite[Sec.~6.C]{Bartels-Lueck(2023foundations)},
  where also more general Hecke algebras are discussed.
\end{example}

\begin{example}\label{ex:calbprime(G;R)}
	Let $R$ be a ring and $G$ be a td-group.
	Assume that $G$ has at least one  compact open subgroup $U$ with the property\footnote{Such a subgroup exists for example if $G$ is reductive $p$-adic and $1/p \in R$, see~\cite[Lemma~1.1]{Meyer-Solleveld(2010)}.} that
	\begin{enumerate}[
                 label=(\thetheorem\alph*),
                 align=parleft, 
                 leftmargin=*,
                 itemsep=2pt
                 ] 
    \item\label{ex:calbprime(G;R):U_0} for all open subgroups $V$ of $U$ the index $[U:V]$ is invertible in $R$.
    \end{enumerate}   
    We can fix one such group $U_0$.  If one normalizes a (left-invariant) Haar measure
    $\mu$ such that $\mu(U_0)=1$, then $\mu$ it takes values in $\IZ[1/n \mid \text{$1/n$
      is invertible in $R$}]$ and one obtains a Hecke algebra $\calh(G;R)$.  In this
    situation we obtain a variant $\calb'(G;R)$ of the category from
    Example~\ref{ex:calb(G;R)}.  Its objects are compact open subgroups
    satisfying~\ref{ex:calbprime(G;R):U_0}.  Morphisms $U \to U'$ are elements of
    $\calh(G;R)$ satisfying
    \begin{equation*} 
      \varphi(u'gu) = \varphi(g) \quad \text{for all $u' \in U'$, $u \in U$}
    \end{equation*}
    as before.  The point is that for such subgroups the measures $\mu(gU)$, $\mu(Ug)$ are
    invertible in $R$ for all $g \in G$\footnote{Indeed
      $\mu(gU) = \mu(U) = [U:U \cap U_0] \cdot \mu(U \cap U_0) = [U:U \cap U_0] \cdot
      [U_0:U \cap U_0]^{-1}$ and
      $\mu(Ug) = \mu(g^{-1}Ug) = [g^{-1}Ug:g^{-1} Ug \cap U_0] \cdot \mu(g^{-1}Ug \cap U_0) = [U:
      U \cap gU_0{g^{-1}}] \cdot [U_0:g^{-1}Ug \cap U_0]^{-1}$.}.  
    Thus formulas from
    Example~\ref{ex:calb(G;R)} still work and $\calb'(G;R)$ is a Hecke category with
    $G$-support.
  \end{example}

\begin{definition}[The category $\calb{[X]}$]\label{def:calbX} Given a category $\calb$
  with $G$-support and a smooth $G$-set $X$ we define the category $\calb[X]$ as follows.
  Objects are pairs $(B,x)$ with $B \in \calb$ and $x \in X$. 
  A morphism $(B,x) \to (B',x')$ is a morphism
  $\varphi \colon B \to B'$ in $\calb$ such that
  $\supp \varphi \subseteq G_{x,x'} = \{ g \in G \mid gx=x' \}$.
\end{definition}

The construction of $\calb[X]$ is natural in $X$ and compatible with disjoint union, i.e.,
if $X = \coprod_{i \in I} X_i$ as $G$-sets, then the canonical functor
\begin{equation}\label{eq:contb-and-coproducts} \coprod_{i \in I} \calb[X_i] \;
  \xrightarrow{\sim} \; \calb[X]
\end{equation}   
is an equivalence of $\IZ$-linear categories.  This reduces the computation of $\calb[X]$
to the case of orbits $G/U$.  Here $U$ is open in $G$, since we are only allowing smooth
$G$-sets.  We write $\calb|_U$ for the subcategory on objects and morphisms with support
in $U$.  If $\calb$ is a Hecke category with $G$-support then,
using~\ref{def:Hecke-category:trans}, there is an equivalence of $\IZ$-linear categories
\begin{equation}\label{eq:calb-on-orbits}
	\calb|_{U} \to \calb[G/U],
\end{equation} 
see~\cite[Lem.~5.5]{Bartels-Lueck(2023foundations)}.

\begin{definition}[$\COP$-assembly map]\label{def:assembly-cop-calb} Let $G$ be a
  td-group and $\calb$ be a category with $G$-support.  The projections $G/U \to G/G$
  induce a map
  \begin{equation}\label{eq:assembly-cop-calb}
    \hocolimunder_{G/U \in \Or_{\COP}(G)} \bfK \big( \calb[G/U] \big)
    \to \bfK \big( \calb[G/G] \big) \simeq \bfK \calb.
  \end{equation}
  We call this the \emph{$\COP$-assembly map for $\calb$}. 
\end{definition}

\begin{remark}
  The notion of categories with $G$-support is very general and allows also for
  pathological examples.  For this reason it is not sensible to conjecture
  that~\eqref{eq:assembly-cop-calb} is in general an equivalence.  However, a significant
  part of our considerations work in the generality of categories with $G$-support.
\end{remark}


\subsection{$l$-uniformly regular coherence and exactness}\label{subsec:regular_and_exact}

Let $\cala$ be an additive category.  The Yoneda embedding $A \mapsto \mor_\cala(\--, A)$
embeds $\cala$ into the abelian category of $\IZ\cala$-modules, i.e., the category of
$\IZ$-linear covariant functors from $\cala$ to $\MODcat{\IZ}$.  The $\IZ\cala$-modules in
the image of this functor are called \emph{finitely generate free}.  A $\IZ\cala$-module
is \emph{finitely presented} if it is the cokernel of a map between finitely generated
free modules.  The additive category $\cala$ is said to be \emph{regular coherent} if any
finitely presented $\IZ\cala$-module has a finite resolution by finitely generated
projective $\IZ\cala$-modules\footnote{These are exactly the direct summands of finitely
  generated free $\IZ\cala$-modules.}.  It is \emph{$l$-uniformly regular coherent} if in
addition the resolution can be chosen to be of length at most $l$.
	
A sequence $A \to A' \to A''$ in $\cala$ is \emph{exact} at $A'$, if its image is exact at
$\mor_\cala(\--, A')$ in $\IZ\cala$-modules.  A functor $F \colon \cala \to \calb$ of
additive categories is \emph{exact}, if it sends sequences that are exact at $A'$ to a
sequence that is exact at $F(A')$.  For a more detailed discussion
see~\cite[Sec.~6]{Bartels-Lueck(2020additive)}.


\subsection{Formulation of the $\COP$-Farrell--Jones Conjecture}\label{subsec:formulation-COP-FJ-conj}

\begin{definition}[Reg]\label{def:regularity-for-COP} 
  A Hecke category with $G$-support is said to satisfy condition (Reg) if for 
  every natural number $d$ there is a natural number $l(d)$ such that for
 every compact open subgroup $U \subseteq G$ the additive category
 $\calb[G/U]_{\oplus}[\IZ^d]$ is $l(d)$-uniformly regular coherent.
\end{definition}

\begin{conjecture}[$\COP$-Farrell--Jones
  Conjecture]\label{conj:ismorphism-conj-coefficients}
  Let $G$ be a td-group and $\calb$ be a Hecke category with $G$-support satisfying (Reg)
  from Definition~\ref{def:regularity-for-COP}.  Then the $\COP$-assembly
  map~\eqref{eq:assembly-cop-calb} for $\calb$ is an equivalence.
\end{conjecture}

As discussed in Subsection~\ref{subsec:homotopy-colimit-in-intro} the $\COP$-assembly
map~\eqref{eq:assembly-cop-calb} for $\calb$ after applying $\pi_n$ can be identified
with~\eqref{eq:COP-assembly-homolgy-theory-cala}.  Thus
Conjecture~\ref{conj:ismorphism-conj-coefficients} is just a restatement of
Conjecture~\ref{conj:COP-FJ-for-Hecke-Cat} from the introduction.


\subsection{Product categories}\label{subsec:product-cats}
We digress briefly to introduce some notation for formal products, that will be useful
later on.  Let $\calc$ be a category.  We define the category $\EPplus\calc$ as follows.
Objects of $\EPplus\calc$ are $n$-tuples of objects of $\calc$, $(C_1, \dots, C_n)$.  Here
$n=0$ is allowed; the empty tuple is the unique $0$-tuple and will be written as
$\ast \in \EPplus\calc$.  Morphisms
$f \colon (C_1, \dots, C_n) \to (C'_1, \dots, C'_{n'})$ are pairs $f=(u,\varphi)$, where
$u \colon \{1,\dots,n'\} \to \{1,\dots,n\}$ and
$\varphi \colon \{1,\dots,n'\} \to \mor_\calc$ are maps such that for each $i' \in \{1,\dots,n'\}$
the morphism $\varphi({i'})$ in $\calc$ is of the shape 
$\varphi({i'}) \colon C_{u(i')} \to C'_{i'}$.  The composition of
\begin{equation*}
  (C_1,\dots,C_n) \xrightarrow{(u,\varphi)} (C'_1,\dots,C'_{n'}) \xrightarrow{(u',\varphi')} (C''_1,\dots,C''_{n''}) 
\end{equation*} 
is $\bigl(u \circ u', i'' \mapsto \varphi'(i'') \circ \varphi(u'(i''))\bigr)$.

As there is a unique map from the empty set to any other set, the empty tuple $\ast$ is a
terminal object in $\EPplus\calc$.  For objects $P = (C_1,\dots,C_n)$, $P' = (C'_1,\dots,C_{n'})$
their product is given by $P \times P' = (C_1,\dots,C_n,C'_1,\dots,C'_{n'})$.  For example
the projection $P \times P' \to P'$ is given by $(u,\varphi)$ with $u(i)=i+n$ and
$\varphi(i') = \id_{C'_{i'}}$ for $i'=1,\dots,n'$.  We write $\EP\calc$ for the full
subcategory of $\EPplus\calc$ obtained by removing the empty product $\ast$.

One advantage of the product category in connection with homotopy colimits is that it
often allows us to disregard all non-identity morphism in $\Or_\calf(G)$.  We write
$\calf(G)$ for the subcategory of $\Or_\calf(G)$ that contains all objects, but only
identity morphisms; for the corresponding subcategory of $\Or(G)$ containing all $G/H$
with $H$ a closed subgroup we write $\All(G)$.  Passing to product categories we obtain
$\EPplus\calf(G)$ and $\EP\calf(G)$.
Thus a morphism
$u \colon (G/H_1, \ldots , G/H_n) \to (G/H'_1, \ldots , G/H'_{n'})$ in
$\EPplus\All(G)$ is a function 
$u \colon  \{1, 2, \ldots, n'\} \to \{1, 2, \ldots, n\}$
satisfying $H_{u(i)} = H'_i$ for all $i$.


\subsection{A reformulation of the $\COP$-Farrell--Jones Conjecture}
Let $G$ be a td-group and $\calb$ be a category with $G$-support.  In the following we
write
\begin{equation*}
  \calb[P] := \calb[G/U_1 \times \cdots \times G/U_n]
\end{equation*}
for $P = (G/U_1,\dots,G/U_n) \in \EP \Or_{\OPEN}(G)$.
We note that $\calb[\ast]$ and $\calb[(G/G,\dots,G/G)]$ are both just (the $\IZ$-linear category underlying) $\calb$.  
In terms of the notation introduced later $\calb[P] = \calb[|P|]$.

\begin{proposition}\label{prop:cop-iso-via-PSub}
  For a family $\calu$ of open subgroups the canonical maps
   induced by the canoncial inclusions
   $\Or_\calu(G) \to \EP\Or_\calu(G)$ and   $\EP\calu(G) \to \EP\Or_\calu(G)$
  \begin{equation*}
    \xymatrix@C=-8ex{\displaystyle \hocolimunder_{G/U \in \Or_\calu(G)}
      \bfK  \big(\calb[G/U]\big) \ar[rd]^{\sim}_{\alpha_1}
      & &  \displaystyle \hocolimunder_{P \in \EP\calu(G)} \bfK \big( \calb[P]\big) \ar[ld]_{\sim}^{\alpha_2}
\\
      & 
      \displaystyle\hocolimunder_{P \in \EP\Or_\calu(G)} \bfK \big( \calb[P]\big)	
    }
  \end{equation*}
  are equivalences.
\end{proposition}

\begin{proof}
  To show that $\alpha_1$ is an equivalence we will use the transitivity
  Lemma~\ref{lem:transitivity} for homotopy colimits.  It thus suffices to show that for
  any $P = (G/U_1,\dots,G/U_n) \in \EP\Or_\calu(G)$ the canonical map 
  \begin{equation}\label{eq:OrUoverP} \hocolimunder_{(G/U,f) \in \Or_\calu(G)
      \downarrow P} \bfK \big( \calb[G/U] \big) \; \xrightarrow{\sim} \; \bfK \big(
    \calb[P] \big)
  \end{equation}
  is an equivalence.  
  As the $U_i$ are open we have $G/U_1 \times \dots \times G/U_n = \coprod_{j \in J} G/W_j$
  with $W_j \in \calu$.    Then
  $\Or_\calu(G) \downarrow P \simeq \coprod_{j} \big( \Or_\calu(G) \downarrow
  G/W_j \big)$ and that $\id_{G/W_j}$ is a terminal object of
  $\Or_\calu(G) \downarrow G/W_j$.  Together with the compatibility of $\calb[\--]$
  with coproducts~\eqref{eq:contb-and-coproducts} this implies that~\eqref{eq:OrUoverP} is
  an equivalence.
    
  Lemma~\ref{lem:hocolim-products-Q-to-P} implies directly that $\alpha_2$ is an
  equivalence.
\end{proof}

  Now we can reformulate 
  Conjecture~\ref{conj:ismorphism-conj-coefficients} by precomposing the assembly
  map~\eqref{eq:assembly-cop-calb} with $\alpha_1^{-1}$ or $\alpha_1^{-1} \circ \alpha_2$,
  thus changing the source
  $\hocolimunder_{G/U \in \Or_{\COP}(G)} \bfK \big( \calb[G/U] \big)$ of the assembly
  map~\eqref{eq:assembly-cop-calb} to
  $\hocolimunder_{P \in \EP\Or_\calu(G)} \bfK \big( \calb[P]\big)$ or
  $\hocolimunder_{P \in \EP\calu(G)} \bfK \big( \calb[P]\big)$.


\section{Controlled algebra}\label{sec:controlled-algebra}


\subsection{The category $\contrCatUcoef{G}{X}{\calb}$}

Let $X$ be a set.
In the following we will often write $2$-tupels in $X$ as $\twovec{x'}{x}$.

\begin{definition}\label{def:contrCatNUcoef}
  Let $X$ be a $G$-set and $\calb$ be a category with $G$-support.  We define the 
  category $\contrCatUcoef{G}{X}{\calb}$ as follows.  Objects are triples $\bfB = (S,\pi,B)$ where
  \begin{enumerate}[
                 label=(\thetheorem\alph*),
                 align=parleft, 
                 leftmargin=*,
                 labelindent=1pt,
                 ] 
  \item\label{def:contrCatNUcoef:S} $S$ is a set,
  \item\label{def:contrCatNUcoef:pi} $\pi \colon S \to X$ is a map,
  \item\label{def:contrCatNUcoef:b} $B \colon S \to \ob \calb$ is a map.
  \end{enumerate}
  Morphisms $\bfB = (S,\pi,B) \to \bfB' = (S',\pi',B')$ in $\contrCatUcoef{G}{X}{\calb}$ are matrices
  $\varphi = (\varphi_s^{s'} \colon B(s) \to B'(s'))_{s \in S, s' \in S'}$ of morphisms in
  $\calb$.  Morphisms are required to be column finite: for each $s \in S$ there are only
  finitely many $s' \in S'$ with $\varphi_s^{s'} \neq 0$.  Composition is matrix
  multiplication (using composition in $\calb$)
  \begin{equation*}
    (\varphi' \circ \varphi)_s^{s''} := \sum_{s'} {\varphi'}_{s'}^{s''} \circ \varphi_s^{s'}.
  \end{equation*}  
\end{definition}
 
The formula for the identity $\id_\bfB$ of $\bfB = (S,\pi,B) \in \contrCatUcoef{G}{X}{\calb}$  is
   \begin{equation*}
   	   (\id_\bfB)_s^{s'} = \begin{cases} \id_{B(s)} & s = s' \\ 0 & \text{else.} \end{cases}
         \end{equation*}

\begin{definition}[Support and finiteness for $\contrCatUcoef{G}{X}{\calb}$]\label{def:support-B(X)}
  The \emph{support} of an object $\bfB = (S,\pi,B)$ in $\calb(X)$ is defined
  to be 
  \begin{equation*}
  	  \suppobj \bfB := \pi(S) \subseteq X.
  \end{equation*} 
  The \emph{support} of a morphism
  $\varphi \colon (S,\pi,B) \to (S',\pi',B')$ in $\contrCatUcoef{G}{X}{\calb}$ is
  \begin{equation*}
    \suppX \varphi := \Big\{ \twovec{\pi'(s')}{g\pi(s)} \,\Big|\, s \in S, s' \in S',
  g \in \supp (\varphi_{s}^{s'})  \Big\} \subseteq X \times X.
  \end{equation*} 
  The $G$-support of a morphism $\varphi$ in $\contrCatUcoef{G}{X}{\calb}$ is
  \begin{equation*}
  	  \suppG \varphi := \bigcup_{s \in S,s' \in S'}  \supp \varphi_s^{s'}.
  \end{equation*}
  We abbreviate    
  \begin{equation*}
    \suppX \bfB := \suppX \id_{\bfB} = \Big\{ \twovec{\pi(s)}{g\pi(s)} \,\Big|\, s \in S, g \in  \supp_G (\id_{B(s)})  \Big\}
  \end{equation*}
  and $\suppG \bfB := \suppG \id_{\bfB}$.
  
  For a subset $A$ of $X$ we will say that $\bfB$ is \emph{finite over $A$}, if $\pi^{-1}(A)$ is finite.
  We will say that $\bfB$ is \emph{finite}, if it is finite over $X$, i.e., if $S$ is finite. 
\end{definition}

For $E, E' \subset X \times X$ we call
\begin{equation*}
  E' \circ E := \{  \twovec{x''}{x} \mid \exists x' \text{with}\; \twovec{x''}{x'} \in E', \twovec{x'}{x} \in E \}
\end{equation*}
the \emph{composition} of $E$ and $E'$.  We call
$E^\op := \{ \twovec{x}{x'} \mid \twovec{x'}{x} \in E \}$ the \emph{opposite} of $E$.  The
product of two subsets $M,M'$ of a group $G$ is
$M' \cdot M := \{ g'g \mid g' \in M', g \in M \}$.  For a subset $M$ of $G$ we write
$M^{-1} := \{ g^{-1} \mid g \in M\}$ for its elementwise inverse.

Note that for $\bfB \xrightarrow{\varphi} \bfB' \xrightarrow{\varphi'} \bfB''$ in
$\contrCatUcoef{G}{X}{\bfB}$ we have
\begin{equation}\label{eq:composition}
	\suppX (\varphi' \circ \varphi) \subseteq \suppX \varphi' \circ (\suppG \varphi' \cdot \suppX \varphi).
\end{equation}

\begin{remark}\label{rem:B(X)-cat-is-additive}        
   We note that $\suppX \bfB$ is not necessarily contained in the diagonal of $X \times X$.
   The category $\contrCatUcoef{G}{X}{\calb}$ is additive; the direct sum comes from disjoint unions, i.e.,   
   \begin{equation*}
   	  (S,\pi,B) \oplus (S',\pi',B') \cong (S \sqcup S',\pi \sqcup \pi',B \sqcup B').  
   \end{equation*} 	
 \end{remark}
 		
\begin{remark}
  Typically $\contrCatUcoef{G}{X}{\calb}$ does not really encode information about $X$;
  any map $f \colon X \to Y$ between non-empty $G$-sets induces an equivalence
  $\contrCatUcoef{G}{X}{\calb} \xrightarrow{\sim} \contrCatUcoef{G}{Y}{\calb}$.
  
  The use of $\contrCatUcoef{G}{X}{\calb}$ will be as a home for interesting subcategories
  that we will exhibit using additional structure on $X$.  The general framework to
  determine subcategories of $\contrCatUcoef{G}{X}{\calb}$ uses the support notions from
  Definition~\ref{def:support-B(X)} and the formalism of $G$-control structures that we
  discuss in Subsection~\ref{subsec:G-control-structures}.
\end{remark}

\begin{remark}[Functoriality]\label{rem:functoriality-calb(X)} The definition of
  $\contrCatUcoef{G}{X}{\calb}$ does not really use a $G$-action on $X$; it is just the
  notion of support that makes use of the $G$-action.  Any map $f \colon X \to Y$ induces
  a functor $f_* \colon \contrCatUcoef{G}{X}{\calb} \to \contrCatUcoef{G}{X}{\calb}$ with
  $f_*(S,\pi,B) = (S,f \circ \pi,B)$.
  \begin{equation*}
    f_* (S,\pi,B) = (S,f\circ \pi,B), \qquad (f_*(\varphi))_{s}^{s'} = \varphi_s^{s'}.
  \end{equation*}
  We have $\suppG f_*(\varphi) = \suppG \varphi$ and $\suppobj f_*(\bfB) = f(\suppobj \bfB)$.
  If $f$ is $G$-equivariant, then $\suppX f_*(\varphi) = f^{\times 2}(\suppX \varphi)$.
  If $f$ is not $G$-equivariant, then there is no general formula that expresses
  $\suppX f_*(\varphi)$ directly in terms of $\suppX \varphi$, but we have
  \begin{equation}\label{eq:support-after-non-equiv-f}
    \suppX f_*(\varphi) \subseteq  f^{\times 2}(\suppX \varphi) \circ
    \Big\{ \twovec{f(gx)}{gf(x)} \; \Big|\; x \in \suppobj \bfB, g \in
    \suppG \varphi \Big\}.
  \end{equation}
  Thus to estimate $\suppX f_*(\varphi)$ we need to estimate the failure of equivariance
  of $f$.
\end{remark}


\subsection{$G$-control structures}\label{subsec:G-control-structures}

\begin{definition}\label{def:G-control-structure} Let $G$ be a group and $X$ be a
  $G$-set.  A $G$-control structure on $X$ is a triple $\mfE = (\mfE_1,\mfE_2,\mfE_G)$
  where
  \begin{itemize}[
                 label=$\bullet$,
                 align=parleft, 
                 leftmargin=*,
                 labelindent=5pt,
                 ] 
  \item $\mfE_1$ is a collection of subsets of $X$ that is closed under finite unions and
    taking subsets;
  \item $\mfE_2$ is a collection of subsets of $X \times X$ that is closed under finite
    unions, taking subsets, opposites, and composition;
  \item $\mfE_G$ is a collection of subsets of $G$ that is closed under finite unions,
    taking subsets, elementwise inverses, and products.  \end{itemize} We require in
  addition that for $M \in \mfE_G$, $E \in \mfE_2$ the product
  $M \cdot E := \big\{ \twovec{gx'}{gx} \; \big| \; g \in M, \twovec{x'}{x} \in E \big\}$
  belongs to $\mfE_2$.
\end{definition}

One might wonder if the condition that $\mfE_G$ is closed under elementwise inverses is
really necessary.  We use this condition in the proof of
Lemma~\ref{lem:filtered-for-controlled}.

\begin{remark}
  In our examples $X$ will always be a topological space and the elements of $\mfE_1$ will
  always have finite intersections with compact subsets of $X$.
\end{remark}

\begin{remark}
  In almost all our examples $\mfE_G$ will be the collection of relatively compact subsets of $G$.
    The only other example for $\mfE_G$ that we use is the
  collection of all subsets of $G$.  It will only be used in Section~\ref{sec:reduction}
  for the proof of the Reduction Theorem~\ref{thm:reduction}.
\end{remark}

\begin{example}[Trivial control structure]\label{ex:trivial-control-structure} Let $X$ be
  a $G$-set.  We obtain a $G$-control structure $\mfE$ on $X$, where $\mfE_1$ is the
  collection of all finite subsets, $\mfE_2$ is the collection of all subsets of the
  diagonal in $X \times X$ and $\mfE_G$ is the collection of all relatively compact subsets of $G$.
\end{example}


\subsection{The category $\contrCatUcoef{G}{\mfE}{\calb}$}

\begin{definition}\label{def:controlled-hecke-cat}
  Let $X$ be a $G$-set, $\mfE = (\mfE_1,\mfE_2,\mfE_G)$ be a $G$-control structure on $X$,
  and $\calb$ be a category with $G$-support.  The additive category
  $\contrCatUcoef{G}{\mfE}{\calb}$ is the following subcategory of
  $\contrCatUcoef{G}{X}{\calb}$.
  \begin{enumerate}[
                 label=(\thetheorem\alph*),
                 align=parleft, 
                 leftmargin=*,
                 labelindent=1pt,
                 labelsep=8pt,
                 itemsep=1pt
                 ] 
               \item\label{def:controlled-hecke-cat:obj} An object $\bfB = (S,\pi,B)$
                 from $\contrCatUcoef{G}{X}{\calb}$ belongs to
   $\contrCatUcoef{G}{\mfE}{\calb}$,  iff it is finite over each point of
   $X$\footnote{I.e., $\pi \colon S \to X$ is finite-to-one}, $\suppobj \bfB \in \mfE_1$,
  $\suppX \bfB \in \mfE_2$ and $\suppG \bfB \in \mfE_G$;
     \item\label{def:controlled-hecke-cat:mor} A morphism $\varphi$ in
       $\contrCatUcoef{G}{X}{\calb}$ between objects from $\contrCatUcoef{G}{\mfE}{\calb}$
       belongs to $\contrCatUcoef{G}{\mfE}{\calb}$ iff
  $\suppX \varphi \in \mfE_2$, $\suppG \varphi \in \mfE_G$ and $\varphi$ is row
  finite\footnote{As morphisms in $\contrCatUcoef{G}{X}{\calb}$ are already required to be column
    finite, this means that $\varphi$ is column and row finite, i.e., for fixed $s$ there
    are only finitely many $s'$ with $\varphi_s^{s'} \neq 0$ and for fixed $s'$ there are
    only finitely many $s$ with $\varphi_s^{s'} \neq 0$.}. 
  \end{enumerate}
\end{definition}

\begin{remark}
  For a smooth $G$-set $X$ and the $G$-control structure $\mfE$ from
  Example~\ref{ex:trivial-control-structure} the category $\contrCatUcoef{G}{\mfE}{\calb}$ is equivalent to
  $(\calb[X])_\oplus$.
\end{remark}

\begin{remark}[Summands]\label{rem:summands} Let $\bfB = (S,\pi,B) \in \contrCatUcoef{G}{\mfE}{\calb}$.
  For $S_0 \subseteq S$ set $\bfB|_{S_0} := (S_0,\pi|_{S_0},B|_{S_0})$.  Consider
  \begin{equation*}
    \bfB|_{S_0} \xrightarrow{i} \bfB \xrightarrow{r} \bfB|_{S_0} \quad \text{with} \quad i_{s_0}^s = r_{s}^{s_0} 
= \begin{cases} \id_{B(s)} & s = s_0; \\ 0 & s \neq s_0. \end{cases}	
  \end{equation*} 
  Then $\id_{\bfB|_{S_0}} = r \circ i$, $\suppobj \bfB|_{S_0} \subseteq \suppobj \bfB$,
  $\suppX i = \suppX r \subseteq \suppX \bfB$.  Altogether $\bfB|_{S_0}$ is a direct
  summand of $\bfB$ in $\contrCatUcoef{G}{\mfE}{\calb}$.
  
  For $Y \subseteq X$ we abbreviate $\bfB|_{Y} := \bfB|_{\pi^{-1}(Y)}$.
\end{remark}

\begin{remark}[Corners]\label{rem:corners} 
  Let $\varphi \colon \bfB = (S,\pi,B) \to \bfB' = (S',\pi',B')$.
  For $Y,Y' \subseteq X$ we obtain summands 
  $\bfB|_{Y} \xrightarrow{i_Y} \bfB \xrightarrow{r^Y} \bfB|_{Y}$ and
  $\bfB'|_{Y'} \xrightarrow{i_{Y'}} \bfB' \xrightarrow{r^{Y'}} \bfB|_{Y'}$ as in Remark~\ref{rem:summands}.
  We define $\varphi_Y^{Y'}$ as the composition
  $\bfB|_Y \xrightarrow{i_Y} \bfB \xrightarrow{\varphi} \bfB' \xrightarrow{r^{Y'}} \bfB'|_{Y'}$.
  Then
  \begin{equation*}
    (\varphi_Y^{Y'})_s^{s'} = \begin{cases} \varphi_s^{s'} & \pi(s) \in Y, \pi'(s') \in Y';
      \\
      0 & \text{else}. \end{cases}
  \end{equation*}   
  If $\twovec{x}{x'} \in \suppX \varphi_Y^{Y'}$, then $x=\pi(s)$, $x'=g\pi'(s')$ for
  $\pi(s) \in Y$, $\pi'(s') \in Y'$, $g \in \suppG \varphi$ and
  $g \in \supp \varphi_s^{s'}$.  Thus
  \begin{equation}\label{eq:support-of-corner}
  	\suppX \varphi_Y^{Y'} \subseteq Y' \times (\suppG \varphi)\cdot Y.
  \end{equation}
\end{remark}  

\begin{lemma}\label{lem:corner}
  Consider the situation of of Remark~\ref{rem:corners}.  Suppose that for all
  $\twovec{x'}{x} \in \suppX \varphi$ with $x \in (\supp_G \varphi) \cdot Y$ we have
  $x' \in Y'$.  Then
  \begin{equation*}
    i_{Y'} \circ \varphi_Y^{Y'} = \varphi \circ i_{Y} \colon \bfB|_Y \to \bfB'.
  \end{equation*}
\end{lemma}
 
\begin{proof}
  We need to check that $\varphi_s^{s'} \neq 0$ with $s \in \pi^{-1}(Y)$ implies
  $\pi'(s') \in Y'$.  If $\varphi_s^{s'} \neq 0$, then there is
  $g \in \supp_G \varphi_s^{s'} \subseteq \suppG \varphi$ and so
  $\twovec{\pi'(s')}{g \pi(s)} \in \suppX \varphi$.  Hence  $\pi'(s') \in Y'$ by assumption.
\end{proof}

\begin{remark}[Shifted copy]\label{rem:X-shifted-copies} 
  Let $\bfB = (S,\pi,B) \in \contrCatUcoef{G}{\mfE}{\calb}$.  
  Let $\sigma \colon S \to X$ be a finite-to-one map.  
  Assume that $\sigma(S) \in \mfE_1$ and that $\pi$ and $\sigma$ are
  $\mfE_2$-equivalent in the sense that
  \begin{equation*}
    E := \Big\{ \twovec{\pi(s)}{\sigma(s)} \; \Big| \;  s \in S \Big\} \in \mfE_2.
  \end{equation*}
  Consider $\bfB_{\sigma} := (S,\sigma,B)$ and
  \begin{equation*}
    \bfB_{\sigma} \xrightarrow{\varphi} \bfB \xrightarrow{\psi} \bfB_\sigma \quad
    \text{with} \quad \varphi_{s}^{s'} = \psi_{s}^{s'} =
    \begin{cases} \id_{B(s)} & s = s'; \\ 0 & s \neq s'. \end{cases}	
  \end{equation*}
  Then $\suppobj \bfB_{\sigma} = \sigma(S_0) \in \mfE_1$,
  $\id_{\bfB_\sigma} = \psi \circ \varphi$, $\id_{\bfB} = \varphi \circ \psi$, and
  \begin{eqnarray*}
    \suppX \varphi
    & = & \
\Big\{ \twovec{\pi(s)}{g \sigma(s)} \; \Big| \; \mid s \in S, g \in \suppG B(s) \Big\}
          \subseteq  \suppX \bfB \circ (\suppG \bfB) \cdot E \in \mfE_2;
\\
    \suppX \psi 
& = &  
\Big\{ \twovec{\sigma(s)}{g \pi(s)} \; \Big| \; \mid s \in S, g \in \suppG B(s) \Big\}
                      \subseteq E^{\op} \circ \suppX \bfB \in \mfE_2. 
  \end{eqnarray*}
  Altogether, $\bfB_{\sigma}$ and $\bfB$ are canonically isomorphic in $\contrCatUcoef{G}{\mfE}{\calb}$.
  We call $\bfB_{\sigma}$ \emph{a shifted copy of $\bfB$}.
\end{remark}


\subsection{Quotients}

\begin{definition}\label{def:E_Y} Let $\mfE = (\mfE_1,\mfE_2,\mfE_G)$ be a $G$-control
  structure on $X$.  Let $\caly$ be a collection of subsets $Y$ of $X$ that is closed
  under finite unions and taking subsets.  Suppose also that for $M \in \mfE_G$ and
  $Y \in \caly$ we have $M \cdot Y \in \caly$.  We obtain a $G$-control structure
  $\mfE|_\caly := (\mfE_1|_\caly ,\mfE_2,\mfE_G)$ on $X$ where
  $\mfE_1|_\caly := \{ F \cap Y \mid F \in \mfE_1, Y \in \caly \}$.
\end{definition}

\begin{definition}\label{def:calb(E,Y)} In the situation of Definition~\ref{def:E_Y} the
  category $\contrCatUcoef{G}{\mfE|_\caly}{\calb}$ is a full subcategory of
  $\contrCatUcoef{G}{\mfE}{\calb}$ and we define
  \begin{equation*}
    \contrCatUcoef{G}{\mfE,\caly}{\calb} := \contrCatUcoef{G}{\mfE}{\calb} \; \big/ \;
    \contrCatUcoef{G}{\mfE|_\caly}{\calb}.
  \end{equation*}
\end{definition}

\begin{lemma}\label{lem:filtered-for-controlled} In the situation of
  Definition~\ref{def:E_Y} the category $\contrCatUcoef{G}{\mfE}{\calb}$ is
  $\contrCatUcoef{G}{\mfE|_\caly}{\calb}$-filtered,
  see Definition~\ref{def:U-filtered}.
\end{lemma}

\begin{proof}
  Let $\bfU_- \xrightarrow{\varphi_-} \bfB \xrightarrow{\varphi_+} \bfU_+$ be morphisms in
  $\contrCatUcoef{G}{\mfE}{\calb}$ 
  with $\bfU_-, \bfU_+ \in \contrCatUcoef{G}{\mfE|_\caly}{\calb}$.  
  Write
  $\bfU_\pm = (S_\pm, \pi_\pm, U_\pm)$ and $\bfB = (T, \rho, B)$.  Let $T_- \subset T$
  consist of all $t$, for which there is $s \in S_-$ with $(\varphi_-)_s^t \neq 0$, and let
  $T_+ \subset T$ consist of all $t$, for which there is $s_+ \in S_+$ with
  $(\varphi_+)^{s}_t \neq 0$.  Set $T_0 := T_- \cup T_+$.  We obtain the summand
  $\bfB|_{T_0}$ of $\bfB$ as in Remark~\ref{rem:summands}.  Clearly $\varphi_\pm$ factors
  over $\bfB|_{T_0}$.  It suffices now to check that $\bfB|_{T_0}$ is isomorphic to an
  object in $\contrCatUcoef{G}{\mfE|_\caly}{\calb}$.  We now use a shifted copy of $\bfB|_{T_0}$.  Choose
  $\sigma \colon T_0 \to X$ such that for every $t \in T_0$ there are either $s \in S_-$,
  $g \in G$ with $(\varphi_-)_s^t(g) \neq 0$ and $\sigma(t) = g\pi_-(s)$ or there are
  $s \in S_+$, $g \in G$ with $(\varphi_+)^{s}_t(g) \neq 0$ and
  $\sigma(t) = (g)^{-1} \pi_+(s)$.  Then $\sigma$ is finite-to-one because $\varphi$ is
  column finite and $\varphi'$ is row finite.  Our choice of $\sigma$ implies
  \begin{equation*}
    E := \Big\{ \twovec{\rho(t)}{\sigma(t)} \; \Big| \; t \in T_0 \Big\}
    \subseteq \suppX \varphi_- \cup (\suppG \varphi_+)^{-1} \cdot (\suppX \varphi_+)^{\op}.
  \end{equation*}  
  Thus\footnote{Here we use in particular that $\mfE_G$ is closed under pointwise
    inverses.} $E \in \mfE_2$ and the shifted copy $(\bfB|_{T_0})_{\sigma}$ of
  $\bfB|_{T_0}$ is isomorphic to $\bfB$, see Remark~\ref{rem:X-shifted-copies}.  By
  construction
  \begin{multline*}
    \suppobj (\bfB|_{T_0})_{\sigma} = \sigma(T_0) \subseteq \\ (\suppG \varphi_-) \cdot
    \supp_1(\varphi_-) \cup (\suppG \varphi_+)^{-1} \cdot \supp_1(\varphi_+) \in \caly
  \end{multline*}
  and $(\bfB|_{T_0})_{\sigma} \in \contrCatUcoef{G}{\mfE|_\caly}{\calb}$ as required.
\end{proof}

Combining Lemma~\ref{lem:filtered-for-controlled} with
Theorem~\ref{thm:Karoubi-filtration} we obtain a fibration sequence
\begin{equation}\label{eq:karoubi-sequence} \bfK (\contrCatUcoef{G}{\mfE|_\caly}{\calb}) \to
  \bfK(\contrCatUcoef{G}{\mfE}{\calb}) \to \bfK(\contrCatUcoef{G}{\mfE,\caly}{\calb}).
\end{equation}
More general, if $\caly_0$ is another collection of subsets of $X$ also satisfying the
conditions from Definition~\ref{def:calb(E,Y)}, and if $\caly_0 \subseteq \caly$, then
\begin{equation}\label{eq:karoubi-sequence-rel} \bfK (\contrCatUcoef{G}{\mfE|_\caly,\caly_0}{\calb}) \to
  \bfK(\contrCatUcoef{G}{\mfE,\caly_0}{\calb}) \to \bfK(\contrCatUcoef{G}{\mfE,\caly}{\calb})
\end{equation}
is a fibration sequence\footnote{\eqref{eq:karoubi-sequence} applies to the two vertical
  and the upper horizontal sequence in
  \begin{equation*}
    \xymatrix{\contrCatUcoef{G}{(\mfE|_{\caly})|_{\caly_0}}{\calb} \ar[r]^{=} \ar[d]
      & \contrCatUcoef{G}{\mfE|_{\caly_0}}{\calb} \ar[d] \\
      \contrCatUcoef{G}{\mfE|_\caly}{\calb} \ar[r] \ar[d]
      &  \contrCatUcoef{G}{\mfE}{\calb} \ar[r] \ar[d] & \contrCatUcoef{G}{\mfE,\caly}{\calb} \ar[d]^{=} \\
      \contrCatUcoef{G}{\mfE|_\caly,\caly_0}{\calb} \ar[r] &  \contrCatUcoef{G}{\mfE,\caly_0}{\calb} \ar[r]
      & \contrCatUcoef{G}{\mfE,\caly}{\calb} \\
    }
  \end{equation*} and the lower horizontal sequence is therefore a fibration sequence in K-theory.}.
We will refer to sequences of additive categories of the form
\begin{equation*}
  \contrCatUcoef{G}{\mfE|_\caly,\caly_0}{\calb} \to
  \contrCatUcoef{G}{\mfE,\caly_0}{\calb} \to \contrCatUcoef{G}{\mfE,\caly}{\calb}
\end{equation*}
as Karoubi sequences.  K-theory takes Karoubi sequences to a fibration sequences of
spectra.


\subsection{Excision}

Let $\mfE$ be a $G$-control structure on $X$.  Let $\caly_0$ and $\caly_1$ be two
collections of subsets of $X$ satisfying the assumptions in Definition~\ref{def:E_Y}.
Then $\caly_0 \cap \caly_1$ also satisfies these assumptions.  The union of $\caly_0$ and
$\caly_1$ may not, but we abuse notation and define $\caly_0 \cup \caly_1$ as the
collection of all sets $Y_0 \cup Y_1$ with $Y_i \in \caly_i$.  
There is a natural functor
\begin{equation}\label{eq:excision-functor} \contrCatUcoef{G}{\mfE|_{\caly_1}, \caly_0 \cap
  \caly_1}{\calb} \to \contrCatUcoef{G}{\mfE|_{\caly_0 \cup \caly_1},\caly_0}{\calb}.
\end{equation}
It is not difficult to check that~\eqref{eq:excision-functor} is surjective on morphism
sets and on isomorphism classes of objects, but it may fail to be injective on morphism
sets.  There are different possible assumptions that guarantee
that~\eqref{eq:excision-functor} is an equivalence.  We later use the following.

\begin{lemma}\label{lem:formal-excision} Assume that for all $Y_1 \in \caly_1$,
  $E \in \mfE_2$ there are $Y_0 \in \caly_0$, $Y_1' \in \caly_1$ such that
  $Y_1 \subseteq Y_1' \cup Y_0$ and
  \begin{equation*}
    (Y_1')^{E} := \{ y' \in X \mid \exists y \in Y \; \text{with} \; \twovec{y'}{y} \in E \} \in \caly_1. 
  \end{equation*}	
  Then~\eqref{eq:excision-functor} is an equivalence.
\end{lemma}

\begin{proof}
  We only need to discuss faithfulness on morphisms.  Let $\varphi \colon \bfB \to \bfB'$
  be a morphism in $\contrCatUcoef{G}{\mfE|_{\caly_1}}{\calb}$.  Assume that $\varphi$ can be factored in
  $\contrCatUcoef{G}{\mfE|_{\caly_0 \cup \caly_1}}{\calb}$ as
  \begin{equation*}
    \bfB \xrightarrow{\varphi_-} \bfX \xrightarrow{\varphi_+} \bfB
  \end{equation*}
  where $\bfX \in \contrCatUcoef{G}{\mfE|_{\caly_0}}{\calb}$, i.e., $\suppobj \bfX \in \caly_0$.  We need to
  produce such a factorization over an $\bfX' \in \contrCatUcoef{G}{\mfE|_{\caly_0 \cap \caly_1}}{\calb}$,
  i.e., $\suppobj \bfX' \in \caly_0 \cap \caly_1$.  We have
  $Y_1 := \suppG \varphi \cdot \suppobj \bfB \in \caly_1$ and
  $E := (\suppG \varphi_-)^{-1} \cdot \suppX \in \bfE_2$.  Applying the assumption we find
  $Y_0 \in \caly_0$, $Y_1' \in \caly_1$ such that $Y_1 \subseteq Y_1' \cup Y_0$ and
  $(Y_1')^{E} \in \caly_1$.  Now $\bfB = \bfB|_{Y_1'} \oplus \bfB|_{Y_0}$.  Write
  $\pi_{Y_1'},\pi_{Y_0} \colon \bfB \to \bfB$ for the corresponding projections.  Then
  $\varphi - \varphi \circ \pi_{Y_1'} = \varphi \circ \pi_{Y_0}$ factors over
  $\bfB|_{Y_0} \in \contrCatUcoef{G}{\mfE|_{\caly_0}}{\calb}$.  This allows us to replace $\varphi$ with
  $\varphi \circ \pi_{Y_1'}$ and $\varphi_-$ with $\varphi_- \circ \pi_{Y_1'}$, or put
  differently, we may assume without loss of generality $\suppobj \bfB \subseteq Y_1'$.
  
  Let now $Y'_0 \subseteq \suppobj \bfX$ consist of all $y_0 \in \suppobj \bfX$ for which
  there are $y_1 \in \suppobj \bfB$ and $g \in \suppG \varphi_-$ with
  $\twovec{y_0}{gy_1} \in \suppX \varphi_-$, i.e., the matrix entry of $\varphi$ for
  $\twovec{y_0}{y_1}$ is non-trivial.  Then $\varphi_-$ factors canonically over the
  inclusion $\bfX|_{Y'_0} \to \bfX$.  This allows us to replace $\bfX$ with
  $\bfX|_{Y'_0}$.  It remains to check that $Y'_0 \in \caly_0 \cap \caly_1$.  As
  $Y'_0 \subseteq \suppobj \bfX \in \caly_0$ we have $Y'_0 \in \caly_0$.  For
  $y_0 \in Y'_0$ there are $y_1 \in \suppobj \bfB \subseteq Y'_1$ and
  $g \in \suppG \varphi_-$ with $\twovec{y_0}{gy_1} \in \suppX \varphi_-$.  This implies
  $y_0 \in (Y'_1)^E$.  As $(Y'_1)^E \in \caly_1$ we now also have $Y'_0 \in \caly_1$.
\end{proof}

\begin{lemma}\label{lem:y-rel-Y-irrelevant} Let $\caly$ and $\caly_0$ be two collections
  of subsets of $X$ satisfying the assumptions from Definition~\ref{def:E_Y}.  Then the
  canonical functor 
  \begin{equation*}
    \contrCatUcoef{G}{\mfE|_{\caly_0},\caly}{\calb} \to \contrCatUcoef{G}{\mfE|_{\caly_0 \cup \caly},\caly}{\calb}
  \end{equation*}
  is an equivalence.
\end{lemma}

\begin{proof}
  The only difference between the two categories is that the category on the right has
  more objects, i.e., objects with support in $\caly$.  But, by Definition, these
  additional objects are trivial (isomorphic to zero).
\end{proof}

\begin{lemma}\label{lem:formal-homotopy-pushout-square} Let $\caly,\caly_0,\caly_1$ be
  collections of subsets of $X$ satisfying the assumptions in Definition~\ref{def:E_Y}.
  Assume that the condition from Lemma~\ref{lem:formal-excision} holds.  Then
  \begin{equation*}
    \xymatrix{\bfK(\contrCatUcoef{G}{\mfE|_{\caly_0 \cap \caly_1}, \caly}{\calb}) \ar[r] \ar[d] &
      \bfK(\contrCatUcoef{G}{\mfE|_{\caly_1}, \caly)}{\calb} \ar[d] \\   
      \bfK(\contrCatUcoef{G}{\mfE|_{\caly_0}, \caly)}{\calb} \ar[r] &
      \bfK(\contrCatUcoef{G}{\mfE|_{\caly_0 \cup \caly_1}, \caly)}{\calb}  }
  \end{equation*}	
  is a homotopy pushout square.
\end{lemma}

\begin{proof}
  We first argue that we may assume that $\caly \subseteq \caly_0 \cap \caly_1$.  Indeed,
  Lemma~\ref{lem:y-rel-Y-irrelevant} allows us to replace $\caly_0$ with
  $\caly_0 \cup \caly$ and $\caly_1$ with $\caly_1 \cup \caly$.  It is not difficult to
  check that the condition from Lemma~\ref{lem:formal-excision} is preserved.

  Now the horizontal homotopy cofibers of the above diagram are determined
  by~\eqref{eq:karoubi-sequence-rel} and are given by 
  $\bfK(\contrCatUcoef{G}{\mfE|_{\caly_0},\caly_0 \cap \caly_1)}{\calb}$ and
  $\bfK(\contrCatUcoef{G}{\mfE|_{\caly_0 \cup \caly_1}, \caly_1)}{\calb}$.  The excision result of
  Lemma~\ref{lem:formal-excision} applies to show that the induced maps between the
  homotopy cofibers is an equivalence.  This implies the assertion.
\end{proof}


\subsection{Swindles}

Let $\mfE$ be a $G$-control structure on $X$.  Let $\caly$ be a collection of subsets of
$X$ satisfying the assumptions in Definition~\ref{def:E_Y}.  
Sometimes it is easy to produce Eilenberg swindles on
$\contrCatUcoef{G}{\mfE,\caly}{\calb}$.  Often such a swindle either comes from some map $f \colon X \to X$
that pushes everything to $\infty$, as in~\ref{lem:formal-swindle:leave} below.  We will
use the following formal result later on.

\begin{lemma}\label{lem:formal-swindle} Assume that there is a $G$-map $f \colon X \to X$
  satisfying
  \begin{enumerate}[
                 label=(\thetheorem\alph*),
                 align=parleft, 
                 leftmargin=*,
                 labelindent=1pt,
                 labelsep=8pt,
                 itemsep=1pt
                 ] 
  \item\label{lem:formal-swindle:leave} for all $x \in X$ there is $n$ such that
    $(f^{\circ n})^{-1}(x) = \emptyset$;
  \item\label{lem:formal-swindle:E_1} for all $Y \in \mfE_1$ we have
    $\bigcup_{n \in \IN} f^n(Y) \in \mfE_1$;
  \item\label{lem:formal-swindle:caly} for all $Y \in \caly$ we have
    $\bigcup_{n \in \IN} f^n(Y) \in \caly$;
  \item\label{lem:formal-swindle:E_2} for all $E \in \mfE_2$ we have
    $\bigcup_{n \in \IN} (f \times f)^{\circ n}(E) \in \mfE_2$;
  \item\label{lem:formal-swindle:shift} for all $M \in \mfE_G$ we have
    $\{ \twovec{f(x)}{gx} \mid x \in X; g \in M \} \in \mfE_2$.
  \end{enumerate}
  Then the K-theory of $\contrCatUcoef{G}{\mfE,\caly}{\calb}$ is trivial.
\end{lemma}

\begin{proof}
  For $\bfB = (S,\pi,B) \in \contrCatUcoef{G}{\mfE}{\calb}$ we define
  $\bfB^{\infty} = (S^{\infty},\pi^{\infty},B^{\infty})$ where $S^\infty = S \times \IN$,
  $\pi^\infty(s,n) = f^{\circ n}(s)$, $B^\infty(s,t) = B(s)$.
  Assumption~\ref{lem:formal-swindle:leave} implies that $\pi^{\infty}$ is finite-to-one.
  Assumption~\ref{lem:formal-swindle:E_1} implies $\suppobj \bfB^{\infty} \in \mfE_1$.
  Assumption~\ref{lem:formal-swindle:E_2} implies $\suppX \bfB^{\infty} \in \mfE_2$.  Thus
  $\bfB^{\infty} \in \contrCatUcoef{G}{\mfE}{\calb}$.
   
  For $\varphi \colon (S,\pi,B) \to (S',\pi',B') \in \contrCatUcoef{G}{\mfE}{\calb}$ we define
  $\varphi^\infty$ by $(\varphi^\infty)^{s',t'}_{s,t} := \varphi_s^{s'}$.
  Assumption~\ref{lem:formal-swindle:E_2} implies $\suppX \varphi^\infty \in \mfE_2$.  As
  $\varphi^\infty$ is also row and column finite (because $\varphi$ is), we have
  $\varphi^\infty \in \contrCatUcoef{G}{\mfE}{\calb}$.  Compatibility with composition is straight forward
  and we obtain an endofunctor $(\--)^\infty$ of $\contrCatUcoef{G}{\mfE}{\calb}$.  For
  $\bfB = (S,\pi,B) \in \contrCatUcoef{G}{\mfE}{\calb}$ let $i_\bfB \colon \bfB \to \bfB^\infty$ and
  $\sh_\bfB \colon \bfB^\infty \to \bfB^\infty$ be induced by the inclusions
  \begin{eqnarray*}
    S  & \to & S \times \IN, \quad s \mapsto (s,0); \\
    S \times \IN & \to & S \times \IN, \quad (s,t) \mapsto (s,t+1).  
  \end{eqnarray*}
  Clearly, $i_\bfB \in \contrCatUcoef{G}{\mfE}{\calb}$.  Assumption~\ref{lem:formal-swindle:shift} (for
  $M = \suppG \bfB$) implies that $\sh_\bfB \in \contrCatUcoef{G}{\mfE}{\calb}$.  
  Now $i \oplus \sh$ is a
  natural isomorphism $(\--)^\infty \oplus \id_{\contrCatUcoef{G}{\mfE}{\calb}} \cong (\--)^\infty$.
  Altogether we defined a swindle on $\contrCatUcoef{G}{\mfE}{\calb}$.
  Assumption~\ref{lem:formal-swindle:caly} ensures that this swindle descends to
  $\contrCatUcoef{G}{\mfE,\caly}{\calb}$.
\end{proof}

We will need a variation of the swindle from Lemma~\ref{lem:formal-swindle}, where we can
swindle towards some $Z \in \caly$ instead of towards $\infty$\footnote{Of course,
  Lemma~\ref{lem:formal-swindle} is implied by Lemma~\ref{lem:formal-swindle-Z} by taking
    $Z = \emptyset$.}.

\begin{lemma}\label{lem:formal-swindle-Z} Assume that there are a $G$-map
  $f \colon X \to X$ and $Z \in \caly$ satisfying
  \begin{enumerate}[
                 label=(\thetheorem\alph*),
                 align=parleft, 
                 leftmargin=*,
                 labelindent=1pt,
                 itemsep=1pt
                 ] 
  \item\label{lem:formal-swindle-Z:leave} for all $x \in X \setminus Z$ there is $n$ such
    that $(f^{\circ n})^{-1}(x) = \emptyset$;
  \item\label{lem:formal-swindle-Z:E_1} for all $Y \in \mfE_1$ we have
    $\bigcup_{n \in \IN} Y_n \in \mfE_1$ where $Y_0 = Y$ and
    $Y_{n+1} = f(Y_n \setminus Z)$;
  \item\label{lem:formal-swindle-Z:caly} for all $Y \in \caly$ we have
    $\bigcup_{n \in \IN} Y_n \in \caly$ where $Y_0 = Y$ and
    $Y_{n+1} = f(Y_n \setminus Z)$;
  \item\label{lem:formal-swindle-Z:E_2} for all $E \in \mfE_2$ we have
    $\bigcup_{n \in \IN} (f \times f)^{\circ n}(E) \in \mfE_2$;
  \item\label{lem:formal-swindle-Z:shift} for all $M \in \mfE_G$ we have
    $\{ \twovec{gx}{f(x)} \mid x \in X; g \in M \} \in \mfE_2$.
  \end{enumerate}
  Then the K-theory of $\contrCatUcoef{G}{\mfE,\caly}{\calb}$ is trivial.
\end{lemma}

\begin{proof}
  In $\contrCatUcoef{G}{\mfE,\caly}{\calb}$ we have $\bfB|_{Z} \cong 0$ and so
  $\bfB \cong \bfB|_{X \setminus Z}$.  Thus we can systematically get ride of everything
  over $Z$.  A swindle on $\contrCatUcoef{G}{\mfE,\caly}{\calb}$ can be constructed almost verbatim as in
  Lemma~\ref{lem:formal-swindle}.  The only difference is that we use a subset
  $S^\infty_-$ in place of $S^\infty$.  To define this subset set $S_0 := S$,
  $\pi_0 := \pi$ and inductively $S_{n+1} := (\pi_n)^{-1}(X \setminus Z)$,
  $\pi_{n+1} := (f \circ \pi_n)|_{S_{n+1}}$.  Then
  $S^\infty_- := \bigcup_{n} S_n \times \{n\}$.  We remark that $\pi_n$ is just the
  restriction of $\pi^{\infty}$ to $S_n \cong S_n \times \{ n\}$.
    
  After restricting everything from $S^\infty$ to $S^\infty_-$, we obtain a swindle on
  $\contrCatUcoef{G}{\mfE,\caly}{\calb}$ as before.
\end{proof}


\section{The $\CVCYC$-Farrell--Jones Conjecture}\label{sec:Farrell-Jones-Conj-for-td}


\subsection{Delooping} 
In order to formulate the $\CVCYC$-Farrell--Jones Conjecture we will need categories
$\contc_G(P;\calb)$ for $P \in \EPplus\All(G)$.  To prepare for their construction in
Definition~\ref{def:C_G-P} later on we discuss the Pedersen-Weibel delooping of K-theory
from~\cite{Pedersen-Weibel(1985)} in our set-up.  Let $X$ be a $G$-set.  Define the
$G$-control structure
$\contstraPW(X) = (\contstraPW_1(X),\contstraPW_2(X),\contstraPW_G(X))$ on $X \times \IN$
as follows.
\begin{itemize}[
                 label=$\bullet$,
                 align=parleft, 
                 leftmargin=*,
                 labelindent=5pt,
                 ] 
\item $\contstraPW_1(X)$ is the collection of all subsets $F$ of $X \times \IN$ for which
  $F \cap X \times \{t\}$ is finite for all $t$;
\item $\contstraPW_2(X)$ is the collection of all $E \subseteq \IN \times \IN$
  with
  \[
    \supp \Big\{ |t-t'| \; \Big|\; \twovec{x',t'}{x,t} \in E \Big\} < \infty;
  \]
\item $\contstraPW_G(X)$ is the collection of all relatively compact subsets of $G$.
\end{itemize}
Let $\caly$ be the collection of subsets $Y$ of $X \times \IN$ that are contained in
$X \times \{1,\dots,N\}$ for some $N$ (depending on $Y$).  Let $\calb$ be a category with
$G$-support.  Then 
\begin{equation*}
  \contrCatUcoef{G}{\contstraPW(X)|_\caly}{\calb} \to  \contrCatUcoef{G}{\contstraPW(X)}{\calb}
  \to \contrCatUcoef{G}{\contstraPW(X),\caly}{\calb}
\end{equation*}
is a Karoubi sequence.  Applying K-theory we obtain a fibration sequence,
see~\eqref{eq:karoubi-sequence}.
Combining this sequence with Lemma~\ref{lem:equivalence_and_swindle} below we obtain
$\Omega \bfK \big( \contrCatUcoef{G}{\contstraPW(X),\caly}{\calb} \big) \simeq \bfK(\calb)$.
Lemma~\ref{lem:equivalence_and_swindle} is standard, but the proof is instructive as we
will use variation thereof later on.
 		
\begin{lemma}\label{lem:equivalence_and_swindle} \
  \begin{enumerate}[
                 label=(\thetheorem\alph*),
                 align=parleft, 
                 leftmargin=*,
                 labelindent=1pt,
                 ] 
  \item\label{lem:equivalence_and_swindle:equiv} There is an equivalence
    $\contrCatUcoef{G}{\contstraPW(X)|_\caly}{\calb} \xrightarrow{\sim} \calb_{\oplus}$
    defined by $(S,\pi,B) \mapsto \oplus_{s \in S} B(s)$;
  \item\label{lem:equivalence_and_swindle:swindle} The K-theory of
    $\contrCatUcoef{G}{\contstraPW(X)}{\calb}$  is trivial.
  \end{enumerate}
\end{lemma}

\begin{proof}
  If $(S,\pi,B) \in \contrCatUcoef{G}{\contstraPW(X)}{\calb}$, then by definition of $\contstraPW_1(X)$, for
  each $t \in \IN$, $\pi^{-1}(X \times \{t\})$ is finite.  If
  $(S,\pi,B) \in \contrCatUcoef{G}{\contstraPW(X)|_\caly}{\calb}$, then in addition $\pi^{-1}(X \times \{t\})$ is
  non-empty for only finitely many $t$.  Thus $S$ is finite and
  $(S,\pi,B) \mapsto \oplus_{s \in S} B(s)$ defines a functor
  $\contrCatUcoef{G}{\contstraPW(X)|_\caly}{\calb} \to \calb_{\oplus}$.  It is straight forward to check that
  this functor is an equivalence.
    
  Let $f \colon X \times \IN \to X \times \IN$ be the shift $(x,t) \mapsto (x,t+1)$.  It
  is not difficult to check that $f$ induces an Eilenberg swindle on
  $\contrCatUcoef{G}{\contstraPW(X)}{\calb}$.  More precisely, Lemma~\ref{lem:formal-swindle} applies to $f$
  and $\contrCatUcoef{G}{\contstraPW(X)}{\calb}$.
\end{proof}

Recall from Subsection~\ref{subsec:comparison-discrete} that for the Farrell--Jones
Conjecture for a discrete group $\Gamma$ the group rings over virtually cyclic subgroups
of $\Gamma$ play a central role.  For td-groups the K-theory of the categories
$\contc_G(P;\calb)$ will take this role.  Let $G$ be a td-group and $V \in \CVCYC$.  Our
difficulty is that, if $V$ is closed but not open in $G$, there is no inclusion of
$\calh(V;R)$ into $\calh(G;R)$.  For a category $\calb$ with $G$-support we can restrict
to $V$ and only consider morphisms whose support is contained in $V$.  However, as the
$G$-support is typically open this is not sensible and it is not clear how one might
exhibit a subcategory associated to $V$.  But once we use the (K-theoretic) deloopings
$\contrCatUcoef{G}{\contstraPW(X),\caly}{\calb}$ of $\calb$ this changes;
$\contrCatUcoef{G}{\contstraPW(X),\caly}{\calb}$ has many subcategories.  For example for
any $G$-control structure $\mfE$ on $X \times \IN$ that is contained in $\contstraPW(X)$
we obtain a subcategory $\contrCatUcoef{G}{\mfE,\caly}{\calb}$.  The $\IN$-factor in
$X \times \IN$ allows us create $G$-control structures that become more restrictive with
$t \to \infty$.  In our definition later on we use this to approximate $V$ by smaller and
smaller neighborhoods of $V$ in $G$ as $t \to \infty$.  For the precise $G$-control
structure we use see Definition~\ref{def:mfc} later on.  We will only change
$\contstraPW_2(X)$ by adding what we call the foliated control condition.  (There are many
possible variations for this $G$-control structure; our choice is carefully tailored to
enable us to prove both the $\CVCYC$-Farrell--Jones Conjecture for p-adic groups and the
Reduction Theorem~\ref{thm:reduction}.)


\subsection{Two functors to $G$-spaces}\label{subsec:realizations-for-P}
We define two functors $\EPplus\All(G) \to G\text{-}\Spaces$ as follows.  We recall from
Subsection~\ref{subsec:product-cats} that objects in $\EPplus\All(G)$ are $n$-tuples
$(G/V_1, \ldots, G/V_n)$, where $n = 0$ is allowed and each $V_i$ is a closed subgroup of
$G$.  Given two such objects $(G/V_1, \ldots, G/V_n)$ and $(G/V_1', \ldots , G/V'_{n'})$,
a morphism $u \colon (G/V_1, \ldots, G/V_n) \to (G/V_1', \ldots, G/V_{n'})$ is given by a
function $u \colon \{1,\ldots, n'\} \to \{1, \ldots, n\}$ for which $V'_i = V_{u(i)}$
holds\footnote{if $n'=0$ there is precisely one such $u$, if $n' \ge 1$ and $n = 0$, then
  there is no such $u$.}.  The first functor\footnote{Unlike the second the first one
  factors over $\EPplus\Or(G)$.}, written as $P \mapsto |P|$, is defined on objects by
\begin{equation*}
  |(G/V_1,\dots,G/V_n)| := G/V_1 \times \cdots \times G/V_n. 
\end{equation*}
It sends a morphism $u \colon (G/V_1,\ldots,G/V_n) \to (G/V'_1,\ldots,G/V'_{n'})$ to the map
\begin{equation*}
  (x_1V_1,\dots,x_nV_n) \mapsto (x_{u(1)}V'_1,\dots,x_{u(n')}V'_{n'}).
\end{equation*}
The second one, written as $P \mapsto |P|^\wedge$, is defined on objects by
\begin{equation*}
  |G/V_1,\dots,G/V_n|^{\wedge} := G \times \cdots \times G = G^n. 
\end{equation*}
It sends a morphism $u \colon (G/V_1,\dots,G/V_n) \to (G/V'_1,\ldots,G/V'_{n'})$ to the map
\begin{equation*}
  (x_1,\dots,x_n) \mapsto (x_{u(1)},\dots,x_{u(n')}).
\end{equation*}
There is a canonical natural transformation $|\cdot|^\wedge \to |\cdot|$ given on
$(G/V_1,\dots,G/V_n)$ by the map
\begin{equation*}
  G \times \cdots \times G \to G/V_1 \times \cdots \times G/V_n,
  \quad (x_1,\dots,x_n) \mapsto (x_1 V_1,\dots,x_n V_n).
\end{equation*}  
For $P \in \EP\All(G)$ the action of $G$ on $|P|^\wedge$ is free, and we can think of
$|P|^\wedge \to |P|$ as a resolution.  For the empty tuple $\ast$, both, $|\ast|$ and
$|\ast|^\wedge$ are the empty product, i.e., a point.


\subsection{$V$-foliated distance}\label{subsec:foliated-distance-V}

Let $G$ be a td-group.  We can equip $G$ with a left invariant proper metric $d_G$ that
generates the topology of $G$, see~\cite[Thm.~4.5]{Haagerup-Przybyszewska(2006)}
or~\cite[Thm.~1.1]{Abels-Manoussos-Noskov-proper-inv-metric(2011)}.  Let $V$ be a closed
subgroup of $G$.  For $g,g' \in G$, $\beta \geq 0, \eta > 0$ we write
\begin{equation*}
  \fold{V}(g,g') < (\beta,\eta),
\end{equation*} 
iff there is $v \in V$ with $d_G(e,v) = d_G(g,gv) \leq \beta$ and $d_G(gv,g') < \eta$.
Similarly, for $g,g' \in G$, $\beta, \eta \geq 0$ we write
$\fold{V}(g,g') \leq (\beta,\eta)$, iff there is $v \in V$ with
$d_G(e,v) = d_G(g,gv) \leq \beta$ and $d_G(gv,g') \leq \eta$.  We will not consider
$< (\beta,0)$. 

The general idea here is to treat traveling in cosets of $V$ different from traveling in
arbitrary directions in $G$.  Typically, $\beta$ will be a bounded number, whereas $\eta$
will be a small number.  This definition is motivated by similar constructions for flow
spaces, see Subsection~\ref{subsec:V-fol-and-FS}. 

\begin{remark}
  We have $\fold{V}(g,g') \leq (\beta,0)$, iff $g^{-1}g' \in V$ and $d_G(g,g') \leq \beta$.
\end{remark}

\begin{remark}\label{rem:d-fol-open-U}
  Let $U$ be an open subgroup of $G$.  Then there is $\eta_0 > 0$ such that the
  $\eta_0$-neighborhood of $U$ is just $U$.  Thus for $g,g' \in G$ and $\eta < \eta_0$ we
  have
  \begin{equation*}
    \fold{U}(g,g') \leq (\beta,\eta) \quad \implies \quad \fold{U}(g,g') \leq (\beta+\eta,0).
  \end{equation*}
\end{remark}

\begin{remark}\label{rem:why-foliated-control}
  In our constructions of controlled categories later on (see Definition~\ref{def:C_G-P})
  we would ideally like to work with a $G$-invariant metric on $G/V$.  Typically we would
  be interested in small distances in $G/V$\footnote{For $V$ open in $G$ we could simply
    work with any discrete metric on $G/V$, for example the metric that put different
    points at distance $1$. The difficulty here arises only if $G/V$ is not discrete.}.
  Often there are however no $G$-invariant metrics on $G/V$ (and neither are there
  $G$-invariant uniform structures on $G/V$).  The notion of $V$-foliated control on $G$
  is (left) $G$-invariant and will serve us as a replacement for $G/V$ with (a
  non-existing) $G$-invariant metric.  One way to think about this replacement is that we
  have to add to points in $G/V$ choices of lifts to $G$, where the choice of lifts is
  only relevant up to bounded distance in the fibers for $G \to G/V$.  On the level of
  flow spaces this corresponds to the difference between parametrized geodesics and their
  images.
\end{remark}


\subsection{$P$-foliated distance}\label{subsec:foliated-distance-P}

There is a natural extension of $V$-foliated distance to products.  For
$P = (V_1,\dots,V_n) \in \EPplus\All(G)$,
$g = (g_1,\dots,g_n), g' = (g'_1,\dots,g'_n) \in |P|^\wedge = G^n$ we
write
\begin{equation*} \fold{P}(g,g') < (\beta,\eta),
\end{equation*}  
iff $\fold{V_i}(g_i,g'_i) < (\beta,\eta)$ for $i=1,\dots,n$.  Similarly, we write
$\fold{P}(g,g') \leq (\beta,\eta)$, iff $\fold{V_i}(g_i,g'_i) \leq (\beta,\eta)$ for
$i=1,\dots,n$.  Note that if $P = \ast$ is the empty tuple, then
$\fold{P}(g,g') < (\beta,\eta)$ and $\fold{P}(g,g') \leq (\beta,\eta)$ are empty
conditions and thus always satisfied.  However, $|\ast|^\wedge$ is just a point so this
is sensible\footnote{Recall that we do not allow $\eta =0$ when considering
  $< (\beta,\eta)$}.

\begin{remark}\label{rem:d-fol-open-P}
  Remark~\ref{rem:d-fol-open-U} also applies to $P \in \EPplus\OPEN(G)$: if
  $\fold{P}(\lambda,\lambda') \leq (\beta,\eta)$ with sufficiently small $\eta$, then
  $\fold{P}(\lambda,\lambda') \leq (\beta+\eta,0)$.
\end{remark}

We will need the following version of the triangle inequality for $\fold{P}$.  Note that
in the statement $\delta$ depends not on $P$.
 
\begin{lemma}[Foliated triangle inequality]\label{lem:fol-triangle-V-fol}
  Let $\alpha \geq 0$.  Then for any $\epsilon > 0$ there is $\delta > 0$ such that for
  $P \in \EPplus\All(G)$ and $g,g',g'' \in |P|^\wedge = G^n$
  \begin{equation*}
    \fold{P}(g,g'), \; \fold{P}(g',g'') \leq (\alpha,\delta) \quad \implies \quad  \fold{P}(g,g'') \leq (2\alpha,\epsilon).
  \end{equation*}
\end{lemma}

\begin{proof}
  This is an easy consequence of~\ref{lem:unif-continuity:V} below.
\end{proof}

\begin{lemma}\label{lem:unif-continuity} \
  \begin{enumerate}[
                 label=(\thetheorem\alph*),
                 align=parleft, 
                 leftmargin=*,
                 labelindent=1pt,
                 ] 
  \item\label{lem:unif-continuity:M} Let $M \subseteq G$ be compact. For any
    $\epsilon > 0$ there is $\delta > 0$ such that for all $g,g' \in G$, $v \in M$ we have
    \begin{equation*}
      d_G(g,g') < \delta \implies d_G(gv,g'v) < \epsilon;
    \end{equation*}
  \item\label{lem:unif-continuity:V} Let $\alpha \geq 0$.  Then for any $\epsilon > 0$
    there is $\delta > 0$ such that for any closed subgroup $V$ of $G$ and
    $g,g',g'' \in G$ we haven
    \begin{equation*}
      \fold{V}(g,g'),  \fold{V}(g',g'') \leq (\alpha,\delta) \quad \implies \quad  \fold{V}(g,g'') \leq (2\alpha,\epsilon).
    \end{equation*}
  \end{enumerate}
\end{lemma}

\begin{proof}
   This is~\cite[Lem.~3.1]{Bartels-Lueck(2023almost)}.
\end{proof}


\subsection{The category $\contc_G(P)$}\label{subsec:category-contc}

\begin{definition}\label{def:mfc}
  Let $P \in \EPplus\All(G)$.  We define the $G$-control structure
  $\contstrc(P) = \big(\contstrc_1(P),\contstrc_2(P),\contstrc_G(P)\big)$ on
  $|P|^\wedge \times \IN$ as follows:
  \begin{itemize}[
                 label=$\bullet$,
                 align=parleft, 
                 leftmargin=*,
                 labelindent=5pt,
                 ] 
  \item $\contstrc_1(P)$ consists of all subsets $F$ of $|P|^\wedge \times \IN$ for which
    $F \cap |P|^\wedge \times \{t\}$ is finite for all $t \in \IN$;
  \item $\contstrc_2(P)$ consists of all subsets $E$ of
    $\big(|P|^\wedge \times \IN\big)^{\times 2}$ satisfying the following two conditions
    \begin{itemize}
    \item \emph{bounded control over $\IN$}: there is $\alpha > 0$ such that for all
      $\twovec{\lambda',t'}{\lambda,t} \in E$ we have $|t-t'|\leq \alpha$;
    \item \emph{foliated control over $|P|^\wedge$}: there is $\beta \geq 0$ such that for
      any $\eta > 0$ there is $t_0$ such that for all $t \geq t_0$ and all
      $\lambda,\lambda',t'$ we have
      \begin{equation*}
        \twovec{\lambda',t'}{\lambda,t} \in E  \quad \implies \quad \fold{P}(\lambda,\lambda') < (\beta,\eta);
      \end{equation*}
    \end{itemize}
  \item $\contstrc_G(P)$ consists of all relatively compact subsets of $G$. 
  \end{itemize}
  We write $\caly(P)$ for the collection of all subsets of $|P|^\wedge \times \IN$ that
  are contained in $|P|^\wedge \times \{0,\dots,N\}$ for some $N$.
\end{definition}

  It is an exercise to check that this is indeed a $G$-control structure.  To check that
$\contstrc_2(P)$ is closed under composition, the triangle inequality from
Lemma~\ref{lem:fol-triangle-V-fol} is used.

\begin{definition}\label{def:C_G-P}
  Let $G$ be a td-group and $\calb$ be a category with $G$-support.  We set
  \begin{equation*}
    \contc_G(P;\calb) := \contrCatUcoef{G}{\contstrc(P),\caly(P)}{\calb}. 
  \end{equation*} 	
  We will often drop $\calb$ from the notation and abbreviate $\contc_G(P) = \contc_G(P;\calb)$.
\end{definition}

The assignment $P \mapsto \contc_G(P)$ is functorial in
$\EPplus\All(G)$\footnote{$P \mapsto \contc_G(P)$ is not strictly functorial in
  $\EPplus\Or(G)$. This can be fixed, using a construction that will appear later in
  Subsection~\ref{subsec:functro-orbit-cat}. But for now we will ignore this.}.  We obtain an
$\EPplus\All(G)$-spectrum
\begin{equation*}
  P \mapsto \bfK \big( \contc_G(P) \big).
\end{equation*}


\subsection{The $\CVCYC$-Farrell--Jones Conjecture}

\begin{definition}[$\CVCYC$-assembly map]\label{def:Farrell-Jones-assembly-map} Let $G$
  be a td-group and $\calb$ be a category with $G$-support.  The maps $P \to \ast$ for
  $P \in \EP\CVCYC(G)$ induce a map
  \begin{equation}\label{eq:FJ-assembly-map}
    \hocolimunder_{P \in \EP\CVCYC(G)} \bfK \big( \contc_G(P;\calb) \big) \;
    \to \; \bfK \big( \contc_G(\ast;\calb) \big).
  \end{equation}
  This is the \emph{$\CVCYC$-assembly map for $\calb$}.
\end{definition}

In light of the Farrell--Jones Conjecture for discrete groups one might expect that the
homotopy colimit in~\eqref{eq:FJ-assembly-map} should be taken over $\Or_{\CVCYC}(G)$
instead of $\EP\CVCYC(G)$.  However, as discussed in
Subsection~\ref{subsec:comparison-discrete} it is an important point that we allow
products here.  On the other hand with slightly different definitions we could use
$\EP\Or_{\CVCYC}(G)$ in place of $\EP\CVCYC(G)$, see
Subsection~\ref{subsec:functro-orbit-cat}.

\begin{conjecture}[$\CVCYC$-Farrell--Jones Conjecture]%
\label{conj:Farrell-Jones-Conj-for-td} $ $ \\
  Let $G$ be a td-group and $\calb$ be a Hecke category with $G$-support. 
  Then the $\CVCYC$-assembly
  map~\eqref{eq:FJ-assembly-map} for $\calb$ is an equivalence.
\end{conjecture}

\begin{remark}\label{rem:CVCYC-FJ-is-about}
  By Proposition~\ref{prop:c-is-b-for-smooth-P} below
  (for $P = \ast$) we have  \begin{equation*}
    \Omega \bfK ( \contc_G(\ast;\calb) ) \simeq  \bfK (\calb_G[\ast]) 
      = \bfK (\calb).
  \end{equation*}
  If $\calb = \calb(G;R)$, then $\bfK(\calb) = \bfK(\calh(G;R))$ and so in this case
  the $\CVCYC$-Farrell--Jones Conjecture~\ref{conj:Farrell-Jones-Conj-for-td} is about
  the K-theory of the Hecke algebra $\calh(G;R)$.
\end{remark}

\begin{remark}\label{rem:CVCYC-FJ-implies-discrete-FJ}
  Let $\Gamma$ be a discrete group and $\cala$ be an additive category with a
  $\Gamma$-action.  One obtains a category $\cala[\Gamma]$ whose objects are the objects
  of $\cala$.  Morphisms $A \to A'$ in $\cala[\Gamma]$ are finite formal sums
  $\sum_{\gamma} \varphi_\gamma \cdot \gamma$ where
  $\varphi_\gamma \colon \gamma A \to A'$ is a morphism in $\cala$.  The K-theoretic
  Farrell--Jones conjecture with coefficients for $\Gamma$ concerns the K-theory of
  $\cala[\Gamma]$.  As $\cala[\Gamma]$ is a Hecke category with $\Gamma$-support in an
  obvious way, one can use Proposition~\ref{prop:cop-iso-via-PSub} and
  Proposition~\ref{prop:c-is-b-for-smooth-P} below to check that for discrete groups the
  $\CVCYC$-Farrell--Jones Conjecture~\ref{conj:Farrell-Jones-Conj-for-td} implies the usual
  K-theoretic Farrell--Jones Conjecture with coefficients.  In fact, for discrete groups the two conjectures are equivalent see~\cite[Remark~5.7]{Bartels-Lueck(2023foundations)}.
\end{remark}

\begin{theorem}[$\CVCYC$-Farrell--Jones Conjecture for reductive $p$-adic groups]%
\label{thm:Farrell-Jones-Conjecture-for-reductive-p-adic}
  $ $\\
  Let $G$ be a reductive $p$-adic group and $\calb$ be a Hecke category with $G$-support. 
  Then the $\CVCYC$-assembly map~\eqref{eq:FJ-assembly-map} for $\calb$ is an
  equivalence.
\end{theorem}

The formal framework of the proof of
Theorem~\ref{thm:Farrell-Jones-Conjecture-for-reductive-p-adic} is discussed in
Section~\ref{sec:formal-framwork}.  The proof is then carried out in
Sections~\ref{sec:cats-D_and_D0} to~\ref{sec:constr-transfer}.


\subsection{Relating $\contc_G(P;\calb)$ to $\calb[P]$}

\begin{proposition}\label{prop:c-is-b-for-smooth-P}
  Let $\calb$ be a category with $G$-support.  
  There is a zig-zag of weak equivalences
  between the $\EPplus\OPEN(G)$-$\Spectra$
  \begin{equation*}
    P \mapsto \Omega\bfK \big( \contc_G(P;\calb) \big) \quad \text{and} \quad P \mapsto \bfK \big( \calb[P] \big).
  \end{equation*}
\end{proposition}

We start the proof of Proposition~\ref{prop:c-is-b-for-smooth-P} with the following
observation.  Let $P \in \EPplus\OPEN(G)$.  As noted in Remark~\ref{rem:d-fol-open-P}, if
$\fold{P}(\lambda,\lambda') \leq (\beta,\eta)$ with small $\eta$, then
$\fold{P}(\lambda,\lambda') \leq (\beta+\eta,0)$.  This implies that (using the foliated
control condition) for $E \in \contstrc_2(P)$ there are $\beta > 0$ and $t_0 \in \IN$ such
that for all $t \geq t_0$ and all $\lambda,\lambda',t'$ we have
\begin{equation*}
  \twovec{\lambda',t'}{\lambda,t} \in E  \quad \implies \quad \fold{P}(\lambda,\lambda') < (\beta,0).
\end{equation*}
In the following definition we strengthen this to all $t$, not just sufficiently large
$t$.  This produce a control structure that is discrete over $|P|$ (with respect to the
projection $|P|^\wedge \to |P|$).

\begin{definition}
  Let $P \in \EPplus\OPEN(G)$.  We define the $G$-control structure
  $\contstrc^\dis(P) = \big(\contstrc^\dis_1(P), \contstrc^\dis_2(P),
  \contstrc^\dis_G(P)\big)$ as follows.  We set $\contstrc^\dis_1(P) := \contstrc_1(P)$,
  $\contstrc^\dis_G(P) := \contstrc_G(P)$ and define $\contstrc^\dis_2(P)$ to consist of
  all $E \in \contstrc_2(P)$ satisfying in addition the following: there is $\beta > 0$
  such that
  \begin{equation*}
    \twovec{\lambda',t'}{\lambda,t} \in E \quad \implies \quad \fold{P}(\lambda, \lambda') \leq (\beta,0).
  \end{equation*}
  We define
  $\contc_G^\dis(P) := \contrCatUcoef{G}{\contstrc^\dis(P),\caly(P)}{\calb}$.
\end{definition}

\begin{lemma}\label{lem:about-calc-dis}
  Let $P \in \EPplus\OPEN(G)$.
  \begin{enumerate}[
                 label=(\thetheorem\alph*),
                 align=parleft, 
                 leftmargin=*,
                 labelindent=1pt,
                 ] 
  \item\label{lem:about-calc-dis:dis-is-no-dis} The inclusion
    $\contc_G^\dis(P) \to \contc_G(P)$ is an equivalence.
  \item\label{lem:about-calc-dis:dis-is-B} The projection $|P|^\wedge \times \IN \to |P|$
    induces an equivalence
    $\contrCatUcoef{G}{\contstrc^\dis(P)|_{\caly(P)}}{\calb} \to (\calb[P])_\oplus$.
  \item\label{lem:about-calc-dis:swindle} The K-theory of the category
    $\contrCatUcoef{G}{\contstrc^\dis(P)}{\calb}$ vanishes.
  \end{enumerate}
\end{lemma}

\begin{proof}
  The first two are easy exercises in the Definitions.  The third comes from the standard
  Eilenberg swindle on $\contrCatUcoef{G}{\contstrc^\dis(P)}{\calb}$ using the shift
  $(\lambda,t) \mapsto (\lambda,t+1)$, i.e., Lemma~\ref{lem:formal-swindle}
  applies\footnote{It is instructive to note that this swindle does not work on
    $\contrCatUcoef{G}{\contstrc(P)}{\calb}$: For $\varphi \in \contrCatUcoef{G}{\contstrc(P)}{\calb}$ there can be
    $\twovec{\lambda',t'}{\lambda,t} \in \suppX \varphi$ where $\lambda$ and $\lambda'$
    have different images in $|P|$, i.e., there is $\eta$ such that
    $\fold{P}(\lambda,\lambda') < (\beta,\eta)$ fails regardless of $\beta$.  Then
    $\twovec{\lambda',t'+n}{\lambda,t+n} \in \suppX \varphi^\infty$ for all $n$ and this
    violates the foliated control condition over $|P|^\wedge$, i.e.,
    $\varphi^\infty \not\in \contrCatUcoef{G}{\contstrc(P)}{\calb}$. In other
    words~\ref{lem:formal-swindle-Z:E_2} fails.}.
\end{proof}

\begin{proof}[Proof of Proposition~\ref{prop:c-is-b-for-smooth-P}]
  The Karoubi sequence 
  \begin{equation*}
    \contrCatUcoef{G}{\contstrc^\dis(P)|_{\caly(P)}}{\calb} \to \contrCatUcoef{G}{\contstrc^\dis(P)}{\calb}
    \to \contrCatUcoef{G}{\contstrc^\dis(P),{\caly(P)}}{\calb} = \contc_G^\dis(P)
  \end{equation*}  
  induces a fibration sequence in K-theory, see~\eqref{eq:karoubi-sequence}.
  Using~\ref{lem:about-calc-dis:swindle} we obtain a weak equivalence
  $\Omega \bfK \big( \contc_G^\dis(P) \big) \xrightarrow{\sim} \bfK
  \big(\contrCatUcoef{G}{\contstrc^\dis(P)|_{\caly(P)}}{\calb} \big)$.
  Now~\ref{lem:about-calc-dis:dis-is-no-dis} and~\ref{lem:about-calc-dis:dis-is-B} give
  the result.
\end{proof}


\section{Formal framework of proof of the $\CVCYC$-Farrell--Jones Conjecture for reductive $p$-adic groups}%
\label{sec:formal-framwork}

The proof of the $\CVCYC$-Farrell--Jones Conjecture for reductive $p$-adic groups
(Theorem~\ref{thm:Farrell-Jones-Conjecture-for-reductive-p-adic}) is organized around two
functors
\begin{equation*}
  \bfD_G( \--;\calb) ,\; \bfD_G^0(\--;\calb) \; \colon \regularPOrGSC \; \to \; \Spectra.  
\end{equation*}  
We will define the source category below and then discuss some properties of these
functors.  Theorem~\ref{thm:Farrell-Jones-Conjecture-for-reductive-p-adic} is then an easy
consequence of these properties.  The functors
$\bfD_G(\--;\calb) = \bfK(\contd_G(\--;\calb))$ and
$\bfD_G^0(\--;\calb) = \bfK(\contd^0_G(\--;\calb))$ will be constructed in
Section~\ref{sec:cats-D_and_D0} as the K-theory of certain additive categories.  The
verification of their properties will occupy Sections~\ref{sec:cats-D_and_D0}
to~\ref{sec:constr-transfer}.  For most of these properties we can work with any category
with $G$-support $\calb$.  The exception is the transfer from
Theorem~\ref{thm:transfer-bfD0}, for which we need $\calb$ to be a Hecke category with
$G$-support.  Often we will drop $\calb$ from the notation and write
$\bfD_G( \--) = \bfD_G( \--;\calb)$ and $\bfD_G^0(\--) = \bfD_G^0(\--;\calb)$.


\subsection{$\calc$-simplicial complexes.}

\begin{definition}[$\calc$-simplicial complexes]
  Let $\calc$ be a category.  A \emph{$\calc$-simplicial complex} is a pair
  $\bfSigma = (\Sigma,C)$, where $\Sigma$ is a simplicial complex and
  $P \colon \simp(\Sigma)^{\op} \to \calc$ is a contravariant functor from the poset of
  simplices of $\Sigma$, ordered by inclusion, to $\calc$.  A map of $\calc$-complexes
  $(\Sigma,C) \to (\Sigma',C')$ is a pair $\bff = (f,\kappa)$, where
  $f \colon \Sigma \to \Sigma'$ is a simplicial map and $\kappa \colon C \to C' \circ f_*$
  is a natural transformation.  Here we write
  $f_* \colon \simp(\Sigma) \to \simp(\Sigma')$ for the map induced by $f$.

  The \emph{dimension} of $\bfSigma = (\Sigma,C)$ is the dimension of $\Sigma$; its
  \emph{$d$-skeleton} is $\bfSigma^d := (\Sigma^d, C|_{\simp(\Sigma^d)})$, where
  $\Sigma^d$ is the $d$-skeleton of $\Sigma$.
\end{definition}

\begin{definition}
  We define $\regularPOrGSC$ as the category of $\EPplus\All(G)$-simplicial complexes and
  write $\regularPOrGSCzero$ for the full subcategory of $0$-dimensional
  $\EPplus\All(G)$-simplicial complexes.
\end{definition}

We can think about $\regularPOrGSCzero$ as being obtained from $\EPplus\All(G)$ by adding
arbitrary coproducts to $\EPplus\All(G)$.  There is a product
\begin{equation}\label{eq:product-POrGSC-with-POrGSCzero}
  \begin{split}
    \regularPOrGSC \times \regularPOrGSCzero \qquad  & \to   \qquad \regularPOrGSC \\
    \big( (\Sigma,P), (\ZerodimRunderlyingSet,Q) \big) \;& \mapsto \; \big(\Sigma \times
    \ZerodimRunderlyingSet, (\sigma,\ZerodimRunderlyingSetElement) \mapsto P(\sigma)
    \times Q(\ZerodimRunderlyingSetElement) \big).
  \end{split}
\end{equation}
We write $\regularPOrGSCzerocalf{\calf}$ for the full subcategory of $\regularPOrGSCzero$
on all $(\ZerodimRunderlyingSet,P)$ where $P$ takes values in $\EP\calf(G)$.  A drawback
of our notation is that $\regularPOrGSCzerocalf{\All(G)} \subsetneq \regularPOrGSCzero$,
because the empty product $\ast$ is not contained in $\EP\All(G)$.  However, typically
$\calf$ will be a proper collection of subgroups, so this should not lead to serious
confusion.

\begin{example}\label{ex:bfJ_F}
  Let $\calf$ be a collection of closed subgroups of $G$.
  We write $\Sigma_\calf(G)$ for the following simplicial complex.  Vertices of
  $\Sigma_\calf(G)$ are pairs $(n,V)$ with $n \in \IN$ and $V \in \calf$.  Vertices
  $(n_0,V_0),\dots,(n_k,V_k)$ form a simplex of $\Sigma_\calf(G)$, if and only if the
  $n_i$ are pairwise distinct\footnote{Alternatively, $\Sigma_\calf(G)$ is the infinite
    join $\ast_{n \in \IN} (\coprod_{F \in \calf} G/F)$.}.  There is an evident functor
  $P_\calf(G) \colon \simp(\Sigma_\calf(G))^{\op} \to \EP\calf(G)$ that sends a simplex
  $\sigma = \big\{(n_0,V_0),\ldots,(n_k,V_k)\big\}$ to $(G/V_0,\ldots,G/V_k)$ with
  $n_0 < \cdots < n_k$, where we choose the numbering such that $n_0 < \cdots < n_k$.  We
  obtain $\bfJ_\calf(G) := \big(\Sigma_\calf(G),P_\calf(G)\big) \in \regularPOrGSC$.
   
  We write $\Sigma^N_\calf(G)$ for the subcomplex of $\Sigma_\calf(G)$ spanned by all
  vertices $(n,V)$ with $n \leq N$\footnote{Alternatively,
    $\Sigma^N_\calf(G) = \ast_{n \leq N} (\coprod_{F \in \calf} G/F)$.}.  Then
  $\Sigma^N_\calf(G)$ is a proper subcomplex of the $N$-skeleton
  $\big(\Sigma_\calf(G)\big)^N$.  For finite $\calf$, $\Sigma^N_\calf(G)$ is a finite
  complex, while the $N$-skeleton $\big(\Sigma_\calf(G)\big)^N$ is never finite.  We set
  $\bfJ^N_\calf(G) := \big(\Sigma^N_\calf(G),P_\calf(G)|_{\simp(\Sigma^N_\calf(G))}\big)$.
  
  We will discuss in Subsection~\ref{subsec:realization-PORG-simpl-cx} realization
  functors from $\regularPOrGSC$ to $G$-spaces.  The realization $|\bfJ_\calf(G)|$ of
  $\bfJ_\calf(G)$ is the numerable classifying spaces for
  $\calf$~\cite[A1]{Baum-Connes-Higson(1994)}, see Example~\ref{ex:bfJ_F-realizations}.
  This motivated the definition of $\bfJ_\calf(G)$.
\end{example}


\subsection{Coefficients of $\bfD_G$}\label{subsec:Coefficients-bfD}

Write
$I \colon \EPplus\All(G) \; \to \; \regularPOrGSCzero$ for the inclusion. 
The underlying simplicial complex of $I(P)$ consist of one vertex which is sent to $P$.
We will show in
Proposition~\ref{prop:coefficients-contd} that there exists a zig-zag of equivalences of
$\EPplus\All(G)$-spectra between $I^*\Omega\bfD_G(\--)$ and
$\bfK \big(\contc_G(\--)\big)$.
To ease notation we will often abbreviate $P = I(P)$ and omit $I^*$ from the notation.


\subsection{Computation of $\bfD_G$ on
  $\regularPOrGSCzero$}\label{subsec:computation-bfD-on-zero}

Let
$(\ZerodimRunderlyingSet,P) \in \regularPOrGSCzero$.  We will show in
Proposition~\ref{prop:computation-bfD-on-zero} that the canonical map

\begin{equation}\label{eq:bfD(B,P)-as-bigvee} \bigvee_{\ZerodimRunderlyingSetElement \in
    \ZerodimRunderlyingSet} \bfD_G(P(\ZerodimRunderlyingSetElement))
  \xrightarrow{\sim} \bfD_G((\ZerodimRunderlyingSet,P))
\end{equation}
is an equivalence.


\subsection{$\bfD_G^0$ determines $\bfD_G$}\label{subsec:bfD-vs-bfD0}

We will construct in
Proposition~\ref{prop:contd-vs-contd0} a diagram in $\regularPOrGSC\text{-}\Spectra$
\begin{equation*}
  \xymatrix{\bfD_G^0(\--) & \bfD_G^0(\--) \ar[r] \ar[l] & \bfD_G^0(\--) \\
    \bfD_G^0(\--) \ar[u] \ar[d] & \bfD_G^0(\--) \ar[r] \ar[l] \ar[u] \ar[d]
    & \bfD_G^0(\--) \ar[u] \ar[d] \\
    \bfD_G^0(\--) & \bfD_G^0(\--) \ar[r] \ar[l] & \bfD_G^0(\--) 
  }	
\end{equation*}
whose homotopy colimit is equivalent to $\bfD_G(\--)$.


\subsection{Homotopy invariance for $\bfD_G^0$}\label{subsec:homotopy-inv-bfD0}

Let
$\ZerodimR = (\ZerodimRunderlyingSet,P) \in \regularPOrGSCzero$.  Let
$\pi \colon \ZerodimRunderlyingSet \times \Delta^d \to \ZerodimRunderlyingSet$ be the projection. 
We obtain
\begin{equation*}
  \bfDelta^d_\ZerodimR := (\ZerodimRunderlyingSet \times \Delta^d,P \circ \pi_*) \in \regularPOrGSC.
\end{equation*}
A choice of a point $x_0 \in |\Delta^d|$ determines an inclusion
$\bfi \colon \ZerodimR \to \bfDelta^d_\ZerodimR$.  We show in
Proposition~\ref{prop:homotopy-inv-bfD0} that $\bfi$ induces an equivalence
\begin{equation*}
  \bfD_G^0(\ZerodimR)   \; \xrightarrow{\sim} \; 	\bfD_G^0(\bfDelta^d_\ZerodimR).
\end{equation*}


\subsection{Excision for $\bfD_G^0$}\label{subsec:excision-bfD0}

Let $\bfSigma = (\Sigma,P)$ in
$\regularPOrGSC$ be $d$-dimensional.  Assume that the vertices of $\Sigma$ are locally
ordered; then any simplex of $\Sigma$ is canonically isomorphic to a standard simplex
$\Delta^k$.  Let $B$ be the set of $d$-simplices of $\Sigma$.  We obtain a canonical map
$f \colon B \times \Delta^d \to \Sigma$.  Let $\pi \colon B \times \Delta^d \to B$ be the
projection.  Set $\widehat\Sigma := B \times \Delta^d$, $\widehat P := p|_B \circ \pi_*$
and $\widehat \bfSigma := ( \widehat\Sigma, \widehat P)$.  Let
$\bff := (f,\kappa) \colon \widehat\bfSigma \to \bfSigma$ where $\kappa$ is defined as
follows.  Let $\tau$ be a simplex of $\widehat\Sigma$.  Then $\tau$ is contained in
$\{ \sigma \} \times \Delta^d$ for some $d$-simplex $\sigma$ of $\Sigma$ and we have
$\widehat P(\tau) = P(\sigma)$ and $f(\tau) \subseteq \sigma$.  We define
$\kappa_\tau \colon \widehat P(\tau) \to P(f(\tau))$ as the evaluation of $P$ on the
inclusion $f(\tau) \subseteq \sigma$.  Let $\bf\Sigma' = (\Sigma',P')$ be the
$(d\,\text{-}1)$-skeleton of $\bfSigma$ and $\widehat \bfSigma' = (\Sigma',P')$ be the
$(d\,\text{-}1)$-skeleton of $\widehat\bfSigma$.  Then $\bff$ restricts to
$\bff' \colon \widehat\bfSigma' \to \bfSigma'$.  We write
$\widehat \iota \colon \widehat\bfSigma' \to \widehat\bfSigma$ and
$\iota \colon \bfSigma' \to \bfSigma$ for the canonical inclusions and obtain
\begin{equation}\label{eq:excison-diagram-hat-no-hat} \xymatrix{\widehat\bfSigma'
    \ar[r]^{\bff'} \ar[d]^{\widehat \iota}
    &  \bfSigma'  \ar[d]^{\iota} \\
    \widehat\bfSigma \ar[r]^{\bff} & \bfSigma. 
  }
\end{equation}
We show in Proposition~\ref{prop:excision-bfD0} that $\bfD_G^0(\--)$ 
takes this diagram to a homotopy pushout diagram of spectra.


\subsection{Skeleton continuity of $\bfD_G^0$}\label{subsec:continuity-bfD0}

Let $\bfSigma \in \regularPOrGSC$.  We show
in Proposition~\ref{prop:continuity-contd0} that the canonical map
\begin{equation*}
  \hocolimunder_{d \in \IN} \bfD_G^0(\bfSigma^d) \xrightarrow{\sim} \bfD_G^0(\bfSigma)
\end{equation*}
is an equivalence.


\subsection{Transfer}\label{subsec:transfer-bfD0}

We use $\bfJ_{\CVCYC}(G)$ from
Example~\ref{ex:bfJ_F} and consider
\begin{eqnarray*}
  \bfD_G^0\big(\bfJ_\CVCYC(G) \times \--;\calb \big) \colon \regularPOrGSC \to \Spectra.
\end{eqnarray*}
The projections $\bfJ_\CVCYC(G) \times P \to P$ induce a projection 
\begin{equation*}
  \bfp \colon \bfD_G^0\big(\bfJ_\CVCYC(G) \times \--;\calb\big) \to \bfD_G^0(\--;\calb)
\end{equation*}
in $\regularPOrGSC$-spectra.
  
\begin{theorem}\label{thm:transfer-bfD0}
  Assume that $G$ is a reductive $p$-adic group and that $\calb$ is a Hecke category with
  $G$-support.  Then the projection $\bfp$ admits a section, i.e., there is
  $\bftr \colon \bfD_G^0(\--;\calb) \to \bfD_G^0 \big(\bfJ_\CVCYC(G) \times
  \--;\calb\big)$ such that $\bfp \circ \bftr$ is equivalent to the identity in
  $\regularPOrGSCzero$-spectra. 
\end{theorem}

\begin{proof}[Proof of Theorem~\ref{thm:Farrell-Jones-Conjecture-for-reductive-p-adic}
  modulo properties of $\bfD_G$ and $\bfD_G^0$]  
  We need to show that the
  $\CVCYC$-assembly map
  \begin{equation}\label{eq:FJ-assembly-map-again}
    \hocolimunder_{P \in \EP\CVCYC(G)} \bfK \big( \contc_G(P) \big) \;
    \to \; \bfK \big( \contc_G(\ast) \big)
  \end{equation} 
  is an equivalence.  By the equivalence from Subsection~\ref{subsec:Coefficients-bfD} we
  can equivalently show that
  \begin{equation}\label{eq:FJ-assembly-map-with-D}
    \hocolimunder_{P \in \EP\CVCYC(G)} \bfD_G(P) \; \to \; \bfD_G(\ast)
  \end{equation} 
  is an equivalence.  We obtain the following factorization
  of~\eqref{eq:FJ-assembly-map-with-D}
  \begin{equation}\label{eq:FJ-assembly-map-factorization}
    \hocolimunder_{P \in \EP\CVCYC(G)} \bfD_G(P) \; \to \;
    \hocolimunder_{(\ZerodimRunderlyingSet,P) \in \regularPOrGSCzerocalf{\CVCYC}}
    \bfD_G(\ZerodimRunderlyingSet,P) \; \to \; \bfD_G(\ast).
  \end{equation}
  For fixed $(\ZerodimRunderlyingSet,P) \in \regularPOrGSCzerocalf{\CVCYC}$ consider the
  canonical map
  \begin{equation}\label{eq:bfD_upper_0(B,P)-as-hocolim} \hocolimunder_{Q \in I
      \downarrow (\ZerodimRunderlyingSet,P)} \bfD_G(Q) \to
    \bfD_G(\ZerodimRunderlyingSet,P)
  \end{equation} 
  where $I$ denotes the inclusion $\EP\CVCYC(G) \to \regularPOrGSCzerocalf{\CVCYC}$.  It
  is not difficult to identify~\eqref{eq:bfD_upper_0(B,P)-as-hocolim}
  with~\eqref{eq:bfD(B,P)-as-bigvee}, which is an equivalence.  The transitivity
  Lemma~\ref{lem:transitivity} for homotopy colimits implies now that the first map
  in~\eqref{eq:FJ-assembly-map-factorization} is an equivalence.  As $\bfD_G$ can be
  expressed as a homotopy colimit in $\bfD_G^0$, see Subsection~\ref{subsec:bfD-vs-bfD0},
  the second map in~\eqref{eq:FJ-assembly-map-factorization} is an equivalence if
  \begin{equation}\label{eq:FJ-assembly-map-D0}
    \hocolimunder_{(\ZerodimRunderlyingSet,P) \in \regularPOrGSCzerocalf{\CVCYC}}
    \bfD_G^0(\ZerodimRunderlyingSet,P) \; \to \; \bfD_G^0(\ast)
  \end{equation} 
  is an equivalence.  Theorem~\ref{thm:transfer-bfD0} implies
  that~\eqref{eq:FJ-assembly-map-D0} is a retract of
  \begin{equation}\label{eq:FJ-assembly-map-D0-J}
    \hocolimunder_{(\ZerodimRunderlyingSet,P) \in \regularPOrGSCzerocalf{\CVCYC}}
    \bfD_G^0 \big(\bfJ_\CVCYC(G) \times (\ZerodimRunderlyingSet,P)\big) \;
    \to \; \bfD_G^0(\bfJ_\CVCYC(G)).
  \end{equation}
  We now use that $\bfD_G^0$ is homotopy invariant
  (Subsection~\ref{subsec:homotopy-inv-bfD0}), satisfies an excision result
  (Subsection~\ref{subsec:excision-bfD0}) and skeleta continuity
  (Subsection~\ref{subsec:continuity-bfD0}).  These properties imply that
  $\bfD_G^0 \big(\bfJ_\CVCYC(G) \times \--\big)$ can in $\regularPOrGSCzero$-spectra be
  constructed as a homotopy colimit of functors of the form
  $\bfD_G^0\big((\ZerodimRunderlyingSet_0,P_0) \times \--\big)$ with
  $(\ZerodimRunderlyingSet_0,P_0) \in \regularPOrGSCzerocalf{\CVCYC}$.
  Lemma~\ref{lem:hocolim-times-Q} implies that for all
  $(\ZerodimRunderlyingSet_0,P_0) \in \regularPOrGSCzerocalf{\CVCYC}$
  \begin{equation*}
    \hocolimunder_{(\ZerodimRunderlyingSet,P) \in \regularPOrGSCzerocalf{\CVCYC}}
    \bfD_G^0 \big((\ZerodimRunderlyingSet_0,P_0) \times (\ZerodimRunderlyingSet,P)\big) \;
    \xrightarrow{\sim} \; \bfD_G^0(\ZerodimRunderlyingSet_0,P_0)
  \end{equation*}
  is an equivalence.  Thus~\eqref{eq:FJ-assembly-map-D0-J} is an equivalence and so
  is~\eqref{eq:FJ-assembly-map-again}.
\end{proof}

\begin{remark}
  It is possible to show that the functor $\bfD_G$, which we construct later on, also
  satisfies homotopy invariance, excision and continuity exactly as $\bfD_G^0$.  Moreover,
  Theorem~\ref{thm:transfer-bfD0} holds also for $\bfD_G$.  Thus it is possible to prove
  the Farrell--Jones Conjecture for reductive $p$-adic groups using only $\bfD_G$.  In
  fact, homotopy invariance, excision and continuity for $\bfD_G$ can be proven in exactly
  the same way as for $\bfD_G^0$.  However, the construction of the transfer map $\bftr$ in
  Theorem~\ref{thm:transfer-bfD0} is technically easier for $\bfD_G^0$ than for $\bfD_G$;
  this is the reason for our small detour through $\bfD_G^0$.  The other way round, we
  cannot replace $\bfD_G$ with $\bfD_G^0$ throughout; the equivalences from
  Subsections~\ref{subsec:Coefficients-bfD} and~\ref{subsec:computation-bfD-on-zero} do
  not hold for $\bfD_G^0$ in place of $\bfD_G$.
\end{remark}


\section{The categories $\contd_G(\bfSigma)$ and $\contd^0_G(\bfSigma)$}%
\label{sec:cats-D_and_D0}

In this section we construct the two functors $\bfD_G$ and $\bfD_G^0$ promised in
Section~\ref{sec:formal-framwork} as the K-theory of functors $\contd_G(\--)$ and
$\contd_G^0(\--)$ to additive categories.  We will need some preparations.


\subsection{Some notation for simplicial complexes}
Let $\Sigma$ be an (abstract) simplicial complex.  We write $\vertices(\Sigma)$ for the
set of vertices of $\Sigma$.  We will write $|\Sigma|$ for the realization of $\Sigma$ to
topological spaces.  For a simplex $\sigma$ of $\Sigma$ we write $\Delta_\sigma$ for the
simplicial subcomplex of $\Sigma$ spanned by $\sigma$ and $\partial \Delta_\sigma$ for its
boundary.  So $\partial \Delta_\sigma$ is obtained from $\Delta_\sigma$ by omitting
$\sigma$.  For a vertex $v \in \vertices(\Sigma)$ we will not distinguish between the
abstract vertex $v$ and the corresponding point $v \in |\Sigma|$.  Any point
$x \in |\Sigma|$ has unique barycentric coordinates,
$x = \sum_{v \in \vertices(\Sigma)} x(v) \cdot v$ with $x(v) \in [0,1]$,
$\sum_{v \in \vertices(\Sigma)} x(v) = 1$, $x(v) \neq 0$ for only finitely many $v$.  Of
course, $\sigma := \{v \mid x(v) \not= 0\}$ forms a simplex with
$x \in |\Delta_\sigma| \setminus |\partial\Delta_\sigma|$.
The $\ell^\infty$-metric\footnote{In general, the topology of the $\ell^\infty$-metric is
  coarser than the weak topology on $|\Sigma|$, but we will mostly only use it on finite
  subcomplexes of $\Sigma$, where both topologies coincide.  Also, on finite dimensional
  subcomplexes the $\ell^\infty$-metric and the $\ell^1$-metric (that we used for example
  in~\cite{Bartels-Lueck-Reich(2008hyper)}) are Lipschitz equivalent.  Using the
  $\ell^\infty$-metric is more convenient here, but there is no substantial difference.}
on $|\Sigma|$ is
\begin{equation}
  d^\infty(x,x') := \max_{v \in \vertices(\Sigma)} |x(v) - x'(v)|.
\end{equation}  
For $\sigma \in \simp(\Sigma)$ we set
\begin{equation*}
  U_{\sigma} := \big\{ x \in |\Sigma| \; \big|  \; \forall v \in \sigma : x(v) > 0 \big\},
\end{equation*}
this is the open star of $\sigma$, i.e., the union of the interiors of those simplices
which contain $\sigma$ as face.  It is an open neighborhood of
$|\Delta_\sigma| \setminus |\partial \Delta_\sigma|$.  For $\epsilon > 0$ we set
\begin{equation*}
  K_{\sigma,\epsilon} := \{ x \in |\Sigma| \mid \forall v \in \sigma : x(v) \geq \epsilon \}.
\end{equation*}
This is a closed subset of $U_\sigma$.  We record that the $K_{\sigma,\epsilon}$ get
larger with decreasing $\epsilon$ and that
$U_\sigma = \bigcup_{\epsilon > 0} K_{\sigma,\epsilon}$.  Moreover $U_\sigma$ is the
$\epsilon$-neighborhood of $K_{\sigma,\epsilon}$ with respect to $d^\infty$.
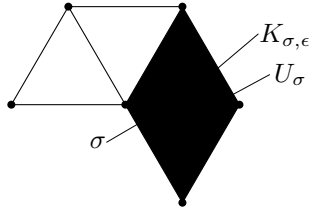
\begin{figure}[H]
  \begin{tikzpicture}
    \fill [black,opacity=.07] (1.5,0) -- (2.25,1.3) -- (3,0) -- (2.25,-1.3) -- (1.5,0);
    \fill [black,opacity=.15] (1.7,0) -- (2.25,1.02) -- (2.8,0) -- (2.25,-1.02) --
    (1.7,0); \draw [line width=.06mm] (1.7,0) -- (2.25,1.02) -- (2.8,0) -- (2.25,-1.02) --
    (1.7,0); \fill [black,opacity=.5] (0,0) circle (1.5pt) (1.5,0) circle (1.5pt) (3,0)
    circle (1.5pt) (.75,1.3) circle (1.5pt) (2.25,1.3) circle (1.5pt) (2.25,-1.3) circle
    (1.5pt); \draw (0,0) -- (1.5,0) -- (3,0) -- (2.25,1.3) -- (.75,1.3) -- (1.5,0) --
    (2.25,1.3) (1.5,0) -- (2.25,-1.3) -- (3,0) (0,0) -- (.75,1.3); \draw [line width=.5mm]
    (1.5,0) -- (3,0); \draw [line width=.06mm] (2.85,.1) -- (3.4,.4); \draw (3.67,.4) node
    {$U_\sigma$}; \draw [line width=.06mm] (2.5,.3) -- (3.25,.94); \draw (3.6,.9) node
    {$K_{\sigma,\epsilon}$}; \draw [line width=.06mm] (2.1,0) -- (1.25,-.5); \draw
    (1.14,-.5) node {$\sigma$};
  \end{tikzpicture}
  \caption{$U_\sigma$ and $K_{\sigma,\epsilon}$}\label{fig:U-sigma-K-sigma-eps}
\end{figure}


\subsection{The $G$-spaces $|\bfSigma|$ and $|\bfSigma|^\wedge$}%
\label{subsec:realization-PORG-simpl-cx}

Write $\Delta \colon \simp(\Sigma) \to \Spaces$ for the functor
$\sigma \mapsto |\Delta_\sigma|$ and define the realization functor
\begin{eqnarray*}
  |\--| \; \colon \regularPOrGSC & \to & \SpacesC{G} \\
  \bfSigma = (\Sigma,P) & \mapsto &  \Delta \times_{\simp(\Sigma)}  |P(-)|.
\end{eqnarray*}
In the construction of $\contc_G(P)$ for $P \in \EPplus\All(G)$ we used the $G$-space
$|P|^\wedge$ from Subsection~\ref{subsec:realizations-for-P}, thought of as a resolution
of $|P|$.  We will need a similar resolution for $|\bfSigma|$.  Define
\begin{eqnarray*}
  |\--|^\wedge \; \colon \regularPOrGSC & \to & \SpacesC{G} \\
  \bfSigma = (\Sigma,P) & \mapsto & \Delta \times_{\simp(\Sigma)} |P(\--)|^\wedge. 
\end{eqnarray*}
The projections maps $|P|^\wedge \to \ast$ induce a map
\begin{equation*}
  p_\bfSigma \colon |\bfSigma|^\wedge \to |\Sigma| = \Delta \times_{\simp{\Sigma}} \ast. 
\end{equation*}
Let $\sigma$ be a simplex of $\Sigma$.  It comes with a canonical map
\begin{equation*}
  |\Delta_\sigma| \times |P(\sigma)|^\wedge \to |\bfSigma|^\wedge.
\end{equation*}
We will write this map as $(x,\lambda) \mapsto [x,\lambda]_\sigma$.  Of course
$p_\bfSigma \big( [x,\lambda]_\sigma \big) = x$.  Altogether these canonical maps define
the projection
\begin{equation*}
  \coprod_{\tau} |\Delta_\tau| \times |P(\tau)|^\wedge \; \xrightarrow{q_\bfSigma} \; |\bfSigma|^{\wedge}
\end{equation*}
which is an identification of topological spaces.  For a simplex $\sigma$ the preimage of
$U_{\sigma} \subseteq |\Sigma|$ under $p_\bfSigma \circ q_\bfSigma$ is
$\coprod_{\tau, \sigma \le \tau} \big( |\Delta_{\tau}| \setminus
\partial_{\sigma}|\Delta_{\tau}| \big) \times |P(\tau)|^\wedge$, where
$\partial_{\sigma}|\Delta_{\tau}| = \bigcup_{\mu \subseteq \tau, \sigma \not\subseteq
  \tau} |\Delta_{\mu}|$.  We define
$\lambda_{\sigma,\tau} \colon \big( |\Delta_{\tau}| \setminus
\partial_{\sigma}|\Delta_{\tau}| \big) \times |P(\tau)|^\wedge \to |P(\sigma)|^\wedge$ to
be the composite of the projection
$\big( |\Delta_{\tau}| \setminus \partial_{\sigma}|\Delta_{\tau}|\big) \times
|P(\tau)|^\wedge \to |P(\tau)|^\wedge$ with the map
$|P(\tau)|^\wedge \to |P(\sigma)|^\wedge$ induced by $\sigma \subseteq \tau$.  One easily
checks that the map
\[
  \coprod_{\tau, \sigma \le \tau} \lambda_{\sigma,\tau} \; \; \colon \coprod_{\tau, \sigma
    \le \tau} \big(|\Delta_{\tau}| \setminus \partial_{\sigma}|\Delta_{\tau}| \big) \times
  |P(\tau)|^\wedge \to |P(\sigma)|^\wedge
\]
factorizes over the projection
$\coprod_{\tau, \sigma \le \tau} |\Delta_{\tau}| \setminus
\partial_{\sigma}|\Delta_{\tau}|\times |P(\tau)|^\wedge \to
(p_\bfSigma)^{-1}(U_{\sigma})$, 
to a map $\lambda_\sigma \colon (p_\bfSigma)^{-1}(U_{\sigma}) \to |P(\sigma)|^\wedge$.  We
note that $\lambda_\sigma \big( [x,\lambda]_\sigma \big) = \lambda$.

\begin{example}[Realizations of $\bfJ_\calf(G)$ and
  $\bfJ^N_\calf(G)$]\label{ex:bfJ_F-realizations}
  Let $\bfJ_\calf(G) = (\Sigma_\calf(G),P_\calf(G))$ be as in Example~\ref{ex:bfJ_F}.  It
  is not hard to check that then
  $|\bfJ_\calf(G)| = \ast_{n \in \IN} (\coprod_{F \in \calf} G/F)$ and
  $|\bfJ_\calf(G)|^\wedge = \ast_{n \in \IN} (\coprod_{F \in \calf} G) = \ast_{n \in \IN}
  (G \times \calf)$ hold.  The canonical projection
  $(\coprod_{F \in \calf} G) = G \times \calf \to \calf$ induces the projection
  $p_{\bfJ_\calf(G)} \colon |\bfJ_\calf(G)|^\wedge = \ast_{n \in \IN} (G \times \calf) \to
  \ast_{n \in \IN} (\calf) = |\Sigma_\calf(G)|$.
	
  Similarly, $|\bfJ^N_\calf(G)| = \ast_{n \leq N} (\coprod_{F \in \calf} G/F)$ and
  $|\bfJ_\calf(G)|^\wedge = \ast_{n \leq N} (\coprod_{F \in \calf} G) = \ast_{n \leq N} (G
  \times \calf)$.  In this description of points in $|\bfJ_\calf(G)|^\wedge$ can be
  written as $z = [t_0 \cdot (g_0,V_0),\dots,t_N \cdot (g_N,V_N)]$ with $t_i \in [0,1]$,
  $g_i \in G$, $V_i \in \calf$ where $\sum t_i = 1$.  In this notation
  $[t_0 \cdot (g_0,V_0),\dots,t_N \cdot (g_N,V_N)] = [t'_0 \cdot (g'_0,H'_0),\dots,t'_{N}
  \cdot (g'_{N},H'_{N})]$ if and only if $t_i = t'_{i}$ for $i = 0,\dots,N$, and
  $(g_i,V_i) = (g'_i,H'_i)$ for all $i$ with $t_i = t'_i\neq 0$.
\end{example}


\subsection{Foliated distance in $|\bfSigma|^\wedge$}\label{subsec:fol-dis-bfSigma-wedge}

We extend the notion of foliated distance from Subsection~\ref{subsec:foliated-distance-P}
to $\regularPOrGSC$.  Let $\bfSigma = (\Sigma,P) \in \regularPOrGSC$ and
$\beta, \eta, \epsilon > 0$.  For $z,z' \in |\bfSigma|^\wedge$ we write
\begin{equation*}
  \fold{\bfSigma}(z,z') < (\beta,\eta,\epsilon),
\end{equation*}
iff the following two conditions are satisfied~\refstepcounter{theorem}
\begin{enumerate}[
                 label=(\thetheorem\alph*),
                 align=parleft, 
                 leftmargin=*,
                 labelindent=1pt,
                 ] 
\item\label{nl:bfSigma-fol-distance:eps}
  $d^\infty\big(p_\bfSigma(z),p_\bfSigma(z')\big) < \epsilon$;
\item\label{nl:bfSigma-fol-distance:beta-eta} for all $\sigma \in \simp(\Sigma)$ with
  $p_\Sigma(z) \in K_{\sigma,\epsilon}$ or $p_\Sigma(z') \in K_{\sigma,\epsilon}$, we
  require
  \begin{equation*}
    \fold{P(\sigma)} \big(\lambda_\sigma(z),\lambda_\sigma(z')\big) < (\beta,\eta).
  \end{equation*}  
\end{enumerate}
Note that the first condition implies that if $p_\Sigma(z) \in K_{\sigma,\epsilon}$ or
$p_\Sigma(z') \in K_{\sigma,\epsilon}$, then both, $p_\Sigma(z)$ and $p_\Sigma(z')$,
belong to $U_{\sigma}$, and both, $\lambda_\sigma(z)$ and $\lambda_\sigma(z')$, are
defined.

We note that this definition is compatible with restrictions to subcomplexes.  More
precisely, let $\Sigma' \subseteq \Sigma$ be a subcomplex and let
$\bfSigma' := (\Sigma',P|_{\simp(\Sigma')})$.  If $\sigma$ is a simplex of $\Sigma'$, then
$K^{\Sigma'}_{\sigma,\epsilon} = \Sigma' \cap K^{\Sigma}_{\sigma,\epsilon}$, where the
upper index indicates in which complex we form $K_{\sigma,\epsilon}$.  Thus for
$z,z' \in |\bfSigma'|^\wedge \subseteq |\bfSigma|^\wedge$ we have
$\fold{\bfSigma'}(z,z') < (\beta,\eta,\epsilon)$ iff
$\fold{\bfSigma}(z,z') < (\beta,\eta,\epsilon)$.

\begin{remark}\label{rem:why-foliated-bfSigma-control}
  Recall that we think of $V$-foliated distance on $G$ as a way to get around the problem
  that there may not exist $G$-invariant metrics on $G/V$, see
  Remark~\ref{rem:why-foliated-control}
	
  Given $\bfSigma = (\Sigma,P) \in \regularPOrGSC$, we would ideally like to equip the
  $G$-space $|\bfSigma|$ with a $G$-invariant metric (and as for $G/V$ we would be
  interested in small distances in $|\bfSigma|$).  However, this space can have isotropy
  groups for which the orbit $G/V$ admits no $G$-invariant metric and then neither does
  $|\bfSigma|$.  The notion of foliated distance in $|\bfSigma|^\wedge$ is our replacement
  for $|\bfSigma|$ with (a non-existing) $G$-invariant metric.  A way to think about this
  replacement is that we add to points in $|\bfSigma|$ a choice of lift to
  $|\bfSigma|^\wedge$, where the choice of lift is only relevant up to bounded distance in
  the fibers for $|\bfSigma|^\wedge \to |\bfSigma|$.  Alternatively, we can think of this
  as adding to points in $|\Sigma|$ a choice of lift to $|\bfSigma|^\wedge$, where the
  choice of lift is only relevant up to foliated distance in the fibers for
  $p_\bfSigma \colon |\bfSigma|^\wedge \to |\Sigma|$.
\end{remark}

\begin{remark}\label{rem:two-stage}
  One should think of $\fold{\bfSigma}(z,z') < (\beta,\eta,\epsilon)$ as a two stage
  condition.  The first stage just uses the images of $z,z'$ in $|\Sigma|$ and requires
  their distance to be $< \epsilon$.  For general $z$ and $z'$ we can not compare the two
  fibers for $p_\bfSigma \colon |\bfSigma|^\wedge \to |\Sigma|$ containing them.  But
  whenever one of $z$ and $z'$ projects into $K_{\sigma,\epsilon}$, then
  $\lambda_\sigma(z),\lambda_\sigma(z') \in P(\sigma)$ are both defined and we require
  $\fold{P(\sigma)}( \lambda_\sigma(z),\lambda_\sigma(z')) < (\beta,\eta)$ in the second
  stage of the condition.
    
  We point out the following detail about the second stage.  Recall
  $K_{\sigma,\epsilon} \subseteq U_{\sigma}$.  One might be tempted to require
  $\fold{P(\sigma)}( \lambda_\sigma(z),\lambda_\sigma(z')) < (\beta,\eta)$ whenever both
  $z$ and $z'$ project into $U_\sigma$, as this suffices for
  $\lambda_\sigma(z),\lambda_\sigma(z')$ to be defined.  However, if neither $z$ nor $z'$
  projects into $K_{\sigma,\epsilon}$ (but their images in $|\Sigma|$ are close), then
  both $z$ and $z'$ are close to $p_{\bfSigma}^{-1}\big(|\Delta_\tau|\big)$ for a simplex
  $\tau$ (of smaller dimension than $\sigma$) and in passing from $z$ to $z'$ one might
  take a short-cut through $p_{\bfSigma}^{-1}(|\Delta_\tau|)$ and mostly avoid
  $p_{\bfSigma}^{-1}(|\Delta_\sigma|)$.  Thus in this situation
  $\fold{P(\sigma)}\big(\lambda_\sigma(z),\lambda_\sigma(z')\big)$ is not necessarily
  relevant in comparing $z$ and $z'$.  Requiring that that at least one of $z$ and $z'$
  projects into $K_{\sigma,\epsilon}$ avoids this problem.
  
  More formally, our formulation of~\ref{nl:bfSigma-fol-distance:eps}
  and~\ref{nl:bfSigma-fol-distance:beta-eta} guarantees (using
  Lemma~\ref{lem:fol-triangle-V-fol}) the following version of the triangle inequality.
  Given $\beta > 0$, $\eta > 0$ there is $\rho > 0$ such that for $\epsilon > 0$,
  $z,z',z'' \in |\bfSigma|^\wedge$ we have
  \begin{equation*}
    \fold{\bfSigma}(z,z') < (\beta,\rho,\epsilon), \fold{\bfSigma}(z',z'') < (\beta,\rho,\epsilon) 
\implies \fold{\bfSigma}(z,z'') < (2\beta,\eta,2\epsilon).
  \end{equation*}
\end{remark}

\begin{example}[Foliated distance for $\bfJ^N_\calf(G)$]\label{ex:bfJ_F-fol-distance}
  For $\bfJ^N_\calf(G) = (\Sigma^N_\calf(G),P_\calf(G))$ from Example~\ref{ex:bfJ_F} we
  can use the join description of
  $|\bfJ^N_\calf(G)|^\wedge = \ast_{n \leq N} (G \times \calf)$ from
  Example~\ref{ex:bfJ_F-realizations} to unravel the definition of the foliated distance
  as follows.  Let $z := [t_0 \cdot (g_0,V_0),\dots,t_N \cdot (g_N,V_N)]$,
  $z' := [t'_0 \cdot (g'_0,V'_0),\dots,t'_{N} \cdot (g'_{N},V'_{N})] \in
  |\bfJ^N_\calf(G)|^\wedge$.  Then $\fold{\bfJ^N_\calf(G)}(z,z') < (\beta,\eta,\epsilon)$
  if and only if
  \begin{enumerate}[label=(\thetheorem\alph*),leftmargin=*]
  \item $|t_i -t'_i| < \epsilon$ for all $i$;
  \item for all $i$ with $\max\{t_i,t'_i\}$ we have $V_i = V'_i$ and
    $\fold{V_i}(g_i,g'_i) < (\beta,\eta)$.
  \end{enumerate}
\end{example}


\subsection{The $G$-control structures $\contstrd(\bfSigma)$ and $\contstrd^0(\bfSigma)$}

Let $\bfSigma = (\Sigma,P) \in \regularPOrGSC$.  In the following definition we will
define a control structure $\contstrd({\bfSigma})$ on
$|\bfSigma|^{\wedge} \times \IN^{\times 2}$.  The two $\IN$-directions will be used to
encode two different control conditions.  The first factor will be used to encode a
foliated control conditions over the $P(\sigma)$, that is compatible with
$\contstrc(P(\sigma))$.  The second $\IN$-factor will be used to encode an
$\epsilon$-control condition over $|\Sigma|$ with respect to $d^\infty$.  In particular
Definition~\ref{def:D_bfSigma} is not symmetric in $t_0$ and $t_1$.  The remarks following
the definition provide some discussion and motivation.

We will use the $\ell^1$-norm $|(t_0,t_1)| = t_0 +t_1$ on $\IN^{\times 2}$.

\begin{definition}\label{def:D_bfSigma}
  Let $\bfSigma = (\Sigma,P) \in \regularPOrGSC$.  We define the $G$-control structure
  $\contstrd({\bfSigma}) = \big( \contstrd_1({\bfSigma}), \contstrd_2({\bfSigma}),
  \contstrd_G({\bfSigma}) \big)$ on the $G$-space
  $|\bfSigma|^{\wedge} \times \IN^{\times 2}$ as follows.
  \begin{enumerate}[
                 label=(\thetheorem\alph*),
                 align=parleft, 
                 leftmargin=*,
                 labelindent=1pt,
                 labelsep=1pt,
                 ] 
  \item\label{def:D_bfSigma:obj} $\contstrd_1({\bfSigma})$ consists of all subsets $F$ of
    $|\bfSigma|^{\wedge} \times \IN^{\times 2}$ satisfying the following conditions
    \begin{itemize}
    \item \emph{Finiteness over $\IN^{\times 2}$}: for all
      $\underlinetupel{t} \in \IN^{\times 2}$ the set
      $F \cap |\bfSigma|^{\wedge} \times \{ \underlinetupel{t} \}$ is finite;
    \item \emph{Compact support in $|\Sigma|$}: for every $t_0 \in \IN$ there exists a
      finite subcomplex $\Sigma_0$ of $\Sigma$ such that
      $F \cap |\bfSigma|^\wedge \times \{ t_0 \} \times \IN \subseteq
      p_\bfSigma^{-1}\big(|\Sigma_0|\big) \times \IN^{\times 2}$;
    \item \emph{Finite dimensional support}:
      $F \subseteq |\bfSigma^d|^\wedge \times \IN^{\times 2}$;
      
    \end{itemize}
  \item\label{def:D_bfSigma:X} $\contstrd_2({\bfSigma})$ consists of all subsets $E$ of
    $\big(|\bfSigma|^{\wedge} \times \IN^{\times 2}\big)^{\times 2}$ satisfying the
    following conditions
    \begin{itemize}
    \item \emph{Bounded control over $\IN^{\times 2}$}: there is $\alpha > 0$ such that if
      $\twovec{z',\underlinetupel{t}'}{z,\underlinetupel{t}} \in E$, then
      $|\underlinetupel{t}-\underlinetupel{t}'|\leq \alpha$;
    \item \emph{Foliated control over $\bfSigma$}: for any $\epsilon > 0$ there is
      $k_0 \in \IN$ such that for all $t_0 \in \IN_{\geq k_0}$ there is $\beta > 0$ such
      that for all $\eta > 0$ there is $k_1 \in \IN$ such that for all
      $t_1 \geq \IN_{\geq k_1}$ and all $z,z' \in |\bfSigma|^\wedge$,
      $\underlinetupel{t}' \in \IN^{\times 2}$, with $\underlinetupel{t} := (t_0,t_1)$ we
      have
      \begin{equation*}
        \qquad \twovec{z',\underlinetupel{t}'}{z,\underlinetupel{t}} \in E
        \; \implies \; \fold{\bfSigma}(z,z') < (\beta,\eta,\epsilon);
      \end{equation*}
    \end{itemize}
  \item\label{def:D_bfSigma:G} $\contstrd_G(\bfSigma)$ consists of all relatively compact
    subsets of $G$.
  \end{enumerate}

  It is an exercise to check that this is indeed a $G$-control structure.  To check that
  $\contstrd_2({\bfSigma})$ is closed under composition,  the triangle inequality from
  Remark~\ref{rem:two-stage} is used.
  
  We define the $G$-control structure
  $\contstrd^0(\bfSigma) = \big(
  \contstrd_1^0(\bfSigma),\contstrd_2^0(\bfSigma),\contstrd_G^0(\bfSigma) \big)$ as
  follows.  Set $\contstrd_1^0(\bfSigma) := \contstrd_1(\bfSigma)$,
  $\contstrd_G^0(\bfSigma) := \contstrd_G(\bfSigma)$.  We define $\contstrd_2^0(\bfSigma)$
  to consist of all $E \in \contstrd_2(\bfSigma)$ satisfying
  \begin{equation*}
    \qquad \twovec{z',\underlinetupel{t}'}{z,\underlinetupel{t}} \in E
    \; \implies \; \underlinetupel{t}' = \underlinetupel{t}.
  \end{equation*} 
\end{definition}

\begin{remark}\label{rem:Sigma-fol-with-quantifiers} Using quantifiers the foliated
  control in Definition~\ref{def:D_bfSigma} reads as
  \begin{align*}
    \forall \epsilon > 0 \;\exists k_0\; \forall t_0 \geq k_0\;
    \exists \beta > 0\; \forall \eta > 0\;  \exists
    & k_1 \; \forall \big( t_1 \geq k_1,   z,z',\underlinetupel{t}' \big) \; \text{we have}
    \\
    &\twovec{z',\underlinetupel{t}'}{z,\underlinetupel{t}}
      \in E \implies \fold{\bfSigma}(z,z') < (\beta,\eta,\epsilon).
  \end{align*}
\end{remark}

\begin{remark}[$\epsilon$-control over $|\Sigma|$]\label{rem:fol-contrl-gives-eps-contrl}
  The foliated control condition in Definition~\ref{def:D_bfSigma} implies that for any
  $\epsilon > 0$ there is $k_0$ such that for all $t_0 \geq k_0$ there is $k_1$ such that
  for all $t_1 \geq k_1$ and all $z,z',\underlinetupel{t}'$ with
  $\twovec{z',\underlinetupel{t}'}{z,\underlinetupel{t}} \in E$ for
  $\underlinetupel{t} = (t_0,t_1)$ we have
  $d^\infty\big(p_\bfSigma(z),p_\bfSigma(z')\big) < \epsilon$.  A possible shape of a
  region in the $\IN^{\times 2}$-plane where $\epsilon$-control holds for a fixed
  $\epsilon > 0$ is illustrated in Figure~\ref{fig:eps-control}.
  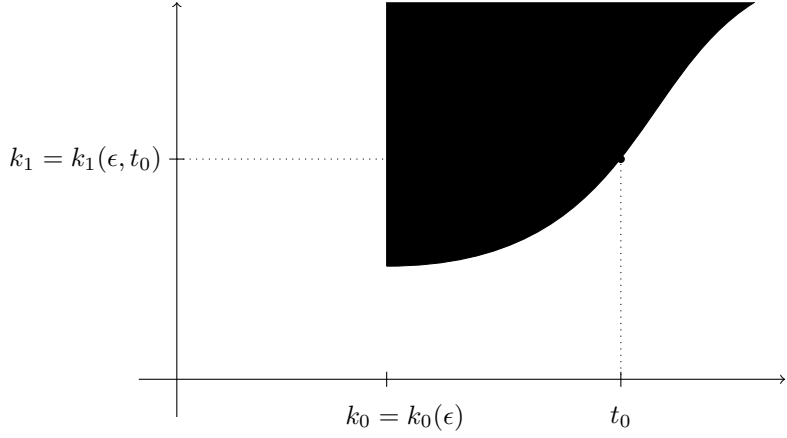
\begin{figure}[H]
    \begin{tikzpicture}
      \useasboundingbox (-2.2,-.7) rectangle (10,5);
      \fill[black,opacity=.15] (2.76,1.5) .. controls (6,1.5) and (6,4) ..  (7.6,5) --
      (2.76,5) -- (2.76,1.5); \draw[->] (-.5,0) -- (8,0); \draw [->] (0,-.5) -- (0,5);
      \draw (2.76,-.1) -- (2.76,.1); \draw(3,-.5) node {$k_0 = k_0(\epsilon)$}; \draw
      (5.84,-.1) -- (5.84,.1); \draw (5.84,-.5) node {$t_0$}; \draw (2.76,1.5) --
      (2.76,5); \draw (2.76,1.5) .. controls (6,1.5) and (6,4) .. (7.6,5); \draw (5,4.3)
      node {$\epsilon$-control over $|\Sigma|$}; \fill [black,opacity=.5] (5.84,2.92)
      circle (1.5pt); \draw (-.1,2.92) -- (.1,2.92); \draw (-1.2,2.92) node
      {$k_1 = k_1(\epsilon,t_0)$}; \draw [dotted] (0,2.92) -- (5.84,2.92) -- (5.84,0);
    \end{tikzpicture}
    \caption{Where we have $\epsilon$-control.}\label{fig:eps-control}
  \end{figure}
\end{remark}
     
 \begin{remark}[$\epsilon$-control and excision]\label{rem:eps-and-excision}
   It is mostly the $\epsilon$-control from Remark~\ref{rem:fol-contrl-gives-eps-contrl}
   that guarantees that $\bfK \big(\contd^0_G(\--)\big)$ as defined below is excisive as
   required in Proposition~\ref{prop:excision-bfD0}.  This is analogous to many other
   similar results in controlled topology/algebra, see for example the construction of the
   homology theory associated to the K-theory spectrum of a ring by Pedersen-Weibel
   in~\cite{Pedersen-Weibel(1989)}.
 \end{remark}

 \begin{remark}[Foliated control over $P(\sigma)$]\label{rem:foliated-control-P-picture}
   The foliated control condition in Definition~\ref{def:D_bfSigma} includes a second
   stage\footnote{The first stage is the $\epsilon$-control condition discussed in
     Remark~\ref{rem:fol-contrl-gives-eps-contrl}} that implies that for certain
   $z,z' \in |\bfSigma|^\wedge$ we have
   ${\fold{P(\sigma)}}\big(\lambda_\sigma(z),\lambda_\sigma(z')\big) < (\beta,\eta)$.  See
   also Remark~\ref{rem:two-stage} where $\fold{\bfSigma}(z,z') < (\beta,\eta,\epsilon)$
   is explained as a two stage condition.  Figure~\ref{fig:beta-eta-control} illustrates
   where this $(\beta,\eta)$-control applies along a vertical ray in the $\IN^{\times 2}$
   plane for fixed $\epsilon > 0$ and $t_0$.  Here $\beta(t_0)$ is fixed along the ray and
   $\eta(t_0,t_1) \to 0$ with $t_1 \to \infty$.
   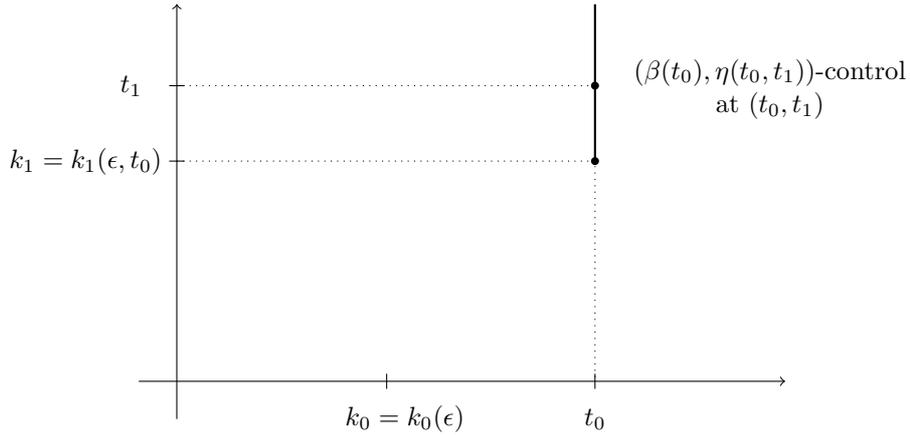
\begin{figure}[H]
     \begin{tikzpicture}
       \useasboundingbox (-2.2,-.7) rectangle (10,5);
       \draw [->] (-.5,0) -- (8,0); \draw [->] (0,-.5) -- (0,5); \draw (2.76,-.1) --
       (2.76,.1); \draw (3,-.5) node {$k_0 = k_0(\epsilon)$}; \draw (5.5,-.1) -- (5.5,.1);
       \draw (5.5,-.5) node {$t_0$};
       \fill [black,opacity=.5] (5.5,2.92) circle (1.5pt); \draw [thick] (5.5,2.92) --
       (5.5,5); \draw (-.1,2.92) -- (.1,2.92); \draw (-1.2,2.92) node
       {$k_1 = k_1(\epsilon,t_0)$}; \fill [black,opacity=.5] (5.5,3.92) circle (1.5pt);
       \draw (-.1,3.92) -- (.1,3.92); \draw (-.6,3.92) node {$t_1$}; \draw (7.8,4.1) node
       {$(\beta(t_0),\eta(t_0,t_1))$-control}; \draw (7.8,3.66) node {at $(t_0,t_1)$};
       \draw [dotted] (0,2.92) -- (5.5,2.92) -- (5.5,0); \draw [dotted] (0,3.92) --
       (5.5,3.92);
     \end{tikzpicture}
     \caption{Where we have $(\beta,\eta)$-control.}\label{fig:beta-eta-control}
   \end{figure}
 \end{remark}

\begin{remark}[On non-uniform compact support]\label{rem:why-non-uniform}
  An important aspect of Definition~\ref{def:D_bfSigma} is that the compact support
  condition in~\ref{def:D_bfSigma:obj} is not uniform over all
  $\underlinetupel{t} \in \IN^{\times 2}$.  This creates some difficulties in the
  computation of $\bfD_G(\ZerodimR)$ for $\ZerodimR \in \regularPOrGSCzero$ in
  Proposition~\ref{prop:computation-bfD-on-zero} (which would be easier using a uniform
  compact support condition).  But the non-uniformness will be crucial for the
  construction of the required transfer in Theorem~\ref{thm:transfer-bfD0}\footnote{This
    issue comes also up in proofs of the Farrell--Jones Conjecture for certain discrete
    groups, but somewhat less visible.  For example the category $\cald^G(Y;\cala)$
    in~\cite[Sec.~3.3]{Bartels-Lueck-Reich(2008hyper)} does use a uniform compact support
    condition, but the proof later also uses the category
    $\calo^G(Y,(Z_n,d_n)_{n \in \IN})$ where the compact support condition is not uniform
    in $n \in \IN$.}.  The construction of the transfer depends on certain maps
  $X \to |\bfJ_{\CVCYC}|^\wedge$ where $X$ is the extended  Bruhat-Tits building associated to $G$,
  see Theorem~\ref{thm:X-to-J}.  The construction of these maps uses an intermediate step
  maps $X \to \FS(X)$, where $\FS(X)$ is the flow space associated to $X$, see
  Subsection~\ref{subsec:factor-over-FS}.  In the construction of these latter maps the
  geodesic flow on $\FS(X)$ is used for arbitrary long times.  Roughly, this has the
  effect that the images of $X$ in $\FS(X)$ are spread out over large parts of $\FS(X)$
  and ultimately we do not have uniform control over the images of the maps
  $X \to |\bfJ_{\CVCYC}|^\wedge$.  This forces us to work with the non-uniform compact
  support condition over $|\Sigma|$.  This is also the reason for the non-uniform nature
  of $\contstrc_1(P)$ in Definition~\ref{def:mfc}.  The non-uniformness of this compact
  support condition in turn force us to work with $P$-foliated control over $|P|^\wedge$
  instead of continuous control of $|P|$, compare Remark~\ref{rem:why-foliated-control}.
\end{remark}

\begin{remark}\label{rem:uniform-fin-dim-support}
  In contrast to the compact support conditions, the finite dimensional support condition
  is uniform in $\IN^{\times 2}$.  This is crucial for (and directly implies) the skeleton
  continuity of $\bfD_G^0$ in Proposition~\ref{prop:continuity-contd0}.
\end{remark}

\begin{definition}\label{def:caly-bfSigma}
  For $\bfSigma \in \regularPOrGSC$ we define $\caly(\bfSigma)$ as the collection of all
  subsets $Y$ of $|\bfSigma|^\wedge \times \IN^{\times 2}$ satisfying the following
  condition: there is $k_0$ such that for each $t_0 \geq k_0$ there is $k_1$ with
  \begin{equation*}
    Y \cap |\bfSigma|^\wedge \times \{t_0\} \times \IN_{\geq k_1} = \emptyset. 
  \end{equation*} 
\end{definition}

\begin{definition}\label{def:cald_G}
  Let $\calb$ be a category with $G$-support.  For $\bfSigma \in \regularPOrGSC$ we apply
  Definition~\ref{def:calb(E,Y)} and define
  \begin{eqnarray*}
    \contd_G(\bfSigma;\calb) & := &  \contrCatUcoef{G}{\contstrd(\bfSigma),\caly(\bfSigma)}{\calb};\\
    \contd^0_G(\bfSigma;\calb) & := & \contrCatUcoef{G}{\contstrd^0(\bfSigma),\caly(\bfSigma)}{\calb}.
  \end{eqnarray*}
  Often we drop $\calb$ from the notation and write
  $\contd_G(\bfSigma) = \contd_G(\bfSigma;\calb)$ and
  $\contd^0_G(\bfSigma) = \contd^0_G(\bfSigma;\calb)$.
\end{definition}

\begin{remark}\label{rem:contd-more-explicit}
  The category $\contd_G(\bfSigma)$ can be described slightly more explicit as
  follows.  Objects of $\contd_G(\bfSigma)$ are objects of
  $\contrCatUcoef{G}{\contstrd(\bfSigma)}{\calb}$.  Morphisms in $\contd_G(\bfSigma)$ are equivalence
  classes of morphisms in $\contrCatUcoef{G}{\contstrd(\bfSigma)}{\calb}$, where
  $\varphi, \psi \colon (S,\pi,B) \to (S',\pi',B')$ are identified, if and on if there is
  $Y \in \caly(\bfSigma)$ such that
  \begin{equation*}
    \varphi_s^{s'} = \psi_s^{s'} 
  \end{equation*} 	
  whenever $s \in S$, $s' \in S'$ satisfy $\pi(s),\pi'(s') \not\in Y$.
\end{remark}

\begin{remark}[$\contd^0_G(\bfSigma)$ as sequences]\label{rem:contd0-more-explicit}
  An advantage of $\bfD_G^0(\bfSigma)$ over $\bfD_G(\bfSigma)$ is that it
  admits the following description.
  For $\underlinetupel{t} \in \IN^{\times 2}$ we can restrict to
  $|\bfSigma|^\wedge \times \{ \underlinetupel{t} \}$ and obtain a functor
  \begin{equation*}
    \res_\underlinetupel{t} \colon \contrCatUcoef{G}{\contstrd^0(\bfSigma)}{\calb}  
    \to  \contrCatUcoef{G}{|\bfSigma|^\wedge}{\calb}.
  \end{equation*} 
  Write $\prodprimeinline_{\IN^{\times 2}} \contrCatUcoef{G}{|\bfSigma|^\wedge}{\calb}$
  for the following category.  Objects are sequences
  $(\bfB_\underlinetupel{t})_{\underlinetupel{t} \in \IN^{\times 2}}$ of objects in
  $\contrCatUcoef{G}{|\bfSigma|^\wedge}{\calb}$.  Morphisms are equivalence classes of
  sequences $(\varphi_\underlinetupel{t})_{\underlinetupel{t} \in \IN^{\times 2}}$ of
  morphisms in $\contrCatUcoef{G}{|\bfSigma|^\wedge}{\calb}$, where two sequences
  $(\varphi_\underlinetupel{t})_{\underlinetupel{t} \in \IN^{\times 2}}$,
  $(\varphi'_\underlinetupel{t})_{\underlinetupel{t} \in \IN^{\times 2}}$ are equivalent
  if there is $k_0$ such that for any $t_0 \geq k_0$ there is $k_1$ such that for all
  $t_1 \geq k_1$ we have $\varphi_{t_0,t_1} = \psi_{t_0,t_1}$.  The above restrictions
  combine to a faithful functor
  \begin{equation}\label{eq:contd0-more-explicit}
    \contd^0_G(\bfSigma) \; \to \; \prodprimedisplay_{\IN^{\times 2}} \contrCatUcoef{G}{|\bfSigma|^\wedge}{\calb}.
  \end{equation}
  A sequence $\bfB = (\bfB_\underlinetupel{t})_{\underlinetupel{t} \in \IN^{\times 2}}$ of
  objects in $\contrCatUcoef{G}{|\bfSigma|^\wedge}{\calb}$ is in the image
  of~\eqref{eq:contd0-more-explicit}, if and only if the following four conditions are
  satisfied:
  \begin{enumerate}[
                 label=(\thetheorem\alph*),
                 align=parleft, 
                 leftmargin=*,
                 ] 
  \item\label{rem:contd0-more-explicit:obj:suppobj}
    $\suppobj \bfB = \big\{ (z,\underlinetupel{t})
    \; \big| \; z \in \suppobj \bfB_\underlinetupel{t} \big\} \in \contstrd^0_1(\bfSigma)$;\\[-1ex]
  \item\label{rem:contd0-more-explicit:obj:suppX}
    $\suppX \bfB = \Big\{ \twovec{z',\underlinetupel{t}}{z,\underlinetupel{t}} \;
    \Big| \; \twovec{z'}{z} \in \suppX \bfB_\underlinetupel{t} \Big\} \in \contstrd^0_2(\bfSigma)$;\\[-1ex]
  \item\label{rem:contd0-more-explicit:obj:suppG} $\suppG \bfB
    =  \bigcup_{\underlinetupel{t} \in \IN^{\times 2}} \suppG \bfB_\underlinetupel{t} \in \contstrd^0_G(\bfSigma)$;\\[-1ex]
  \item\label{rem:contd0-more-explicit:obj:finite} $\bfB_\underlinetupel{t}$ is finite for
    all $\underlinetupel{t}$\footnote{This condition comes from the finite over points
      condition in~\ref{def:controlled-hecke-cat:obj} and the finiteness over $\IN^{\times 2}$
      in~\ref{def:D_bfSigma:obj}.}.
  \end{enumerate}
  A sequence of morphisms is in the image of~\eqref{eq:contd0-more-explicit} if and only
  if it is equivalent to a sequence
  $(\varphi_\underlinetupel{t})_{\underlinetupel{t} \in \IN^{\times 2}}$ of morphisms in
  $\contrCatUcoef{G}{|\bfSigma|^\wedge}{\calb}$ satisfying
  \begin{enumerate}[resume,
                 label=(\thetheorem\alph*),
                 align=parleft, 
                 leftmargin=*,
                 ] 
  \item\label{rem:contd0-more-explicit:mor:suppX}
    $\suppX \varphi = \big\{ \twovec{z',\underlinetupel{t}}{z,\underlinetupel{t}} \; \big|
    \; \twovec{z'}{z} \in \suppX \varphi_\underlinetupel{t} \big\} \in
    \contstrd^0_2(\bfSigma)$;
  \item\label{rem:contd0-more-explicit:mor:suppG}
    $\suppG \varphi = \bigcup_{\underlinetupel{t} \in \IN^{\times 2}} \suppG
    \varphi_\underlinetupel{t} \in \contstrd^0_G(\bfSigma)$.
  \end{enumerate}
\end{remark}

For $\bfSigma = (\Sigma,P) \in \regularPOrGSC$ and $Q \in \EPplus\All(G)$ we have
$\bfSigma \times Q = (\Sigma,\sigma \mapsto P(\sigma) \times Q)$ as a special case
of~\eqref{eq:product-POrGSC-with-POrGSCzero}.  Note that
$|\bfSigma \times Q|^\wedge = |\bfSigma|^\wedge \times |Q|^\wedge$.

\begin{lemma}\label{lem:D_upper_0_bfSigma-prod-P} 
  Let $\bfSigma \in \regularPOrGSC$ and $Q \in \EPplus\All(G)$.
  \begin{enumerate}[
                 label=(\thetheorem\alph*),
                 align=parleft, 
                 leftmargin=*,
                 labelindent=1pt,
                 ] 
  \item\label{lem:D_upper_0_bfSigma-prod-P:1} For $F \in \contstrd_1^0(\bfSigma)$,
    $F' \in \contstrd_1^0(\ZerodimR)$ we have
    \begin{equation*}
      \big\{ (z,\lambda,\underlinetupel{t})  \mid (z,\underlinetupel{t})
      \in F, (\lambda,\underlinetupel{t}) \in F' \big\}
      \in \contstrd_1^0(\bfSigma \times Q);
    \end{equation*}  	
  \item\label{lem:D_upper_0_bfSigma-prod-P:2} For $E \in \contstrd_2^0(\bfSigma)$,
    $E' \in \contstrd_2^0(\ZerodimR)$ we have
    \begin{equation*}
      \Big\{ \twovec{z',\lambda',\underlinetupel{t}}{z,\lambda,\underlinetupel{t}}
      \, \Big| \, \twovec{z',\underlinetupel{t}}{z,\underlinetupel{t}} \in E,
      \twovec{\lambda',\underlinetupel{t}}{\lambda,\underlinetupel{t}} \in E' \Big\}
      \in \contstrd_2^0(\bfSigma \times Q).
    \end{equation*}  	
  \end{enumerate}
\end{lemma}

\begin{proof}
  This is an easy exercise in the definitions.
\end{proof}

\begin{definition}\label{def:bfD-bfD0}
  For a category $\calb$ with $G$-support we define the two functors
  $\bfD_G(\--;\calb),\, \bfD_G^0(\--;\calb) \;\colon \regularPOrGSC \to \Spectra$ by
  \begin{equation*}
    \bfD_G(\bfSigma;\calb) := \bfK \big( \contd_G(\bfSigma;\calb) \big) \quad \text{and}
    \quad \bfD_G^0(\bfSigma;\calb) := \bfK \big( \contd^0_G(\bfSigma;\calb) \big).
  \end{equation*}
  We often abbreviate $\bfD_G(\--) = \bfD_G(\--;\calb)$ and $\bfD_G^0(\--) = \bfD_G^0(\--;\calb)$.
\end{definition}


\section{Properties of $\contd_G(\--)$ and $\contd_G^0(\--)$}%
\label{sec:properties-cats-D_upper_D0}


\subsection{Computation of $\contd_G(P)$}

We write again $I \colon \EPplus\All(G) \to \regularPOrGSC$ for the inclusion.

\begin{proposition}\label{prop:coefficients-contd}
  There is a zig-zag of equivalences of$\EPplus\All(G)$-spectra between
  $I^*\Omega\bfK ( \contd_G(\--) )$ and $\bfK ( \contc_G(\--) )$.
\end{proposition}

The proof of Proposition~\ref{prop:coefficients-contd} will need a preparation. 

\begin{definition}
  Let $P \in \EPplus\All(G)$.  Let $\caly_0(P)$ be the collection of all subsets $Y$ of
  $|P|^\wedge \times \IN \times \IN$ that are contained in
  $|P|^\wedge \times \IN \times \{ 0,\dots,N \}$ for some $N$ (depending on $Y$).  We
  define
  \begin{eqnarray*}
    \contd^{\fin}_G(P) & := & \contrCatUcoef{G}{\contstrd(P)|_{\caly(P)},\caly_0(P)}{\calb}; \\
    \contd^{\sw}_G(P) & := & \contrCatUcoef{G}{\contstrd(P),\caly_0(P)}{\calb}.
  \end{eqnarray*}   
  Here $\caly(P)$ is from Definition~\ref{def:caly-bfSigma}.
\end{definition}

As $\contd_G(P) = \contrCatUcoef{G}{\contstrd(P),\caly(P)}{\calb}$ and
$\caly_0(P) \subseteq \caly(P)$ we obtain a Karoubi sequence
\begin{equation}\label{eq:karoubi-D-fin-D-sw-D-for-P} \contd^{\fin}_G(P) \; \to \;
  \contd^{\sw}_G(P) \; \to \; \contd_G(P).
\end{equation}

\begin{lemma}\label{lem:about-contd-sw-fin}
  Let $P \in \EPplus\All(G)$.
  \begin{enumerate}[
                 label=(\thetheorem\alph*),
                 align=parleft, 
                 leftmargin=*,
                 labelindent=1pt,
                 ] 
  \item\label{lem:about-contd-sw-fin:fin} The inclusion
    $|P|^\wedge \times \IN \to |P|^\wedge \times \IN^{\times 2}$,
    $(\lambda,t) \mapsto (\lambda,0,t)$ induces an equivalence
    $\contc_G(P) \to \contd^\fin(P)$;
  \item\label{lem:about-contd-sw-fin:sw} The K-theory of $\contd^{\sw}_G(P)$ vanishes.
  \end{enumerate}
\end{lemma}

\begin{proof} 
  The first statement is an easy exercise in the definitions.  The second comes from a
  standard Eilenberg swindle on $\contd^{\sw}_G(P)$ using the shift
  $(\lambda,t_0,t_1) \mapsto (\lambda,t_0+1,t_1)$; formally we use
  Lemma~\ref{lem:formal-swindle}.  
  Let $P \in \EPplus\All(G)$.
  Let $\bfSigma \in \regularPOrGSCzero$ be the object whose  underlying simplicial complex consist of one vertex which is sent to $P$, i.e., $\bfSigma = I(P)$.
  Now the point is that for $z,z' \in |\bfSigma|^\wedge = |P|^\wedge$ we have for any $\epsilon > 0$
  \begin{equation*}
    \fold{\bfSigma}(z,z') < (\beta,\eta,\epsilon) \quad \iff \quad \fold{P}(z,z') < (\beta,\eta).
  \end{equation*}
  Thus the conditions in~\ref{def:D_bfSigma:X} are constant in the first coordinate.  This
  is used to verify~\ref{lem:formal-swindle-Z:E_2}, the other assumptions of
  Lemma~\ref{lem:formal-swindle} are straight forward to check.
\end{proof}

\begin{proof}[Proof of Proposition~\ref{prop:coefficients-contd}]
  The Karoubi sequence~\eqref{eq:karoubi-D-fin-D-sw-D-for-P} induces a fibration sequence
  in K-theory, see~\eqref{eq:karoubi-sequence-rel}.  Thus~\ref{lem:about-contd-sw-fin:fin}
  and~\ref{lem:about-contd-sw-fin:sw} give the result.
\end{proof}


\subsection{Computation of $\contd_G(\--)$ on
  $\regularPOrGSCzero$}\label{subsec:computation-bfD-on-zero-really}

\begin{proposition}\label{prop:computation-bfD-on-zero} 
  Let $\ZerodimR = (\ZerodimRunderlyingSet,P) \in \regularPOrGSCzero$.  The canonical map
  \begin{equation*}
    \bigvee_{\ZerodimRunderlyingSetElement \in \ZerodimRunderlyingSet}
    \bfK(\contd_G(P(\ZerodimRunderlyingSetElement))) \xrightarrow{\sim} \bfK(\contd_G(\ZerodimR))
  \end{equation*}
  is an equivalence.
\end{proposition}

The proof of Proposition~\ref{prop:computation-bfD-on-zero} requires some preparations.

\begin{definition}\label{def:contddis}
  Let $\ZerodimR = (\ZerodimRunderlyingSet,P) \in \regularPOrGSCzero$.  We define the
  $G$-control structure
  $\contstrd^\dis(\ZerodimR) =
  \big(\contstrd^\dis_1(\ZerodimR),\contstrd^\dis_2(\ZerodimR),\contstrd^\dis_G(\ZerodimR)\big)$
  as follows.  We set $\contstrd^\dis_1(\ZerodimR) := \contstrd_1(\ZerodimR)$ and
  $\contstrd^\dis_G(\ZerodimR) := \contstrd_G(\ZerodimR)$.  We define
  $\contstrd^\dis_2(\ZerodimR)$ to consist of all $E \in \contstrd_2(\ZerodimR)$ that are
  $0$-controlled over $\ZerodimRunderlyingSet$, i.e., satisfy the following.  Let
  $\twovec{\lambda',\underlinetupel{t}'}{\lambda,\underlinetupel{t}} \in E$.  Write
  $\ZerodimRunderlyingSetElement$ and $\ZerodimRunderlyingSetElement'$ for the images of
  $\lambda$ and $\lambda'$ in $\ZerodimRunderlyingSet$ under the projection
  $|\ZerodimR|^\wedge \to \ZerodimRunderlyingSet$.  We require
  $\ZerodimRunderlyingSetElement=\ZerodimRunderlyingSetElement'$.
	
  Let $\caly_0(\ZerodimRunderlyingSet)$ be the collection of all subsets $Y$ of
  $|\ZerodimR|^\wedge \times \IN \times \IN$ that are contained in
  $|\ZerodimR|^\wedge \times \IN \times \{ 0,\dots,N \}$ for some $N$ (depending on $Y$).
  We define
  \begin{eqnarray*}
    \contd_G^{\dis,\fin}(\ZerodimR)
    & := &
           \contrCatUcoef{G}{\contstrd^\dis(\ZerodimR)|_{\caly(\ZerodimR)},\caly_0(\ZerodimR)}{\calb}; \\
    \contd_G^{\dis,\sw}(\ZerodimR)
    & := &
           \contrCatUcoef{G}{\contstrd^\dis(\ZerodimR),\caly_0(\ZerodimR)}{\calb}; \\ 
    \contd_G^\dis(\ZerodimR)
    & := &
           \contrCatUcoef{G}{\contstrd^\dis(\ZerodimR),\caly(\ZerodimR)}{\calb}. 
  \end{eqnarray*}
\end{definition}

\begin{remark}\label{rem:contrstrd-is-almost-contrstrddis}
  The $\epsilon$-control aspect of the foliated control condition in~\ref{def:D_bfSigma:X}
  implies that for all $E \in \contstrd_2(\ZerodimR)$ there is $Y \in \caly(\ZerodimR)$
  such that
  \begin{equation*}
    E \cap Y \times Y \; \in \; \contstrd^\dis_2(\ZerodimR). 
  \end{equation*}
\end{remark}

\begin{lemma}\label{lem:about-constr-dis}
  Let $\ZerodimR = (\ZerodimRunderlyingSet,P) \in \regularPOrGSCzero$.
  \begin{enumerate}[
                 label=(\thetheorem\alph*),
                 align=parleft, 
                 leftmargin=*,
                 labelindent=1pt,
                 ]   
    \item\label{lem:about-constr-dis:dis-fin} The canonical map
    \begin{equation*}
      \bigoplus_{\ZerodimRunderlyingSetElement \in \ZerodimRunderlyingSet}
      \contd_G^{\dis,\fin}(P(\ZerodimRunderlyingSetElement)) \xrightarrow{\sim} \contd_G^{\dis,\fin}(\ZerodimR)
    \end{equation*}
    is an equivalence
  \item\label{lem:about-constr-dis:dis-sw} The K-theory of
    $\contd_G^{\dis,\sw}(\ZerodimR)$ vanishes.
  \item\label{lem:about-constr-dis:dis} The inclusion
    $\contd_G^\dis(\ZerodimR) \to \contd_G(\ZerodimR)$ is an equivalence.
  \end{enumerate}
\end{lemma}

\begin{proof}
  The first statement follows from the compact support condition over $|\Sigma| = M$
  in~\ref{def:D_bfSigma:obj} and the fact that we have $0$-control over over $M$.  The
  second statement uses the Eilenberg swindle on $\contd_G^{\dis,\sw}(\ZerodimR)$ induced
  by the map $(\lambda,t_0,t_1) \mapsto (\lambda,t_0+1,t_1)$, see
  Lemma~\ref{lem:formal-swindle}.  To verify~\ref{lem:formal-swindle:E_2} it is again
  important that $0$-control over $B$ is in $\contstrd^\dis(\ZerodimR)$ enforced for all
  $t_0 \in \IN$\footnote{Neither~\ref{lem:about-constr-dis:dis-fin}
    nor~\ref{lem:about-constr-dis:dis-sw} hold if we use $\contstrd(\ZerodimR)$ instead of
    $\contstrd^\dis(\ZerodimR)$}.  The third statement is an easy exercise in the
  definitions and uses the observation from
  Remark~\ref{rem:contrstrd-is-almost-contrstrddis}.
\end{proof}

\begin{proof}[Proof of Proposition~\ref{prop:computation-bfD-on-zero}] 
  By~\ref{lem:about-constr-dis:dis} it suffices to prove the assertion for
  $\contd_G^{\dis}$ in place of $\contd_G$.  The Karoubi sequence
  \begin{equation*}
    \contd_G^{\dis,\fin}(\ZerodimR) \; \to \; \contd_G^{\dis,\sw}(\ZerodimR) \; \to \; \contd_G^{\dis}(\ZerodimR) 
  \end{equation*} 
  induces a fibration sequence in K-theory, see~\eqref{eq:karoubi-sequence-rel}.
  Using~\ref{lem:about-constr-dis:dis-sw} we obtain an equivalence
  $\Omega \bfK \big(\contd_G^{\dis}(\--)\big) \xrightarrow{\sim} \bfK \big(
  \contd_G^{\dis,\fin}(\--) \big)$.  Thus it suffices to prove the assertion with
  $\contd_G^{\dis,\fin}$ in place of $\contd_G^\dis$.  In this formulation the assertion
  follows from~\ref{lem:about-constr-dis:dis-fin}.
\end{proof}


\subsection{The K-theory of $\contd_G^0(\bfSigma)$ determines the K-theory of
  $\contd_G(\bfSigma)$}%
\label{subsec:contd-vs-contd0}

\begin{proposition}\label{prop:contd-vs-contd0}
  There exists a diagram in $\regularPOrGSC\text{-}\Spectra$
  \begin{equation*}
    \xymatrix{\bfK \big( \contd^0_G(\--)\big) & \bfK \big( \contd^0_G(\--)\big) \ar[r] \ar[l]
      & \bfK \big( \contd^0_G(\--)\big) \\
      \bfK \big( \contd^0_G(\--)\big) \ar[u] \ar[d] & \bfK \big( \contd^0_G(\--)\big) \ar[r] \ar[l] \ar[u] \ar[d]
      & \bfK \big( \contd^0_G(\--)\big) \ar[u] \ar[d] \\
      \bfK \big( \contd^0_G(\--)\big) & \bfK \big( \contd^0_G(\--)\big) \ar[r] \ar[l]
      & \bfK \big( \contd^0_G(\--)\big) 
    }	
  \end{equation*}
  whose homotopy colimit is equivalent to $\bfK \big( \contd_G(\--)\big)$.
\end{proposition}

The proof of Proposition~\ref{prop:contd-vs-contd0} requires some preparations.

\begin{definition}
  We define the $G$-control structure
  \begin{equation*}
  	\contstrd^{.5}(\bfSigma) = \big(
  \contstrd_1^{.5}(\bfSigma),\contstrd_2^{.5}(\bfSigma),\contstrd_G^{.5}(\bfSigma) \big)
  \end{equation*}
  as follows.  Set $\contstrd_1^{.5}(\bfSigma) := \contstrd_1(\bfSigma)$,
  $\contstrd_G^{.5}(\bfSigma) := \contstrd_G(\bfSigma)$.  We define
  $\contstrd_2^{.5}(\bfSigma)$ to consist of all $E \in \contstrd_2(\bfSigma)$ satisfying
  \begin{equation*}
    \qquad \twovec{z',t'_0,t'_1}{z,t_0,t_1} \in E \; \implies \; t_0 = t'_0.
  \end{equation*} 
  We set
  \begin{equation*}
    \contd^{.5}_G(\bfSigma)  :=  \contrCatUcoef{G}{\contstrd^{.5}(\bfSigma),\caly(\bfSigma)}{\calb}.
  \end{equation*}
\end{definition}

Note that
$\contstrd_2^{0}(\bfSigma) \subseteq \contstrd_2^{.5}(\bfSigma) \subseteq
\contstrd_2(\bfSigma)$ and so
$\contd^{0}_G(\bfSigma) \subseteq \contd^{.5}_G(\bfSigma) \subseteq \contd_G(\bfSigma)$.

\begin{lemma}\label{lem:two-pushout}
  There are homotopy pushouts in $\regularPOrGSC\text{-}\Spectra$
  \begin{equation*}
    \xymatrix{\bfK \big( \contd^0_G(\--)\big) \ar[r] \ar[d] & \bfK \big( \contd^0_G(\--)\big) \ar[d]
      & \bfK \big( \contd^{.5}_G(\--)\big) \ar[r] \ar[d] & \bfK \big( \contd^{.5}_G(\--)\big) \ar[d] \\
      \bfK \big( \contd^0_G(\--)\big) \ar[r]  & \bfK \big( \contd^{.5}_G(\--)\big)
      & \bfK \big( \contd^{.5}_G(\--)\big) \ar[r] & \bfK \big( \contd_G(\--)\big).
    }	 
  \end{equation*}
\end{lemma}

We will only give the construction of the right homotopy pushout square in
Lemma~\ref{lem:two-pushout}; the construction of the left homotopy pushout square is
entirely analogous.

\begin{definition}
  For $X \subseteq \IN$ let $\caly^X(\bfSigma)$ be the collection of all subsets of
  $|\bfSigma|^\wedge \times X \times \IN$.  We define
  \begin{equation*}
    \contd^X_G(\bfSigma) := \contrCatUcoef{G}{\contstrd(\bfSigma)|_{\caly^X(\bfSigma)},\caly(\bfSigma)}{\calb}.
  \end{equation*}
\end{definition}

The definition amounts to replacing $|\bfSigma|^\wedge \times \IN \times \IN$ with
$|\bfSigma|^\wedge \times X \times \IN$.  Note that because of the bounded control
requirement over $\IN \times \IN$ in~\ref{def:D_bfSigma:X} the category
$\contd^X_G(\bfSigma)$ depends on $X$ with the metric it inherits from $\IN$.  In
particular, the properties of $\contd^X_G(\bfSigma)$ depend on the coarse
structure of $X$.

We now choose natural numbers $0=a_0 < a_1 < a_2 < \dots$ such that
$a_{n+1} - a_n \to \infty$ as $n \to \infty$.  We set
\begin{eqnarray*}
  A & := & [a_0,a_3] \cup [a_4,a_7] \cup [a_{\,8},\,a_{11}] \cup [a_{12},a_{15}] \cup \dots; \\
  B & := & [a_2,a_5] \cup [a_6,a_9] \cup [a_{10},a_{13}] \cup [a_{14},a_{17}] \cup \dots, 
\end{eqnarray*} 
where $[a,b]$ is the discrete interval $\{a,a+1,\dots,b\}$.  The point is that each of the
three sets $A$, $B$ and $A \cap B$ is the infinite union of intervals, where both the
length and the distance between successive grow to $\infty$.  We have
\refstepcounter{theorem}
\begin{enumerate}[
                 label=(\thetheorem\alph*),
                 align=parleft, 
                 leftmargin=*,
                 labelsep=8pt,
                 ] 
\item\label{nl:A-B-distance} for any $r$ there is $R$ such that for $a \in A \setminus B$
  and $b \in B$ we have either $a,b \leq R$ or $|b-a| \leq r$;
\item\label{nl:a_i-a_j-distance} for any $r$ there is $R$ such that if
  $|a_i -a_j| \leq r$ and $a_i \neq a_j$, then $a_i,a_j \leq R$.
\end{enumerate}

\begin{lemma}\label{lem:A-B-A-cap-B-pushout} The square
  \begin{equation*}
    \xymatrix{\bfK\bigl(\contd^{A \cap B}_G (\--)\bigr) \ar[r] \ar[d] & \bfK\bigl(\contd_G^A(\--)\bigr) \ar[d] \\
      \bfK\bigl(\contd_G^B(\--) \bigr) \ar[r] & \bfK\bigl(\contd_G(\--)\bigr) 
    }
  \end{equation*}	
  is a homotopy pushout square in $\regularPOrGSC\text{-}\Spectra$.
\end{lemma}
 
\begin{proof} 
  Let $\bfSigma \in \regularPOrGSC$.  We check that
  Lemma~\ref{lem:formal-homotopy-pushout-square} applies to the above square evaluated on
  $\bfSigma$.  Let $Y_A \in \caly^A(\bfSigma)$, $E \in \contstrd_2(\bfSigma)$ be given.
  Set $Y_B := Y_A \cap \big(|\bfSigma|^\wedge \times B \times \IN \big)$ and
  $Y'_A := Y_A \setminus |\bfSigma|^\wedge \times B \times \IN$.  Then
  $Y_A \in \caly^A(\bfSigma)$, $Y_B \in \caly^B(\bfSigma)$ and $Y_A = Y'_A \cup Y_B$.  The
  bounded control condition from~\ref{def:D_bfSigma:X} together with~\ref{nl:A-B-distance}
  implies that there is $R > 0$ such that for all $(z,t_0,t_1) \in (Y'_A)^E$ we have
  $t_0 \in A$ or $t_0 \leq R$.  Thus $(Y'_A)^E \in \caly^A(\bfSigma)$ and
  Lemma~\ref{lem:formal-homotopy-pushout-square} applies.
\end{proof}

\begin{lemma}\label{lem:all-are-point5} In $\regularPOrGSC\text{-}\Spectra$, the functors
  $\bfK\bigl( \contd^{A \cap B}_G (\--)\bigr)$, $\bfK\bigl( \contd_G^A(\--)\bigr)$,
  $\bfK\bigl( \contd_G^B(\--)\bigr) $ are all equivalent to
  $\bfK\bigl( \contd_G^{.5}(\--)\bigr) $.
\end{lemma}

\begin{proof}
  The argument is almost the same in all three cases.  We treat
  $\bfK\bigl( \contd_G^A(\--)\bigr)$.  Let $A_0 := \{a_0,a_4,a_8,a_{12},\dots \}$.  We
  claim that for any $\bfSigma \in \regularPOrGSC$
  \begin{enumerate}[
                 label=(\thetheorem\alph*),
                 align=parleft, 
                 leftmargin=*,
                 labelsep=10pt,
                 ]  
   \item\label{nl:IN-A_0} the bijection $\IN \to A_0$, $i \mapsto a_{4i}$ induces an
    equivalence $\contd_G^{.5}(\bfSigma) \to \contd_G^{A_0}(\bfSigma)$ of categories;
  \item\label{nl:A_0-A} the inclusion $A_0 \to A$ induces an equivalence
    $\contd_G^{A_0}(\bfSigma) \to \contd_G^{A}(\bfSigma)$ in K-theory.
  \end{enumerate}
  Clearly,~\ref{nl:IN-A_0} and~\ref{nl:A_0-A} together give us an equivalence
  $\bfK\bigl( \contd_G^{.5}(\--)\bigr) \to \bfK\bigl( \contd_G^A(\--)\bigr)$.
	
  The difference between $\contstrd^{.5}(\bfSigma)$ and $\contstrd(\bfSigma)$ is that for
  $\twovec{z',t_0',t_1'}{z,t_0,t_1'} \in E \in \contstrd_2(\bfSigma)$ we can have
  $t_0' \neq t_0$, while this does not happen for $\contstrd_2^{.5}(\bfSigma)$.  However,
  if $t_0 = a_i$ and $t'_0 = a_{i'}$ then~\ref{nl:a_i-a_j-distance} implies that either
  $t_0 = t'_0$ or $t_0,t'_0$ are bounded.  More formally, for
  $E \in \contstrd_2(\bfSigma)$ there is $Y \in \caly(\bfSigma)$ such that
  $(E \setminus Y \times Y) \cap \big(|\bfSigma|^\wedge \times A_0 \times \IN\big)^{\times
    2} \in \contstrd_2^{.5}(\bfSigma)$ and from this it is not difficult to
  verify~\ref{nl:IN-A_0}.
	
  The Karoubi sequence
  \begin{equation*}
    \contd_G^{A_0}(\bfSigma) \to \contd_G^{A}(\bfSigma)
    \to \contrCatUcoef{G}{\contstrd(\bfSigma)|_{\caly^A(\bfSigma)},\caly^{A_0}(\bfSigma)}{\calb}
  \end{equation*}
  induce a fibration sequence in K-theory, see~\eqref{eq:karoubi-sequence-rel}.  To
  prove~\ref{nl:A_0-A} it suffice to show that
  $\contrCatUcoef{G}{\contstrd(\bfSigma)|_{\caly^A(\bfSigma)},\caly^{A_0}(\bfSigma)}{\calb}$ admits an Eilenberg
  swindle.  Set $Z := |\bfSigma|^\wedge \times A_0 \times \IN$.  Consider
  $f \colon |\bfSigma|^\wedge \times A \times \IN \to |\bfSigma|^\wedge \times A \times
  \IN$ with
  \begin{equation*}
    f(z,t_0,t_1) := \begin{cases}
      (z,t_0,t_1) & t_0 \in A_0 \\ (z,t_0-1,t_1) & 	t_0 \not\in A_0. 
    \end{cases}
  \end{equation*}
  It is not difficult to check that Lemma~\ref{lem:formal-swindle-Z} applies and we obtain
  a swindle on $\calb^{\wedge}_G(\mfE(\bfSigma)|_{\caly^A(\bfSigma)},\caly^{A_0}(\bfSigma))$.
  $\contrCatUcoef{G}{\mfE(\bfSigma)|_{\caly^A(\bfSigma)},\caly^{A_0}(\bfSigma)}{\calb}$.
\end{proof}
 
\begin{proof}[Proof of Lemma~\ref{lem:two-pushout}]
  The existence of the right hand square follows directly from
  Lemma~\ref{lem:A-B-A-cap-B-pushout} and Lemma~\ref{lem:all-are-point5}.  The left hand
  square can be constructed by a similar argument.
\end{proof}

\begin{proof}[Proof of Proposition~\ref{prop:contd-vs-contd0}]
  This follows from Lemma~\ref{lem:two-pushout}.
\end{proof}


\subsection{Homotopy invariance for $\bfD_G^0$}\label{subsec:homotopy-inv-contd0}

Let $\ZerodimR = (\ZerodimRunderlyingSet,P) \in \regularPOrGSCzero$.  Let
$\pi \colon \ZerodimRunderlyingSet \times \Delta^n \to \ZerodimRunderlyingSet$ be the projection.  
  We obtain
$\bfDelta^d_\ZerodimR = (\ZerodimRunderlyingSet \times \Delta^d,P \circ \pi_*) \in
\regularPOrGSC$ as in Subsection~\ref{subsec:homotopy-inv-bfD0}.  A choice of a point
$x_0 \in |\Delta^d|$ determines an inclusion
$\bfi \colon \ZerodimR \to \bfDelta^d_\ZerodimR$.

\begin{proposition}\label{prop:homotopy-inv-bfD0} The inclusion $\bfi$ induces an
  equivalence
  \begin{equation*}
    \bfK \big( \contd_G^0 (\ZerodimR) \big) \; \xrightarrow{\sim} \;
    \bfK \big( \contd_G^0 (\bfDelta^d_\ZerodimR) \big).
  \end{equation*}
\end{proposition}
 
\begin{proof}
  We have $|\bfDelta^d_\ZerodimR|^\wedge = |\ZerodimR|^\wedge \times |\Delta^d| $.  Let
  $\caly(x_0)$ be the collection of all subsets of
  $|\ZerodimR|^\wedge \times \{ x \} \times \IN \times \IN \subset
  |\bfDelta^d_\ZerodimR|^\wedge \times \IN \times \IN $.  Then
  $\contd_G^0 (\ZerodimR) = \contrCatUcoef{G}{\contstrd^0(\ZerodimR),\caly(\ZerodimR)}{\calb}$ is
  equivalent to
  $\contrCatUcoef{G}{\contstrd^0(\bfDelta^d_\ZerodimR)|_{\caly(x_0)},\caly(\bfDelta^d_\ZerodimR)}{\calb}$.
  We obtain a Karoubi sequence
  \begin{align*}
    \contrCatUcoef{G}{\contstrd^0(\bfDelta^d_\ZerodimR)|_{\caly(x_0)},\caly(\bfDelta^d_\ZerodimR)}{\calb}
    \to \contrCatUcoef{G}{\contstrd^0(\bfDelta^d_\ZerodimR),\caly(\bfDelta^d_\ZerodimR)}{\calb}
    & \\ & \hspace{-80pt}
     \to \contrCatUcoef{G}{\contstrd^0(\bfDelta^d_\ZerodimR),\caly(\bfDelta^d_\ZerodimR) \cup \caly(x_0)}{\calb}
  \end{align*}
  as in~\eqref{eq:karoubi-sequence-rel}.  It suffices to show that the K-theory of right
  most category of this sequence is trivial.  To this end we produce an Eilenberg swindle
  and use Lemma~\ref{lem:formal-swindle-Z} with
  $Z = |\ZerodimR|^\wedge \times \{ x_0 \} \times \IN \times \IN$.  The point of the
  swindle is that we can contract $|\Delta^d|$ linearly to $\{x_0\}$.  The
  $\epsilon$-control part (see~\ref{nl:bfSigma-fol-distance:eps}) of the foliated control
  condition in~\ref{def:D_bfSigma:X} requires a bit of care here: with increasing $t_0$ we
  need to push slower and slower.  More precisely, we construct
  $f \colon |\ZerodimR|^\wedge \times |\Delta^d| \times \IN \times \IN \to
  |\ZerodimR|^\wedge \times |\Delta^d| \times \IN \times \IN$ as follows.  Write
  $f_{t_0} \colon |\Delta^d| \to |\Delta^d|$ for the map that sends $x$ to the point $x'$
  on the straight line from $x$ to $x_0$ in $|\Delta^d|$, whose distance from $x$ is
  $\min\{1/t_0,d(x,x_0)\}$.  We can set
  \begin{equation*}
    f(\lambda,x,t_0,t_1) := (\lambda,f_{t_0}(x),t_0,t_1).
  \end{equation*} 
  The foliated control condition in~\ref{def:D_bfSigma:X} also involves foliated control
  over $|(P \circ \pi_*)(\{m\} \times \sigma)|^\wedge = |P(m)|^\wedge$
  for each $m \in \ZerodimRunderlyingSet$,
  see~\ref{nl:bfSigma-fol-distance:beta-eta}.    
  The map $f$ preserves this condition,
  because $f$ acts as the identity on the $|\ZerodimR|^\wedge$-coordinate and because
  $|(P \circ \pi_*)(\{m\} \times \sigma)$ depends on $m$ but not on $\sigma$ (otherwise
  the passage from a subsimplex to a larger simplex could create problems).  With this
  observation it is not difficult to check that the assumptions of
  Lemma~\ref{lem:formal-swindle-Z} are satisfied.
\end{proof}


\subsection{Excision for $\contd^0_G(\--)$}\label{subsec:excision-bfD0-really}

Let $\bfSigma = (\Sigma,P) \in \regularPOrGSC$ be $d$-dimensional.  Let
$\bfSigma' = (\Sigma',P')$ be its $(d\,\text{-}1)$-skeleton.  For $W \subseteq |\Sigma|$
we let $|W|^\wedge$ be the preimage of $W$ under
$|\bfSigma|^\wedge \xrightarrow{p_\bfSigma} |\Sigma|$ and define $\caly_W$ as the
collection of all subsets of $|W|^\wedge \times \IN^{\times 2}$.  We set
\begin{equation*}
  \contd^{0,W}_G(\bfSigma) :=
  \contrCatUcoef{G}{\contstrd^0(\bfSigma)|_{\caly_W},\caly(\bfSigma)}{\calb},
\end{equation*}
i.e., $\contd^{0,W}_G(\bfSigma)$ is the full subcategory of $\contd^{0,W}_G(\bfSigma)$ on
all objects whose support is contained in $|W|^\wedge \times \IN^{\times 2}$.  Setting
$\caly_W^+ := \caly_W \cup \caly(\bfSigma)$, we can apply
Lemma~\ref{lem:y-rel-Y-irrelevant} to conclude that $\contd^{0,W}_G(\bfSigma)$ is
canonically equivalent
to $\contrCatUcoef{G}{\contstrd^0(\bfSigma)|_{\caly^+_W},\caly(\bfSigma)}{\calb}$.

Fix $0 < \epsilon < 1/(d+1)$.  Let $N$ be the (open\footnote{We could equally well work
  with closed neighborhoods.}) $\epsilon$-neighborhood of $|\Sigma'|$ in $|\Sigma|$
(always with respect to the $l^{\infty}$-metric).  The choice of $\epsilon$  guarantees
$|\Sigma'| \subseteq N \subsetneq |\Sigma|$.  
Let $M$ be the complement of the $\epsilon/2$-neighborhood of
$|\Sigma'|$ in $|\Sigma|$.  Thus $|\Sigma| = N \cup M$.

\begin{lemma}\label{lem:thick-is-not-thick}
  The functor $\contd_G^{0}(\bfSigma') \to \contd_G^{0,N}(\bfSigma)$ induced by the
  inclusion $\bfSigma' \to \bfSigma$ yields an equivalence in K-theory.
\end{lemma}

\begin{proof}
  First note that the inclusion $\bfSigma' \to \bfSigma$ induces an equivalence
  $\contd_G^{0}(\bfSigma') \to \contd_G^{0,|\Sigma'|}(\bfSigma)$.  We have a Karoubi
  sequence
  \begin{align*}
    \contrCatUcoef{G}{\contstrd^0(\bfSigma)|_{\caly_{|\Sigma'|}},\caly(\bfSigma)}{\calb} \to
    \contrCatUcoef{G}{\contstrd^0(\bfSigma)|_{\caly_N},\caly(\bfSigma)}{\calb} & \\ & \hspace{-80pt} \to
    \contrCatUcoef{G}{\contstrd^0(\bfSigma|_{\caly_N},\caly(\bfSigma) \cup \caly_{|\Sigma'|})}{\calb} 
  \end{align*}
  as in~\eqref{eq:karoubi-sequence-rel}.  The first map can be identified with
  $\contd_G^{0,|\Sigma'|}(\bfSigma) \to \contd_G^{0,N}(\bfSigma)$.  Thus we need to show
  that the K-theory of the third term is trivial.  To this end we will use
  Lemma~\ref{lem:formal-swindle} to construct an Eilenberg swindle on
  $\contrCatUcoef{G}{\contstrd^0(\bfSigma)|_{\caly(N)},\caly(\bfSigma) \cup
  \caly_{|\Sigma'|}}{\calb}$.  The swindle will be similar to the one constructed in the
  proof of homotopy invariance in Proposition~\ref{prop:homotopy-inv-bfD0}.  This time we
  swindle towards $|\bfSigma'|^\wedge$ in $|N|^\wedge$.  Let
  $Z := |\bfSigma'|^\wedge \times \IN \times \IN$.  Let $p \colon N \to |\Sigma'|$ be the
  radial projection, i.e., if $\sigma$ is a $d$-simplex of $\Sigma$ and
  $x \in N \cap |\Delta_\sigma|$, then $p(x)$ is the unique point in
  $|\partial \Delta_\sigma|$ such that $x$ lies on the straight line between $p(x)$ and
  the barycenter of $\sigma$.  Let $f_t(x)$ be the point on the straight line from $x$ to
  $p(x)$ of distance $\min\{1/(t+1),d(x,p(x))$ from $x$.  Let
  $f^\wedge_t \colon |N|^\wedge \to |N|^\wedge$ be the map that sends
  $[x,\lambda]_\sigma \in |N|^\wedge$ to $[f_t(x),\lambda]_\sigma$.  We now define
  $f \colon |N|^\wedge \times \IN \times \IN \to |W|^\wedge \times \IN \times \IN$ by
  \begin{equation*}
    f(z,t_0,t_1) = (f^\wedge_{t_0},t_0,t_1).
  \end{equation*}
  It is not too difficult to check that the assumptions of
  Lemma~\ref{lem:formal-swindle-Z} are satisfied.  As in the proof of
  Proposition~\ref{prop:homotopy-inv-bfD0}, it is important here that $f_t$ pushes less
  with increasing $t$, in order to preserve the $\epsilon$-control part
  (see~\ref{nl:bfSigma-fol-distance:eps}) of the foliated control condition
  in~\ref{def:D_bfSigma:X}.  As $f_t$ pushes linearly towards the boundary of
  $|\Delta_\sigma|$, it does preserve the sets $K_{\tau,\epsilon}$
  from~\ref{nl:bfSigma-fol-distance:beta-eta}.  With these observations it is easy to
  control the interaction of $f^\wedge_t$ with the foliated control condition
  in~\ref{def:D_bfSigma:X}.
\end{proof}

 

\begin{proposition}\label{prop:M-N-MV} The diagram
  \begin{equation*}
    \xymatrix{\contd_G^{0,M \cap N}(\bfSigma) \ar[r] \ar[d] & \contd_G^{0,M}(\bfSigma) \ar[d]
      \\ \contd_G^{0,N}(\bfSigma) \ar[r] & \contd_G^{0}(\bfSigma)
    }
  \end{equation*}
  is a homotopy pushout.
\end{proposition}

\begin{proof}
  It is not difficult to check that Lemma~\ref{lem:formal-homotopy-pushout-square}
  applies, where we use $\caly_0 := \caly_N^+$, $\caly_1 := \caly_M^+$ and observe that
  $\caly_0 \cap \caly_1 = \caly_{M \cap N}$ and that, as $|\Sigma| = M \cup N$,
  $\contd_G^0(\bfSigma)|_{\caly_{M \cup N}} = \contd_G^0(\bfSigma)$.  To verify that the
  assumption from Lemma~\ref{lem:formal-excision} (as required in
  Lemma~\ref{lem:formal-homotopy-pushout-square}) is satisfied, observe that the distance
  between points in $M \setminus N$ and in $N \setminus M$ is uniformly bounded from below
  by some $\delta > 0$.  Let $Y_1 \in \caly_M^+$ and $E \in \contstrd_2^0(\bfSigma)$ be
  given.  The $\epsilon$-control requirement in~\ref{def:D_bfSigma:X} for $E$ gives us
  some $k_0 \in \IN$ such that if $\twovec{z',t'_0,t'_1}{z,t_0,t_1} \in E$ with
  $t_0=t'_0 \geq k_0$, then the distance of $p_\bfSigma(z)$ and $p_\bfSigma(z')$ in
  $|\Sigma|$ is $<\delta$; in particular $p_\bfSigma(z)$ and $p_\bfSigma(z')$ are either
  both in $M$ or in $N$.  Then with
  $Y_0 := |\bfSigma|^\wedge \times \IN_{\leq k_0} \in \caly(\bfSigma) \subseteq
  \caly_N^+$, $Y'_1 := Y_1 \setminus Y_1 \in \caly_M^+$ we have $(Y'_1)^E \in \caly_M^+$.
\end{proof}

Let now $\widehat \bfSigma = (\widehat \Sigma, \widehat P)$ be as in
Subsection~\ref{subsec:excision-bfD0}, i.e., $\widehat \Sigma = B \times \Delta^d$, where
$B$ is the set of $d$-simplices of $\Sigma$.  We can also apply the previous definitions
to $\widehat \bfSigma$ and obtain $\widehat N$,
$\contd_G^{0,\widehat M}(\widehat \bfSigma)$ and so on.

\begin{lemma}\label{lem:hat-M-vs-M} The functors
  \begin{equation*}
    \contd^{0,\widehat M \cap \widehat N }_G( \widehat \bfSigma )
    \to \contd^{0,M \cap N }_G(  \bfSigma )   \quad \text{and}
    \quad \contd^{0,\widehat M}_G(\widehat \bfSigma) \to \contd^{0,M}_G(\bfSigma)
  \end{equation*}
  induced by the projection $\widehat \bfSigma \xrightarrow{\bff} \bfSigma$ are
  equivalences.
\end{lemma}

\begin{proof}
  The projection $|\widehat\Sigma| \xrightarrow{f} \Sigma|$ restricts to an bijection
  $\widehat M \to M$.  The restrictions of the $l^\infty$-metrics of $|\widehat \Sigma|$
  and $|\Sigma|$ to $\widehat M$ and $M$ are different, but agree on path components.
  Moreover, in both cases the distance between points in different path components is
  bounded from below by a universal constant (depending only on $d$).  With these
  observations it is not difficult to verify the assertion.
\end{proof}

\begin{proposition}\label{prop:excision-bfD0} The diagram
  \begin{equation}\label{another-pushout}
    \xymatrix{\bfK \big( \contd^0_G( \widehat\bfSigma') \big) \ar[d]^{\widehat \iota_*} \ar[r]^{\bff'_*}
      &  \bfK \big( \contd^0_G(\bfSigma')\big) \ar[d]^{\iota_*}
      \\ \bfK \big( \contd^0_G(\widehat\bfSigma)\big) \ar[r]^{\bff_*}
      &
      \bfK \big( \contd^0_G(\bfSigma)\big) 
    }
  \end{equation} 
  obtained by applying K-theory to~\eqref{eq:excison-diagram-hat-no-hat} is a homotopy
  pushout diagram.
\end{proposition}

\begin{proof}
  Consider
  \begin{equation*}
    \xymatrix{\contd^{0,\widehat M \cap \widehat N }_G( \widehat \bfSigma ) \ar[r] \ar[d] \ar@{} [dr] |{(1)}  & 
      \contd^{0,\widehat M}_G(\widehat \bfSigma ) \ar[d] & 
      \contd^{0,\widehat M \cap \widehat N }_G( \widehat \bfSigma ) \ar[r] \ar[d] \ar@{} [dr] |{(3)} & 
      \contd^{0,\widehat M}_G(\widehat \bfSigma) \ar[d]  \\
      \contd^{0,\widehat N}_G(\widehat \bfSigma) \ar[r] \ar[d] \ar@{} [dr] |{(2)}   
      & 
      \contd^{0}_G(\widehat \bfSigma ) \ar[d] & 
      \contd^{0,M \cap N}_G(\bfSigma ) \ar[r] \ar[d] \ar@{} [dr] |{(4)} & 
      \contd^{0,M}_G(\bfSigma) \ar[d] \\
      \contd^{0,N}_G(\bfSigma) \ar[r] & 
      \contd^{0}_G(\bfSigma) & 
      \contd^{0,N}_G(\bfSigma) \ar[r] & 
      \contd^{0}_G(\bfSigma)  
    }
  \end{equation*}	
  We will argue that (2) is a homotopy pushout in K-theory.  This will give the assertion
  of the proposition, as Lemma~\ref{lem:thick-is-not-thick} (which also applies to
  $\widehat \bfSigma$) allows us to replace $\contd^{0,\widehat N}_G(\widehat \bfSigma)$
  with $\contd^{0}_G(\widehat \bfSigma')$ and $\contd^{0,N}_G(\bfSigma)$ with
  $\contd^{0}_G(\bfSigma')$.  (Note that the positions of the right/top and left/bottom
  corners in~\eqref{another-pushout} and in (2)  are switched.)
   
  Proposition~\ref{prop:M-N-MV} (which also applies to $\widehat \bfSigma$) tells us that (1)
  and (4) are homotopy pushouts in K-theory.  Lemma~\ref{lem:hat-M-vs-M} implies that (3)
  is a homotopy pushouts in K-theory as well.  As the combinations of (1) with (2) and
  (3) with (4) agree, this implies that (4) is homotopy pushouts in K-theory.
\end{proof}


\subsection{Skeleton continuity of $\contd^0_G(\--)$}\label{subsec:continuity-contd0}
       
\begin{proposition}\label{prop:continuity-contd0}
  For any $\bfSigma \in \regularPOrGSC$ the canonical map
  \begin{equation*}
    \hocolimunder_{d \in \IN} \bfK \big( \contd_G^0(\bfSigma^d) \big)
    \; \to \; \bfK \big(\contd_G^0(\bfSigma)\big)
  \end{equation*}   
  is an equivalence.
\end{proposition}

\begin{proof}
  The finite dimensional support condition in Definition~\ref{def:D_bfSigma} directly
  implies that $\contd_G^0(\bfSigma)$ is the directed union of the
  $\contd_G^0(\bfSigma^d)$ and this gives the result.
\end{proof}


\section{Outline of the construction of the transfer}\label{sec:outline-transfer}

We now assume that $\calb$ is a Hecke category with $G$-support.
As before we abbreviate $\contd^0_G(\bfSigma) = \contd^0_G(\bfSigma;\calb)$.
Theorem~\ref{thm:transfer-bfD0} asserts that the maps \begin{equation}\label{eq:needs-splitting} \bfK \contd_G^0\big(\bfJ_\CVCYC(G) \times
  \ZerodimR\big) \; \xrightarrow{\bfp_M} \; \bfK \contd_G^0(\ZerodimR)
\end{equation} 
induced by the projections
$\bfJ_\CVCYC(G) \times \ZerodimR \to \ZerodimR$ admit sections $\bftr_\ZerodimR$ that are natural in $\ZerodimR \in \regularPOrGSCzero$.  In
this outline we will concentrate on the case where $\ZerodimR = \ast$ is the terminal
object\footnote{i.e., $\ast = (\Delta^0,\ast_{\EPplus\All(G)})$.} in $\regularPOrGSC$; the
general case requires no real additional input.


\subsection{Sequences}\label{subsec:sequences} To construct a section 
to~\eqref{eq:needs-splitting} for $\ZerodimR = \ast$ we will work with finite chain complexes
and construct a homotopy coherent\footnote{See Appendix~\ref{app:homotopy-coherent}.}
functor
\begin{equation}\label{eq:is-splitting-ast} \contd_G^0(\ast) \to \ch_\fin \Idem
  \contd_G^0\big(\bfJ_\CVCYC(G)\big).
\end{equation}
We use the sequence description of $\contd_G^0\big(\bfSigma)$ from
Remark~\ref{rem:contd0-more-explicit} as a subcategory of
$\prodprimeinline_{\IN^{\times 2}} \contrCatUcoef{G}{|\bfSigma|^\wedge}{\calb}$.  We will for each
$\underlinetupel{t} \in \IN^{\times 2}$ construct a (homotopy coherent) functor
\begin{equation*}
  \tilde F_\underlinetupel{t} \; \colon \; \contrCatUcoef{G}{\ast}{\calb} \;
  \to \; \ch_\fin \big( \Idem \contrCatUcoef{G}{|\bfJ_\CVCYC(G)|^\wedge}{\calb} \big)
\end{equation*}
such that their product
\begin{equation*}
  \prodprimedisplay_{\IN^{\times 2}} \tilde F_\underlinetupel{t} \; \colon \; \contrCatUcoef{G}{\ast}{\calb}\; \to \;
  \prodprimedisplay_{\IN^{\times 2}} \ch_\fin \big(\Idem \contrCatUcoef{G}{|\bfJ_\CVCYC(G)|^\wedge}{\calb}\big)
\end{equation*}
restricts to the desired functor~\eqref{eq:is-splitting-ast}.  This boils down to
verifying the
conditions~\ref{rem:contd0-more-explicit:obj:suppobj},~\ref{rem:contd0-more-explicit:obj:suppX},~%
\ref{rem:contd0-more-explicit:obj:suppG},~\ref{rem:contd0-more-explicit:obj:finite},~%
\ref{rem:contd0-more-explicit:mor:suppX},~\ref{rem:contd0-more-explicit:obj:suppG} spelled out in
Remark~\ref{rem:contd0-more-explicit}.

Objects in $\contd_G^0(\ast)$ are sequences
$\bfB = (\bfB_\underlinetupel{t})_{\underlinetupel{t} \in \IN^{\times 2}}$ of objects in
$\contrCatUcoef{G}{\ast}{\calb}$, such that $\suppG \bfB_\underlinetupel{t} \subseteq K$
for all $\underlinetupel{t}$ for some compact subset $K$ that does not depend on
$\underlinetupel{t}$.  (In fact, we will pass to a subcategory of
$\contrCatUcoef{G}{\ast}{\calb}$, and have a bit more control over the
$\suppG \bfB_\underlinetupel{t}$.)  Write
$(\tilde F_\underlinetupel{t} (\bfB_\underlinetupel{t}))_n$ for the $n$-th chain module of
$\tilde F_\underlinetupel{t} (\bfB_\underlinetupel{t})$.  In order for
$\prodprimeinline_{\IN^{\times 2}} \tilde F_t$ to restrict to the desired functor, we
will, among other things, require that
$\suppobj (\tilde F_\underlinetupel{t}(\bfB_\underlinetupel{t}))_n \subseteq
|\bfJ_{\CVCYC}(G)|^\wedge$ is finite for all $\underlinetupel{t} \in \IN^{\times 2}$.
Morphisms in $\contd_G^0(\ast)$ are (equivalence classes of) sequences of morphisms
$\varphi = (\varphi_\underlinetupel{t})_{\underlinetupel{t} \in \IN^{\times 2}}$ in
$\contrCatUcoef{G}{\ast}{\calb}$ such that $\suppG \varphi_\underlinetupel{t} \subseteq K$
for all $\underlinetupel{t}$ for some compact subset $K$ that does not depend on
$\underlinetupel{t}$.  In order for
$\prodprimeinline_{\IN^{\times 2}} \tilde F_\underlinetupel{t}$ to restrict to the desired
functor, we will, among other things, need to verify the foliated control condition
from~\ref{def:D_bfSigma:X}.  This means roughly the following.  Given a compact subset $K$
of $G$ we need $\beta > 0$, $\eta_\underlinetupel{t}, \epsilon_\underlinetupel{t} > 0$ for
$\underline{t} = (t_0,t_1) \in \IN^{\times 2}$ with $\eta_\underlinetupel{t} \to 0$ as
$t_1 \to \infty$ (for fixed $t_0$) and $\epsilon_{\underlinetupel{t}} \to 0$ as
$t_0 \to \infty$ (uniform in $t_1$), such that if $\varphi$ is a morphism in
$\contrCatUcoef{G}{\ast}{\calb}$ with $\suppG \varphi \subseteq K$, then
\begin{equation*}
  \qquad \twovec{z'}{z} \in \suppX \tilde F_{\underlinetupel{t}}(\varphi) \;
  \implies \; \fold{\bfSigma}(z,z') < (\beta,\eta_\underlinetupel{t},\epsilon_\underlinetupel{t}).
\end{equation*}
These two required properties, finiteness for objects, 
  and foliated control for morphisms,
are in tension with each other\footnote{For example, it is not difficult to construct a
  functor
  $F \colon \contrCatUcoef{G}{\ast}{\calb}
  \to \ch_\fin \big(\contrCatUcoef{G}{|\bfJ_\CVCYC(G)|^\wedge}{\calb}\big)$ such
  that for $\twovec{z}{z} \in \suppX \tilde F_{t}(\varphi_{t})$  
  we
  even have $z=z'$, but such an $F$ will fail the required finiteness property for
  objects, see Remark~\ref{rem:shortcomings-singular-ch-cx}.}.  Let $X$ be the extended  Bruhat-Tits
building associated to $G$.  The functors $\tilde F_\underlinetupel{t}$ are constructed as
a composition
\begin{equation*}
  \contrCatUcoef{G}{\ast}{\calb} \; \xrightarrow{F_\underlinetupel{t}} \; \ch_\fin \Idem \contrCatUcoef{G}{X}{\calb}  \;
  \xrightarrow{(f_\underlinetupel{t})_*} \; \ch_\fin \big( \Idem \contrCatUcoef{G}{|\bfJ_\CVCYC(G)|^\wedge}{\calb}\big)
\end{equation*}
where the $F_\underlinetupel{t}$ are given by a tensor product with certain complexes over
$X$ and $f_\underlinetupel{t} \colon X \to |\bfJ_\CVCYC(G)|^\wedge$ are certain maps.  We
give a brief outline for both below.


\subsection{The diagonal tensor product}

Given two smooth $G$-representations $V$, $W$ we can equip $V \otimes W$ with the diagonal
action $g \cdot (v \otimes w) = gv \otimes gw$.  We obtain a functor
\[
  V \otimes \-- \; \colon \;\Rep(G) \; \to \; \Rep(G)
\]
on categories of smooth representations.  If the underlying vector space of $V$ is finite
dimensional, then this functor preserves finitely generated and projective
representations.  If $V$ is the permutation representation of a smooth $G$-set $\Sigma$,
then for Hecke category $\calb$ with $G$-suport this can be generalized\footnote{The
  formula in Subsection~\ref{subsec:diagonal-otimes} looks a priori different. The
  translation between the two functors uses a shearing isomorphism.} to
\[
  \Sigma \otimes \-- \; \colon \; \contrCatUcoef{G}{\ast}{\calb}
  \; \to \;  \Idem \contrCatUcoef{G}{\ast}{\calb}.
\] 
Moreover, a map $c \colon \Sigma \to X$ determines a lift of this functor to
\[
  (\Sigma,c) \otimes \-- \; \colon \; \contrCatUcoef{G}{\ast}{\calb}
  \; \to \; \Idem \contrCatUcoef{G}{X}{\calb}.
\]
In Subsection~\ref{subsec:S_upper_G} we define a $\IZ$-linear category $\cals^G(X)$, whose
objects are pairs $(\Sigma,c)$ as above, and in
Subsection~\ref{subsec:diagonal-otimes}  we obtain a functor
\[
  \-- \otimes \-- \; \colon \; \cals^G(X) \times \contrCatUcoef{G}{\ast}{\calb}
  \; \to \; \Idem \contrCatUcoef{G}{X}{\calb}.
\]
The functor $(\Sigma,c) \otimes \--$ will provide a splitting for
$\contrCatUcoef{G}{X}{\calb} \to \contrCatUcoef{G}{\ast}{\calb}$ if $\Sigma = G/G$, but
this does not give us enough flexibility to construct a sequence $F_t$ that will satisfy
the foliated control condition from~\ref{def:D_bfSigma:X}.  We can view the singular chain
complex $\bfS_*(X)$ of $X$ as a chain complex over $\cals^G(X)$\footnote{For technical
  reason we will have to replace $X$ with the set $S(X)$ of singular simplices in $X$ and
  view $\bfS_*(X)$ as a chain complex over $\cals^G(S(X))$; applying the barycenter map
  $S(X) \to X$ we then obtain a chain complex over $\cals^G(X)$.}, see
Subsection~\ref{subsec:singular-ch-cx-in-cals}, and obtain
\[
  \bfS_*(X) \otimes \-- \; \colon \; \contrCatUcoef{G}{\ast}{\calb}
  \; \to \; \ch \Idem \contrCatUcoef{G}{X}{\calb}.
\]
As $X$ is contractible this functor is much closer to providing a splitting and is much
better compatible with control conditions for morphisms,
see~\ref{lem:supp-rho-otimes-varphi:suppX}.  The remaining problem is that the singular
chain complex is very large and this will lead to conflict with the finiteness
conditions~\ref{def:D_bfSigma:obj}, see Remark~\ref{rem:shortcomings-singular-ch-cx}.  Of
course, as $X$ is contractible, $\bfS_*(X)$ is finite up to homotopy.  But such a homotopy
involves moving through $X$ and incorporating it into our construction would again lead to
conflict with the foliated control condition from~\ref{def:D_bfSigma:X}.  The solution is
a compromise between $X$ and a point.  We will use large balls $B_\underlinetupel{t}$ in
$X$.  Moreover, in place of the singular complex we will use the simplicial complex of a
suitable (fine) triangulation of $B_\underlinetupel{t}$.  The balls $B_\underlinetupel{t}$
are not $G$-invariant.  To resolve this we use that $B_\underlinetupel{t} \subseteq X$ is
a deformation retract via the radial projection $X \to B_\underlinetupel{t}$.  This way
$B_\underlinetupel{t}$ inherits a homotopy coherent $G$-action from the $G$-action on $X$.
Altogether we will construct a homotopy coherent functor
\begin{equation*}
  F_\underlinetupel{t} \colon \contrCatUcoef{G,U_\underlinetupel{t}}{\ast}{\calb}
  \; \to \; \ch_\fin \Idem \contrCatUcoef{G}{X}{\calb}, 
\end{equation*}
see Subsection~\ref{subsec:tensorproduct-subcomples-of-singular}.  Here
$U_\underlinetupel{t}$ is the compact open subgroup of $G$ that fixes the ball
$B_\underlinetupel{t}$ pointwise and
$\contrCatUcoef{G,U_\underlinetupel{t}}{\ast}{\calb} \subseteq
\contrCatUcoef{G}{\ast}{\calb}$ is the full subcategory on objects $(S,\pi,B)$ where
$\supp B(s) \subseteq U_t$ for all $s$.  Passing to this subcategory is not a serious
restriction, by the support cofinality property~\ref{def:Hecke-category:cofinal} for
$\calb$, the idempotent completions of
$\contrCatUcoef{G,U_\underlinetupel{t}}{\ast}{\calb}$ and $\contrCatUcoef{G}{\ast}{\calb}$
coincide.  The deformation of $X$ onto $B_\underlinetupel{t}$ still moves through $X$,
creating conflict with the foliated control condition from~\ref{def:D_bfSigma:X}.  We will
outline how this conflict is resolved in the next subsection.


\subsection{The maps $X \to |\bfJ_\CVCYC(G)|^\wedge$}

Recall that morphisms in $\contd_G^0(\ast)$ have relatively compact $G$-support.  Using
the equivalence relation on morphisms in $\contd_G^0(\ast)$ (or in
$\prodprimeinline \contrCatUcoef{G}{\ast}{\calb}$) this means that in the construction of the
$\tilde F_\underlinetupel{t}$  for each fixed $\underlinetupel{t}$ we only have to
control its interaction with a relatively compact set in $G$, specified later in
Subsection~\ref{subsec:data-choosing} and denoted $M_\underlinetupel{t}$.  The important
point is that $M_\underlinetupel{t} = M_{t_0,t_1} \to \infty$ as $t_0 \to \infty$.  In a
similar way we will for fixed $\underlinetupel{t}$ not need to worry about the deformation
of all of $X$ onto $B_\underlinetupel{t}$, but only about its restriction to the
$L_\underlinetupel{t}$-neighborhood $B_\underlinetupel{t}^{(L_\underlinetupel{t})}$ of
$B_\underlinetupel{t}$.  Here it will be important that
$L_\underlinetupel{t} = L_{t_0,t_1} \to \infty$ as $t_1 \to \infty$\footnote{Typically the
  radii of the $B_\underlinetupel{t}$ will grow much quicker than the
  $L_\underlinetupel{t}$.}.  This leads to the following requirements for the maps
$f_\underlinetupel{t} \colon X \to |\bfJ_\CVCYC(G)|^\wedge$.
\begin{enumerate}[label=(\alph*),leftmargin=*]
\item The restriction of $f_\underlinetupel{t}$ to
  $B_\underlinetupel{t}^{(L_\underlinetupel{t})}$ should be
  $M_\underlinetupel{t}$-equivariant up to a $\bfJ_\CVCYC(G)$-foliated error-term; i.e.,
  we have control over
  \begin{equation*}
    \fold{\bfJ_\CVCYC(G)}(gf_\underlinetupel{t}(x),f_\underlinetupel{t}(gx))
  \end{equation*}
  for $g \in M_\underlinetupel{t}$ and
  $x \in B_\underlinetupel{t}^{(L_\underlinetupel{t})}$.  The precise formulation
  is~\ref{nl:f:G};
  
\item The restriction of $f_\underlinetupel{t}$ to the tracks of the radial deformation of
  $B_\underlinetupel{t}^{(L_\underlinetupel{t})}$ to $B_\underlinetupel{t}$ is constant up
  to a $\bfJ_\CVCYC(G)$-foliated error-term; i.e., we have control over \begin{equation*}
    \fold{\bfJ_\CVCYC(G)}(f_\underlinetupel{t}(x),f_\underlinetupel{t}(\pi_{R'}(x)))
  \end{equation*}
  for $x \in B_\underlinetupel{t}^{(L_\underlinetupel{t})}$ and $\pi_{R'}(x)$ on the
  geodesic between $x$ and its image in $B_\underlinetupel{t}$ under the radial projection
  $B_\underlinetupel{t}^{(L_\underlinetupel{t})} \to B_\underlinetupel{t}$.  The precise
  formulation is~\ref{nl:f:H}.
\end{enumerate}
There is third requirement~\ref{nl:f:rho} for $f_\underlinetupel{t}$.  This should be
thought of as a substitute for continuity of $f_\underlinetupel{t}$.  In fact, with a more
careful choice for the resolution $|\bfJ\CVCYC(G)|^\wedge \to |\bfJ\CVCYC(G)|$ we could
arrange for the $f_\underlinetupel{t}$ to be continuous, but we found it more convenient
to allow some non-continuity for the $f_\underlinetupel{t}$.  
The construction of the
$f_\underlinetupel{t}$ uses a geodesic flow on $X$ and
depends on the fact that $X$ is $\CAT(0)$. 
It is outlined in Appendix~\ref{app:X-to-J}, details are worked out in~\cite{Bartels-Lueck(2023almost)}.


\section{The category $\cals^G(\Omega)$}\label{sec:cals}

Throughout this section we fix a smooth $G$-space $X$.
We assume that for
$K \subseteq X$ compact, the pointwise isotropy group $G_K = \bigcap_{x \in K} G_x$ is an
open subgroup of $G$. 
Later $X$ will be the extended Bruhat-Tits building associated to a reductive
$p$-adic group.


\subsection{The category $\cals^G(\Omega)$}\label{subsec:S_upper_G}

\begin{definition}\label{def:S_upper_G}
  For a smooth $G$-set $\Omega$ we define the additive category $\cals^G(\Omega)$ as
  follows.  Objects are pairs $\bfV = (\Sigma,c)$ where $\Sigma$ is a smooth $G$-set and
  $c \colon \Sigma \to \Omega$ is a $G$-map.  A morphism
  $\rho \colon \bfV = (\Sigma,c) \to \bfV' = (\Sigma',c')$ is an
  $\Sigma \times \Sigma'$-matrix
  $(\rho_\sigma^{\sigma'})_{\sigma \in \Sigma, \sigma' \in \Sigma'}$ over $\IZ$ satisfying
  the following two conditions
  \begin{enumerate}[
                 label=(\thetheorem\alph*),
                 align=parleft, 
                 leftmargin=*,
                 labelsep=10pt,
                 ]
  \item\label{def:S_upper_G:col-finite} for all $\sigma \in \Sigma$ the set
    $\{ \sigma' \in \Sigma' \mid \rho_\sigma^{\sigma'} \neq 0 \}$ is finite;
  \item\label{def:S_upper_G:G} for all $g \in G$, $\sigma \in \Sigma$, $\sigma' \in \Sigma'$ we have
    $\rho_{g\sigma}^{g\sigma'} = \rho_{\sigma}^{\sigma'}$.
  \end{enumerate}
  The support of $\rho$ is
  \begin{equation*}
    \suppX \rho := \Big\{ \twovec{c'(\sigma')}{c(\sigma)} \,\Big|\, \rho_{\sigma}^{\sigma'}
    \neq 0 \Big\} \subseteq \Omega \times \Omega.
  \end{equation*}
  The support of $\bfV$ is $\suppobj \bfV := c(\Sigma)$.
  Composition is matrix multiplication
  $(\rho' \circ \rho)_\sigma^{\sigma''} := \sum_{\sigma'} {\rho'}_{\sigma'}^{\sigma''} \circ \rho_\sigma^{\sigma'}$. 
  We will say that $\bfV$ is finite, if $\Sigma$ is finite. 
\end{definition}

The identity $\id_\bfV$ of $\bfV = (\Sigma,c)$ is given by
$(\id_\bfV)_\sigma^{\sigma'} = 1$ for $\sigma = \sigma'$ and
$(\id_\bfV)_\sigma^{\sigma'} = 0$ for $\sigma \not= \sigma'$.


\subsection{The singular chain complex of $X$ as a chain complex over $\cals^G(S(X))$}%
\label{subsec:singular-ch-cx-in-cals}

Let $S_n (X)$ be the set of singular $n$-simplices of $X$.  
Let $S(X)$ be the union of the $S_n(X)$.
By our assumption on $X$ this is a smooth  $G$-set via the $G$-action on $X$.    
Let $c_n \colon S_n (X) \to S(X)$ be the inclusion.
We obtain
$(S_n (X),c_n) \in \cals^G(S(X))$.  
We define $\partial_n \colon (S_n (X),c_n) \to (S_{n-1} (X),c_{n-1})$ by
\begin{equation}\label{eq:boundary-map} 
  (\partial_n)_\sigma^{\sigma'} := \begin{cases}
    (-1)^i & \text{if $\sigma'$ is the $i$-th face of $\sigma$;} \\ 0 & \text{else}.
  \end{cases} 
\end{equation}
We write $\bfS_*(X) \in \ch \cals^G(S(X))$ for the chain complex in $\cals^G(S(X))$ obtained this way.

Later we will use the $G$-map $\bary \colon S(X) \to X$ that sends
$\sigma \colon |\Delta^n| \to X$ to the image of the barycenter of $|\Delta^n|$ under
$\sigma$.  Then $\bary_*(\bfS_*(X)) \in \ch \cals^G(X)$\footnote{A $G$-map $f \colon \Omega \to \Omega'$ induces a map $f_* \colon
 \ch \cals^G(\Omega) \to \cals^G(\Omega')$ by composition.}.  
Working in $\cals^G(S(X))$ and
not in $\cals^G(X)$ will allow us to consider certain restrictions in
Subsection~\ref{subsec:estimates-over-S(X)}.


\subsection{Subspace}
For some purposes the singular chain complex $\bfS_*(X)$ is too big.  To replace it by
smaller chain complexes we will pass from $X$ to a subspace (and later to a subspace
equipped with a triangulation).  Typically, the subspace will not be $G$-invariant, and we
will have to pass to an open subgroup as well.

Let $U \subseteq G$ be an open subgroup.  Let $\ball \subseteq X$ be a $U$-invariant
subspace\footnote{Later on $\ball$ will be contained in the $U$-fixed point set $X^U$}.  We
write $S_n (\ball)$ for the set of singular simplices of $\ball$.  We obtain the chain complex
$\bfS_*(\ball)$ over $\cals^U(S(X))$ whose $n$-th chain module is $(S_n \ball , c_n|_{S_n (\ball)})$.  The
boundary map is still given by~\eqref{eq:boundary-map}.

Let $\ball'$ be a further $U$-invariant subspace.  A $U$-equivariant map
$f \colon \ball \to \ball'$ induces a chain map
$f_* \colon \bfS_* (\ball) \to \bfS_* (\ball')$ in $\ch \cals^U(S(X))$ with
\begin{equation*}
	(f_*)_\sigma^{\sigma'} :=  	\begin{cases}
		1 & \text{if} \; \sigma' = f \circ \sigma; \\ 0 & \text{else.} 
	\end{cases} 
\end{equation*}
A $U$-equivariant homotopy $H \colon \ball \times [0,1] \to \ball'$ with $H(-,0) = f_0$,
$H(-,1) = f_1$ determines a chain homotopy $H_* \colon (f_0)_* \simeq (f_1)_*$ by the
usual formula.  In order to give the formula in detail, we write $v_k$ for the $k$-th
vertex of $\Delta^n$ and let $i_j \colon |\Delta^{n+1}| \to |\Delta^n| \times [0,1]$ be
the affine map determined by
\begin{equation*}
	i_j(v_k) := \begin{cases}
                   (v_k,0) & j \leq k; \\  (v_{k-1},1) &  j > k.   	
                 \end{cases}
               \end{equation*}
Then\footnote{Usually $H_*$ is not written as a matrix; for this reason the formula may not look
familiar at a first glance.} 
\begin{equation*}
	(H_*)_{\sigma}^{\sigma'} = \sum_{j \colon \sigma' = H \circ (\sigma \times \id_{[0,1]}) \circ i_j } (-1)^j.
\end{equation*}

\begin{lemma}\label{lem:support-H_ast} For $H \colon \ball \times [0,1] \to \ball'$ we have
  \begin{equation*}
    \twovec{\sigma'}{\sigma} \in \suppX H_*  \quad \implies
    \quad \Image(\sigma') \subseteq \Image(H \circ \sigma \times \id_{[0,1]}).
  \end{equation*}
\end{lemma}

\begin{proof}
  This is a direct consequence of the formula for $H_*$ reviewed above.
\end{proof}

\begin{lemma}\label{lem:chain-homotopy-subdivided}
   Let $H \colon \ball \times [0,1] \to \ball'$ be a homotopy between $f_0$ and $f_1$.
   Let $d_X$ be a metric on $X$. 
   Assume that $\ball$ is compact.
   Assume that for all $s \in [0,1]$ and all $x,x' \in B$ we have
   \begin{equation*}
   	  d_X\big(H(x,s),H(x',s)\big) \leq d_X(x,x').
   \end{equation*}
   Let $\epsilon > 0$.  Then there is a chain homotopy
   $\widetilde{H} \colon \bfS_*(\ball) \to \bfS_{*+1}(\ball')$   between $(f_0)_*$ and
   $(f_1)_*$     with the following property.  Let $\twovec{\sigma'}{\sigma} \in \suppX
   \widetilde{H}$\footnote{The support of a graded map is the union of the supports of the maps in all degrees.}.    Then
   \begin{enumerate}[label=(\alph*),leftmargin=*]
   \item if the diameter of the image of $\sigma$ is $< \kappa$, then the diameter of the
     image of $\sigma'$ is $<\kappa+\epsilon$;
   \item $\Image \sigma' \subseteq \Image(H \circ \sigma \times \id_{[0,1]})$.
   \end{enumerate}
 \end{lemma}
 
 \begin{proof}
   As $\ball$ is compact we can find $\delta > 0$ such that
   $d_X(H(x,s),H(x,s')) < \epsilon/2$ for all $x \in D$ and all $s,s' \in [0,1]$ with
   $|s-s'| < \delta$.  Now we choose $0=s_0 < s_1 < \cdots < s_n =1$ with
   $|s_{i+1} - s_i| < \delta$.  Then
   $\widetilde{H} := \sum_{i=1}^{n} (H|_{\ball \times [s_{i-1},s_i]})_*$  is a chain homotopy between $(f_0)_*$ and
   $(f_1)_*$.  Now
   $\supp \widetilde{H}_{\sigma}^{\sigma'} \subseteq \bigcup_{i=1}^{n} \supp
   \big((H|_{\ball \times [s_{i-1},s_i]})_*\big)$ and it is easy to check that $\widetilde{H}$ has the required properties.
 \end{proof}


 \subsection{Triangulations}\label{subsec:triangulations}
 Assume that the $U$-invariant subspace $\ball$ of $X$ is equipped with a $U$-invariant
 triangulation, i.e., $\ball = |K|$ for a simplicial complex $K$ with a smooth $U$-action.
 Assume that the triangulation is locally ordered, i.e., there is a partial order on the
 set of vertices that restricts to a linear order on the vertex set of every simplex.
 This ensures that every simplex in the triangulation determines a unique singular simplex
 of $\ball$ and therefore of $X$.  In particular, we can view the set $\simp_n(\ball)$ of
 $n$-simplices of $\ball$ as a subset of $S_n \ball$.  We obtain the chain complex
 $\bfC_*(\ball) \in \ch \cals^U(\res_G^U S(X))$, whose $n$-th chain module is
 $(\simp_n(\ball),c_n|_{\simp_n(\ball)})$.  It comes with an inclusion
 $i \colon \bfC_*(\ball) \to \bfS_*(\ball)$ defined by
\begin{equation*}
  i_\sigma^{\sigma'} := \begin{cases}
    1 & \text{if} \; \sigma' =  \sigma; \\ 0 & \text{else}.
  \end{cases} 
\end{equation*} 
Let $d_X$ be a metric on $X$.  For a singular simplex $\sigma \colon \Delta^n \to X$ we
write $\Image \sigma$ for its image and $(\Image \sigma)^\epsilon$ for the open
$\epsilon$-neighborhood of $\Image \sigma$. 

\begin{lemma}\label{lem:singular-vs-simplicial} Assume that all simplices of the
  triangulation of $\ball$ are of diameter $< \epsilon$.  In $\ch \cals^U(S(X))$ there
  exists a chain map $r \colon \bfS_*(\ball) \to \bfC_*(\ball)$ with
  $r \circ i = \id_{\bfC_*(\ball)}$ and a chain homotopy
  $H \colon \id_{\bfS_*(\ball)} \simeq i \circ r$ such that if
  $\twovec{\sigma'}{\sigma} \in \suppX r \cup \suppX H$, then
  $\Image \sigma' \subseteq (\Image \sigma)^\epsilon$.
\end{lemma}

\begin{proof}
  This a minor variation of~\cite[Lem.~6.9]{Bartels-Lueck-Reich(2008hyper)}.  Let $\calk$
  be the poset of subcomplexes of $K$ ordered by inclusion.  We view $\calk$ as a category
  and work now in the abelian category of ${\IZ\calk}$-modules\footnote{That is in the
    category of functors $\calk \to \MODcat{\IZ}$.}.  For $K_0 \in \calk$ let
  $\underline{C}_*(K_0)$ be the simplicial chain complex of $K_0$ and
  $\underline{S}_*(K_0)$ be the singular chain complex of $|K_0|$.  This defines chain
  complexes $\underline{C}_*$ and $\underline{S}_*$ of ${\IZ\calk}$-modules and it is not
  difficult to check that the underlying ${\IZ\calk}$-modules $\underline{C}_n$,
  $\underline{S}_n$ are free (and thus projective) for all $n$. 
  Write
  $\underline{i} \colon \underline{C}_* \to \underline{S}_*$ for the inclusion.  Each
  $\underline{i}_n \colon \underline{C}_n \to \underline{S}_n$ is the inclusion of a
  direct summand.  Moreover, $\underline{i}$ induces an isomorphism in homology (taken in
  the category of $\IZ\calk$-modules).  By general results in homological algebra (in
  abelian categories) it follows that there exists a chain map
  $\underline{r} \colon \underline{S}_* \to \underline{C}_*$ with
  $\underline{r} \circ \underline{i} = \id_{C_*}$ and a homotopy
  $\underline{H} \colon \id_{S_*} \simeq i \circ r$.  We can set $r := \underline{r}(K)$,
  $H := \underline{H}(K)$.  The additional properties of $r$ and $H$ follow from the
  functoriality in $\calk$ of $\underline{r}$ and $\underline{H}$: Suppose
  $r_\sigma^{\sigma'} \neq 0$.   Let $K_\sigma$ be the smallest subcomplex of $K$ such that
  $|K_\sigma|$ contains $\Image \sigma$, so $\sigma \in \underline{S}_*(K_\sigma)$.  As
  $\underline{r}(K_\sigma) \colon \underline{S}_*(K_\sigma) \to \underline{C}_*(K_\sigma)$
  it  follows that $\sigma'$ is a simplex of $K_\sigma$.  Now
  $|K_\sigma| \subseteq (\Image \sigma)^\epsilon$ as the diameter of simplices in $K$ are
  of diameter $< \epsilon$.  Thus $\Image \sigma' \subseteq (\Image \sigma)^\epsilon$.
  The same argument applies to $H$.
\end{proof}


\section{Diagonal tensor products}\label{sec:tensor-products}

Throughout this section we fix a $G$-set $\Lambda$ and a smooth $G$-space $X$.  Later we
will have $\Lambda = |\ZerodimR|^\wedge$ for $\ZerodimR \in \regularPOrGSCzero$, while for $X$ we will take
the extended Bruhat-Tits building associated to a reductive $p$-adic group.  We also fix a
smooth $G$-set $\Omega$.  Later $\Omega = S(X)$ will be the set of singular simplices of
$X$.
  
To simplify the discussion we assume throughout this section that $\calb$ is the category
$\calb(G;R)$ from Example~\ref{ex:calb(G;R)}.   This means in particular that if
$\bfB = (S,\pi,B)$ is an object from $\contrCatUcoef{G}{\Lambda}{\calb}$, then the $B(s)$
for $s \in S$ are compact open subgroups of $G$.  Also, if
$\varphi \colon \bfB = (S,\pi,B) \to (S',\pi',S')$ is a morphism in
$\contrCatUcoef{G}{\Lambda}{\calb}$, then the $\varphi_s^{s'}$ are elements of the Hecke
algebra $\calh(G;R)$ satisfying $\varphi_s^{s'}(a'ga) = \varphi_s^{s'}(g)$ for all
$a \in B(s)$, $a' \in B(s')$.

Everything\footnote{More precisely, there is a diagonal tensor product such that
  Lemmas~\ref{lem:id-is-id}, and~\ref{lem:properties-diag-tensor}, and the
  identities~\eqref{eq:ind-res-is-ind},~\eqref{eq:ind-res-is-ind-morphisms} are still
  valid, see~\cite[Lem.~7.41,~7.43,~7.44]{Bartels-Lueck(2023foundations)}.  Everything else in this section just depends on these results.}  we need also
works for general Hecke categories with $G$-support and is treated in detail
in~\cite[Sec.~7]{Bartels-Lueck(2023foundations)}.


\subsection{Precursor}~\label{subsec:Precursor}
Our first goal in this section is the construction of a bilinear functor
\begin{equation*}
  \-- \otimes \-- \; \colon \cals^G(\Omega) \times \contrCatUcoef{G}{\Lambda}{\calb} \to \Idem \contrCatUcoef{G}{\Omega \times \Lambda}{\calb}.
\end{equation*} 
The construction is easier under some simplifying assumptions.  
So we assume for this subsection
that $G$ is discrete and consider the full subcategory
$\contrCatUcoef{G}{\Lambda}{\calb^{1}}$ of $\contrCatUcoef{G}{\Lambda}{\calb}$ on the
objects $(S,\pi,B)$,  for  which $B(s)$ is the trivial subgroup for all $s$.  
  For
$\bfV = (\Sigma,c) \in \cals^G(X)$ and
$\bfB = (S,\pi,B) \in \contrCatUcoef{G}{\Lambda}{\calb^{1}}$, we can then define
\begin{equation}\label{eq:V-otimes-B-wrong-B(s)}
  \bfV \otimes \bfB := (\Sigma \times S,c \times \pi, (\sigma,s) \mapsto B(s)).
\end{equation}
For $\rho \colon (\Sigma,c) \to (\Sigma',c')$ and
$\varphi \colon (S,\pi,B) \to (S',\pi',B')$ we can define
\begin{equation}\label{eq:rho-otimes-varphi}
  (\rho \otimes \varphi)^{\sigma',s'}_{\sigma,s}(g) := \rho_{g\sigma}^{\sigma'}\cdot \varphi_s^{s'}(g).
\end{equation}
The general idea is that we use the $G$-action on $\Sigma$ to twist morphisms in
$\contrCatUcoef{G}{\Lambda}{\calb^{1}}$\footnote{The category
  $\contrCatUcoef{G}{\Lambda}{\calb^{1}}$ is equivalent to the category of free
  $R[G]$-modules. Under this equivalence our functor is equivalent to
  $((\Sigma,c),M) \mapsto \IZ[\Sigma] \otimes_{\IZ} M$, where the tensor product over
  $\IZ$ is equipped with the diagonal action of $G$.}.  In the general case these formulas
do not define a functor; for example we can have
$(\rho \otimes \varphi)^{\sigma',s'}_{\sigma,s} \neq (\rho \otimes
\varphi)^{\sigma',s'}_{\sigma,s}(a'ga)$ for $a' \in B'(s')$, $a \in B(s)$.  To account for
this we will replace $(\sigma,s) \mapsto B(s)$ with
$(\sigma,s) \mapsto G_\sigma \cap B(s)$.  The only remaining drawback is that the functor
obtained this way does not preserve units.  We will correct this using the idempotent
completion.

\begin{remark}[Non-unital categories]
  In this remark we will contemplate categories without units.  Let
  $\calb^\nounit := \underline{\calh(G;R)}$ be the $\IZ$-linear category with exactly one
  object whose endomorphism ring is the Hecke algebra $\calh(G;R)$.  As $\calh(G;R)$ does
  not have a unit, $\calb^\nounit$ is a non-unital category, but its idempotent completion
  $\Idem \calb^\nounit$ has units.  Via $U \mapsto \frac{\chi_U}{\mu(U)}$ the category
  $\calb = \calb(G;R)$ can be identified with a full subcategory of $\Idem \calb^\nounit$.
  Definition~\ref{def:contrCatNUcoef} also makes sense for $\calb^\nounit$ in place of
  $\calb$ and we obtain the non-unital category $\contrCatNUcoef{G}{X}{\calb}$.  Now the
  formulas~\eqref{eq:V-otimes-B-wrong-B(s)} and~\eqref{eq:rho-otimes-varphi} define a
  functor (of non-unital categories)
  \begin{equation*}
    \cals^G(\Omega) \times \contrCatNUcoef{G}{\Lambda}{\calb}
    \to \contrCatNUcoef{G}{\Omega \times \Lambda}{\calb}.
  \end{equation*}
  Of course this functor sends idempotents to idempotents and so induces a functor (of
  unital categories)
  \begin{equation*}
    \cals^G(\Omega) \times \Idem\contrCatNUcoef{G}{\Lambda}{\calb}
    \to \Idem\contrCatNUcoef{G}{\Omega \times \Lambda}{\calb}.
  \end{equation*}
  One can now identify $\contrCatNUcoef{G}{\Omega}{\calb}$ with a full subcategory of
  $\Idem \contrCatNUcoef{G}{\Omega}{\calb}$ and use this to obtain
  \begin{equation*}
    \cals^G(\Omega) \times \contrCatUcoef{G}{\Lambda}{\calb}
    \to \Idem \contrCatUcoef{G}{\Omega \times \Lambda}{\calb}.
  \end{equation*}
  This is essentially what we will do in Subsections~\ref{subsec:diagonal-otimes_0}
  and~\ref{subsec:diagonal-otimes}, although we will avoid non-unital categories and
  instead give the resulting formulas more directly.
\end{remark}

\subsection{The diagonal tensor product $\otimes_0$}\label{subsec:diagonal-otimes_0}

We define
\begin{equation*} 
  \-- \otimes_0 \-- \; \colon \cals^G(\Omega) \times \contrCatUcoef{G}{\Lambda}{\calb}
  \to  \contrCatUcoef{G}{\Omega \times \Lambda}{\calb}
\end{equation*}
as follows.  For $\bfV = (\Sigma,c) \in \cals^G(X)$ and
$\bfB = (S,\pi,B) \in \contrCatUcoef{G}{\Lambda}{\calb}$ we set
\begin{equation*}
  \bfV \otimes_0 \bfB := (\Sigma \times S,c \times \pi, (\sigma,s) \mapsto G_\sigma \cap B(s)).
\end{equation*} 
For morphisms $\rho \colon \bfV = (\Sigma,c) \to \bfV'= (\Sigma',c')$ in $\cals^G(\Omega)$
and $\varphi \colon \bfB = (S,\pi,B) \to \bfB' = (S',\pi',B')$ in
$\contrCatUcoef{G}{\Lambda}{\calb}$ we define
\begin{equation*}
  (\rho \otimes_0 \varphi)^{s,'\sigma'}_{s,\sigma}(g) := \rho_{g\sigma}^{\sigma'}\cdot \varphi_s^{s'}(g)
\end{equation*}
as in~\eqref{eq:rho-otimes-varphi}.

We will check in Lemmas~\ref{lem:otimes0-well} and~\ref{lem:composition-otimes0} below
that $\rho \otimes_0 \varphi$ is well defined and compatible with composition.  While we
do not claim that $\id_\bfV \otimes_0 \id_{\bfB}$ is $\id_{\bfV \otimes_0 \bfB}$,
compatibility with composition implies that $\id_\bfV \otimes_0 \id_{\bfB}$ is an
idempotent endomorphism of $\bfV \otimes_0 \bfB$.

\begin{lemma}\label{lem:otimes0-well}
  Let $\rho \colon \bfV = (\Sigma,c) \to \bfV'= (\Sigma',c')$ in $\cals^G(\Omega)$ and
  $\varphi \colon \bfB = (S,\pi,B) \to \bfB' = (S',\pi',B')$ in
  $\contrCatUcoef{G}{\Lambda}{\calb}$.  Then
  $\rho \otimes_0 \varphi \colon \bfV \otimes_0 \bfB \to \bfV' \otimes \bfB'$ is a
  morphism in $\contrCatUcoef{G}{\Omega \times \Lambda}{\calb}$
\end{lemma}

\begin{proof}
  For $a \in G_\sigma \cap B(s)$, $a' \in G_{\sigma'} \cap B(s')$ we have
  \begin{align*}
    (\rho \otimes_0 \varphi)^{\sigma',s'}_{\sigma,s}(a'ga) = \rho_{a'ga\sigma}^{\sigma'}\cdot \varphi_s^{s'}(a'ga) =
    & \rho_{g\sigma}^{(a')^{-1}\sigma'}\cdot \varphi_s^{s'}(g)
    \\ = &\rho_{g\sigma}^{\sigma'}\cdot \varphi_s^{s'}(g) = (\rho \otimes_0 \varphi)^{\sigma',s'}_{\sigma,s}(g),
  \end{align*}
  so for fixed $\sigma,\sigma',s,s'$,
  $(\rho \otimes_0 \varphi)^{\sigma',s'}_{\sigma,s} \colon G_\sigma \cap B(s) \to
  G_{\sigma'} \cap B(s')$ is a morphism in $\calb = \calb(G;R)$.  We also need to check
  that $\rho \otimes_0 \varphi$ is column finite.  Fix $(\sigma,s)$.  We need to check
  that there are only finitely many $(\sigma',s')$ with
  $(\rho \otimes \varphi)_{\sigma,s}^{\sigma',s'} \neq 0$.  As $\varphi$ is column finite
  there is $S'_0 \subset S'$ finite such that $\varphi_s^{s'} \neq 0$ implies
  $s' \in S'_0$.  The $\varphi_{s}^{s'}$ are compactly supported.  Thus there is
  $M \subseteq G$ compact such that $\varphi_{s}^{s'}(g) \neq 0$ implies $g \in M$.  As
  $\Sigma$ is a smooth $G$-set, the set $M\cdot\sigma \subseteq \Sigma$ is finite.  As
  $\rho$ is column finite there is $\Sigma'_0 \subset \Sigma'$ finite such that
  $\rho_{g\sigma}^{\sigma'} \neq 0$ with $g \in M$ implies $\sigma' \in \Sigma'_0$.  Now
  if $(\rho \otimes \varphi)_{s,\sigma}^{s',\sigma'} \neq 0$, then for some $g \in G$ we
  have $\rho^{\sigma'}_{g\sigma} \neq 0$ and $\varphi_{s}^{s'}(g) \neq 0$.  Thus
  $(\sigma',s') \in \Sigma'_0 \times S'_0$.
\end{proof}

\begin{lemma}\label{lem:composition-otimes0}
  Let $(S,\pi) \xrightarrow{\varphi} (S',\pi') \xrightarrow{\varphi'} (S'',\pi'')$ be
  composable morphisms in $\contrCatUcoef{G}{\Lambda}{\calb}$ and
  $(\Sigma,c) \xrightarrow{\rho} (\Sigma',c') \xrightarrow{\rho'} (\Sigma'',c'')$ be
  composable morphisms in $\cals^G(\Omega)$.  Then
  \begin{equation*}
    (\rho' \otimes_0 \varphi') \circ (\rho \otimes_0 \varphi) = (\rho' \circ \rho) \otimes_0 (\varphi' \circ \varphi).
  \end{equation*}
\end{lemma}

\begin{proof} 
  \begin{align*}
    \big((\rho'  \otimes_0 \varphi') \circ (\rho \otimes_0 &\varphi)\big)_{\sigma,s}^{\sigma'',s''}(g) 
    \;\;= \sum_{\sigma' \in \Sigma',s \in S'}  \Big( (\rho' \otimes_0 \varphi')^{\sigma'',s''}_{\sigma',s'} \circ (\rho \otimes_0 \varphi)^{\sigma',s'}_{\sigma,s} \Big) (g) \\
    & = \sum_{\sigma' \in \Sigma',s \in S'} \int_{x \in G}  (\rho' \otimes_0 \varphi')^{\sigma'',s''}_{\sigma',s'}(gx) \circ (\rho \otimes_0 \varphi)^{\sigma',s'}_{\sigma,s}  (x^{-1}) \\ 
    & = \sum_{\sigma' \in \Sigma',s \in S'} \int_{x \in G}  (\rho')^{\sigma''}_{gx\sigma'} \cdot (\varphi')^{s''}_{s'}(gx) \cdot \rho^{\sigma'}_{x^{-1}\sigma} \cdot \varphi^{s'}_{s}  (x^{-1}) \\ 
    & = \sum_{\sigma' \in \Sigma',s \in S'} \int_{x \in G}  (\rho')^{\sigma''}_{gx\sigma'} \cdot \rho^{\sigma'}_{x^{-1}\sigma} \cdot (\varphi')^{s''}_{s'}(gx) \cdot \varphi^{s'}_{s}  (x^{-1}) \\ 
    & = \sum_{\sigma' \in \Sigma',s \in S'} \int_{x \in G}  (\rho')^{\sigma''}_{gx\sigma'} \cdot \rho^{gx\sigma'}_{g\sigma} \cdot (\varphi')^{s''}_{s'}(gx) \cdot \varphi^{s'}_{s}  (x^{-1}) \\ 
    & = \sum_{\sigma' \in \Sigma',s \in S'} \int_{x \in G}  (\rho')^{\sigma''}_{\sigma'} \cdot \rho^{\sigma'}_{g\sigma} \cdot (\varphi')^{s''}_{s'}(gx) \cdot \varphi^{s'}_{s}  (x^{-1}) \\ 
 & = (\rho' \circ \rho)^{\sigma''}_{g\sigma} \cdot    (\varphi' \circ \varphi)^{s''}_{s} (g)  \\
    & = \big((\rho' \circ \rho) \otimes_0 (\varphi' \circ \varphi)\big)^{\sigma'',s''}_{\sigma,s} (g).
  \end{align*}     
\end{proof}


\subsection{The diagonal tensor product $\otimes$}\label{subsec:diagonal-otimes}

For $\bfV  \in \cals^G(\Omega)$ and $\bfB  \in \contrCatUcoef{G}{\Lambda}{\calb}$ we set
\begin{equation*}
	\bfV \otimes \bfB := ( \bfV \otimes_0 \bfB, \id_\bfV \otimes_0 \id_{\bfB} ).
\end{equation*} 
For morphisms
$\rho \colon \bfV  \to \bfV'$ in $\cals^G(X)$ and
$\varphi \colon \bfB  \to \bfB'$ in $\contrCatUcoef{G}{\Lambda}{\calb}$ we define
\begin{equation*} 
   (\rho \otimes \varphi) := (\rho \otimes_0 \varphi) \colon \bfV \otimes \bfB \to \bfV' \otimes \bfB'.
\end{equation*}
Now $\id_\bfV \otimes \id_\bfB = \id_{\bfV \otimes \bfB}$.
Altogether we have now defined a bilinear functor
\begin{equation}\label{eq:diagonal-otimes}
  \-- \otimes \-- \; \colon \cals^G(\Omega) \times \contrCatUcoef{G}{\Lambda}{\calb}
  \to  \Idem \contrCatUcoef{G}{\Omega \times \Lambda}{\calb}.
\end{equation}
The following observation will often allow us to get rid of idempotent completions.

\begin{lemma}\label{lem:id-is-id}
  Let $\bfV = (\Sigma,c) \in \cals^G(\Omega)$ and $\bfB = (S,\pi,B) \in \contrCatUcoef{G}{\Lambda}{\calb}$.  If
  $\Sigma$ is fixed pointwise by all $B(s)$, then $\bfV \otimes \bfB = \bfV \otimes_0 \bfB$.
\end{lemma}

\begin{proof}
  The content of the lemma is that $\id_{\bfV} \otimes_0 \id_{\bfB}$ is the identity of
  $\bfV \otimes_0 \bfB$, not just an idempotent.  Indeed, since
  $G_\sigma \cap B(s) = B(s)$ for all $\sigma$ and $s$ we have
  \begin{equation*}
    (\id_{\bfV} \otimes_0 \id_{\bfB})_{\sigma,s}^{\sigma',s'}(g) = (\id_\bfV)_{g\sigma}^{\sigma'}
    (\id_\bfB)_s^{s'} (g)  = (\id_{\bfV \otimes_0 \bfB})_{\sigma,s}^{\sigma',s'}(g). 
  \end{equation*}    	
\end{proof}

For $E \subseteq \Omega \times \Omega$ and $E' \subseteq \Lambda \times \Lambda$
we use the following convention
\begin{equation}\label{eq:cheated-times}
  E \times E' := \Big\{ \twovec{x',\lambda'}{x,\lambda} \; \Big| \; \twovec{x'}{x} \in E, \twovec{\lambda'}{\lambda}
  \in E' \Big\} \subseteq (\Omega \times \Lambda)^{\times 2}.
\end{equation}

\begin{lemma}\label{lem:properties-diag-tensor}
   Let $\bfV = (\Sigma,c) \in \cals^G(\Omega)$ and $\bfB = (S,\pi,B) \in \contrCatUcoef{G}{\Lambda}{\calb}$.
   \begin{enumerate}[
                 label=(\thetheorem\alph*),
                 align=parleft, 
                 leftmargin=*,
                 labelindent=2pt,
                 labelsep=10pt,
                 ]
               \item\label{lem:finiteness_and_suppobj-diag-tensor-obj:finite} If $\bfV$ and $\bfB$ are finite,
                 then $\bfV \otimes_0 \bfB$ is finite as well;
   	\item\label{lem:finiteness_and_suppobj-diag-tensor-obj:suppobj}
   	      $\suppobj (\bfV \otimes_0 \bfB) = \suppobj \bfV \times \suppobj \bfB$. 
   \end{enumerate}
   Let $\rho \colon \bfV = (\Sigma,c) \to \bfV' = (\Sigma',c')$ in $\cals^G(\Omega)$,
  $\varphi \colon \bfB = (S,\pi,B) \to \bfB' = (S',\pi',B')$ in $\contrCatUcoef{G}{\Lambda}{\calb}$.  Then
  for $\rho \otimes \varphi$ in $\contrCatUcoef{G}{\Omega \times \Lambda}{\calb}$ we have
  \begin{enumerate}[
                 label=(\thetheorem\alph*),
                 align=parleft, 
                 leftmargin=*,
                 labelindent=2pt,
                 labelsep=10pt,
                 resume
                 ]
               \item\label{lem:supp-rho-otimes-varphi:suppX}
                 $\suppX (\rho \otimes \varphi) \; \subseteq \; \suppX \rho \times \suppX \varphi$;
               \item\label{lem:supp-rho-otimes-varphi:suppG}
                 $\suppG (\rho \otimes \varphi) \; \subseteq \; \suppG \varphi$.
  \end{enumerate}
\end{lemma}

\begin{proof}
	These claims are straight forward from the definitions.
	We give some details for~\ref{lem:supp-rho-otimes-varphi:suppX}. 
  Let $\twovec{x',\lambda'}{x,\lambda} \in \suppX (\rho \otimes \varphi)$.  Then there are
  $\sigma \in \Sigma$, $\sigma' \in \Sigma'$, $s \in S$, $s' \in S'$, $g \in G$ with
  $(\rho \otimes \varphi)_{\sigma,s}^{\sigma',s'}(g) \neq 0$ and $x = gc(\sigma)$,
  $x'=c'(\sigma')$, $\lambda = g\pi(s)$, $\lambda' = \pi'(s')$.  By definition
  $(\rho \otimes \varphi)_{\sigma,s}^{\sigma',s'}(g) =
  \rho^{\sigma'}_{g\sigma} \cdot \varphi_s^{s'}(g)$ and thus
  $ \twovec{x'}{x} = \twovec{c'(\sigma')}{gc(\sigma)} = \twovec{c'(\sigma')}{c(g\sigma)}
  \in \suppX \rho$ and
  $\twovec{\lambda'}{\lambda} = \twovec{\pi'(s')}{g\pi(s)} \in \suppX \varphi$.
\end{proof}


\subsection{Tensor product with a singular chain complex}%
\label{subsec:tensor-product-singular-chain-cx}

Consider the singular chain complex $\bfS_*(X) \in \ch \cals^G(S(X))$ from
Subsection~\ref{subsec:singular-ch-cx-in-cals}.  Using the diagonal tensor
product~\eqref{eq:diagonal-otimes} we obtain a functor
\begin{equation*}
  \bfS_*(X)\, \otimes \; \--  \; \colon \contrCatUcoef{G}{\Lambda}{\calb}
  \, \to \, \ch \Idem \big( \contrCatUcoef{G}{S(X) \times \Lambda}{\calb}\big).
\end{equation*}
Note that for a morphism $\varphi \colon \bfB \to \bfB'$ in
$\contrCatUcoef{G}{\Lambda}{\calb}$ by~\ref{lem:supp-rho-otimes-varphi:suppX} we have
\begin{equation}\label{eq:supp-id-otimes-varphi} \suppX \big(\id_{\bfS_*(X)} \otimes
  \varphi\big)
  \subseteq \Big\{ \twovec{\sigma,\lambda'}{\sigma,\lambda} \,\Big|\, \sigma \in S(X),
  \twovec{\lambda'}{\lambda} \in \suppX \varphi \Big\}.
\end{equation}

\begin{remark}[Shortcomings of $\bfS_*(X)$]\label{rem:shortcomings-singular-ch-cx}
  Let $J := |\bfJ_\calf(G)|^{\wedge}$,
  $\ZerodimR \in \regularPOrGSCzero$, and $\Lambda := |\ZerodimR|^\wedge$.  Let
  $f \colon S(X) \to |\bfJ_\calf(G)|$ be a $G$-equivariant map.  Suppose also that
  $\widehat f \colon S(X) \to |\bfJ_\calf(G)|^\wedge$ is a lift of $f$.  In light
  of~\eqref{eq:supp-id-otimes-varphi} one might hope, that
  $(\widehat f \times \id_{|P|^\wedge})_* ( \bfS_*(X) \otimes -)$ induces a
  functor\footnote{We use the sequence description of $\contd^0_G(\--)$ from
    Remark~\ref{rem:contd0-more-explicit}.}
  \begin{eqnarray*}
    \contd^0_G(P)
    & \to &
      \ch \Idem \contd^0_G\big(\bfJ_\calf(G)  \times P\big),
    \\
    (\bfB_{\underlinetupel{t}})_{\underlinetupel{t} \in \IN^{\times  2}}
    & \mapsto &
                \big((\widehat f \times \id_{|P|^\wedge})_*
                (\bfS_*(X) \otimes \bfB_{\underlinetupel{t}})\big)_{\underlinetupel{t} \in  \IN^{\times  2}},  
  \end{eqnarray*}
  but this is not the case.  Let
  $(\bfB_{\underlinetupel{t}})_{\underlinetupel{t} \in \IN^{\times 2}} =
  (S_\underlinetupel{t},\pi_\underlinetupel{t},B_\underlinetupel{t})_{\underlinetupel{t}
    \in \IN^{\times 2}} \in \contd^0_G(P)$.  Typically
  $\big((\widehat f \times \id_{|P|^\wedge})_*(\bfS_*(X) \otimes
  \bfB_\underlinetupel{t})\big)_{\underlinetupel{t} \in \IN^{\times 2}}$ will fail in two
  ways to define a chain complex in $ \Idem \contd^0_G(P \times \bfJ_\calf)$.

  Firstly, the boundary maps
  $(\widehat f \times \id_{|P|^\wedge})_*(\partial_n \otimes \id_{\bfB_t})$ do not satisfy
  the required control conditions to define morphisms in
  $\contd^0_G(\bfJ_\calf(G) \times P)$.  For example the $\epsilon$-control condition over
  $|J_\calf|$ from Remark~\ref{rem:fol-contrl-gives-eps-contrl} will typically fail.  The
  support of
  $(\widehat f \times \id_P)_*(\partial_n \otimes \id_{\bfB_\underlinetupel{t}})$ is the
  set of all
  \begin{equation*}
    \twovec{\widehat f(x'),\lambda}{g\widehat f(x),g\lambda}
    \in ( |\bfJ_\calf|^\wedge \times |P|^\wedge)^{\times 2}
  \end{equation*} 
  with $g \in \suppG B_\underlinetupel{t}(s)$ for some $s \in S$ with
  $\pi_\underlinetupel{t}(s) = x$ the barycenter of a singular $n$-simplex $\sigma$ of
  $X$, and $x'$ the barycenter of a face of $\sigma$.  Since we can always arrange for the
  $G$-supports of the $B_\underlinetupel{t}(s)$ to be small, the appearance of $g$ in the
  above formula is not the main problem.  The real problem comes from the difference
  between $\widehat f(x)$ and $\widehat f(x')$.  This difference shrinks if we use only
  small singular simplices $X$.  Therefore, in the construction of a functor
  $\contd^0_G(P) \to \contd^0_G(\bfJ_\calf \times P)$ we will need to work with chain
  complexes of simplices that get smaller as $\underlinetupel{t} \to \infty$.

  Secondly,
  $\big((\widehat f \times \id_{|P|^\wedge})_*(\bfS_n(X) \otimes
  B_\underlinetupel{t})\big)_{\underlinetupel{t} \in \IN^{\times 2}}$ typically does
  not define an object of $\contd^0_G(\bfJ_\calf \times P)$.  The singular chain complex
  is simply to big: $S_n(X)$ is infinite and so~\ref{rem:contd0-more-explicit:obj:finite}
  will fail.  In order to overcome this problem we will replace $X$ with a suitable large
  ball $\ball$ in $X$ (which is compact).  We will also replace singular simplices with
  simplices from a triangulation of $\ball$.  The set of $n$-simplices $\simp_n \ball$ is no
  longer a smooth $G$-set, as $\ball$ is not $G$-invariant in $X$, but $\simp_n \ball$ will be
  invariant for a compact open subgroup of $G$.  Moreover, the $G$-action on $X$ still
  induces a homotopy coherent action of $G$ on $\ball$.  Theorem~\ref{thm:X-to-J} provides
  maps $X \to |\bfJ_\CVCYC(G)|^\wedge$.  For the restrictions of these maps to large balls
  we control the failure of equivariance relative to this homotopy coherent action.  Again
  we will need the construction to vary in $\underlinetupel{t} \in \IN^{\times 2}$.
\end{remark}


\subsection{Restriction to open subgroups}\label{subsec:restriction}

As discussed in Remark~\ref{rem:shortcomings-singular-ch-cx}, in applications of the
tensor product later on we will not always be able to work with smooth $G$-sets, but will
also need to consider $U$-invariant sets for an open subgroup of $G$.  
To formalize this we discuss restrictions and inductions.

Let $U$ be an open subgroup of $G$.  
Write $\res_G^U \colon G\text{-}\Sets \to U\text{-}\Sets$ for the restriction functor.  
It induces a restriction functor
\begin{eqnarray*}
  \cals^G(\Omega) & \to & \cals^U(\res_G^U \Omega), \\
  (\Sigma,c) & \mapsto & (\res_G^U \Sigma, \res_G^U c),
\end{eqnarray*}
that we will also denote by $\res_G^U$.  We write
$\ind_U^G \colon \contrCatUcoef{U}{\res_G^U \Lambda}{\calb} \to
\contrCatUcoef{G}{\Lambda}{\calb}$ for the canonical inclusion\footnote{More precisely,
  writing $\res_G^U \calb$ for the subcategory of $\calb$ whose morphisms have support in
  $U$ we have
  $\ind_U^G \colon \contrCatUcoef{U}{\res_G^U \Lambda}{\res_G^U \calb} \to
  \contrCatUcoef{G}{\Lambda}{\calb}$.}.  This inclusion identifies
$\contrCatUcoef{U}{\res_G^U \Lambda}{\calb}$ with the subcategory of
$\contrCatUcoef{G}{\Lambda}{\calb}$ consisting of all objects and morphisms with
$G$-support in $U$.  Let us briefly write
\begin{equation*}
	\begin{array}{rlclcl}
		\otimes^G  \colon & \cals^G(\Omega) & \times  & \contrCatUcoef{G}{\Lambda}{\calb}
  & \to & \Idem \contrCatUcoef{G}{\Omega \times \Lambda}{\calb}; \\
  \otimes^U \colon & \cals^U(\res_G^U \Omega) & \times & \contrCatUcoef{U}{\res_G^U \Lambda}{\calb} 
  & \to & \Idem \contrCatUcoef{U}{\res_G^U (\Omega \times \Lambda)}{\calb},	
  \end{array}
\end{equation*}
for the tensor products.  Directly from the definition it follows that for
$\bfV \in \cals^G(\Omega)$ and $\bfB \in \contrCatUcoef{U}{\res_G^U \Lambda}{\calb}$ we have
\begin{equation}\label{eq:ind-res-is-ind} \ind_U^G ( \res_G^U \bfV \otimes^U \bfB) = \bfV
  \otimes^G \ind_U^G \bfB.
\end{equation}
Similarly, for morphisms $\rho \colon \bfV \to \bfV'$ in $\cals^G(\Omega)$ and
$\varphi \colon \bfB \to \bfB'$ in $\contrCatUcoef{U}{\res_G^U \Lambda}{\calb}$ we have
\begin{equation}\label{eq:ind-res-is-ind-morphisms} 
  \ind_U^G ( \res_G^U \rho \otimes^U \varphi) = \rho \otimes^G \ind_U^G \varphi.
\end{equation}
Because of these identities we will later often drop $\ind_U^G$ and $\res_G^U$ from the
notation and simply write $\otimes = \otimes^G = \otimes^U$.


\subsection{The category
  $\contrCatUcoef{G,U}{\Lambda}{\calb}$}\label{subsec:calh-G-U} 
Let $U$ be an open subgroup of $G$.  We write $\contrCatUcoef{G,U}{\Lambda}{\calb}$ for
the full subcategory of $\contrCatUcoef{G}{\Lambda}{\calb}$ on all objects with
$G$-support in $U$.  The induction
$\ind_U^G \colon \contrCatUcoef{U}{\res_{G}^U\Lambda}{\calb} \to
\contrCatUcoef{G}{\Lambda}{\calb}$ from Subsection~\ref{subsec:restriction} factors
through the inclusion
$\contrCatUcoef{G,U}{\Lambda}{\calb} \subseteq \contrCatUcoef{G}{\Lambda}{\calb}$.
Moreover, $\contrCatUcoef{G,U}{\Lambda}{\calb}$ is the full subcategory on objects in the
image of $\ind_U^G$.  Because of support cofinality~\ref{def:Hecke-category:cofinal} for
$\calb$ the inclusion
$\contrCatUcoef{G,U}{\Lambda}{\calb} \subseteq \contrCatUcoef{G}{\Lambda}{\calb}$ induces
an equivalence on idempotent completions.  Thus we can often work with
$\contrCatUcoef{G,U}{\Lambda}{\calb}$ in place of $\contrCatUcoef{G}{\Lambda}{\calb}$.


\subsection{Tensor product with subcomplexes of
  $\bfS_*(X)$}\label{subsec:tensorproduct-subcomples-of-singular}
Assume we are given~\refstepcounter{theorem}
\begin{enumerate}[
                 label=(\thetheorem\alph*),
                 align=parleft, 
                 leftmargin=*,
                 labelindent=2pt,
                 labelsep=10pt,
                 ]
               \item\label{nl:U-B} a compact open subgroup $U$ of $G$ and a compact
                 $U$-invariant subspace $\ball$ of the $U$-fixed points $X^U$ with a
                 locally ordered triangulation.
\end{enumerate}
Then we obtain the simplicial chain complex $\bfC_*(\ball) \in \ch \cals^U(\res_G^U S(X))$, see
Subsection~\ref{subsec:triangulations}.  The tensor product with $\bfC_*(\ball)$ then yields a
functor
\begin{equation*}
  \bfC_*(\ball) \, \otimes^U \, \-- \; \colon \contrCatUcoef{U}{\res_G^U\Lambda}{\calb} 
  \, \to \,  \ch  \contrCatUcoef{U}{\res_G^U(S(X) \times \Lambda)}{\calb}. 
\end{equation*}
Here we do not need the idempotent completion because of Lemma~\ref{lem:id-is-id}.
Write $i \colon \bfC_*(\ball) \to \res_G^U \bfS_*(X)$ for the inclusion in
$\ch \cals^U(\res_G^U S(X))$.  Assume we are given in addition
\begin{enumerate}[start=2,
                 label=(\thetheorem\alph*),
                 align=parleft, 
                 leftmargin=*,
                 labelindent=2pt,
                 labelsep=10pt,
                 ]
\item\label{nl:r-H} a chain map $r \colon \res_G^U \bfS_*(X) \to \bfC_*(\ball)$ in
  $\ch\cals^U(\res_G^U S(X))$ with $r \circ i = \id_{\bfC_*(\ball)}$ and a chain homotopy
  $H \colon i \circ r \simeq \id_{\res_G^U\bfS_*(X)}$.
\end{enumerate}
In $\ch \contrCatUcoef{U}{\res_G^U (S(X) \times \Lambda)}{\calb}$ we now obtain for each
$\bfB \in \contrCatUcoef{U}{\res_G^U\Lambda}{\calb}$
\begin{equation*} 
  \xymatrix{\bfC_*(\ball) \otimes^U \bfB \ar@<.5ex>[rrr]^{i\otimes^U{\id_\bfB}} & & &
    \res_G^U \bfS_*(X) \otimes^U \bfB \ar@<.5ex>[lll]^{r\otimes^U{\id_\bfB}} 
  }
\end{equation*}
where
$(r \otimes^U {\id_\bfB}) \circ (i\otimes^U {\id_\bfB}) = \id_{\bfC_*(\ball) \otimes^U \bfB}$
and
$H\otimes^U {\id_\bfB} \colon (i\otimes^U{\id_\bfB}) \circ \, (r \,\otimes^U {\id_\bfB})
\simeq \id_{\res_G^U\bfS_*(X) \otimes \bfB}$.  Applying $\ind_U^G$ and
using~\eqref{eq:ind-res-is-ind} we obtain
\begin{equation*}
  \xymatrix{\ind_U^G \big(\bfC_*(\ball) \otimes^U \bfB\big) \ar@<.5ex>[rrr]^{\ind_U^G (i\otimes^U{\id_\bfB})}
    & & &
    \bfS_*(X) \otimes^G \ind_U ^G \bfB \ar@<.5ex>[lll]^{\ind_U^G (r\otimes^U{\id_\bfB})} 
  }
\end{equation*}
in $\ch \Idem \contrCatUcoef{G}{S(X) \times \Lambda}{\calb}$.  
The main advantage of
$\ind_U^G (\bfC_*(\ball) \otimes^U \bfB)$ over $\bfS_*(X) \otimes^G \ind_U^G \bfB$ is that it
is smaller and has better chances of satisfying control conditions.  Its disadvantage is
that, as $\ball$ is only $U$-invariant and not $G$-invariant, morphisms
$\varphi \colon \ind_U^G \bfB \to \ind_U^G \bfB'$ in $\contrCatUcoef{G,U}{\Lambda}{\calb}$ do not induce
maps $\ind_U^G (\bfC_*(\ball) \otimes^U \bfB) \to \ind_U^G (\bfC_*(\ball) \otimes^U \bfB')$.  But
$\id_{\bfS_*(X)} \otimes \varphi$ is defined and we can use the composition
\begin{equation}\label{eq:with-lots-of-ind-and-res} \xymatrix{\ind_U^G \big(\bfC_*(\ball)
    \otimes^U \bfB\big) \ar[rrr]^{\ind_U^G (i\otimes^U{\id_\bfB})} & & &
    \bfS_*(X) \otimes^G \ind_U ^G \bfB \ar[d]^{\id_{\bfS_*(X)} \otimes^G \varphi} \\
    \ind_U^G \big(\bfC_*(\ball) \otimes^U \bfB'\big) & & & \bfS_*(X) \otimes^G \ind_U ^G \bfB'
    \ar[lll]_{\ind_U^G (r\otimes^U{\id_\bfB})} }
\end{equation}
instead.  While this is not strictly compatible with composition, the homotopies
$\ind_U^G (H\otimes^U {\id_\bfB})$ guarantee that it is compatible with composition up to
coherent homotopy.  From now on we will simply the notation as alluded to in
Subsection~\ref{subsec:restriction} and drop $\ind_U^G$ and $\res_G^U$ from the notation
and simply write $\otimes = \otimes^G = \otimes^U$.
Thus~\eqref{eq:with-lots-of-ind-and-res} abbreviates to
\begin{equation*}
  \xymatrix{\bfC_*(\ball) \otimes \bfB \ar[rr]^{i \otimes {\id_\bfB}}
    & & \bfS_*(X) \otimes \bfB \phantom{'.} \ar[d]^{\id_{\bfS_*(X)} \otimes \varphi} \\ 
    \bfC_*(\ball) \otimes \bfB' & & \bfS_*(X) \otimes \bfB'.
    \ar[ll]_{r\otimes{\id_\bfB'}} 
  }
\end{equation*}

In summary, we can use the data chosen in~\ref{nl:U-B} and~\ref{nl:r-H} to define a
homotopy coherent functor, see Definition~\ref{def:homotopy-coherent-functor},
\begin{equation}\label{eq:homotopy-coherent-F} F = (F^0,F^1,\dots) \;
  \colon \contrCatUcoef{G,U}{\Lambda}{\calb} \to \ch  \contrCatUcoef{G}{S(X) \times \Lambda}{\calb}
\end{equation}
as follows.  For $\bfB \in \contrCatUcoef{G,U}{\Lambda}{\calb}$ we set
\begin{equation*}
  F^0(\bfB) := \bfC_*(\ball) \otimes \bfB.
\end{equation*}
For a chain
\begin{equation*}
  \bfB_0 \xleftarrow{\varphi_1} \bfB_1 \xleftarrow{\varphi_2} \dots \xleftarrow{\varphi_n} \bfB_n
\end{equation*}
of composable morphisms in $\contrCatUcoef{G,U}{\Lambda}{\calb}$ we define
\begin{align*}
  F^n(\varphi_1,\dots,\varphi_n) := (r \otimes
  & \id_{\bfB_0}) \circ (\id_{\bfS_*(X)} \otimes \varphi_1) \circ (H \otimes \id_{\bfB_1}) \circ \dots
  \\ & \dots \circ (H \otimes \id_{\bfB_{n-1}}) \circ (\id_{\bfS_*(X)} \otimes \varphi_n) \circ  (i \otimes \id_{\bfB_n}).
\end{align*} 
It is not difficult to check that this defines a homotopy coherent functor, compare
Example~\ref{ex:def-functor-to-homotopy-coherent}.

In Section~\ref{sec:support-estimates} we will discuss the effect
of~\eqref{eq:homotopy-coherent-F} on $\suppX$.  For now we record the following easy
facts.

\begin{lemma}\label{lem:suppG_and_suppobj-for-diagonal-tensor} \
   \begin{enumerate}[
                 label=(\thetheorem\alph*),
                 align=parleft, 
                 leftmargin=*,
                 labelindent=2pt,
                 labelsep=10pt,
                 itemindent=-3pt
                 ]
               \item\label{lem:suppG_and_suppobj-for-diagonal-tensor:finite}
                 If $\bfB$ is finite, then all chain modules $(F^0(\bfB))_n$ of $F^0(\bfB)$ are finite as well;
               \item\label{lem:suppG_and_suppobj-for-diagonal-tensor:obj}
                 We have $\suppobj (F^0(\bfB))_n \subseteq \simp_n(\ball) \times \suppobj \bfB$
                 and $\suppG (F^0(\bfB))_n \subseteq \suppG \bfB$;
               \item\label{lem:suppG_and_suppobj-for-diagonal-tensor:boundary}
                 The $n$-th boundary map $\partial^{F^{0}(\bfB)}_n$ of $F^0(\bfB)$
                 satisfies $\suppG \partial^{F^{0}(\bfB)}_n \subseteq \suppG \bfB$.
   \end{enumerate}	
\end{lemma}

\begin{proof}
  This follows from~\ref{lem:finiteness_and_suppobj-diag-tensor-obj:finite},~\ref{lem:finiteness_and_suppobj-diag-tensor-obj:suppobj}
  and~\ref{lem:supp-rho-otimes-varphi:suppG}.
\end{proof}


\subsection{Projection back to $\Lambda$}\label{subsec:back-to-Z} The
projection $\pr \colon S(X) \times \Lambda \to \Lambda$ induces a functor
\begin{equation*}
  P \colon  \ch  \contrCatUcoef{G}{X \times \Lambda}{\calb} \to  \ch  \contrCatUcoef{G}{\Lambda}{\calb}.
\end{equation*}
We are interested in the composition of $P$ with $F$ from~\eqref{eq:homotopy-coherent-F}
\begin{equation*}
  P \circ F  \colon  \contrCatUcoef{G,U}{\Lambda}{\calb} \to \ch \contrCatUcoef{G}{\Lambda}{\calb}.
\end{equation*} 
Let
\begin{equation*}
  I \colon \contrCatUcoef{G,U}{\Lambda}{\calb} \to  \ch  \contrCatUcoef{G}{\Lambda}{\calb}
\end{equation*}
be the inclusion\footnote{Recall that $\contrCatUcoef{G,U}{\Lambda}{\calb}$ is a subcategory of
$\contrCatUcoef{G}{\Lambda}{\calb}$.}.  We construct a natural transformation
$\tau \colon P \circ F \to I$.  Write $\star$ for the one-point space. Let $p(\bfC_*(\ball))$
be the image of $\bfC_*(\ball)$ under the functor $\cals^U(X) \to \cals^U(\star)$ induced by
the projection $X \to \star$.  We can identify
$(P \circ F)(\bfA) \cong p(\bfC_*(\ball)) \otimes \bfB$, for $\bfB \in \contrCatUcoef{G,U}{\Lambda}{\calb}$.
Let $\bbone{G} := (\star,\id_{\star}) \in \cals^G(\star)$.  We have
$\bbone{U} := \res_G^U \bbone{G} = (\star,\id_{\star}) \in \cals^U(\star)$.  We can
identify $I(\bfB) \cong \bbone{U} \otimes \bfB$.  The projection $\ball \to \star$ induces an
augmentation $\epsilon \colon p(\bfC_*(\ball)) \to \bbone{U}$.  Now for
$\bfB \in \contrCatUcoef{G,U}{\Lambda}{\calb}$ we define
\begin{equation}\label{eq:def-tau} \tau_\bfB := \epsilon \otimes
  \id_\bfB \colon p(C_*(\ball)) \otimes \bfB \to \bbone{U} \otimes \bfB.
\end{equation}
We will need the notion of a strict natural transformation between homotopy coherent
functors, see Definition~\ref{def:strict-nat-transf}.

\begin{lemma}\label{lem:tau-is-transformation} 
  Under the canonical
  identifications $(P \circ F)(\bfB) \cong p(\bfC_*(\ball)) \otimes \bfB$ and
  $I(\bfB) \cong \bbone{U} \otimes \bfB$ the maps~\eqref{eq:def-tau} define a strict
  natural transformation $\tau \colon P \circ F \to I$, see
  Example~\ref{ex:strict-nat-transf}.
\end{lemma}

\begin{proof}
  This is a straight forward exercise in the definitions.
\end{proof}

\begin{lemma}\label{lem:tau-is-weak-equiv} Suppose that $\ball$ is
  contractible.  Let $\bfB \in \contrCatUcoef{G,U}{\Lambda}{\calb}$.  Then in
  $\ch \Idem \contrCatUcoef{G,U}{\Lambda}{\calb}$ there are a chain map
  $f \colon \bbone{U} \otimes \bfB \to p(\bfC_*(\ball)) \otimes \bfB$, and chain homotopies
  $h \colon f \circ \tau_\bfB \simeq \id_{p(\bfC_*(\ball)) \otimes \bfB}$,
  $k \colon \tau_\bfB \circ f \simeq \id_\bbone{U} \otimes \bfB$ such that
  \begin{equation*} \begin{aligned} \suppX f, \suppX h, \suppX h' & \subseteq \suppX \bfB,
      \\ \suppG f,\suppG h, \suppG h' & \subseteq \suppG \bfB \subseteq U.
    \end{aligned}\end{equation*}
  In particular, $\tau_\bfB$ is a homotopy equivalence.
\end{lemma}

\begin{proof}
  Under the assumptions on $\ball$, $p(\bfC_*(\ball))$ and $\bbone{U}$ are homotopy equivalent in
  $\ch\cals^U(\star)$ (because the simplicial chain complex of $\ball$ is homotopy equivalent
  to the simplicial chain complex of a point).  Thus in $\ch\cals^U(\star)$ there are a
  chain map $f_0 \colon \bbone{U} \to p(\bfC_*(\ball))$, and chain homotopies
  $h_0 \colon f_0 \circ \epsilon \simeq \id_{p(\bfC_*(\ball))}$,
  $k_0 \colon \epsilon \circ f_0 \simeq \id_{\bbone{U}}$.  Now set
  $f := f_0 \otimes \id_\bfB$, $h := h_0 \otimes \id_\bfB$, and
  $k := k_0 \otimes \id_\bfB$.  The claims about $\suppX$ and $\suppG$ follow
  from~\ref{lem:supp-rho-otimes-varphi:suppX} and~\ref{lem:supp-rho-otimes-varphi:suppG}.
\end{proof}


\section{Support estimates for homotopy coherent functors}%
\label{sec:support-estimates}

Let $X$ be a $G$-space equipped with a $G$-invariant metric $d_X$.
We assume that for $K \subseteq X$ compact the pointwise isotropy group $G_K$ is open in $G$.
Let $J$ and $\Lambda$ be further $G$-spaces.  
Let $\calb$ be a Hecke category with $G$-support.

We will refine the construction of the homotopy coherent functor from
Subsection~\ref{subsec:tensorproduct-subcomples-of-singular} and will be interested in its
effect on supports of objects and morphisms.  Its construction and analysis will depend on
a list of data.  This section is very formal and will be used later to check that a
sequence of homotopy coherent functors does descend to $\contd^0_G(\--)$ as discussed in
Subsection~\ref{subsec:sequences}.


\subsection{Set-up}
Throughout this section we fix
\begin{itemize}[label=-,                 
                 align=parleft, 
                 leftmargin=*,
                 labelindent=5pt,
                 ] 
\item numbers $L \in \IN$, $\rho > 0$;
\item a compact open subgroup $U \subseteq G$ and a compact subset $M \subseteq G$ with
  $U \subseteq M$;
\item a subspace $\ball \subseteq X$ with a locally ordered triangulation all whose simplices
  are of diameter $< \rho$;
\item a sequence of further subspaces
  $\ball=\ball^{(0)} \subseteq \ball^{(1)} \subseteq \dots \subseteq \ball^{(L)}$ with
  $M \cdot \ball^{(l-1)} \subseteq \ball^{(l)}$ for $l=1,\dots,L$; moreover we require that
  $\ball^{(L)}$ is pointwise fixed by $U$;
\item a retraction $r^0 \colon X \to \ball$ for the inclusion $i^0 \colon \ball \to X$, i.e.,
  $r^0 \circ i^0 = \id_\ball$;
\item a homotopy $H^0 \colon i^0 \circ r^0 \simeq \id_X$; we will assume that $H^0$ is
  non-expanding, i.e., for all $\tau \in [0,1]$, $x,x' \in X$ we require
  $d_X(H^0(x,\tau),H^0(x',\tau)) \leq d_X(x,x')$; we also assume that $H^0$ preserves the
  $\ball^{(l)}$, i.e., we require $H^0(\ball^{(l)} \times [0,1]) \subseteq \ball^{(l)}$;
\item a map $f \colon X \to J$;
\item a subset $E \subseteq J \times J$.
\end{itemize}
We do \emph{not} assume that $f$ is $G$-equivariant.

\begin{remark}[Some explanations for the list of data]\label{rem:some-explanations}
  Later
  \begin{itemize}[label=-,                 
                 align=parleft, 
                 leftmargin=*,
                 labelindent=5pt,
                 ] 
  \item $X$ will be the extended Bruhat-Tits building for the reductive $p$-adic group $G$;
  \item the $\ball^{(i)}$ will be an increasing sequence of balls around a common center;
  \item $r^0$ will be the radial projection and $H^0$ will be the associated radial
    homotopy;
  \item $U$ will be the pointwise isotropy group of $\ball^{(L)}$;
  \item $\Lambda$ will be $|\ZerodimR|^\wedge$ for some
    $\ZerodimR \in \regularPOrGSCzero$;
  \item $J$ will be $|\bfJ_{\CVCYC}(G)|^\wedge$;
  \item $f \colon X \to J$ will come from Theorem~\ref{thm:X-to-J}.
  \end{itemize}
  Below $U$, $B$, $r^0$, and $H^0$ will be used to construct a homotopy coherent functor
  \begin{equation*}
    \contrCatUcoef{G,U}{\Lambda}{\calb} \to \ch  \contrCatUcoef{G}{S(X) \times \Lambda}{\calb}
  \end{equation*}
  as in Subsection~\ref{subsec:tensorproduct-subcomples-of-singular}.  
  Let $\bary \colon S(X) \to X$ be the map that sends a singular simplex to its barycenter.  
  We can then compose with $( f \circ \bary \times \id_\Lambda)_*$ and obtain a homotopy coherent functor
  \begin{equation*}
    \contrCatUcoef{G,U}{\Lambda}{\calb} \to \ch  \contrCatUcoef{G}{J \times \Lambda}{\calb}
  \end{equation*}
  and will be interested in estimates for the support of objects and morphisms under this
  latter functor, see Proposition~\ref{prop:supp-tildeF_l} below.  For these estimates we will
  bound the $G$-support of morphisms and objects by $M$ and we will treat chains of at
  most $L$-composable morphisms.  The upper bound for $\suppX$ will be in terms of $E$.
\end{remark}

We will work under the following assumptions throughout this section.

\begin{assumption}\label{assm:E_0}\
  \begin{enumerate}[
                 label=(\thetheorem\alph*),
                 align=parleft, 
                 leftmargin=*,
                 labelindent=2pt,
                 labelsep=10pt,
                 ]
  \item\label{assm:E_0:G} $\Big\{ \twovec{f(gx)}{gf(x)} \; \Big| \; x \in \ball^{(L)}, g \in M \Big\} \subseteq E$;\\
  \item\label{assm:E_0:H} $\Big\{ \twovec{f(H^0(x,\tau))}{f(x)} \; \Big| \; x \in \ball^{(L)}, \tau \in [0,1] \Big\} \subseteq E$;\\
  \item\label{assm:E_0:eps}
    $\Big\{ \twovec{f(x')}{f(x)} \; \Big| \; x,x' \in \ball^{(L)}, d_X(x,x') \leq (L+1) \rho
    \Big\} \subseteq E$. 
  \end{enumerate}
\end{assumption}


\subsection{From $S(X)$ to $J$}
Set 
\begin{eqnarray*}
	E_X  & :=  &  \big\{ \twovec{x'}{x} \; \big| \; x,x' \in \ball^{(L)}, d_X(x,x') \leq (L+1)\rho \big\} \\
             & & \qquad \cup \; \big\{ \twovec{H^0(x,\tau)}{x} \; \big| \; x \in \ball^{(L)}, \tau \in [0,1] \big\}
                 \quad \subseteq X \times X; \\
	E_S & := & (\bary \times \bary)^{-1}(E_X) \qquad \subseteq S(X) \times S(X).
\end{eqnarray*}
We note that $E_X \subseteq E_X^{\circ 2}$ and $E_S \subseteq E_S^{\circ 2}$.
We use again the convention for products from~\eqref{eq:cheated-times}.

\begin{lemma}\label{lem:extra-E}
  Let $E' \subseteq \Lambda \times \Lambda$,
  $\Phi \colon \bfB \to \bfB' \in \contrCatUcoef{G}{S(X) \times \Lambda}{\calb}$, and
  $\nu \in \IN$.  Suppose $\suppobj \bfB \subseteq S(\ball^{(L)}) \times \Lambda$,
  $\suppX \Phi \subseteq E_S^{\circ \nu} \times E'$, and $\suppG \Phi \subseteq M$.  Then
  $\suppX (f \circ \bary \times \id_\Lambda)_*\Phi \subseteq E^{\circ (\nu+1)} \times E'$.
\end{lemma}

\begin{proof}
  By~\eqref{eq:support-after-non-equiv-f}, and since $\bary$ and $\id_\Lambda$ are
  $G$-equivariant, $\suppX (f \circ \bary \times \id_\Lambda)_*\Phi$ is contained in
  \begin{equation*}
    (f \circ \bary \id_\Lambda)^{\times 2}(E_S^{\circ \nu} \times E') \circ 
    \Big\{ \twovec{f(g\bary(\sigma)),g\lambda}{g f(\bary(\sigma)),g\lambda}
    \, \Big| \, (\sigma,\lambda) \in S(\ball^{(L)}) \times \Lambda, g \in M \Big\}.
  \end{equation*}
  We have
  \begin{multline*}
    (f \circ \bary \id_\Lambda)^{\times 2}(E_S^{\circ \nu} \times E') \subseteq ((f \circ
    \bary)^{\times 2}(E_S^{\circ \nu})) \times E' \\ \subseteq ((f \circ \bary)^{\times
      2}(E_S))^{\circ \nu} \times E' \subseteq E^{\circ \nu} \times E'
  \end{multline*}
  and, by~\ref{assm:E_0:G}
  \begin{equation*}
    \Big\{ \twovec{f(g\bary(\sigma)),g\lambda}{g f(\bary(\sigma)),g\lambda}
    \, \Big| \, (\sigma,\lambda) \in S(\ball^{(L)}) \times \Lambda, g \in M \Big\}
    \subseteq E \times \Big\{ \twovec{\lambda}{\lambda} \;\big| \; \lambda \in \Lambda \Big\}.
  \end{equation*}
  Thus
  \begin{equation*}
    \suppX (f \circ \bary \times \id_\Lambda)_*\Phi
    \subseteq \Big(E^{\circ \nu} \times E'\Big)
    \circ \Big(E \times \Big\{ \twovec{\lambda}{\lambda} \;\big| \; \lambda
    \in \Lambda \Big\} \Big) \subseteq E^{\circ (\nu+1)} \times E'.
  \end{equation*}
\end{proof}
 

\subsection{Construction of $r$ and $H$} 
We write $\bfC_*(\ball) \in \ch \cals^U(S(X))$ for the simplicial chain complex of $\ball$
and $\bfS_*(X) \in \ch \cals^G(S(X))$ for the singular chain complex of $X$.  Let
$i \colon \bfC_*(\ball) \to \bfS_*(X)$ be the inclusion.  For $k,l \leq L$ we define
$S^{k,l} \subseteq S(X)$ as the collection of all singular simplices in $X$ that are
contained in $\ball^{(l)}$ and are of diameter $< k\rho$.

\begin{lemma}\label{lem:r-and-H} 
  There is a chain map $r \colon \bfS_*(X) \to \bfC_*(\ball)$ and a chain homotopy
  $H \colon \bfS_*(X) \to \bfS_*(X)$ in $\ch \cals^U(S(X))$ with
  $r \circ i = \id_{\bfC_*(B)}$, $H \colon i \circ r \simeq \id_{\bfS_*(X)}$ satisfying
  the following: if $\twovec{\sigma'}{\sigma} \in \suppX r \cup \suppX H$ with
  $\sigma \in S^{k,l}$ and $k \leq L-1, l \leq L$, then
  \begin{enumerate}[
                 label=(\thetheorem\alph*),
                 align=parleft, 
                 leftmargin=*,
                 labelindent=1pt,
                 labelsep=10pt,
                 ]
               \item\label{lem:r-and-H:in-S-delta-l} $\sigma' \in S^{k+1,l}$;
               \item\label{lem:r-and-H:in-E-S}
                 $\twovec{\sigma'}{\sigma} \in E_{S}^{\circ 2}$.
               \end{enumerate}
             \end{lemma}

\begin{proof}
  We will use $\Image \sigma$ and $(\Image \sigma)^\rho$ as introduced in
  Subsection~\ref{subsec:triangulations}.  We write
  $i^1 \colon \bfC_*(\ball) \to \bfS_*(\ball)$ for the inclusion.  By
  Lemma~\ref{lem:singular-vs-simplicial} there is
  $r^1 \colon \bfS_*(\ball) \to \bfC_*(\ball)$ with $r^1 \circ i^1 = \id_{\bfC_*(\ball)}$
  and a chain homotopy $H^1 \colon \bfS_*(\ball) \to \bfS_{*+1}(\ball)$ for
  $i^1 \circ r^1 \simeq \id_{\bfS_*(\ball)}$.  Moreover,
  Lemma~\ref{lem:singular-vs-simplicial} also yields
  \begin{enumerate}[
                 label=(\thetheorem\alph*),
                 align=parleft, 
                 leftmargin=*,
                 labelindent=1pt,
                 labelsep=8pt,
                 resume ]
               \item\label{enum:diameters-r-H} if
                 $\twovec{\sigma'}{\sigma} \in \suppX (r^1) \cup \in \suppX(H^1)$, then
                 $\Image \sigma' \subseteq (\Image \sigma)^{\rho}$.
               \end{enumerate}
               Next we use that $H^0$ is not expanding and apply
               Lemma~\ref{lem:chain-homotopy-subdivided} to the homotopy $H^0$ and obtain
               a chain homotopy
               $\widetilde{H}^0 \colon (i^0)_* \circ (r^0)_* \simeq \id_{\bfS_*(X)}$ such
               that
  \begin{enumerate}[
                 label=(\thetheorem\alph*),
                 align=parleft, 
                 leftmargin=*,
                 labelindent=1pt,
                 labelsep=8pt,
                 resume ]
               \item\label{enum:tilde-H} if
                 $\twovec{\sigma'}{\sigma} \in \suppX (\widetilde{H}^0)$, then
                 $\diam \Image (\sigma') < \diam \Image (\sigma) + \rho$ and
                 $\Image \sigma' \subseteq \Image(H^0 \circ (\sigma \times \id_{[0,1]}))$.
               \end{enumerate}
               We have $i = (i^0)_* \circ i^1 \colon \bfC_*(\ball) \to \bfS_*(X)$ and set
               \begin{eqnarray*}
                 r & := & r^1 \circ (r^0)_* \colon \bfS_*(X) \to \bfC_*(\ball);\\
                 H & := & (i^0)_* \circ H^1 \circ (r^0)_*  + (\widetilde{H}^0)_* \colon i \circ r \simeq \id_{\bfS_*(X)}.
               \end{eqnarray*}
               Suppose $\twovec{\sigma'}{\sigma} \in \suppX r$ with $\sigma \in S^{k,l}$.
               As $r$ is a map to $\bfC_*(\ball)$, $\sigma'$ is simplex in the
               triangulation of $\ball$, in particular $\sigma' \in S^{k+1,l}$.  Now
               $r_\sigma^{\sigma'} \neq 0$ implies that there is $\tau$ with
               $(r^1)_\tau^{\sigma'} \neq 0$ and $((r^0)_*)_\sigma^\tau \neq 0$.  The
               latter means $\tau = r^0 \circ \sigma = H^0(\sigma(\--),1)$.  In
               particular, $\twovec{\tau}{\sigma} \in E_S$.  As $H^0$ is non-expanding,
               $\diam \Image \tau \leq \diam \Image \sigma < k\rho$.
               By~\ref{enum:diameters-r-H}, $\Image \sigma' \subseteq (\Image \tau)^\rho$.
               Thus
               $d_X(\bary(\sigma'),\bary(\tau)) \leq \rho + \diam \Image \tau < (k+1)\rho
               \leq L \rho$ and $\twovec{\sigma'}{\tau} \in E_S$.  Therefore
               $\twovec{\sigma'}{\sigma} \in E_S^{\circ 2}$.

               Suppose $\twovec{\sigma'}{\sigma} \in \suppX H$ with $\sigma \in S^{k,l}$.
               Then $\twovec{\sigma'}{\sigma} \in \suppX \widetilde{H}$ or
               $\twovec{\sigma'}{\sigma} \in\suppX ((i^0)_* \circ H^1 \circ (r^0))$.  In
               the first case,~\ref{enum:tilde-H} implies
               $\diam \sigma' < \diam \sigma + \rho < (k+1)\rho$.  Moreover,
               $\Image \sigma' \subseteq H_0(\Image \sigma \times [0,1])$.  As $H^0$
               preserves $\ball^{(l)}$, this implies $\sigma' \in S^{k+1,l}$.  Also
               $\bary(\sigma') = H_0(x,t)$ for some $t \in [0,1]$ and
               $x \in \Image \sigma$.  Let $\tau := H_0(\sigma(\--),t)$.  Then
               $\twovec{\sigma}{\tau} \in E_S$.  Also $\bary(\sigma') \in \Image \tau$.
               As $H_0(\--,t)$ is non-expanding,
               $d_X(\bary(\sigma'),\bary(\tau)) < \diam \Image \tau < \diam \Image \sigma
               < k \rho$.  Thus $\twovec{\tau}{\sigma} \in E_S$.  Altogether,
               $\twovec{\sigma'}{\sigma} \in E_S^{\circ 2}$.
    
               In the second case there are $\tau$, $\tau'$ with
               $\twovec{\sigma'}{\tau'} \in \suppX (i^0)_*$,
               $\twovec{\tau'}{\tau} \in \suppX H^1$, and
               $\twovec{\tau}{\sigma} \in \suppX (r^0)_*$.  Now
               $\twovec{\sigma'}{\tau'} \in \suppX (i^0)_*$ means $\tau' = \sigma'$, while
               $\twovec{\tau}{\sigma} \in \suppX (r^0)_*$ means
               $\tau = r^0 \circ \sigma = H^0(\sigma(\--),1)$.  In particular,
               $\twovec{\tau}{\sigma} \in E_S$.  Also, as $H^0$ is non-expanding,
               $\diam \Image \tau \leq \diam \Image \sigma < k\sigma$.
               Using~\ref{enum:diameters-r-H}, $\twovec{\tau'}{\tau} \in \suppX H^1$
               implies $\Image \tau' \subseteq (\Image \tau)^\rho$.  In particular,
               $d_X(\bary(\tau'),x) < \rho$ for some $x \in \Image \tau$.  Moreover,
               $\diam \Image \sigma' = \diam \Image \tau' < \diam \Image \tau + \rho <
               (k+1)\rho$.  As $H^1$ is a map to $\bfS_*(\ball)$, we have
               $\Image \sigma' = \Image \tau' \subseteq \ball$.  In particular,
               $\sigma' \in S^{k+1,l}$.  Now
               \begin{multline*}
                 d_X(\bary(\sigma'),\bary(\tau)) = d_X(\bary(\tau'),\bary(\tau)) \\ \leq
                 d_X(\bary(\tau'),x) + d_X(x,\bary(\tau)) < \rho + \diam \Image \tau < \rho
                 + k\rho = (k+1)\rho
               \end{multline*}
               implies $\twovec{\sigma'}{\tau} \in E_S$.  Thus
               $\twovec{\sigma'}{\tau} \in E_S^{\circ 2}$.
             \end{proof}


\subsection{Estimates over $S(X) \times \Lambda$}\label{subsec:estimates-over-S(X)}
We now fix $r$ and $H$ as in Lemma~\ref{lem:r-and-H}.  
We obtain a homotopy coherent functor as in~\eqref{eq:homotopy-coherent-F}
\begin{equation*}
  F = (F^0,F^1, \dots, ) \colon \contrCatUcoef{G,U}{\Lambda}{\calb} \to \ch   \contrCatUcoef{G}{S(X) \times \Lambda}{\calb}.
\end{equation*}
Recall that $F_0(\bfB) = \bfC(\ball)_* \otimes \bfB$.
Its $n$-th chain module is $F_0(\bfB)_n = \bfC(\ball)_n \otimes \bfB$.
Its $n$-th boundary map is $\dd_n^{F_0(\bfB)} = \dd^{\bfC_*(\ball)}_n \otimes \id_\bfB$. 

\begin{lemma}\label{lem:supp-F_0}
    Let $\bfB \in \contrCatUcoef{G,U}{\Lambda}{\calb}$.
    Then
	\begin{enumerate}[
                 label=(\thetheorem\alph*),
                 align=parleft, 
                 leftmargin=*,
                 labelindent=1pt,
                 labelsep=10pt,
                 ]
		\item\label{lem:supp-F_0:chain-modules}
		  $\suppG F_0(\bfB)_n \subseteq U$, $\suppX F_0(\bfB)_n \subseteq E_S \times \suppX \bfB$; 
		\item\label{lem:supp-F_0:partial}
		  $\suppG \dd_n^{F_0(\bfB)} \subseteq U$, $\suppX \dd_n^{F_0(\bfB)} \subseteq E_S \times \suppX \bfB$. 	
    \end{enumerate}
\end{lemma}

\begin{proof}
  We have
  $\suppX \id_{\bfC_n(\ball)} \subseteq \{ \twovec{\sigma}{\sigma} \mid \sigma \in
  S(\ball) \} \subset E_S$.  Diameters of simplices in the triangulation of $\ball$ are
  $< \rho$.  If $(\partial_n^{\bfC_*(\ball)})_{\sigma}^{\sigma'} \neq 0$, then
  $\Image \sigma' \subseteq \Image \sigma$ and so
  $d_X(\bary{\sigma'},\bary{\sigma}) < \rho$.  Therefore
  $\suppX \partial_n^{\bfC_*(\ball)} \subseteq E_S$.  We have
  $\suppG \bfB = \suppG \id_\bfB \subseteq U$, as
  $\bfB \in \contrCatUcoef{G,U}{\Lambda}{\calb}$.  Both~\ref{lem:supp-F_0:chain-modules},
  and~\ref{lem:supp-F_0:partial} follow now from Lemma~\ref{lem:properties-diag-tensor}.
\end{proof}

We will now use summands and corners as in Remarks~\ref{rem:summands}
and~\ref{rem:corners} relative to $S^{k,l} \subseteq S(X)$.  For
$\bfB \in \contrCatUcoef{G,U}{\Lambda}{\calb}$ we abbreviate
$(\bfS_n(X) \otimes \bfB)_{k,l} := (\bfS_n(X) \otimes \bfB)|_{S^{s,l} \times \Lambda}$ and
write
\begin{equation*}
  (\bfS_n(X) \otimes \bfB)_{k,l} \xrightarrow{i_{k,l}} \bfS_n(X) \otimes \bfB
  \xrightarrow{r^{k,l}} (\bfS_n(X) \otimes \bfB)_{k,l}  
\end{equation*}
for the corresponding inclusion and retraction.  For
$\Phi \colon \bfS_n(X) \otimes \bfB \to \bfS_{n'}(X) \otimes \bfB'$ we set
$\Phi_{k,l}^{k',l'} := r^{k',l'} \circ \Phi \circ i_{k,l}$.  We also set
$(i \otimes \id_\bfB)^{k',l'} := r^{k',l'} \circ (i \otimes \id_\bfB)$, and
$(r\otimes \id_\bfB)_{k,l} := (r\otimes \id_\bfB) \circ i^{k,l}$.
This means 
\begin{eqnarray*}
	(\Phi_{k,l}^{k',l'})_{\sigma,s}^{\sigma',s'} & = & 
                                                           \begin{cases} \Phi_{\sigma,s}^{\sigma',s'} \hspace{7.1ex}
                                                             & \sigma \in S^{k,l}, \sigma' \in S^{k',l'}; \\ 0 & \text{else}; 	
	     \end{cases} 
	     \\
    ((i \otimes \id_\bfB)^{k',l'})_{\sigma,s}^{\sigma',s'} & = &  
       \begin{cases} (i \otimes \id_\bfB)_{\sigma,s}^{\sigma',s'} &  \sigma' \in S^{k',l'}; \\ 0 & \text{else}; 	
	     \end{cases} 
	     \\
	  ((r\otimes \id_\bfB)_{k,l})_{\sigma,s}^{\sigma',s'} & = &   
	  \begin{cases} (r \otimes \id_\bfB)_{\sigma,s}^{\sigma',s'} &  \sigma \in S^{k,l}; \\ 0 & \text{else}.	
	     \end{cases} 
\end{eqnarray*}

\begin{lemma}\label{lem:restrictions} 
  Let $k,l \leq L$, $\bfB, \bfB' \in \contrCatUcoef{G,U}{\Lambda}{\calb}$,
  $\varphi \colon \bfB \to \bfB' \in \contrCatUcoef{G,U}{\Lambda}{\calb}$ with
  $\suppG \varphi \subseteq M$.  Set $E' := \suppX \varphi \cup \suppX \bfB$.  Then
   	\begin{enumerate}[
                 label=(\thetheorem\alph*),
                 align=parleft, 
                 leftmargin=*,
                 labelindent=1pt,
                 labelsep=10pt,
                 itemsep=3pt
                 ]
  \item\label{lem:restrictions:varphi-i}
    $i_{k,l+1} \circ (\id_{\bfS_*(X)} \otimes \varphi)_{k,l}^{k,l+1}
    = (\id_{\bfS_*(X)} \otimes \varphi) \circ i_{k,l}$ (provided $l+1 \leq L)$;
  \item\label{lem:restrictions:varphi-supp}
    $\suppX \big( (\id_{\bfS_*(X)} \otimes \varphi)_{k,l}^{k,l+1} \big) \subseteq E_S \times E'$ (provided $l+1 \leq L)$;   
  \item\label{lem:restrictions:H-i}
   $i_{k+1,l} \circ (H \otimes \id_{\bfB})_{k,l}^{k+1,l} = (H \otimes \id_{\bfB}) \circ i_{k,l}$ (provided $k+1 \leq L)$; 
   \item\label{lem:restrictions:H-supp}
  	$\suppX  \big( (H \otimes \id_{\bfB})_{k,l}^{k+1,l} \big) \subseteq E_S^{\circ 2} \times E'$ (provided $k+1 \leq L)$; 
   \item\label{lem:restrictions:i-i}
   $i_{k,l} \circ (i \otimes \id_{\bfB})^{k,l} = i \otimes \id_{\bfB}$;
  \item\label{lem:restrictions:i-supp}
  	$\suppX  \big( (i \otimes \id_{\bfB})^{k,l} \big) \subseteq  E_S^{\circ 2} \times E'$;   
  \item\label{lem:restrictions:r-i}
   $(r \otimes \id_{\bfB})_{k,l} = (r \otimes \id_{\bfB}) \circ i_{k,l}$;
  \item\label{lem:restrictions:r-supp}    
  	$\suppX \big(  (r \otimes \id_{\bfB})_{k,l} \big) \subseteq E_S^{\circ 2} \times E'$. 
  \end{enumerate}
\end{lemma}

\begin{proof}
  From~\ref{lem:supp-rho-otimes-varphi:suppG} we obtain
  $\suppG \id_{\bfS_*(X)} \otimes \varphi \subseteq \suppG \varphi \subseteq M$.  Using
  $M \cdot \ball^{(l)} \subseteq \ball^{(l+1)}$ we obtain
  $(\suppG \id_{\bfS_*(X)} \otimes \varphi) \cdot S^{k,l} \subseteq M \cdot S^{k,l}
  \subseteq S^{k,l+1}$.  If
  $\twovec{\sigma',\lambda'}{\sigma,\lambda} \in \suppX \id_{\bfS_*(X)} \otimes \varphi$,
  then, by~\ref{lem:supp-rho-otimes-varphi:suppX},
  $\twovec{\sigma'}{\sigma} \in \suppX \id_{\bfS_*(X)}$ and so $\sigma = \sigma'$.  Now
  Lemma~\ref{lem:corner}
  gives~\ref{lem:restrictions:varphi-i}. Using~\ref{lem:supp-rho-otimes-varphi:suppX}
  and~\eqref{eq:support-of-corner} we obtain
  \begin{equation*}
    \suppX (\id_{\bfS_*(X)} \otimes \varphi)_{k,l}^{k,l+1} 
    \subseteq 
    \Big\{ \twovec{\sigma,\lambda'}{\sigma,\lambda} \Big| \sigma \in S^{k,l+1},
    \twovec{\lambda'}{\lambda} \in \suppX \varphi  \Big\} 
  \end{equation*}
  and this implies~\ref{lem:restrictions:varphi-supp}.
	
  By~\ref{lem:supp-rho-otimes-varphi:suppG} $\suppG H \otimes \id_\bfB$,
  $\suppG r \otimes \id_\bfB$, and $\suppG i \otimes \id_\bfB$ are all contained in
  $\suppG \bfB \subseteq U$.  As $U$ fixes $\ball^{(l)}$, we have
  $U \cdot S^{k,l} = S^{k,l}$.  If
  $\twovec{\sigma',\lambda'}{\sigma,\lambda} \in \suppX H \otimes \id_\bfB$, then,
  by~\ref{lem:supp-rho-otimes-varphi:suppX}, $\twovec{\sigma'}{\sigma} \in \suppX H$.  If,
  in addition, $\sigma \in S^{k,l}$, then, by~\ref{lem:r-and-H:in-S-delta-l},
  $\sigma' \in S^{k+1,l}$ and by~\ref{lem:r-and-H:in-E-S}
  $\twovec{\sigma'}{\sigma} \in E_S^{\circ 2}$.  Now~\ref{lem:restrictions:H-i} follows
  from Lemma~\ref{lem:corner} and~\ref{lem:restrictions:H-supp} follows
  from~\ref{lem:supp-rho-otimes-varphi:suppX} and~\eqref{eq:support-of-corner}.
  As~\ref{lem:r-and-H:in-S-delta-l} and~\ref{lem:r-and-H:in-E-S} also apply to $r$, and
  hold directly by definition for
  $i$,~\ref{lem:restrictions:i-i},~\ref{lem:restrictions:i-supp},
  and~\ref{lem:restrictions:r-supp} follow from the same argument.
  Finally,~\ref{lem:restrictions:r-i} is the definition of $(r \otimes \id_{\bfB})_{k,l}$.
\end{proof}

\begin{proposition}\label{prop:supp-F_l-over-S(X)} Let
  $\bfB_0 = (S_0,\pi_0,B_0) \xleftarrow{\varphi_1} \dots \xleftarrow{\varphi_l} \bfB_l =
  (S_l,\pi_l,B_l)$ be a chain of composable morphisms in
  $\contrCatUcoef{G,U}{\Lambda}{\calb}$.  Assume $l \leq L$ and
  $\suppG(\varphi_j) \subseteq M$ for all $j$.  Let
  $E' := \bigcup \suppX \varphi_j \cup \bigcup_j \suppX \bfB_j \subseteq \Lambda \times
  \Lambda$.  Then there are
  $\Phi_i \colon \bfB'_i \to \bfB'_{i+1} \in \contrCatUcoef{G,U}{S(X) \times
    \Lambda}{\calb}$, $i=0,\dots,2l$ such that
  \begin{equation*}
    F^l(\varphi_1,\dots,\varphi_l) = \Phi_{2l} \circ \dots \circ \Phi_{0}
  \end{equation*}
  and $\suppG \Phi_i \subseteq M$, $\suppX \Phi \subseteq E_S^{\circ 2} \times E'$,
  $\suppobj \bfB'_i \subseteq S(\ball) \times \Lambda$ for all $i$.
\end{proposition}

\begin{proof}
	Recall that by construction
   \begin{multline*}
     F^l(\varphi_1,\dots,\varphi_l) = \\
     \big(r \otimes \id_{\bfB_0} \big) \circ \big( \id_{\bfS_* (X)} \otimes \varphi_1 \big) 
     \circ \big( H \otimes \id_{\bfB_1} \big)  
     \circ \dots 
     \\ \dots \circ 
     \big( H \otimes \id_{\bfB_{l-1}} \big) \circ \big( \id_{\bfS_* (X)} \otimes
     \varphi_l \big) \circ \big( i \otimes \id_{\bfB_l} \big).
   \end{multline*}
   Using~\ref{lem:restrictions:varphi-i},\ref{lem:restrictions:H-i},\ref{lem:restrictions:i-i},
   and~\ref{lem:restrictions:r-i} we can rewrite this as
   \begin{multline*}
     F^l(\varphi_1,\dots,\varphi_l) = \\
     \big(r \otimes \id_{\bfB_0} \big)_{l,l} \circ \big( \id_{\bfS_* (X)} \otimes \varphi_1 \big)^{l,l}_{l,l-1}
     \circ \big( H \otimes \id_{\bfB_1} \big)^{l,l-1}_{l-1,l-1} \circ 
     \dots 
     \\ \dots \circ 
     \big( H \otimes \id_{\bfB_{l-1}} \big)_{1,1}^{2,1} \circ \big( \id_{\bfS_* (X)} \otimes
     \varphi_l \big)_{1,0}^{1,1} \circ \big( i \otimes \id_{\bfB_l} \big)^{1,0}.
   \end{multline*}
   By~\ref{lem:supp-rho-otimes-varphi:suppG} (and since $U \subseteq M$) for each of these
   factors $\suppG$ is contained in $M$.
   By~\ref{lem:restrictions:varphi-supp},\ref{lem:restrictions:H-supp},\ref{lem:restrictions:i-supp},
   and~\ref{lem:restrictions:r-supp}, $\suppX$ of each factor, is contained in
   $E_S^{\circ 2} \times E'$\footnote{Here we use that $E_S \subseteq E_S^{\circ 2}$ for
     our specific $E_S$.}.  Finally, the domains of the factors are of the form
   $(\bfS_n(X) \otimes \bfB_i)_{k,l}$ (or $\bfC(\ball) \otimes \bfB_0$) and have therefore
   $\suppobj$ contained in $S(\ball^{(L)})$.
 \end{proof}


\subsection{Estimates over $J \times \Lambda$}

We now consider the homotopy coherent functor
\begin{equation*}
  \widetilde F := (f \circ \bary \times \id_\Lambda)_* \circ F \colon \contrCatUcoef{G,U}{\Lambda}{\calb}
  \to \ch \contrCatUcoef{G}{J \times \Lambda}{\calb}.
\end{equation*}

\begin{proposition}\label{prop:supp-tildeF_l}\
	\begin{enumerate}[
                 label=(\thetheorem\alph*),
                 align=parleft, 
                 leftmargin=*,
                 labelindent=1pt,
                 labelsep=10pt,
                 ]
		\item\label{prop:supp-tildeF_l:chain-modules}
		  Let $\bfB \in \contrCatUcoef{G,U}{\Lambda}{\calb}$ and $E' := \suppX \bfB$. Then
		    \begin{equation*}
		        \begin{array}{rlrl}
                          \suppG \widetilde F_0(\bfB)_n \hspace{-5pt} & \subseteq U,
                          & \suppX \widetilde F_0(\bfB)_n \hspace{-5pt} & \subseteq E^{\circ 2} \times E', \\
                          \suppG \dd_n^{\widetilde F_0(\bfB)}  \hspace{-5pt} & \subseteq U,
                          & \suppX \dd_n^{\widetilde F_0(\bfB)} \hspace{-5pt}& \subseteq E^{\circ 2} \times E';
		        \end{array}
		    \end{equation*} 
		\item\label{prop:supp-tildeF_l:morphisms} Let
  $\bfB_0 = (S_0,\pi_0,B_0) \xleftarrow{\varphi_1} \dots \xleftarrow{\varphi_l} \bfB_l =
  (S_l,\pi_l,B_l)$ be a chain of composable morphisms in $\contrCatUcoef{G,U}{\Lambda}{\calb}$.  
  Assume $l \leq L$ and $\suppG(\varphi_j) \subseteq M$ for all $j$.  
  Let $E' := \bigcup \suppX \varphi_j \cup \bigcup_j \suppX \bfB_j$.  Then
  \begin{equation*}
    \begin{array}{rl}
    \suppX \widetilde F^l(\varphi_1,\dots,\varphi_l) \hspace{-5pt}& \subseteq
     (M^{2l}(E^{\circ 3} \times E'))^{\circ (2l+1)}, \\
     \suppG \widetilde F^l(\varphi_1,\dots,\varphi_l) \hspace{-5pt} & \subseteq  M^{2l+1}.
     \end{array}
  \end{equation*}
      \end{enumerate}
\end{proposition}

The point here is that the upper bounds are in terms of $E$, $E'$, $M$ and $l$ and
independent of $\bfB$ and the $\varphi_i$.

\begin{proof}
  Lemma~\ref{lem:supp-F_0} directly implies $\suppG \widetilde F_0(\bfB)_n$,
  $\suppG \dd_n^{\widetilde F_0(\bfB)} \subseteq U$.  Lemma~\ref{lem:supp-F_0} together
  with Lemma~\ref{lem:extra-E} implies
  $\suppX \widetilde F_0(\bfB)_n \subseteq E^{\circ 2} \times E'$,
  $\suppX \dd_n^{\widetilde F_0(\bfB)}\subseteq E^{\circ 2} \times E'$\footnote{As $U$
    fixes $\ball$ pointwise we could here use $E$ instead of $E^{\circ 2}$, but the
    precise form of the estimates is not important.}.
  Thus~\ref{prop:supp-tildeF_l:chain-modules} holds.
	
  Applying $(f \circ \bary \times \id_\Lambda)_*$ to the factorization from
  Proposition~\ref{prop:supp-F_l-over-S(X)} gives a factorization
  \begin{equation*}
    \widetilde F^l(\varphi_1,\dots,\varphi_l) = \widetilde \Phi_{2l} \circ \dots \circ \widetilde \Phi_{0}. 
  \end{equation*}
  The estimates for the $\Phi_i$ from Proposition~\ref{prop:supp-F_l-over-S(X)} translate
  to the $\tilde \Phi_i$.  We get $\suppG \widetilde \Phi_i \subseteq M$ directly from
  Proposition~\ref{prop:supp-F_l-over-S(X)}.  Using in addition Lemma~\ref{lem:extra-E} we
  get $\suppX \widetilde \Phi_i \subseteq E^{\circ 3} \times E'$.  As $\suppG$ is
  submultiplicative we get
  $\suppG \widetilde F^l(\varphi_1,\dots,\varphi_l) \subseteq M^{2l+1}$.  Using the
  composition formula~\ref{eq:composition} for $\suppX$ it is not difficult to obtain an
  explicit bound for $\suppX \widetilde F^l(\varphi_1,\dots,\varphi_l)$, for example
  $(M^{2l}(E^{\circ 3} \times E'))^{\circ (2l+1)}$ works.
  Thus~\ref{prop:supp-tildeF_l:morphisms} holds.
\end{proof}

  
 \section{Construction of the transfer}\label{sec:constr-transfer}

 Let $G$ be a reductive $p$-adic 
 group and $X$ be the associated extended Bruhat-Tits building.  For each
 $\underlinetupel{t} \in \IN^{\times 2}$ we will construct the data considered in
 Section~\ref{sec:support-estimates}, i.e.,
 \begin{itemize}[label=-,                 
                 align=parleft, 
                 leftmargin=*,
                 labelindent=5pt,
                 ] 
 \item numbers $L_\underlinetupel{t} \in \IN$, $\rho_\underlinetupel{t} > 0$;
 \item compact open subgroups $U_\underlinetupel{t} \subseteq G$ and a compact subset
   $M_\underlinetupel{t} \subseteq G$ with $U_\underlinetupel{t} \subseteq M$;
 \item a subspace $\ball_\underlinetupel{t} \subseteq X$ with a locally ordered triangulation
   all whose simplices are of diameter $< \rho_\underlinetupel{t}$;
 \item a sequence of further subspaces
   $\ball_\underlinetupel{t}=\ball^{(0)}_\underlinetupel{t} \subseteq \ball^{(1)}_\underlinetupel{t}
   \subseteq \dots \subseteq \ball^{(L_\underlinetupel{t})}_\underlinetupel{t}$ with
   $M_\underlinetupel{t} \cdot \ball^{l-1}_\underlinetupel{t} \subseteq
   \ball_\underlinetupel{t}^{l}$ for $l=1,\dots,L_t$, and such that $\ball^{L}_\underlinetupel{t}$
   is pointwise fixed by $U_\underlinetupel{t}$;
 \item a retraction $r^0_\underlinetupel{t} \colon X \to \ball_\underlinetupel{t}$ for the inclusion
   $i^0_\underlinetupel{t} \colon \ball_\underlinetupel{t} \to X$, i.e.,
   $r^0_\underlinetupel{t} \circ i^0_\underlinetupel{t} = \id_{\ball_\underlinetupel{t}}$;
 \item a homotopy
   $H^0_\underlinetupel{t} \colon i^0_\underlinetupel{t} \circ r^0_\underlinetupel{t}
   \simeq \id_X$, that is not expanding, i.e., for all $\tau \in [0,1]$, $x,x' \in X$ we
   will have
   $d_X(H^0_\underlinetupel{t}(x,\tau),H^0_\underlinetupel{t}(x',\tau)) \leq d_X(x,x')$;
 \item a map $f_\underlinetupel{t} \colon X \to |\bfJ_{\CVCYC}(G)|^\wedge$;
 \item a subset
   $E_\underlinetupel{t} \subseteq |\bfJ_{\CVCYC}(G)|^\wedge \times
   |\bfJ_{\CVCYC}(G)|^\wedge$.
 \end{itemize}
 Let $\ZerodimR \in \regularPOrGSCzero$.  As in Section~\ref{sec:support-estimates} (with
 $\Lambda = |\ZerodimR|^\wedge$) we obtain then for every $\underlinetupel{t}$ a homotopy
 coherent functor
 \begin{equation*}
   F_\underlinetupel{t}  = (F_\underlinetupel{t}^0,F_\underlinetupel{t}^1,\dots)
   \colon \contrCatUcoef{G,U_\underlinetupel{t}}{|\ZerodimR|^\wedge}{\calb}
   \to \ch \contrCatUcoef{G}{S(X) \times |\ZerodimR|^\wedge}{\calb}
 \end{equation*}
 and its composition
 \begin{equation*}
   \widetilde F_\underlinetupel{t} := ((f_\underlinetupel{t} \circ \bary) \times \id_{|\ZerodimR|^\wedge})_*
   \circ F_\underlinetupel{t}
   \colon \contrCatUcoef{G,U_\underlinetupel{t}}{|\ZerodimR|^\wedge}{\calb}
   \to \ch  \contrCatUcoef{G}{|\bfJ_{\CVCYC}(G)|^\wedge \times |\ZerodimR|^\wedge}{\calb}.
 \end{equation*}
 We will verify Assumption~\ref{assm:E_0} for each $\underlinetupel{t}$.  Thus the support
 estimates from  Proposition~\ref{prop:supp-tildeF_l} will be
 available.  These estimates will allow us to combine the
 $(f_\underlinetupel{t} \times \id_{|\ZerodimR|^\wedge})_* \circ F_\underlinetupel{t}$ to
 obtain a homotopy coherent functor
 \begin{equation*}
   \widetilde F_\ZerodimR \colon \contd^0_G(\ZerodimR) \to \ch  \contd^0_G(\bfJ_{\CVCYC}(G))
 \end{equation*}
 and Lemmas~\ref{lem:tau-is-transformation} and~\ref{lem:tau-is-weak-equiv} will imply that
 the $\tilde F_\ZerodimR$ induce the desired transfer for
 \begin{equation*}
   \bfK \contd^0_G\big(\bfJ_{\CVCYC}(G) \times -\big) \to \bfK \contd^0_G(- )
 \end{equation*}
 in $\regularPOrGSCzero\text{-}\Spectra$.


 \subsection{Retractions to balls in $X$}

 We write $d_X$ for the $\CAT(0)$-metric of $X$.  We fix a base point $x_0 \in X$ and
 write $\ball_R$ for the closed ball in $X$ centered at $x_0$.  We write
 $\pi_R \colon X \to \ball_R$ for the radial projection.  We define the radial homotopy
 $H_R \colon i_R \circ r_R \simeq \id_X$ with
 $H_R(x,\tau) = \pi_{\tau d_X(x,x_0)+(1-\tau)R}(x)$.  The $\CAT(0)$-condition implies that
 $H_R$ is contracting, i.e.,
 \begin{equation*}
   d_X \big(H^X_R(x,\tau),H_R^X(x',\tau) \big) \leq d_X(x,x')
 \end{equation*} 
 for all $R \geq 0$, $\tau \in [0,1]$ and $x,x' \in X$.  Let $U := G_{x_0}$ be the
 stabilizer of $x_0$ in $G$.  As the action of $G$ on $X$ is smooth and proper, this is a
 compact open subgroup of $G$.


 \subsection{Choosing the data}\label{subsec:data-choosing}
    
 We choose for $j \in \IN$ numbers $\epsilon_j > 0$, $\eta_j >0$, $L_j \in \IN$ and
 compact subsets $M_j \subseteq G$ such that
 \begin{equation*}
   \epsilon_j \to 0, \eta_j \to 0, L_j \to \infty \; \text{as} \; j \to \infty
 \end{equation*} 
 and that for any $K \subseteq G$ compact we have $K \subseteq M_j$ for all but finitely
 many $j$.  We also assume that each $M_j$ contains an open subgroup\footnote{To construct
   $M_j$ we can choose a metric on $G$ and take $M_j$ to be the closed ball of radius $j$
   around the unit in $G$. Since $G$ is a td-group, $M_j$ contains a compact open subgroup
   of $G$.}.  We set $\newD_j := L_j \cdot \max\{ d_X(x_0,gx_0) \mid g \in M_j \}$. 
   
 Let $N$ be the first number appearing in Theorem~\ref{thm:X-to-J}.  Given $t_0 \in \IN$,
 we obtain from Theorem~\ref{thm:X-to-J} applied to $M = M_{t_0}$ and
 $\epsilon = \epsilon_{t_0}$ a number $\beta_{t_0}$ and $\calv_{t_0} \subseteq \CVCYC$
 finite.  Given further $t_1 \in \IN$, we obtain again from Theorem~\ref{thm:X-to-J}
 applied to $\eta = \eta_{t_1}$ and $L = \newD_{t_1}$ numbers $R_{t_0,t_1} > 0$,
 $\rho'_{t_0,t_1} > 0$ and a map
 \begin{equation*}
   f_{t_0,t_1} \colon X \to |\bfJ^N_{\calv_{t_0}}(G)|^\wedge  
 \end{equation*}  
 such that~\refstepcounter{theorem}
 \begin{enumerate}[
                 label=(\thetheorem\alph*),
                 align=parleft, 
                 leftmargin=*,
                 labelsep=6pt,
                 ]
 \item\label{nl:f:G} for $x \in \ball_{R_{t_0,t_1}+\newD_{t_1}}$, $g \in M_{t_0}$ we have
   \begin{equation*}
     \fold{\bfJ}\big(f_{t_0,t_1}(gx),gf_{t_0,t_1}(x)\big) <
     (\beta_{t_0},\eta_{t_1},\epsilon_{t_0});
   \end{equation*}
 \item\label{nl:f:H} for $x \in \ball_{R_{t_0,t_1}+\newD_{t_1}}$, $R' \geq R_{t_0,t_1}$ we have
   \begin{equation*}
     \fold{\bfJ}\big(f_{t_0,t_1}(x),f_{t_0,t_1}(\pi_{R'}(x))\big) < (\beta_{t_0},\eta_{t_1},\epsilon_{t_0});
   \end{equation*}
 \item\label{nl:f:rho} for all $x,x' \in X$ with $d_X(x,x') < \rho'_{t_1}$ we have
   \begin{equation*}
     \fold{\bfJ}\big(f_{t_0,t_1}(x),f_{t_0,t_1}(x')\big)
     < (\beta_{t_0},\eta_{t_1},\epsilon_{t_0}).
   \end{equation*}                    
 \end{enumerate}
 Here and later we abbreviate $\fold{\bfJ} = \fold{\bfJ_{\CVCYC}(G)}$.  For
 $\underlinetupel{t}=(t_0,t_1)$ we now define respectively choose
 \begin{itemize}[label=-,                 
                 align=parleft, 
                 leftmargin=*,
                 labelindent=5pt,
                 ] 
 \item $L_\underlinetupel{t} := L_{t_1}$,
   $\rho_\underlinetupel{t} := \rho'_\underlinetupel{t}/ (L_t+1)$,
   $M_\underlinetupel{t} := M_{t_0}$, $\ball_\underlinetupel{t} := \ball_{R_{t}}$; as balls in a
   building the $\ball_\underlinetupel{t}$ can be triangulated;
 \item
   $\ball_\underlinetupel{t}^{(l)} := \ball_{R_{t_0,t_1} + \frac{l}{L_\underlinetupel{t}}
     \newD_{t_1}}$ for $l=0,\dots,L_t$;
 \item $U_\underlinetupel{t}$ a compact open subgroup
   that fixes $\ball_\underlinetupel{t}^{(L_\underlinetupel{t})}$
   pointwise and is contained in $M_\underlinetupel{t}$;\\[-2ex]
 \item $i^0_\underlinetupel{t} \colon \ball_\underlinetupel{t} \to X$ the inclusion;\\[-2ex]
 \item $r^0_\underlinetupel{t} := \pi_{R_{\underlinetupel{t}}} \colon X \to \ball_\underlinetupel{t}$,
   the radial projection;\\[-2ex]
 \item $f_\underlinetupel{t} := f_{t_0,t_1}$;\\[-2ex]
 \item $H^0_\underlinetupel{t} := H_{R_\underlinetupel{t}}$ the radial homotopy
   $i_\underlinetupel{t}^0 \circ r_\underlinetupel{t}^0 \simeq \id_X$;\\[-2ex]
 \item
   $E_\underlinetupel{t} := \big\{ \twovec{z}{z'} \; \big| \; \fold{\bfJ}(z,z') <
   (\beta_{t_0},\eta_{t_1},\epsilon_{t_0}) \big\}$.
 \end{itemize}
 We note that~\ref{nl:f:G},~\ref{nl:f:H} and~\ref{nl:f:rho} imply that
 Assumption~\ref{assm:E_0} holds.  We can now apply the construction from
 Section~\ref{sec:support-estimates} and obtain for
 $\ZerodimR \in \regularPOrGSCzero$ the
 homotopy coherent functor
 \begin{equation*}
    ((f_\underlinetupel{t} \circ \bary) \times \id_{|\ZerodimR|^\wedge})_* \circ F_\underlinetupel{t}
   \colon \contrCatUcoef{G,U_\underlinetupel{t}}{|\ZerodimR|^\wedge}{\calb}
   \to \ch \contrCatUcoef{G}{|\bfJ_{\CVCYC}(G)|^\wedge \times |\ZerodimR|^\wedge}{\calb}.
 \end{equation*}


 \subsection{The $\tilde F_t$ combine to a homotopy coherent functor on
   $\contd^0_{G,\bfU}(\ZerodimR)$.}

 We will now use the $G$-control structure $\contstrd^0(\bfSigma)$ from
 Definition~\ref{def:D_bfSigma} for $\bfSigma = \bfJ_{\CVCYC}(G)$.

\begin{lemma}\label{lem:E-union-of-E_t} 
  Let
  $E := \Big\{ \twovec{z',\underlinetupel{t}}{z,\underlinetupel{t}} \; \Big| \;
  \underlinetupel{t} \in \IN^{\times 2}, \twovec{z'}{z} \in E_\underlinetupel{t} \Big\}$.
  Then $E \in \contstrd^0_2(\bfJ_{\CVCYC}(G))$.
\end{lemma}

\begin{proof}   
  We need to verify the foliated control condition from Definition~\ref{def:D_bfSigma}.
  Let $\epsilon > 0$ be given.  Choose $k_0$ such that $\epsilon_{t_0} < \epsilon$ for all
  $t_0 \geq k_0$.  Let $t_0 \geq k_0$ be given.  Set $\beta := \beta_{t_0}$.  Let
  $\eta > 0$ be given.  Choose $k_1$ such that $\eta_{t_1} < \eta$ for all $t_1 \geq k_1$.
  Let $t_1 \geq k_1$ be given. Set $\underlinetupel{t} := (t_0,t_1)$.  Let
  $z,z' \in |\bfJ_{\CVCYC}(G)|^\wedge$ be given with
  $\twovec{z',\underlinetupel{t}}{z,\underlinetupel{t}} \in E$.  By definition of $E$ we
  then have $\twovec{z'}{z} \in E_{\underlinetupel{t}}$.  By definition of
  $E_{\underlinetupel{t}}$ we then have
  $\fold{\bfJ}(z,z') < (\beta_{t_0},\eta_{t_1},\epsilon_{t_0})$.  Since
  $\beta = \beta_{t_0}$, $\eta_{t_1} < \eta$ and $\epsilon_{t_0} < \epsilon$ this implies
  \begin{equation*}
    \fold{\bfJ}(z,z') < (\beta,\eta,\epsilon),
  \end{equation*}
  as required.
\end{proof}

Following Remark~\ref{rem:contd0-more-explicit} we will view $\contd^0_{G}(\ZerodimR)$ and
$\contd^0_G(\bfJ_{\CVCYC}(G) \times \ZerodimR)$ respectively as subcategories of
$\prodprimeinline_{\IN^{\times 2}} \contrCatUcoef{G}{|\ZerodimR|^\wedge}{\calb}$ and
$\prodprimeinline_{\IN^{\times 2}} \contrCatUcoef{G}{|\bfJ_{\CVCYC}(G) \times \ZerodimR|^\wedge}{\calb}$
respectively.  Thus objects are sequences
$\bfB = (\bfB_\underlinetupel{t})_{\underlinetupel{t} \in \IN^{\times}}$ of objects in
$\contrCatUcoef{G}{|\ZerodimR|^\wedge}{\calb}$
and $\contrCatUcoef{G}{|\bfJ_{\CVCYC}(G) \times \ZerodimR|^\wedge}{\calb}$
respectively satisfying the conditions spelled out in
Remarks~\ref{rem:contd0-more-explicit}.  Similarly morphisms are equivalence classes of
sequences $(\varphi_\underlinetupel{t})_{\underlinetupel{t} \in \IN^{\times 2}}$.  We
recall that sequences
$(\varphi_\underlinetupel{t})_{\underlinetupel{t} \in \IN^{\times 2}}$ and
$(\varphi'_\underlinetupel{t})_{\underlinetupel{t} \in \IN^{\times 2}}$ are equivalent, if
there is $k_0$ such that for all $t_0 \geq k_0$ there is $k_1$ such that for all
$t_1 \geq k_1$ we have $\varphi_{t_0,t_1} = \varphi'_{t_0,t_1}$.  In particular, we can
ignore all $\varphi_{t_0,t_1}$ with $t_0$ or $t_1$ small.  We define
$\contd^0_{G,\bfU}(\ZerodimR)$ as the full subcategory of $\contd^0_{G}(\ZerodimR)$ on all
objects $\bfB = (\bfB_\underlinetupel{t})_{\underlinetupel{t} \in \IN^{\times 2}}$ with
$\suppG \bfB_\underlinetupel{t} \subseteq U_\underlinetupel{t}$ for all $\underlinetupel{t}$.
We can now define the homotopy coherent functor
\begin{equation*}
  F_\ZerodimR = (F_\ZerodimR^0,F_\ZerodimR^1,\dots) \; \colon \contd^0_{G,\bfU}(\ZerodimR)
  \to \ch \contd^0_G \big(\bfJ_{\CVCYC}(G) \times \ZerodimR\big).	
\end{equation*} 
Let $\bfB = (\bfB_\underlinetupel{t})_{\underlinetupel{t} \in \IN^{\times 2}}$ be an
object of $\contd^0_{G}(\ZerodimR)$.  We define
\begin{equation*}
  F_\ZerodimR^0(\bfB) := \Big(((f_\underlinetupel{t} \circ \bary)
\times \id_{|\ZerodimR|^\wedge})_*(F_\underlinetupel{t}^0(\bfB_\underlinetupel{t}))\Big)_{\underlinetupel{t} \in \IN^{\times 2}}.
\end{equation*}
Let $\bfB_0 \xleftarrow{\varphi_1} \dots \xleftarrow{\varphi_l} \bfB_l$ be a chain of
composable morphisms in $\contd^0_G(\ZerodimR)$.  Write
$\varphi_i = \big((\varphi_i)_\underlinetupel{t}\big)_{\underlinetupel{t} \in \IN^{\times
    2}}$.  We
define
\begin{equation*}
  F_\ZerodimR^l(\varphi_1,\dots,\varphi_l)
  := \Big(((f_\underlinetupel{t} \circ \bary) \times \id_{|\ZerodimR|^\wedge})_*
  \big(F_\underlinetupel{t}^l\big((\varphi_1)_\underlinetupel{t},\dots,(\varphi_l)_\underlinetupel{t}\big)\big)\Big)_{\underlinetupel{t} \in \IN^{\times 2}}.
\end{equation*}
We write $\big(F^0_\ZerodimR(\bfB)\big)_n$ for the $n$-chain module of
$F_\ZerodimR^0(\bfB)$ and $\partial_n^{F^0_\ZerodimR(\bfB)}$ for the $n$-th boundary map.

\begin{lemma}\label{lem:F_P-is-ok}
  \begin{enumerate}[label=(\alph*),leftmargin=*]
  \item\label{lem:F_P-is-ok:chain-modules} $\big(F^0_\ZerodimR(\bfB)\big)_n$ as above is
    an object in $\contd_G(\bfJ_{\CVCYC}(G) \times \ZerodimR)$;
  \item\label{lem:F_P-is-ok:boundary} $\partial_n^{F^0_\ZerodimR(\bfB)}$
    as above is a morphism in $\contd_G(\bfJ_{\CVCYC}(G) \times \ZerodimR)$;\\[-2ex]
  \item\label{lem:F_P-is-ok:morphisms} $F_\ZerodimR^l(\varphi_1,\dots,\varphi_l)$ as above
    is a morphism in $\ch \contd_G(\bfJ_{\CVCYC}(G) \times \ZerodimR)$.
  \end{enumerate}
\end{lemma}

\begin{proof}~\ref{lem:F_P-is-ok:chain-modules} We need to verify the four
  conditions~\ref{rem:contd0-more-explicit:obj:suppobj}
  to~\ref{rem:contd0-more-explicit:obj:finite}.
  
  For~\ref{rem:contd0-more-explicit:obj:suppobj} we need to check
  $\suppobj \big(F^0_\ZerodimR(\bfB)\big)_n \in \contd^0_1(\bfJ_{\CVCYC}(G) \times
  \ZerodimR)$.  Let
  \begin{equation*}
    F := \bigcup_{\underlinetupel{t} \in \IN^{\times 2}}
    f_\underlinetupel{t}(\bary_n (\ball_{\underlinetupel{t}})) \times \{ \underlinetupel{t} \}.   
  \end{equation*}
  We claim that $F \in \contstrd_1^0(\bfJ_{\CVCYC}(G))$.  This amounts to checking the
  three conditions in~\ref{def:D_bfSigma:obj} for $F$: Finiteness over $\IN^{\times 2}$ is
  clear as $\ball_{\underlinetupel{t}}$ is compact and has only finitely many simplices.  For the other two
  conditions we use that $f_\underlinetupel{t} = f_{t_0,t_1}$ is a map to
  $|\bfJ^N_{\calv_{t_0}}(G)|^\wedge \subseteq |\bfJ_{\CVCYC}(G)|^\wedge$.  Thus $F$ has
  finite dimensional support.  As $\calv_{t_0}$ is finite $|J^N_{\calv_{t_0}}(G)|$ is a
  finite subcomplex of $|J_{\CVCYC}(G)|$.  It follows that $F$ has compact support in
  $|J_{\CVCYC}(G)|$.  So $F \in \contstrd_1^0(\bfJ_{\CVCYC}(G))$.    
  Set $F' := \suppobj \bfB \in \contstrd_1^0(\ZerodimR)$.
  By~\ref{lem:finiteness_and_suppobj-diag-tensor-obj:suppobj}
  \begin{equation*}
    \suppobj ((f_\underlinetupel{t} \circ \bary)\times \id_{|\ZerodimR|^\wedge})_*  (F_{\underlinetupel{t}}^0(\bfB_\underlinetupel{t}))
    = f_\underlinetupel{t}(\bary_n(\ball_\underlinetupel{t})) \times \suppobj \bfB_t.
  \end{equation*}
  Thus  
  \begin{eqnarray*}
    \suppobj \big(F^0_\ZerodimR(\bfB)\big)_n  & = &
                                                    \bigcup_{\underlinetupel{t} \in \IN^{\times 2}}
                                                    \suppobj ((f_\underlinetupel{t} \circ \bary) \times \id_{|\ZerodimR|^\wedge})_*
          (F_{\underlinetupel{t}}^0(\bfB_\underlinetupel{t}) \times \{\underlinetupel{t}\}) \\
    & = & 
    \big\{ (z,\lambda,\underlinetupel{t})
    \mid (z,\underlinetupel{t}) \in F, (\lambda,\underlinetupel{t}) \in F' \big\}.
  \end{eqnarray*}
  Thus
  $\suppobj \big(F^0_\ZerodimR(\bfB)\big)_n \in \contstrd^0_1(\bfJ_{\CVCYC}(G) \times
  \ZerodimR)$ by~\ref{lem:D_upper_0_bfSigma-prod-P:1}.
      
  For~\ref{rem:contd0-more-explicit:obj:suppX} we need to check
  $\suppX \big(F^0_\ZerodimR(\bfB)\big)_n \in \contstrd^0_2(\bfJ_{\CVCYC}(G) \times
  \ZerodimR)$.  By~\ref{prop:supp-tildeF_l:chain-modules} 
  \begin{equation*}
    \suppX (f_\underlinetupel{t} \times \bary)
    \subseteq E_\underlinetupel{t}^{\circ 2} \times \suppX \bfB_\underlinetupel{t}.
  \end{equation*}
  We now use $E \in \contstrd_2^0(\bfJ_{\CVCYC}(G))$ from Lemma~\ref{lem:E-union-of-E_t}.
  Then
  \begin{equation*}
    \suppX (F^0_\ZerodimR(\bfB))_n
    \; \subseteq \; \Big\{ \twovec{z',\lambda',\underlinetupel{t}}{z,\lambda,\underlinetupel{t}}
    \, \Big| \, \twovec{z',\underlinetupel{t}}{z,\underlinetupel{t}} \in E^{\circ 2},
    \twovec{\lambda',\underlinetupel{t}}{\lambda,\underlinetupel{t}}
    \in \suppX \bfB \Big\} 
  \end{equation*}
  which belongs to $\contstrd_2^0(\bfJ_{\CVCYC}(G) \times \ZerodimR)$
  by~\ref{lem:D_upper_0_bfSigma-prod-P:2}.  
    
  For~\ref{rem:contd0-more-explicit:obj:suppG} we need to check
  $\suppG \big(F^0_\ZerodimR(\bfB)\big)_n \in \contstrd^0_G(\bfJ_{\CVCYC}(G) \times
  \ZerodimR)$.  By~\ref{lem:suppG_and_suppobj-for-diagonal-tensor:obj} we have 
  $\suppG \big(F^0_\ZerodimR(\bfB)\big)_n \subseteq \suppG \bfB$.  The latter is
  relatively compact as $\bfB \in \contd_G^0(\ZerodimR)$.  Thus
  $\suppG \big(F^0_\ZerodimR(\bfB)\big)_n$ is relatively compact as well.
    
  Finally, the finiteness condition~\ref{rem:contd0-more-explicit:obj:finite}
  for $\big(F^0_\ZerodimR(\bfB)\big)_n$
  follows from~\ref{lem:suppG_and_suppobj-for-diagonal-tensor:finite} as
  the $\bfB_{\underlinetupel{t}}$ are finite. 
  \medskip \\ \noindent~\ref{lem:F_P-is-ok:boundary} We need to
  verify~\ref{rem:contd0-more-explicit:mor:suppX}
  and~\ref{rem:contd0-more-explicit:mor:suppG} for $\partial_n^{F^0_\ZerodimR(\bfB)}$.
  This can be done exactly in the same way as for the corresponding
  properties~\ref{rem:contd0-more-explicit:obj:suppX}
  and~\ref{rem:contd0-more-explicit:obj:suppG} for $\big(F^0_\ZerodimR(\bfB)\big)_n$
  under~\ref{lem:F_P-is-ok:chain-modules}.  
  For~\ref{rem:contd0-more-explicit:mor:suppG} this
  uses~\ref{lem:suppG_and_suppobj-for-diagonal-tensor:boundary}
  in place of~\ref{lem:suppG_and_suppobj-for-diagonal-tensor:obj}.
  \medskip
  \\ \noindent~\ref{lem:F_P-is-ok:morphisms} We need to
  verify~\ref{rem:contd0-more-explicit:mor:suppX}
  and~\ref{rem:contd0-more-explicit:mor:suppG} for
  $F_\ZerodimR^l(\varphi_1,\dots,\varphi_l)$.
   
  Choose $k_0$ such that $\suppG \varphi_i \subseteq M_{k_0}$ for $i=1,\dots,l$.  Since we
  are working in $\contd^0_G(\bfJ_{\CVCYC}(G) \times \ZerodimR)$ we can ignore all
  $\underlinetupel{t} = (t_0,t_1)$ with $t_0 \leq k_0$.  Thus for the rest of the argument
  we can and will always assume $t_0 \geq k_0$.  From~\ref{prop:supp-tildeF_l:morphisms} we obtain
  that
  \begin{equation*}
    \suppG 	\big(((f_\underlinetupel{t} \bary) \times \id_{|\ZerodimR|^\wedge})_*(F_\underlinetupel{t}^l)
    \big((\varphi_1)_\underlinetupel{t},\dots,(\varphi_l)_\underlinetupel{t}\big)\big) \subseteq M_{k_0}^{2l+1}
  \end{equation*}
  for all $\underlinetupel{t}$.  In particular, the union of these $G$-supports is
  relatively compact in $G$.  Thus
  $\suppG F_\ZerodimR^l(\varphi_1,\dots,\varphi_l) \in \contstrd_G^0(\bfJ_{\CVCYC}(G)
  \times \ZerodimR)$ as required by~\ref{rem:contd0-more-explicit:mor:suppG}.
  	
  Set
  $E'_\underlinetupel{t} := \bigcup_j \suppX ((\varphi_j)_\underlinetupel{t}) \cup
  \bigcup_j \suppX ((\bfB_j)_\underlinetupel{t})$.  From~\ref{prop:supp-tildeF_l:morphisms} we obtain
  \begin{multline}\label{eq:supp-f_t-F_tl} \suppX
    (((f_\underlinetupel{t} \circ \bary) \times
    \id_{|\ZerodimR|^\wedge})_*(F_\underlinetupel{t}^l)\big((\varphi_1)_\underlinetupel{t},\dots,(\varphi_l)_\underlinetupel{t}\big)
    \\
    \subseteq \Big\{ \twovec{z',\lambda'}{z,\lambda} \Big| \twovec{z'}{z} \in
    \big(M_{k_0}^{2l}\cdot E_\underlinetupel{t}^{\circ 3}\big)^{\circ (2l+1)},
    \twovec{\lambda'}{\lambda} \in \big(M_{k_0}^{2l}\cdot
    E'_\underlinetupel{t}\big)^{\circ (2l+1)} \Big\}.
  \end{multline}
  We use again $E \in \contstrd_2^0(\bfJ_{\CVCYC}(G))$ from Lemma~\ref{lem:E-union-of-E_t}
  and set
  \begin{eqnarray*}
    \tilde E := \big( M_{k_0}^{2l} \cdot E^{\circ 3} \big)^{\circ(2l+1)} \in \contstrd_2^0(\bfJ_{\CVCYC}(G)).
  \end{eqnarray*}
  We also have, as the $\varphi_j$ and the $\bfB_j$ are from
  $\contd^0_{G,\bfU}(\ZerodimR)$,
  \begin{equation*}
    E' := \Big\{ \twovec{\lambda',\underlinetupel{t}}{\lambda,\underlinetupel{t}} \; \Big| \;
    \twovec{\lambda'}{\lambda} \in E'_\underlinetupel{t} \Big\} \in \contstrd_2^0(\ZerodimR)
  \end{equation*} 
  and also $\tilde E' := (M_{k_0}^{2l} \cdot E')^{\circ (2l+1)} \in \contstrd_2^0(\ZerodimR)$.
  Now~\eqref{eq:supp-f_t-F_tl} implies
  \begin{equation*}
    \suppX F_\ZerodimR^l(\varphi_1,\dots,\varphi_l)
    \subseteq \Big\{ \twovec{z',\lambda',\underlinetupel{t}}{z,\lambda,\underlinetupel{t}} \; \Big| \;
    \twovec{z',\underlinetupel{t}}{z,\underlinetupel{t}} \in \tilde E,
    \twovec{\lambda',\underlinetupel{t}}{\lambda,\underlinetupel{t}} \in \tilde E' \Big\}.
  \end{equation*}
  The latter belongs to $\contstrd_2^0(\bfJ_{\CVCYC}(G) \times \ZerodimR)$
  by~\ref{lem:D_upper_0_bfSigma-prod-P:2}.  Thus
  $\suppX F_\ZerodimR^l(\varphi_1,\dots,\varphi_l) \in \contstrd_2^0(\bfJ_{\CVCYC}(G)
  \times \ZerodimR)$ as required by~\ref{rem:contd0-more-explicit:mor:suppX}.
\end{proof}

\begin{proof}[Conclusion of the proof of Theorem~\ref{thm:transfer-bfD0}]
  We have the following diagram.
  \begin{equation*}
    \xymatrix{\\
      \contd^0_{G,\bfU}(\ZerodimR) \ar@/^9ex/[rrrr]^{I} \ar[r]^<<<<<{F_\ZerodimR}
      & \ch_{\fin}  \contd^0_G(\bfJ_{\CVCYC}(G) \times \ZerodimR) \ar[rrr]^<<<<<<<<<<<<<<<{(\pr_\ZerodimR)_*}
      & & & \ch_{\fin}  \contd^0_G(\ZerodimR) 
    }
  \end{equation*}	
  The functor $(\pr_\ZerodimR)_*$ is induced by the projection
  $\bfJ_{\CVCYC}(G) \times \ZerodimR \to \ZerodimR$, $F_\ZerodimR$ is the homotopy
  coherent functor defined above (well-defined by Lemma~\ref{lem:F_P-is-ok}), and $I$ is
  the inclusion.  We can then apply K-theory and obtain
  \begin{equation*}
    \xymatrix{\\
      \bfK \big( \contd^0_{G,\bfU}(\ZerodimR) \big) \ar@/^9ex/[rrrr]^{\bfK(I)} \ar[r]^<<<<<{\bfK(F_\ZerodimR)}
      & \bfK \big( \contd^0_G(\bfJ_{\CVCYC}(G) \times \ZerodimR) \big) \ar[rrr]^<<<<<<<<<<<<<<<{\bfp_\ZerodimR = \bfK((\pr_\ZerodimR)_*)}
      & & &\bfK \big(  \contd^0_G(\ZerodimR) \big). 
    }
  \end{equation*}	
  Strictly speaking the homotopy coherent functor $F_\ZerodimR$ induces a zig-zag in
  K-theory, see Remark~\ref{rem:K-zig-zag-homotopy-coherent}.  Since
  $\contd^0_{G,\bfU}(\ZerodimR) \big) \to \contd^0_G(\ZerodimR)$ induces an equivalence of
  idempotent completions, compare Subsection~\ref{subsec:calh-G-U}, the map $\bfK(I)$ is a
  weak equivalence.  It is not difficult to see from the construction that the diagram is
  natural in $\ZerodimR$, basically because the tensor product functor is natural for
  induced functors in both entries.  There is a strict natural transformation
  $\tau_\ZerodimR \colon (\pr_\ZerodimR)_* \circ F_\ZerodimR \to I$ by weak equivalences.
  It is given by applying the construction from Subsection~\ref{subsec:back-to-Z} for each
  $t \in \IN^{\times 2}$, see Lemma~\ref{lem:tau-is-transformation}.  On each object
  $\tau_\ZerodimR$ evaluates to a chain homotopy equivalence, see
  Lemma~\ref{lem:tau-is-weak-equiv}.  It follows that
  $(\pr_\ZerodimR)_* \circ F_\ZerodimR$ and $I$ induce homotopic maps in K-theory, see
  Remark~\ref{rem:K-zig-zag-strict-transf}. 
  Altogether,
  \begin{equation*}
  \xymatrix{\bfK \big(  \contd^0_G(\ZerodimR) \big) & &
     \bfK \big( \contd^0_{G,\bfU}(\ZerodimR) \big)  
     \ar[ll]_{\bfK(I)}^{\sim} \ar[rr]^<<<<<<<<<<{\bfK(F_\ZerodimR)}
     & & 
     \bfK \big( \contd^0_G(\bfJ_{\CVCYC}(G) \times \ZerodimR) \big) 
  }
  \end{equation*}
  is the required section $\bftr_\ZerodimR$ for $\bfp_\ZerodimR$.
\end{proof}

This concludes the proof of the $\CVCYC$-Farrell--Jones Conjecture for reductive $p$-adic
groups, i.e., of Theorem~\ref{thm:Farrell-Jones-Conjecture-for-reductive-p-adic}.


\section{Reduction from $\CVCYC$ to $\COM$}\label{sec:reduction}

\begin{theorem}[Reduction]\label{thm:reduction}
  Let $G$ be a td-group and let $\calb$ be Hecke category with $G$-support satisfying (Reg) from
  Definition~\ref{def:regularity-for-COP}.
  Assume that the $\CVCYC$-assembly
  map~\eqref{eq:FJ-assembly-map} for $\calb$
  \begin{equation*} 
    \hocolimunder_{P \in \EP\CVCYC(G)} \bfK \big( \contc_G(P) \big) \;
    \to \; \bfK \big( \contc_G(\ast) \big) 
  \end{equation*}
  is an equivalence.  Then the $\COP$-assembly map~\eqref{eq:assembly-cop-calb} for
  $\calb$ 
  \begin{equation*} 
    \hocolimunder_{G/U \in \Or_{\COP}(G)} \bfK \big( \calb[G/U] \big)
    \to \bfK \big( \calb[G/G] \big) \simeq \bfK \calb
  \end{equation*}  
  is also an equivalence.
\end{theorem}

Modulo further results the proof of Theorem~\ref{thm:reduction} will be given in
Subsection~\ref{subsec:contC}. 


\subsection{Functoriality in the orbit category}\label{subsec:functro-orbit-cat}

Recall that the definition of $\contc_G(P)$ used the (free) $G$-space $|P|^\wedge$.  The
construction of this space is not functorial for $P \in \EPplus\Or(G)$, only for
$P \in \EPplus\All(G)$.  For this reason, $\contc_G(P)$ is not functorial in
$\EPplus\Or(G)$.  But this is not a serious issue and we will now tweak the definitions to
remedy this.
  
Let $\calb$ be a category with $G$-support and $\mfE = (\mfE_1,\mfE_2,\mfE_G)$ be a
$G$-control structure on $X$.  Let $\caly$ be a collection of subsets of $X$ as in
Definition~\ref{def:E_Y}.  We say that two maps $\pi, \rho \colon S \to X$ are
\emph{$\mfE_2$-equivalent} if
\begin{equation*}
  \Big\{ \twovec{\rho(s)}{\pi(s)} \; \Big| \; s \in S \Big\} \in \mfE_2.
\end{equation*}
In the definition of $\contrCatUcoef{G}{\mfE,\caly}{\calb}$ we can replace the maps
$\pi \colon S \to X$ with their $\mfE_2$-equivalence classes and obtain the category
$\contrCatUcoefover{G}{\mfE,\caly}{\calb}$.  In more detail, objects of
$\contrCatUcoefover{G}{\mfE,\caly}{\calb}$ are triples $\bfB = (S,[\pi],B)$ such that
$(S,\pi,B)$ is an object of $\contrCatUcoef{G}{\mfE,\caly}{\calb}$.  The point here is
that, if $\pi$ and $\rho$ are $\mfE_2$-equivalent, then $(S,\pi,B)$ is an object of
$\contrCatUcoef{G}{\mfE,\caly}{\calb}$, iff $(S,\rho,B)$ is an object of
$\contrCatUcoef{G}{\mfE,\caly}{\calb}$.  A morphism
$\bfB = (S,[\pi],B) \to \bfB' = (S',[\pi'],B')$ in
$\contrCatUcoefover{G}{\mfE,\caly}{\calb}$ is a matrix
$\varphi = (\varphi^{s'}_s \colon B(s) \to B'(s'))_{s \in S, s' \in S'}$ of morphisms in
$\calb$ that also defines a morphism $(S,\pi,B) \to (S',\pi',B')$ in
$\contrCatUcoef{G}{\mfE,\caly}{\calb}$.  Again, the point here is that, if $\pi$ and
$\rho$ are $\mfE_2$-equivalent, $\pi'$ and $\rho'$ are $\mfE_2$-equivalent, then $\varphi$
defines a morphism in $(S,\pi,B) \to (S',\pi',B')$ in
$\contrCatUcoef{G}{\mfE,\caly}{\calb}$, iff $\varphi$ defines a morphism in
$(S,\rho,B) \to (S',\rho',B')$ in $\contrCatUcoef{G}{\mfE,\caly}{\calb}$.  The functor
\begin{equation}\label{eq:contrCatUcoef-to-contrCatUcoefover}
  \contrCatUcoef{G}{\mfE,\caly}{\calb} \xrightarrow{\sim} \contrCatUcoefover{G}{\mfE,\caly}{\calb}, \quad(S,\pi,B) \mapsto (S,[\pi],B)	
\end{equation} 
is an equivalence.  Essentially, $\contrCatUcoefover{G}{\mfE,\caly}{\calb}$ is obtained
from $\contrCatUcoef{G}{\mfE,\caly}{\calb}$ by identifying $(S,\pi,B)$ and $(S,\rho,B)$
along a canonical isomorphism, whenever $\pi$ and $\rho$ are $\mfE_2$-equivalent.

Recall $\contc_G(P) = \contrCatUcoef{G}{\contstrc(P),\caly(P)}{\calb}$ from
Definition~\ref{def:C_G-P}.  We now define
$\contcover_G(P) := \contrCatUcoefover{G}{\contstrc(P),\caly(P)}{\calb}$.
From~\eqref{eq:contrCatUcoef-to-contrCatUcoefover} we obtain a natural equivalence
\begin{equation}\label{eq:contc-to-contcover}
  \contc_G(P) \xrightarrow{\sim} \contcover_G(P).
\end{equation} 
Let $f \colon P \to P'$ be a map in $\EPplus\Or G$.  It induces a map
$|f| \colon |P| \to |P'|$.  There is always a lift of $|f|$ to
$|P|^\wedge \to |P'|^\wedge$, but it is typically not unique.
 
\begin{lemma}\label{lem:equiv-lifts} Let
  $\widehat f_0, \widehat f_1 \colon |P|^\wedge \to |P|^\wedge$ be lifts of $|f|$.  Then
  $\widehat f_0 \times \id_\IN, \widehat f_1 \times \id_\IN \colon |P|^\wedge \times \IN
  \to |P|^\wedge \times \IN$ are $\contstrc_2(P)$-equivalent.
\end{lemma}

\begin{proof}
  To simplify notation, we just treat the case $P = G/H$, $P'=G/H'$.  Then
  $f = |f| \colon G/H = |G/H| \to G/H' = |G/H'|$.  Any $G$-map $f \colon G/H \to G/H'$ is
  of the form $xH \mapsto xgH'$ where $g^{-1}Hg \subseteq H'$.  Now for the two lifts
  $\widehat f_0, \widehat f_1 \colon G = |G/H|^\wedge \to G = |G/H'|^\wedge$ there are
  $h'_0,h'_1 \in H'$ with $\widehat f_0(x) = xgh'_0$, $\widehat f_1(x) = xgh'_1$ for all
  $x \in G$.  With $\beta := d(h'_0,h'_1)$ then for all $x \in G$
  \begin{equation*}
    \fold{H'}(\widehat f_0(x),\widehat f_1(x)) \leq (\beta,0).
  \end{equation*}
  This implies that $\widehat f_0 \times \id_\IN, \widehat f_1 \times \id_\IN$ are
  $\contstrc_2(P)$-equivalent.
\end{proof}

Lemma~\ref{lem:equiv-lifts} implies that the assignment
\begin{equation*}
  P \mapsto \contcover_G(P)
\end{equation*}
is functorial in $\EPplus\Or(G)$. 
This allows us to consider the bottom row in the following diagram
\begin{equation}\label{eq:product-assembly-orbit}
	\xymatrix{\displaystyle \hocolimunder_{P \in \EP\COP(G)}  \bfK \big( \contc_G(P) \big)
	    \ar[r] \ar[d]^{\sim} &
	  \displaystyle \hocolimunder_{P \in \EP\CVCYC(G)}  \bfK \big( \contc_G(P) \big) 
	    \ar[r] \ar[d]^{\sim} &
	  \bfK \big( \contc_G(\ast) \big) \ar[d]^{\sim}
	  \\
	  \displaystyle \hocolimunder_{P \in \EP\COP(G)}  \bfK \big( \contcover_G(P) \big)
	    \ar[r] \ar[d]^{\sim} &
	  \displaystyle \hocolimunder_{P \in \EP\CVCYC(G)}  \bfK \big( \contcover_G(P) \big) 
	    \ar[r] \ar[d]^{\sim} &
	  \bfK \big( \contcover_G(\ast) \big) \ar[d]^{=}
	  \\
	  \displaystyle \hocolimunder_{P \in \EP\Or_{\COP}(G)}  \bfK \big( \contcover_G(P) \big)
	    \ar[r]  &
	  \displaystyle \hocolimunder_{P \in \EP\Or_{\CVCYC}(G)}  \bfK \big( \contcover_G(P) \big) 
	    \ar[r] &
	  \bfK \big( \contcover_G(\ast) \big). 
	  \\
    }
\end{equation}
We claim that the vertical maps in this diagram are all equivalences.
For the first three vertical maps this follows from~\eqref{eq:contc-to-contcover}.  
For the first two vertical maps in the bottom row, this is an application of
Lemma~\ref{lem:hocolim-products-Q-to-P}.


\subsection{The category $\contc^\allG_G(P)$}\label{subsec:contC} 

We will need slightly bigger categories than $\contc_G(P)$ and $\contcover_G(P)$, where we
relax the condition on the supports in $G$.  This will be important in
Subsection~\ref{subsec:analysis-of-C_upper_0(M)}.  For smooth $P$ this will not change the
K-theory, but we do not know this for general $P$.

Define the $G$-control structure
$\mfC^\allG(P) = (\mfC^\allG_1(P),\mfC^\allG_2(P),\mfC^\allG_G(P))$ by
$\mfC^\allG_1(P) := \contstrc_1(P)$, $\mfC^\allG_2(P) := \contstrc_2(P)$ and
\begin{equation*}
  \mfC^\allG_G(P) := \text{All subsets of $G$}.
\end{equation*}
We define
\begin{equation*}
  \contc^\allG_G(P) :=  \contrCatUcoef{G}{\mfC^\allG(P),\caly(P)}{\calb}
  \quad \text{and} \quad \contcover^\allG_G(P) :=  \contrCatUcoefover{G}{\mfC^\allG(P),\caly(P)}{\calb}.  
\end{equation*}
As $\contstrc_G(P) \subseteq \mfC^\allG_G(P)$, there are inclusions 
$\contc_G(P) \to \contc^\allG_G(P)$ and $\contcover_G(P) \to \contcover^\allG_G(P)$.

\begin{lemma}\label{lem:c-is-C-for-smooth}
  Suppose that $P$ is smooth, i.e., that $P \in \EPplus\OPEN(G)$.  Then the inclusions
  $\contc_G(P) \to \contc^\allG_G(P)$ and $\contcover_G(P) \to \contcover^\allG_G(P)$
  induce equivalences in K-theory.
\end{lemma}

\begin{proof}
  Proposition~\ref{prop:c-is-b-for-smooth-P} holds, with exactly the same proof, for
  $\contc^\allG_G(P)$ in place of $\contc_G(P)$ as well\footnote{This works in the generality of
    categories with $G$-support in place of $\Hecke{G}(\cala)$.}.  This implies the
  result for $\contc_G(P) \to \contc^\allG_G(P)$.
  For $\contcover_G(P) \to \contcover^\allG_G(P)$ it follows now from~\eqref{eq:contrCatUcoef-to-contrCatUcoefover}.
\end{proof}

\begin{theorem}\label{thm:contC_G(P)-for-CVCYC-via-COMP}
  Suppose that $\calb$ satisfies (Reg) from
  Definition~\ref{def:regularity-for-COP}.  Then for all
  $P \in \EP \Or_\CVCYC(G)$ the canonical map
  \begin{equation}\label{eq:PCOM-over-P-to-P-C}
    \hocolimunder_{(Q,f) \in \EP\Or_\COM(G) \downarrow P} \bfK \big( \contcover^\allG_G(Q) \big) \;
    \xrightarrow{\sim} \bfK \big( \contcover^\allG_G(P) \big) 
  \end{equation}
  is an equivalence.
\end{theorem}

The proof of Theorem~\ref{thm:contC_G(P)-for-CVCYC-via-COMP} will be given later.  We do
not know whether or not Theorem~\ref{thm:contC_G(P)-for-CVCYC-via-COMP} holds also for
$\contcover_G(\--)$ instead of $\contcover^\allG_G(\--)$.  We note that for
Theorem~\ref{thm:contC_G(P)-for-CVCYC-via-COMP} it is important that we used
$\EP\Or_\COM(G)$; in contrast the category $\EP\COM(G) \downarrow P$ is typically empty.

\begin{proof}[Proof of Theorem~\ref{thm:reduction} modulo
  Theorem~\ref{thm:contC_G(P)-for-CVCYC-via-COMP}]
  Consider the following diagram.
  \begin{equation*}
    \xymatrix{\displaystyle \hocolimunder_{P \in \EP\Or_\COP(G)}  \bfK \big( \contcover_G(P) \big) 
      \ar[r]^\sim \ar[d]_{\alpha_1}^\sim
      & 
      \displaystyle \hocolimunder_{P \in \EP\Or_\COP(G)}  \bfK \big( \contcover^\allG_G(P) \big) 
      \ar[d]_{\widehat \alpha_1}^\sim 
      \\
      \displaystyle \hocolimunder_{P \in \EP\Or_\COM(G)}  \bfK \big( \contcover_G(P) \big) 
      \ar[r] \ar[d]_{\alpha_2}
      & 
      \displaystyle \hocolimunder_{P \in \EP\Or_\COM(G)}  \bfK \big( \contcover^\allG_G(P) \big) 
      \ar[d]_{\widehat \alpha_2}^\sim 
      \\
      \displaystyle \hocolimunder_{P \in \EP\Or_\CVCYC(G)}  \bfK \big( \contcover_G(P) \big) 
      \ar[r] \ar[d]_{\mathit{FJ}}^\sim 
      & 
      \displaystyle \hocolimunder_{P \in \EP\Or_\CVCYC(G)}  \bfK \big( \contcover^\allG_G(P) \big) 
      \ar[d]
      \\
      \bfK \big( \contcover_G(\ast) \big) 
      \ar[r]^\sim
      &
      \bfK \big( \contcover^\allG_G(\ast) \big) 
    }
  \end{equation*}
  We will argue below that all maps labeled with an $\sim$ are equivalences.  A diagram
  chase shows then that all other maps in the diagram are equivalences as well.  In
  particular the left vertical composition is an equivalence.
  Proposition~\ref{prop:cop-iso-via-PSub} in combination with the equivalences
  from~\eqref{eq:product-assembly-orbit} implies that this vertical composition is
  equivalent to~\eqref{eq:assembly-cop-calb}.

  The analog of the map labeled $\mathit{FJ}$ for $\contc_G(P)$ instead of
  $\contcover_G(P)$ and $\EP\CVCYC(G)$ instead of $\EP\Or_\CVCYC(G)$ is an equivalence by
  assumption of Theorem~\ref{thm:reduction}.  The equivalences
  from~\eqref{eq:product-assembly-orbit} imply that the map labeled $\mathit{FJ}$ is an
  equivalence as well.

  Lemma~\ref{lem:c-is-C-for-smooth} implies that the top and bottom horizontal maps are
  equivalences.  Any compact subgroup of a td-group is contained in a compact open
  subgroup, see Lemma~\ref{lem:small-compact-open}.  This implies that for all
  $P \in \EP\COM(G)$ the category $P \downarrow \EP\COP(G)$ is non-empty.  Thus
  Lemma~\ref{lem:hocolim-products-Q-to-P} implies that $\alpha_1$ and $\widehat \alpha_1$
  are equivalences.

  The map $\widehat \alpha_2$ is an equivalence by
  Theorem~\ref{thm:contC_G(P)-for-CVCYC-via-COMP} and the transitivity
  Lemma~\ref{lem:transitivity} for homotopy colimits.
\end{proof}


\subsection{The category $\contcover^{\allG,0}_G(P)$}
To prove Theorem~\ref{thm:contC_G(P)-for-CVCYC-via-COMP} we introduce a further category
as an intermediate step.  In many ways this is similar to the passage from
$\contd_G(\bfSigma;\calb)$ to $\contd^0_G(\bfSigma;\calb)$, see
Definition~\ref{def:cald_G}.

For $P \in \EPplus\All(G)$ we define the $G$-control structure
\begin{equation*}
	\contstrc^{\allG,0}(P) = (\contstrc^{\allG,0}_1(P),\contstrc^{\allG,0}_2(P),\contstrc^{\allG,0}_G(P))
\end{equation*}
as follows. Set
$\contstrc^{\allG,0}_1(P) := \contstrc^{\allG}_1(P) = \contstrc_1(P)$,
$\contstrc^{\allG,0}_G(P) := \contstrC_G(P)$.  We define $\contstrc^{\allG,0}_2(P)$ to consist of all
$E \in \contstrc^{\allG}_2(P) = \contstrc_2(P)$ satisfying
\begin{equation*}
  \qquad \twovec{\lambda',t'}{\lambda,t} \in E \; \implies \; {t}' = {t}.
\end{equation*} 
Set
\begin{equation*}
  \contc^{\allG,0}_G(P) := \contrCatUcoef{G}{\contstrc^{\allG,0}(P),\caly(P)}{\calb} \quad \text{and} \quad
  \contcover^{\allG,0}_G(P) := \contrCatUcoefover{G}{\contstrc^{\allG,0}(P),\caly(P)}{\calb}.
\end{equation*}

\begin{proposition}\label{prop:pushout-3C0-C}
  There are homotopy pushout diagrams of functors $\EPplus\All(G) \to \Spectra$
  \begin{equation*}
    \xymatrix{\bfK \big( \contc^{\allG,0}_G(\--) \big) \ar[d] \ar[r] & 
      \bfK \big( \contc^{\allG,0}_G(\--) \big) \ar[d] &
      \bfK \big( \contcover^{\allG,0}_G(\--) \big) \ar[d] \ar[r] & 
      \bfK \big( \contcover^{\allG,0}_G(\--) \big) \ar[d] \\
      \bfK \big( \contc^{\allG,0}_G(\--) \big) \ar[r] &
      \bfK \big( \contc^\allG_G(\--) \big) & 
      \bfK \big( \contcover^{\allG,0}_G(\--) \big) \ar[r] &
      \bfK \big( \contcover^\allG_G(\--) \big)
    }
  \end{equation*}
\end{proposition}

\begin{proof}
  For the diagram on the left this is analog to the proof of Proposition~\ref{prop:contd-vs-contd0}, more precisely to
  the construction of either of the pushout squares from Lemma~\ref{lem:two-pushout}.
  For the  diagram on the right it follows now from~\eqref{eq:contrCatUcoef-to-contrCatUcoefover}.
\end{proof}

Proposition~\ref{prop:pushout-3C0-C} implies that it suffices to prove
Theorem~\ref{thm:contC_G(P)-for-CVCYC-via-COMP} for $\contcover^{\allG,0}_G(\--)$ in place of
$\contcover^\allG_G(\--)$.


\subsection{Structure of the remainder of the proof}
For the proof of Theorem~\ref{thm:contC_G(P)-for-CVCYC-via-COMP} for
$\contcover^{\allG,0}_G(\--)$ we will introduce further variations of $\contc_G(\--)$ and
argue along the following diagram
\begin{equation}\label{eq:diagram-for-CVCYC-via-COMP} 
  \xymatrix{\displaystyle\hocolimunder_{(Q,f) \in \EP\Or_\COM(G) \downarrow P} \bfK \big(
    \contcover^{\allG,0}_G(Q) \big) \ar[r]  & \bfK \big( \contcover^{\allG,0}_G(P)
    \big) 
    \\
    \displaystyle\hocolimunder_{\underline{\Gamma}} \bfK \big( \contcover^{\allG,0}_G( M ) \big)
    \ar[u]^{\sim}_{(1)} \ar[r] \ar[d]_{\sim}^{(3)} & \bfK \big( \contcover^{\allG,0}_G( M
    )[\Gamma] \big) \ar[u]^{\sim}_{(2)} \ar[d]_{\sim}^{(4)}
    \\
    \displaystyle\hocolimunder_{\underline{\Gamma}} \bfK \big( \contcover^{\allG,0,\nowedge}_G( M )
    \big) \ar[r] & \bfK \big( \contcover^{\allG,0,\nowedge}_G( M )[\Gamma] \big)
    \\
    \displaystyle\hocolimunder_{\underline{\Gamma}} \bfK \big( \contc^{\allG,0,\nowedge}_G( M ) \big)
    \ar[u]^{\sim}_{(4)} \ar[r]^{\sim}_{(7)} & \bfK \big( \contc^{\allG,0,\nowedge}_G( M )[\Gamma]
    \big) \ar[u]^{\sim}_{(6)}.  
    }
\end{equation}
Details will be worked out in the remainder of this section.


\subsection{The group $\Gamma$}\label{subsec:Gamma} Here we discuss the equivalences (1)
and (2) from~\eqref{eq:diagram-for-CVCYC-via-COMP}.

Fix $P = (G/V_1,\dots,G/V_n)$ with $V_i \in \CVCYC$.  Let $K_i \subseteq V_i$ be the
maximal compact open subgroup of $V_i$.  Set $M := (G/K_1,\dots,G/K_n)$.  The quotients
$\Gamma_i := V_i / K_i$ are either infinite cyclic or trivial.  Let
$\Gamma := \Gamma_1 \times \cdots \times
\Gamma_n$. 
Then $\Gamma$ is a finitely generated free abelian group of rank at most $n$.  There are
canonical maps $h_i \colon \Gamma_i \to \homendo_{\Or(G)}(G/K_i)$, sending $cK_i \in C_i$
to $G/K_i \to G/K_i, gK_i \mapsto gcK_i$.  These combine to an action of $\Gamma$ on $M$
by morphisms in $\EP\Or(G)$.  This induces an action of $\Gamma$ on
$\contcover^{\allG,0}_G(M)$ and we can form $\contcover^{\allG,0}_G(M)[\Gamma]$.  The projection
$\pi \colon M \to P$ is $\Gamma$-equivariant for the trivial action of $\Gamma$ on $P$.
Thus $\pi_* \colon \contcover^{\allG,0}_G(M) \to \contcover^{\allG,0}_G(P)$ is $\Gamma$-equivariant for
the trivial action on $\contcover^{\allG,0}_G(P)$ and induces a functor
\begin{equation*}
  R \colon \contcover^{\allG,0}_G(M)[\Gamma] \; \to \; \contcover^{\allG,0}_G(P).
\end{equation*}
The functor $R$ induces (2) in~\eqref{eq:diagram-for-CVCYC-via-COMP}.  We write
$\underline{\Gamma}$ for the category with exactly one object $\ast_\Gamma$ whose
endomorphisms are given by $\Gamma$.  The action of $\Gamma$ on $M$ determines a functor
$h \colon \underline{\Gamma} \to \EP\Or_\COM(G)\downarrow P$ that sends $\ast_\Gamma$ to
$\pi \colon M \to P$.  In turn $h$ induces a map
\begin{equation*}
  \hocolimunder_{\underline{\Gamma}} \bfK \big( \contcover^{\allG,0}_G( M ) \big) \;
  \to  \; \hocolimunder_{(Q,f) \in \EP\Or_\COM(G) \downarrow P} \bfK \big( \contcover^{\allG,0}_G(Q) \big);
\end{equation*}
this is (1) in~\eqref{eq:diagram-for-CVCYC-via-COMP}.  The inclusion
$\contcover^{\allG,0}_G(M) \to \contcover^{\allG,0}_G(M)[\Gamma]$ induces the assembly map from the
second row of~\eqref{eq:diagram-for-CVCYC-via-COMP}
\begin{equation}\label{eq:assembly-for-Gamma}
  \hocolimunder_{\underline{\Gamma}} \bfK \big( \contcover^{\allG,0}_G( M ) \big)
  \to \bfK \big( \contcover^{\allG,0}_G( M )[\Gamma] \big).
\end{equation}

\begin{lemma}\label{lem:reduction-to-Gamma-assembly}
  \
  \begin{enumerate}[label=(\alph*),leftmargin=*]
  \item\label{lem:reduction-to-Gamma-assembly:R} The functor $R$ induces an equivalence
    in K-theory, i.e., (2) in~\eqref{eq:diagram-for-CVCYC-via-COMP} is an equivalence;
  \item\label{lem:reduction-to-Gamma-assembly:h} the functor $h$ induces an equivalence,
    i.e., (1) in~\eqref{eq:diagram-for-CVCYC-via-COMP} is an equivalence.
   \end{enumerate}
\end{lemma}

\begin{proof}\ref{lem:reduction-to-Gamma-assembly:R}
  We will first show that for $\bfB \in \contcover^{\allG,0}_G(P)$, there is
  $(\widetilde\bfB,p) \in \Idem \contcover^{\allG,0}_G(M)$ such that $\bfB \cong R(\widetilde\bfB,p)$ in
  $\Idem \contcover^{\allG,0}_G(P)$.  
  Write $\bfB = (S,[\pi],B)$.
  Let $U_0 \subseteq U_1 \subset \dots$ be a sequence of compact
  open subgroups of $G$ with $\diam U_i \to 0$ as $i \to \infty$.  Write
  $\pi(s) := (g_1(s),\dots,g_n(s),t(s)) \in G^n \times \IN = |P|^\wedge \times \IN$.  Set
  \begin{equation*}
    K(s) := \supp B(s) \cap g_1(s) U_{t(s)} (g_1(s))^{-1} \cap \dots \cap g_n(s) U_{t(s)} (g_n(s))^{-1}.
  \end{equation*}
  Using~\ref{def:Hecke-category:cofinal} we set $\widetilde\bfB := (s,\Pi,B|_{K(s)})$.
  By design $\widetilde\bfB \in \contcover^{\allG,0}_G(P)$.  
  Define $i \colon \bfB \to \widetilde\bfB$ and $r \colon \widetilde\bfB \to \bfB$ with
  \begin{equation*}
  	 i_s^{s'} = \begin{cases} i_{B(s),K(s)} & s=s' \\ 0 & \text{else} \end{cases} \quad \text{and} \quad
  	 r_s^{s'} = \begin{cases} r_{B(s),K(s)} & s=s' \\ 0 & \text{else} \end{cases}.
  \end{equation*}
  Again by~\ref{def:Hecke-category:cofinal} we have $r \circ i = \id_\bfB$.  Thus
  $p := i \circ r$ is an idempotent on $\widetilde\bfB$ and
  $\bfB \cong R(\widetilde\bfB,p)$ as promised.  It is an exercise in the definitions to
  check that $R$ is bijective on morphism sets.  Altogether, $R$ induces an equivalence on
  idempotent completions and thus in K-theory.
  \medskip \\
  \noindent~\ref{lem:reduction-to-Gamma-assembly:h} This follows from cofinality
  Lemma~\ref{lem:cofinality} for homotopy colimits; we need to verify that $h$ is right
  cofinal.  Fix $f \colon Q \to P$ from $\EP\Or_\COM(G) \downarrow P$. We need to show
  that $(Q,f) \downarrow h$ is contractible.  Objects in $(Q,f) \downarrow h$ are
  morphisms $Q \xrightarrow{g} M$ in $\EP\Or_\COM(G)$ such that $f = \pi \circ g$.
  Morphisms in $(Q,f) \downarrow h$ are commutative diagrams
  \begin{equation*}
    \xymatrix{Q \ar[r]^{g} \ar[d]^{=} & M \ar[r]^{\pi} \ar[d]^{\gamma}
      & P \ar[d]^{=} \\ Q \ar[r]^{g'} & M \ar[r]^{\pi} & P}
  \end{equation*}  
  We check that $(Q,f) \downarrow h$ is equivalent to a point.  Recall
  $P = (G/V_1,\dots,G/V_n)$ and $M = (G/K_1,\dots,G/K_n)$, where $K_i$ is the maximal
  compact open subgroup of $V_i$.  Write $Q = (G/L_1,\dots,G/L_m)$ and $f = (u,\varphi)$
  with $u \colon \{1,\dots,n\} \to \{1,\dots,m\}$ and
  $\varphi(i) \colon G/L_{u(i)} \to G/V_i$.  Then there are $x_1,\dots,x_n \in G$ with
  $\varphi(i)yL_{u(i)} = yx_iV_i$ for all $y \in G$, and
  $x_i^{-1}L_{u(i)}x_i \subseteq V_i$.  As $x_i^{-1}L_{u(i)}x_i$ is compact and $K_i$ is
  the unique maximal compact subgroup, we have $x_i^{-1}L_{u(i)}x_i \subseteq K_i$.  Let
  $g := (u,\psi) \colon Q \to M$ with $\psi(i) \colon G/L_{u(i)} \to G/K_i$ given by
  $\psi(i)yL_{u(i)} = yx_iK_i$ for all $y \in G$.  Then $f = \pi \circ g$.  Thus
  $(Q,f) \downarrow h$ is non-empty.  If $g' = (u',\psi') \colon Q \to M$ is another map
  with $f = \pi \circ g'$, then $u=u'$ and there are $v_i \in V_i$ with
  $\psi'(i)yL_{u(i)} = yx_iv_iK_i$. Now the $v_i$ define an isomorphism
  $\gamma \colon M \to M$ with $\gamma \circ g=g'$ and 
  this $\gamma$ is the only morphism with this property.  Thus any two objects in
  $(Q,f) \downarrow h$ are uniquely isomorphic.
\end{proof}


\subsection{The category $\contcover^{\allG,0,\nowedge}(M)$}
Here we discuss the equivalences (3), (4), (5) and (6)
from~\eqref{eq:diagram-for-CVCYC-via-COMP}.

Let $M = (G/K_1 \times \cdots \times G/K_n)$ be as before.  We give a slight variation of
the categories $\contc^{\allG,0}_G(M)$ and $\contcover^{\allG,0}_G(M)$.  The fact that
$|M|^\wedge \to |M|$ has compact fibers (since $M \in P\Or_\COM$) will allow us to base
the definition on $|M|$ instead of $|M|^\wedge$.  See also
Remark~\ref{rem:why-foliated-control}.

Let $\mu_i$ be a Haar measure on $K_i$.  We can integrate the left-invariant metric $d_G$
on $G$ to a left-invariant metric $d_{G,i}$ on $G$ that is in addition right
$K$-invariant, $d_{G,i}(g,g') := \int_{K_i} d_G(gk,g'k) d\mu_i(k)$ and obtain a
left-invariant metric
\begin{equation*}
  d_{G/K_i}(gK_i,g'K_i) := \min_{k \in K} d_{G,i}(gk,g')
\end{equation*} 
on $G/K_i$.   
We obtain a left-invariant metric $d_{|M|}$ on $|M|$
with
\begin{equation*}
  d_{|M|}\big((g_1K_1,\dots,g_nK_n), (g'_1K_1,\dots,g'_nK_n)\big) := \max_{i} d_{G/K_i,}(g_iK_i,g'_iK_i).
\end{equation*} 
We define the $G$-control structure
$\contstrc^{\allG,0,\nowedge}(M) =
\big(\contstrc^{\allG,0,\nowedge}_1(M),\contstrc^{\allG,0,\nowedge}_2(M),\contstrc^{\allG,0,\nowedge}_G(M)\big)$
on $|M| \times \IN$ as follows:
\begin{itemize}[
                 label=$\bullet$,
                 align=parleft, 
                 leftmargin=*,
                 labelindent=5pt,
                 ] 
\item $\contstrc^{\allG,0,\nowedge}_1(M)$ consists of all subsets $F$ of $|M| \times \IN$ for which
  $F \cap |M| \times \{t\}$ is finite for all $t \in \IN$;
\item $\contstrc^{\allG,0,\nowedge}_2(M)$ consists of all subsets $E$ of
  $\big(|M| \times \IN\big)^{\times 2}$ satisfying the following two conditions
  \begin{itemize}
  \item \emph{$0$-control over $\IN$}: if $\twovec{\lambda',t'}{\lambda,t} \in E$ then
    $t=t'$;
  \item \emph{metric control over $|M|$}: 
    for any $\eta > 0$ there is $t_0$ such
    that for all $t \geq t_0$ and all $\lambda,\lambda'$ we have
    \begin{equation*}
      \twovec{\lambda',t}{\lambda,t} \in E  \quad \implies \quad d_{|M|}(\lambda,\lambda')  < \eta;
    \end{equation*}
  \end{itemize}
\item $\contstrc^{\allG,0,\nowedge}_G(M)$ consists of all subsets of $G$.
\end{itemize}
Let $\caly^\nowedge(M)$ be the collection of all subsets $Y$ of $|M| \times \IN$, for which
there is $d$ with $Y \subseteq |M| \times \IN_{\leq d}$.  We define
\begin{eqnarray*}
  \contc^{\allG,0,\nowedge}(M)
  & := & \contrCatUcoef{G}{\contstrc^{\allG,0,\nowedge}(M),\caly^\nowedge(M)}{\calb}; \\
  \contcover^{\allG,0,\nowedge}(M)
  & := & \contrCatUcoefover{G}{\contstrc^{\allG,0,\nowedge}(M),\caly^\nowedge(M)}{\calb}.
\end{eqnarray*}
Let $p \colon |M|^\wedge \to |M|$ be the canonical projection.  As the $K_i$ are compact and
therefore of finite diameter, it is not difficult to check that $\contstrc^{\allG,0,\nowedge}_2(M)$ is
exactly the image of $\contstrc^{\allG,0}_2(M)$ under
$(p \times \id_\IN)^{\times 2} \colon (|M|^\wedge \times \IN)^{\times 2} \to (|M| \times
\IN)^{\times 2}$.  This in turn implies that $p \times \id$ induces an equivalence
$\contc^{\allG,0}_G(M) \xrightarrow{\sim} \contc^{\allG,0,\nowedge}_G(M)$.
Applying~\eqref{eq:contrCatUcoef-to-contrCatUcoefover}
we obtain the equivalence $\contc^{\allG,0}_G(M) \xrightarrow{\sim} \contc^{\allG,0,\nowedge}_G(M)$
and this yields the equivalences (3) and (4) from~\eqref{eq:diagram-for-CVCYC-via-COMP}.
Also from~\eqref{eq:contrCatUcoef-to-contrCatUcoefover} we obtain an equivalence
$\contc^{\allG,0,\nowedge}_G(M) \xrightarrow{\sim} \contcover^{\allG,0,\nowedge}_G(M)$ and this induces the
equivalences (5) and (6) from~\eqref{eq:diagram-for-CVCYC-via-COMP}.


\subsection{The limit category}\label{subsec:sequence-cat}
We briefly digress to recall a result from~\cite{Bartels-Lueck(2020additive)}.  
Consider a nested sequence of categories
\begin{equation*}
  \calc = \calc_0 \supseteq \calc_1 \supseteq \calc_2 \supseteq \cdots.
\end{equation*}
Associated to this is the \emph{sequence category} $\cals(\calc_*)$.  This is the
subcategory of $\prod_{m \geq 0} \calc$, whose objects are sequences $(C_m)_{m \geq 0}$
such that for any $l$ there is $m_0$ with $C_m \in \calc_l$ for all $m \geq m_0$.
Morphisms are sequences $(\varphi_m)_{m \geq 0}$ such that for any $l$ there is $m_0$ with
$\varphi_m \in \calc_l$ for all $m \geq m_0$.  The sum $\bigoplus_{m \geq 0} \calc$ is a
subcategory of $\cals(\calc_*)$ and we call the quotient category
\begin{equation*}
  \call(\calc_*) := \cals(\calc_*) \bigg/ \bigoplus_{m \geq 0} \calc_0
\end{equation*}
the \emph{limit category}.  We reviewed $l$-uniform regular coherence and exactness in
Subsection~\ref{subsec:regular_and_exact}.  Assume now that \refstepcounter{theorem}
\begin{enumerate}[label=(\thetheorem\alph*),
                 align=parleft, 
                 leftmargin=*,
                 labelsep=10pt,
                 ]                
\item\label{nl:calc-exact} for any $d \in \IN$ there is $m_0$ such that the inclusion
  $\calc_{m+1}[\IZ^d] \to \calc_{m}[\IZ^d]$ is exact for all $m \geq m_0$;
\item\label{nl:calc-reg-coherent} for any $d \in \IN$ there are $m_0,l \in \IN$ such that
  $\calc_m[\IZ^d]$ is $l$-regular coherent for all $m \geq m_0$.
\end{enumerate}
Here we use the trivial action to form $\calc_m[\IZ^d]$.
Then~\cite[Thm.~14.1]{Bartels-Lueck(2020additive)} asserts the following.  Consider an
action of a finitely generated free abelian group $\Gamma$ on $\calc$ with the following
property: for $\gamma \in \Gamma$, $j \in \IN$ there is $i \in \IN$ such that
$\gamma(\calc_i) \subseteq \calc_j$.  Then the assembly map
\begin{equation}\label{eq:assembly-for-Gamma-limit}
  \hocolimunder_{\underline{\Gamma}} \bfK \big( \call(\calc_*) \big)
  \xrightarrow{\sim} \bfK \big( \call(\calc_*)[\Gamma] \big)
\end{equation}
is an equivalence.  (In~\cite{Bartels-Lueck(2020additive)} the result is formulated using
an equivariant homology theory instead of a homotopy colimit, but the two formulations are
equivalent, see~\cite{Davis-Lueck(1998)}.)


\subsection{Analysis of $\contc^{\allG,0,\nowedge}_G(M)$}\label{subsec:analysis-of-C_upper_0(M)}
Finally, we discuss the equivalence (7) from~\eqref{eq:diagram-for-CVCYC-via-COMP}.

Fix again $M \in \EP\Or_\COM(G)$ and write $M = (G/K_1,\dots,G/K_n)$ with $K_r \in \COM$.
Using Lemma~\ref{lem:small-compact-open} we fix a choice of
compact open subgroups $(U_{r,i})_{r=1,\dots,n, i \in \IN_{\geq 1}}$ of
$G$ such that for each $r=1,\dots,n$ we have
\begin{equation*}
  U_{r,1} \supseteq U_{r,2} \supseteq \dots \;\;\supseteq  \bigcap_{i} U_{r,i} = K_r. 	
\end{equation*}
Set $U_{r,0} := G$\footnote{This ensure that $\calc_0$ defined later carries a 
  $\Gamma$-action.} and $Q_i := (G/U_{1,i},\dots,G/U_{n,i})$.  We write
$p^i \colon |M| \to |Q_i|$ and, for $j \geq i$, $p_j^i \colon |Q_j| \to |Q_i|$ for the
canonical projections.  We define the $G$-control structure
$\mfE^i = (\mfE^i_1,\mfE^i_2,\mfE^i_G)$ on $|M|$, where $\mfE^i_1$ is the collection of all
finite subsets of $|M|$, $\mfE^i_2$ consists of all $E \subseteq |M| \times |M|$
satisfying
\begin{equation}\label{eq:0-control-over-Q_i} \twovec{\lambda'}{\lambda} \in E \;
  \implies \; p^i(\lambda) = p^i(\lambda'),
\end{equation}
and $\mfE^i_G$ is the collection of all subsets of $G$.
We abbreviate $\calc_i := \contrCatUcoef{G}{\mfE^i}{\calb}$ and obtain a nested sequence
of categories $\calc_0 \supseteq \calc_1 \supseteq \calc_2 \supseteq \dots$.  The metric
control condition for $E \in \contstrCtilde^0_2(M)$ is equivalent to
\begin{itemize}
\item[-] for any $i$ there is $t_0$ such that for all $t \geq t_0$ and all
  $\lambda, \lambda'$ we have
  \begin{equation*}
    \twovec{\lambda',t}{\lambda,t} \in E \implies p^i(\lambda) = p^i(\lambda').
  \end{equation*}
\end{itemize}
From this observation it is easy to deduce that there is an equivalence
$\contc^{\allG,0,\nowedge}_G(M) \xrightarrow{\sim} \call(\calc_*)$ that sends an object $\bfB = (S,\pi,B)$ to the
sequence $(S_i,\pi_i,B_i)_{i \in \IN}$, where $S_i = \pi^{-1}(|M| \times \{ i \})$,
$B_i = B|_{S_i}$ and $\pi_i$ is the composition
\begin{equation*}
  S_i \xrightarrow{\pi|_{S_i}} |M| \times \{i\} \xrightarrow{\equiv} |M|.
\end{equation*}
We point out that for the equivalence $\contc^{\allG,0,\nowedge}_G(M) \xrightarrow{\sim} \call(\calc_*)$
it is important that we work with $\mfC^\allG_G(P) =\contstrc^{\allG,0,\nowedge}_G(M) = \text{All subsets of $G$}$;
if one were to use only compact subsets of $G$ instead, then the limit category would be strictly bigger. 

Let now $\Gamma = \Gamma_1 \times \cdots \times \Gamma_n$ be as in
Subsection~\ref{subsec:Gamma}.  The group $\Gamma_r$ does not necessarily normalize
$U_{r,i}$, but for each $\gamma \in \Gamma_r$ we still have
$\bigcap_i (U_{r,i})^\gamma = (K_r)^\gamma = K_r$.  Using the compactness of the $U_{r,i}$
for $i>0$, it is  easy to check that for fixed $\gamma$ and $j$ there is $i$ with
$(U_{r,i})^\gamma \subseteq U_{r,j}$.  This implies that the action of $\Gamma$ on $M$
induces an action acts on the nested sequence $C_*$ as in
Subsection~\ref{subsec:sequence-cat}.  
Moreover $\contc^{\allG,0,\nowedge}_G(M) \to \call(\calc_*)$ is $\Gamma$-equivariant and an equivalence.
It induces therefore an equivalence
$\contc^{\allG,0,\nowedge}_G(M)[\Gamma] \to \call(\calc_*)[\Gamma]$.  It follows that (7)
from~\eqref{eq:diagram-for-CVCYC-via-COMP} is an equivalence, if and only
if~\eqref{eq:assembly-for-Gamma-limit} is an equivalence for our present choice of
$\calc_*$.

\begin{lemma}\label{lem:nested-is-reg_and_exact} The nested sequence $\calc_*$ defined
  above satisfies~\ref{nl:calc-exact}.  If $\calb$ satisfies  (Reg) from
  Definition~\ref{def:regularity-for-COP}, then $\calc_*$ also
  satisfies~\ref{nl:calc-reg-coherent}.
\end{lemma}

\begin{proof}
  We start by examining $\calc_i = \contrCatUcoef{G}{\mfE^i}{\calb}$ for $i > 0$.  For
  $\lambda \in |Q_i|$ let $G_\lambda$ be the isotropy group of $\lambda$.  This is a
  finite intersection of conjugates of the $U_{i,r}$ and thus compact open in $G$.  Let
  $\bfB = (S,\pi,B) \in \calc_i$ and $s \in S$.  The control
  condition~\eqref{eq:0-control-over-Q_i} implies $\supp B(s) \subseteq G_{p_i(\pi(s))}$,
  in other words $B(s) \in \calb|_{G_{p_i(\pi(s))}}$.  For $\lambda \in |Q_i|$ we obtain,
  still using~\eqref{eq:0-control-over-Q_i}, a functor
  \begin{eqnarray*}
    F_{i,\lambda} \colon \quad \calc_i
    & \to &
            (\calb|_{G_\lambda})_\oplus \\
    \bfB
    & \mapsto &
                \bigoplus_{p_i(\pi(s)) = \lambda} B(s).
  \end{eqnarray*}
  By definition of $\mfE^i_1$ the set $S$ is finite and so the above sum is finite and
  $F_{i,\lambda}(\bfB) = 0$ for all but finitely many $\lambda$.  
  The $F_{i,\lambda}$ combine to an equivalence
  \begin{equation}\label{eq:calc_i-as-Hecke}
    \begin{array}{ccc}
      \calc_i   & \xrightarrow{\sim} & \bigoplus_{\lambda \in |Q_i|} (\calb|_{G_\lambda})_\oplus \\[3pt]
      \bfB  & \mapsto & (F_{i,\lambda}(\bfB))_{\lambda \in |Q_i|}.
    \end{array}
  \end{equation}
    Since the property $l$-uniformly regular passes to direct sums over arbitrary index
    sets and is invariant under equivalence of additive categories and the passage going
    from an additive category $\cala$ to $\cala[\IZ^d]$ is compatible with infinite direct
    sums over arbitrary index sets, (Reg) from
  Definition~\ref{def:regularity-for-COP}
    implies~\ref{nl:calc-reg-coherent}.
    
  For $j > i$ and $\lambda \in |Q_j|$ we have $G_{\lambda} \subseteq G_{p_j^i(\lambda)}$
  as $p_j^i$ is $G$-equivariant.  The inclusion $\calc_j \subseteq \calc_i$ corresponds
  under~\eqref{eq:calc_i-as-Hecke} to the functor
  \begin{equation*}
    \bigoplus_{\lambda \in |Q_j|}  (\calb|_{G_{\lambda}})_\oplus \to
    \bigoplus_{\kappa \in |Q_i|}  (\calb|_{G_{\kappa}})_\oplus
  \end{equation*}
  that sends the $\lambda$-summand into the $\kappa$-summand by the functor induced from
  the inclusion $G_\lambda \subseteq G_\kappa$ where $\kappa = p_j^i(\lambda)$. 
  By~\cite[Lem.~7.51]{Bartels-Lueck(2023foundations)}  each of these is exact.
  This implies~\ref{nl:calc-exact}. 
\end{proof}

Lemma~\ref{lem:nested-is-reg_and_exact} allows us to
apply~\cite[Thm.~14.1]{Bartels-Lueck(2020additive)} as reviewed in
Subsection~\ref{subsec:sequence-cat} to the nested sequence constructed above.
So~\eqref{eq:assembly-for-Gamma-limit} and therefore also
(7) appearing in~\eqref{eq:diagram-for-CVCYC-via-COMP} is an equivalence.

\begin{proof}[Formal conclusion of the proof of Theorem~\ref{thm:contC_G(P)-for-CVCYC-via-COMP}
  and Theorem~\ref{thm:reduction}.]   
  Altogether we have shown that the top horizontal map
  in~\eqref{eq:diagram-for-CVCYC-via-COMP} is an equivalence.  As noted before, because of
  the pushout from Proposition~\ref{prop:pushout-3C0-C}, this also implies
  that~\eqref{eq:PCOM-over-P-to-P-C} is an equivalence as claimed in
  Theorem~\ref{thm:contC_G(P)-for-CVCYC-via-COMP}. We have already explained in
  Subsection~\ref{subsec:contC} that Theorem~\ref{thm:reduction} follows from
  Theorem~\ref{thm:contC_G(P)-for-CVCYC-via-COMP}.
\end{proof}

\begin{remark}\label{rem:value_on_compact_subgroups}
  At least in the case  $M \in \EP\Or_\COM(G)$ we can give an explanation of the value of
 $\bfK \big( \contc_G(M) \big)$.  Write 
 $M = (G/K_1,\dots,G/K_n)$ with $K_r \in \COM$ and fix a
 choice of compact open subgroups $(U_{r,i})_{r=1,\dots,n, i \in \IN_{\geq 1}}$ of
$G$ such that for each $r=1,\dots,n$ we have
$U_{r,1} \supseteq U_{r,2} \supseteq \dots \;\;\supseteq  \bigcap_{i} U_{r,i} = K_r$.
 Put $Q_i = (G/K_{1,i}, \ldots, G/K_{n,i})$ for  $i \in \IN_{\geq 1}$.

Then there is zigzag of weak homotopy equivalences from $\bfK \big( \contc_G(M) \big)$
to the homotopy inverse limit $\holim_{i \to \infty} \bfK \big( \contc_G(M_i) \big)$.
In particular there is for every
  $n \in \IZ$ a short exact sequence
  \begin{multline*}
    0 \to \invlim^1_{i \to \infty} \pi_{n+1}\bigl(\bfK(\contc_G(Q_i))\big)
    \to \pi_n\bigl(\bfK(\contc_G(M))\big) \\
  \to \invlim_{i \to \infty}\pi_{n}\bigl(\bfK(\contc_G(Q_i))\big) \to 0.
  \end{multline*}
  Since we do not need this result in this paper, we omit its proof.
  It is very unlikely that the corresponding statement holds, if
  we drop the assumption that each $K_i$ is compact.
\end{remark}


\appendix
\renewcommand{\thesubsection}{\Alph{section}.{\sc\roman{subsection}}}


\section{Homotopy colimits}\label{app:homotopy-colimits}


\subsection{Cofinality and transitivity}

We recall two basic facts about homotopy
colimits~\cite[§9]{Dwyer-Kan(1984a)}.

Let $F \colon \cala \to \calb$ be functor.  Let $B \in \calb$.  We
write $B \downarrow F$ for the following category.  Objects are pairs
$(A,\beta)$ with $A \in \cala$ and $\beta \colon B \to F(A)$ in
$\calb$.  Morphisms $(A,\beta) \to (A',\beta')$ are morphisms
$\alpha \colon A \to A'$ in $\cala$ with
$F(\alpha) \circ \beta = \beta'$.  We write $F \downarrow B$ for the
following category.  Objects are pairs $(A,\beta)$ with $A \in \cala$
and $\beta \colon F(A) \to B$ in $\calb$.  Morphisms
$(A,\beta) \to (A',\beta')$ are morphisms $\alpha \colon A \to A'$ in
$\cala$ with $ \beta \circ F(\alpha) = \beta'$.  If $F$ is the
inclusion of the subcategory $\cala$, then we write
$B \downarrow \cala$ and $\cala \downarrow B$ instead of
$B \downarrow F$ and $F \downarrow B$.

A functor $F \colon \cala \to \calb$ between small categories is said
to be \emph{right cofinal}, if the nerve of $B \downarrow F$ is
contractible (and in particular non-empty) for every $B \in \calb$.
 
\begin{lemma}[Cofinality]\label{lem:cofinality}
  Let $F \colon \cala \to \calb$ be a right cofinal functor between
  small categories.  Then for any $\bfD \colon \calb \to \Spectra$ the
  canonical map
  \begin{equation*}
    \hocolimunder_{\cala} F^*\bfD   \; \xrightarrow{\sim} \; \hocolimunder_{\calb} \bfD 
  \end{equation*}
  is an equivalence.
\end{lemma}

\begin{proof}
  This is~\cite[9.4]{Dwyer-Kan(1984a)}.
\end{proof}

\begin{lemma}[Transitivity]\label{lem:transitivity}
  Let $F \colon \cala \to \calb$ be a functor between small
  categories.  Consider functors $\bfD \colon \cala \to \Spectra$,
  $\bfE \colon \calb \to \Spectra$ and a natural transformation
  $\tau \colon \bfD \to F^*\bfE$.  For $B \in \calb$ write
  $v_B \colon F \downarrow B \to \cala$ for the forgetful functor.
  Suppose that for any $B \in \calb$ the canonical map
  \begin{equation*}
    \hocolimunder_{F \downarrow B} v_B^*\bfD \; \xrightarrow{\sim} \; \bfE(B)
  \end{equation*}  
  is an equivalence.  Then the canonical map
  \begin{equation*}
    \hocolimunder_{\cala} \bfD   \; \xrightarrow{\sim} \; \hocolimunder_{\calb} \bfE 
  \end{equation*}
  is an equivalence as well.
\end{lemma}

\begin{proof}
  The assumption implies that $\bfE$ is the homotopy push down
  $F_*\bfD$ of $\bfD$ along $F$.  The assertion follows
  from~\cite[9.4]{Dwyer-Kan(1984a)}.
\end{proof}


\subsection{Homotopy colimits and categories with product}%
\label{appsubsec:homotopy-colimits_and_products}

Let $\calp$ be a small category with all non-empty finite products.
Let $\calp_+$ be the category obtained from $\calp$ by adding a
terminal object $\ast$.  For example for $\calp = \EP \cala$, we have
$\calp_+ = \EPplus\cala$.  Let $\bfD \colon \calp_+ \to \Spectra$ be a
functor.  The unique maps $P \to \ast$ induce the canonical map
\begin{equation*}
  \hocolimunder_{P \in \calp} \bfD(P)  \; \to \; \bfD(\ast).
\end{equation*}

\begin{lemma}\label{lem:hocolim-times-Q}
  Fix $Q \in \calp$.  Then the canonical map
  \begin{equation*}
    \hocolimunder_{P \in \calp} \bfD(P \times Q)  \; \to \; \bfD(Q)
  \end{equation*}
  is an equivalence.
\end{lemma}
  
\begin{proof}
  Consider
  \begin{equation*}
    \calp \; \xrightarrow{f_Q}  \; \big(\calp\downarrow Q\big) \;
    \xrightarrow{v_Q}  \; \calp  \; \xrightarrow{\bfD} \;  \Spectra 
  \end{equation*}
  where $f_Q$ sends $P$ to the canonical projection $P \times Q \to Q$
  and $v_Q$ is the forgetful functor.  It is easy to verify that $f_Q$
  is right cofinal.  Obviously $(Q,\id_Q)$ is terminal in
  $\calp \downarrow Q$.  Thus, using Lemma~\ref{lem:cofinality},
  \begin{equation*}
    \hocolimunder_{P \in \calp} \bfD(P \times Q)
    = \hocolimunder_{\calp} f_Q^*v_Q^*\bfD \xrightarrow{\sim}
    \hocolimunder_{\calp\downarrow Q} v_Q^*\bfD \xrightarrow{\sim} v_Q^*\bfD(Q,\id_Q) = \bfD(Q). 
  \end{equation*}
\end{proof}

\begin{proposition}\label{prop:hocolim-iso-for-smallest}
  Let $\cala$ be a small category.  Let $\cale$ be the smallest
  collection of functors $\bfE \colon \calp_+ \cala \to \Spectra$ that
  is closed under hocolimits, retracts and contains all functors of
  the form $P \mapsto \bfD(P \times Q)$ with $Q \in \calp$ and
  $\bfD \colon \calp \to \Spectra$.  Then for any $\bfE \in \cale$ the
  canonical map
  \begin{equation*}
    \hocolimunder_{P \in \calp} \bfE(P)  \; \xrightarrow{\sim} \; \bfE(\ast).
  \end{equation*} 
  is an equivalence.
\end{proposition}

\begin{proof}
  Homotopy colimits and retracts of equivalences are again
  equivalences.  The result follows thus from
  Lemma~\ref{lem:hocolim-times-Q}.
\end{proof}

\begin{lemma}\label{lem:finite-products-contractible-nerv}
  The nerv of $\calp$ is contractible (provided that $\calp$ is
  non-empty).
\end{lemma}

\begin{proof}  
  Let $Q$ be a fixed object of $\calp$.  Let
  $c_Q, x_Q \colon \calp \to \calp$ be the functors $c_Q(P) = Q$, $c_Q(\varphi) = \id_Q$,
  $x_Q(P) = P \times Q$, $x_Q(\varphi) = (\varphi \times \id_Q)$.  There are evident natural transformations
  $\id_\calp \leftarrow x_Q \to c_Q$.  On the nerv $\id_\calp$ induces
  the identity, $c_Q$ induces a constant map, and the two maps are
  homotopic.
\end{proof}

\begin{lemma}\label{lem:hocolim-products-Q-to-P}
  Let $\calq$ be a further category with all non-empty finite
  products.  Let $F \colon \calq \to \calp$ be a product preserving
  functor.  Assume that $P \downarrow F$ is non-empty for all
  $P \in \calp$.  Then $F$ is right cofinal.  In particular, for any
  $\bfD \colon \calp \to \Spectra$ the canonical map
  \begin{equation*}
    \hocolimunder_\calq F^* \bfD \; \xrightarrow{\sim} \; \hocolimunder_{\calp} \bfD 
  \end{equation*}   
  is an equivalence.
\end{lemma}

\begin{proof}
  Using that $\calq$ has all non-empty finite products and that $F$
  preserves finite products, it is not difficult to check that
  $P \downarrow F$ also has all non-empty finite products.
  Lemma~\ref{lem:finite-products-contractible-nerv} implies that $F$
  is right cofinal.  The statement about homotopy colimits follows
  from Lemma~\ref{lem:cofinality}.
\end{proof}


\section{K-theory of $\dg$-categories}\label{app:K-theory-dg-cat}

In order to obtain induced maps in K-theory for homotopy coherent
functors we use K-theory for $\dg$-categories.  We briefly review its
construction.  We write $\matheurm{Cat}_{\dg}$ for the category of
small $\dg$-categories~\cite{Keller(2006ICM)}.  For a $\dg$-category
$\calc$ and objects $C,C' \in \calc$, we write $\calc(C,C')$ for the
chain complex of morphisms from $C$ to $C'$.  For a $\dg$-category
$\calc$ the category $H_0(\calc)$ has the same objects as $\calc$ and
for $C,C' \in \calc$ the set of morphisms $C \to C'$ in $H_0(\calc)$
is given by $H_0(\calc(C,C'))$. 
A functor $F \colon \calc \to \cald$ of $\dg$-categories is a
\emph{quasi-equivalence} if
\begin{equation*}
  F_* \colon \calc(C,C') \to \cald(F(C),F(C'))
\end{equation*}
is a quasi-isomorphism for all $C, C'$, and it induces an equivalence
$H_0(\calc) \to H_0(\calc)$.

\begin{theorem}[\cite{Cisinski-Tabuada(2011),Keller(2006ICM),Schlichting(2006)}]\label{thm:dg-cat-K-theory}
  There exists a functor
  \begin{equation*}
    \bfK \colon \matheurm{Cat}_{\dg} \to \Spectra
  \end{equation*}
  with the following properties:
  \begin{enumerate}[label=(\alph*),leftmargin=*]
  \item\label{thm:dg-cat-K-theory:agreement} Restricted to additive
    categories $\bfK$ is weakly equivalent to the usual
    (non-connective) K-theory
    functor; \item\label{thm:dg-cat-K-theory:ch(A)} For an additive
    category $\cala$ the inclusion of $\cala$ into the category
    $\ch_{\fin} \cala$ of finite chain complexes induces a weak
    equivalence $\bfK(\cala) \to \bfK
    (\ch_{\fin}(\cala))$; \item\label{thm:dg-cat-K-theory:quasi} If
    $F \colon \calc \to \calc'$ in $\matheurm{Cat}_{\dg}$ is a
    quasi-equivalence, then $\bfK(F)$ is a weak equivalence.
  \end{enumerate}
\end{theorem}

\begin{proof}[Sketch of proof]
  K-theory of a $\dg$-category $\calc$ can be defined as the
  Waldhausen K-theory of the category of perfect
  $\calc$-modules~\cite[Sec.~5.2]{Keller(2006ICM)}
  and~\cite{Cisinski-Tabuada(2011)}.
	 
  We review the construction of
  Schlichting~\cite[Sec.~6.4]{Schlichting(2006)}\footnote{Strictly
    speaking Schlichting only considers $\operatorname{dga}$s, not
    $\dg$-categories, but his construction generalizes in a straight
    forward manner.}.  Let $\calc$ be a $\dg$-category.  A
  $\calc$-module is a $\dg$-functor
  $M \colon \calc^{\op} \to \ch(\MODcat{\IZ})$.  A map of
  $\calc$-modules $M \to M'$ is a natural transformation by chain
  maps.  A $\calc$-module is free on $v \in M(C)_k$ if
  $f \mapsto f^*(v)$ defines an isomorphism $\calc(-,C)[k] \to M(-)$
  of $\calc$-modules.  A finite cell $\calc$-module is a
  $\calc$-module $M$ together with a finite filtration
  $M_0 \subseteq M_1 \subseteq \dots \subseteq M_n = M$ such that the
  quotients $M_i/M_{i-1}$ are free on a finite number of
  generators.     We write $\calc$-$\matheurm{cell}$ for the
  category of finite cell $\calc$-modules.  If $\cala$ is an additive
  category, then the category of finite cell $\cala$-modules can be
  identified with the $\dg$-category $\ch_{\fin}(\cala)$ of finite
  chain complexes over $\cala$.  For $\calc = \ch_{\fin}(\cala)$ the
  Yoneda embedding of $\ch_{\fin}(\cala)$ into the category of
  $\ch_{\fin}(\cala)$-modules identifies $\ch_{\fin}(\cala)$ with
  $\ch_{\fin}(\cala)$-$\matheurm{cell}$.  A sequence
  \begin{equation*}
    M \xrightarrow{i} M' \xrightarrow{p} M''
  \end{equation*}
  in $\calc$-$\matheurm{cell}$ is said to be \emph{exact}, if
  $M(C) \to M'(C) \to M''(C)$ is exact in $\ch(\IZ)$ for each
  $C \in \calc$.  With this notion of exact sequences
  $\calc$-$\matheurm{cell}$ is an exact category.
	
  The exact category $\calc$-$\matheurm{cell}$ has enough injectives
  and projectives.\footnote{$I$ is injective (projective) in an exact
    category $\cale$ if for any exact sequence
    $A \xrightarrow{i} B \xrightarrow{p} C$ the map
    $\cale(B,I) \xrightarrow{i^*} \cale(A,I)$ (the map
    $\cale(P,B) \xrightarrow{p_*} \cale(P,C)$) is
    onto~\cite[Sec.~5]{Keller(2006derived-uses)}.}  The classes of
  projectives and injectives in this exact category coincide and are
  given by the contractible cell modules.  A map
  $\calc$-$\matheurm{cell}$ factors over an injective (projective), if
  and only if it chain nullhomotopic. (All this can be proven using
  cones of $\calc$-modules exactly as
  in~\cite[Sec.6.4]{Schlichting(2006)}.)
	
  A \emph{Frobenius category} is an exact category $\cale$ with enough
  projectives and injectives such that the classes of injectives
  and projectives coincide.  The associated stable category is the
  quotient of $\cale$ by the subcategory of projectives (injectives),
  i.e., morphisms are identified if they factor over a projective
  (injective).  A map of Frobenius categories is an exact functor that
  preserves projectives (injectives).  The category of cell modules
  $\calc$-$\matheurm{cell}$ is a Frobenius category and the associated
  stable category is the homotopy category of
  $\calc$-$\matheurm{cell}$.
	
  Schlichting's~\cite[Sec.~12.1]{Schlichting(2006)} non-connective
  K-theory functor
  \begin{equation*}
    \bfK^{\text{Frob}} \colon \text{Frobenius categories} \to \Spectra  
  \end{equation*}
  has the following property: if a map $\cale \to \cale'$ of Frobenius
  categories induces an equivalence of stable categories, then
  $\bfK^{\text{Frob}} (\cale) \to \bfK^{\text{Frob}}(\cale)$ is an
  equivalence~\cite[Cor.~1]{Schlichting(2006)}.
	
  For a $\dg$-category $\calc$ one then defines
  \begin{equation*}
    \bfK (\calc) := \bfK^{\text{Frob}}( \calc\text{-}\matheurm{cell}).
  \end{equation*}
  For an additive category $\cala$ we have
  $\cala$-$\matheurm{cell} \simeq \ch_{\fin}(\cala)$, see above.
  Now~\ref{thm:dg-cat-K-theory:agreement} follows
  from~\cite[Thm.~5]{Schlichting(2006)}.
  For~\ref{thm:dg-cat-K-theory:ch(A)} we can use that for an additive
  category $\cala$ the inclusion $\cala \to \ch_{fin}(\cala)$ induces
  an equivalence on cell modules and therefore in K-theory.
  For~\ref{thm:dg-cat-K-theory:quasi} we use that for a
  quasi-equivalence $\calc \to \calc'$ the induced functor on the
  stable categories associated to the categories of cell modules
  (i.e., on the homotopy categories of cell modules) is an
  equivalence.
\end{proof}

We record the following consequence of the quasi-isomorphism
invariance of K-theory.
  
\begin{corollary}\label{cor:quasi-iso-between-dg-functors-homotopy-in-K} Let
  $F, F' \colon \calc \to \cald$ be $\dg$-functors and let
  $\tau \colon F \to F'$ be a natural transformation of
  $\dg$-functors.  Assume that $H_0(\tau) \colon H_0(F) \to H_0(F')$
  is an isomorphism between the functors
  $H_0(F), H_0(F') \colon H_0(\calc) \to H_0(\cald)$.  Then $\bfK (F)$
  and $\bfK (F')$ are homotopic.
\end{corollary}

\begin{proof}
  This follows from
  Theorem~\ref{thm:dg-cat-K-theory}~\ref{thm:dg-cat-K-theory:quasi}
  using the path $\dg$-category $P(\cald)$ associated to $\cald$
  (see~\cite[2.9]{Drinfeld(2004)}
  and~\cite[Def.~4.1]{Tabuada(2007PhD)}) by a standard argument.

  Objects of $P(\cald)$ are diagrams $D \xrightarrow{f} D'$ in $\cald$
  with $f$ closed of degree $0$ whose homology class is an isomorphism
  in $H_0(\cald)$.  Given objects $D_0 \xrightarrow{f_0} D'_0$,
  $D_1 \xrightarrow{f_1} D'_1$, the complex
  $P(\cald)(D_0 \xrightarrow{f_0} D'_0,D_1 \xrightarrow{f_1} D'_1)$ is
  the homotopy fiber of
  \begin{eqnarray*}
    \cald(D_0,D'_0) \oplus \cald(D_1,D'_1) & \to & \cald(D_0,D'_1) \\
    (\varphi,\varphi') & \mapsto & f_1 \circ \varphi - \varphi' \circ f_0.
  \end{eqnarray*} 
  Composition is defined by the formula
  \begin{equation*}
    (\varphi_1,\varphi'_1,s_1) \circ (\varphi_0,\varphi'_0,s_0)
    := (\varphi_1 \circ \varphi_0, \varphi'_1 \circ \varphi'_0, \varphi'_1 \circ s_0 + s_1 \circ \varphi_0).
  \end{equation*}
  There are $\dg$-functors $p,p' \colon P(\cald) \to \cald$ with
  $p(D \xrightarrow{f} D') = D$, $p'(D \xrightarrow{f} D') = D'$.  We
  also have $i \colon \cald \to P(\cald)$ with
  $i(D) = (D\xrightarrow{\id_D}D)$ and
  $\Psi \colon \calc \to P(\cald)$ with
  $\Psi(C) = (F(C) \xrightarrow{\tau_C} F'(C))$.  Then
  $p \circ i = p' \circ i = \id_\cald$.  Moreover, $i$ is a
  quasi-isomorphism~\cite[Lem.~4.3]{Tabuada(2007PhD)}.  Using
  Theorem~\ref{thm:dg-cat-K-theory}~\ref{thm:dg-cat-K-theory:quasi}
  this implies $\bfK(p) \simeq \bfK(p')$.  We also have
  $F = p \circ \Psi$ and $F' = p' \circ \Psi$.  Thus
  $\bfK(F) \simeq \bfK(F')$.
\end{proof}


\section{Homotopy coherent functors}\label{app:homotopy-coherent}

Our construction of the transfer depends on induced maps in K-theory
for homotopy coherent functors.  Here we give a self contained
construction of such induced maps that is tailored to our specific
needs.  It is very pedestrian, by no means a complete theory, and
makes no direct contact with the elegant language of stable
$\infty$-categories.   We include this for completeness.

\begin{definition}\label{def:homotopy-coherent-functor} Let $\cala$ and $\calb$ be
  additive categories.  A \emph{homotopy coherent functor}
  $F = (F^0,F^1,\dots) \colon \cala \to \ch \calb$ consists of the
  following data.
  \begin{enumerate}[label=(\thetheorem\alph*),
                   align=parleft, 
                 leftmargin=*,
                 ] 
  \item For any object $A \in \cala$ an object $F^0(A) \in \ch \calb$;
  \item For any chain
    \begin{equation*}
      A_0 \xleftarrow{\varphi_1} A_1 \xleftarrow{\varphi_2} \dots \xleftarrow{\varphi_n} A_n
    \end{equation*} 
    of morphisms in $\cala$ a map
    $F^n(\varphi_1,\dots,\varphi_n) \colon F^0(A_n) \to F^0(A_0)$ of
    degree $n-1$.
  \end{enumerate}
  $F$ is required to satisfy
  \begin{multline*}
    d F^n(\varphi_1,\dots,\varphi_n) = \sum_{j=1}^{n-1} (-1)^j
    \Big(F^{n-1}(\varphi_1,\dots,\varphi_j \circ \varphi_{j+1},\dots
    \varphi_n) \\- F^j(\varphi_1,\dots,\varphi_j) \circ
    F^{n-j}(\varphi_{j+1},\dots,\varphi_n)\Big).
  \end{multline*}
  Here $d$ is the differential in
  $\hom_{\ch \calb}(F(A_n),F(A_0))$\footnote{We use
    $d(\psi) = d_{B'} \circ \psi - (-1)^{|\psi|} \psi \circ d_B$ for
    $\psi \colon B \to B'$ in $\ch \calb$ of degree $|\psi|$.}.  The
  $F^n$ are also be required to be multi-linear in the $\varphi_i$.
\end{definition}

\begin{example}\label{ex:functor-as-homotopy-coherent} Every functor
  $f \colon \cala \to \ch \calb$ can be viewed as a homotopy coherent
  functor $F$ with $F^0(A) = f(A)$, $F^1(\varphi) = f(\varphi)$,
  $F^n \equiv 0$ for $n \geq 2$.
\end{example}

\begin{example}[Deformation of a functor to a homotopy coherent functor]%
\label{ex:def-functor-to-homotopy-coherent}
  Let $f \colon \cala \to \ch \calb$ be a $\IZ$-linear functor.
  Suppose we are given for each $A \in \cala$ a diagram in $\ch \calb$
  \begin{equation*}
    B_A \xrightarrow{i_A} f(A) \xrightarrow{r_A} B_A  
  \end{equation*}
  with $r_A \circ i_A = \id_{B_A}$ and a chain homotopy
  $H_A \colon i_A \circ r_A \simeq \id_{f(A)}$.  Then we obtain a
  homotopy coherent functor $F = (F^0,F^1,\dots)$ with $F(A) = B_A$
  for $A \in \cala$ and
  \begin{multline*}
    F^n(\varphi_1,\dots,\varphi_n) = \\r_{A_0} \circ f(\varphi_1)
    \circ H_{A_1} \circ f(\varphi_2) \circ \dots \circ
    f(\varphi_{n-1}) \circ H_{A_{n-1}} \circ f(\varphi_n) \circ
    i_{A_n} \hspace{3ex}
  \end{multline*}
  for
  \begin{equation*}
    A_0 \xleftarrow{\varphi_1} A_1 \xleftarrow{\varphi_2} \dots \xleftarrow{\varphi_n} A_n
  \end{equation*} 
  in $\cala$.
\end{example}

\begin{definition}\label{def:strict-nat-transf} Let $F, G \colon \cala \to \ch \calb$ be
  homotopy coherent functors.  A strict natural transformation
  $\tau \colon F \to G$ consists of chain maps
  $\tau_A \colon F^0(A) \to G^0(A)$ for all $A \in \cala$ such that
  for any chain
  \begin{equation*}
    A_0 \xleftarrow{\varphi_1} A_1 \xleftarrow{\varphi_2} \dots \xleftarrow{\varphi_n} A_n
  \end{equation*} 
  of morphisms in $\cala$ the diagram
  \begin{equation*}
    \xymatrix{F(A_n) \ar[r]^{\tau_{A_n}} \ar[d]_{F^n(\varphi_1,\dots,\varphi_n)}
      & G(A_n) \ar[d]^{G^n(\varphi_1,\dots,\varphi_n)} \\ F(A_0) \ar[r]^{\tau_{A_0}} & G(A_0) }
  \end{equation*}     
  commutes.
\end{definition}

\begin{example}\label{ex:strict-nat-transf} Let
  $f \colon \cala \to \ch \calb$,
  \begin{equation*}
    B_A \xrightarrow{i_A} f(A) \xrightarrow{r_A} B_A  
  \end{equation*}
  and $H_A$ be as in
  Example~\ref{ex:def-functor-to-homotopy-coherent}.  Let
  $F \colon \cala \to \ch \calb$ be the homotopy coherent functor
  associated to this data.
	 
  Let $g \colon \cala \to \ch \calb$ a further functor, which we also
  view as a homotopy coherent functor, see
  Example~\ref{ex:functor-as-homotopy-coherent}.
	
  Let $\tau \colon f \to g$ be a natural transformation.  Suppose we
  are given chain maps $p_A \colon B_A \to g(A)$ such that
  $\tau_A = p_A \circ r_A$, $p_A = \tau_A \circ i_A$,
  $0 = \tau_A \circ H_A$.  Then $p_A$ determines a strict natural
  transformation $F \to g$.
\end{example}

\begin{definition}
  We write $\Int$ for the following category.  The objects of $\Int$
  are linearly ordered sets of the form
  $[0,n] := \{0 < 1 < \cdots < n\}$ with $n \in \IN_{>0}$.  Maps
  $[0,n] \to [0,m]$ are strictly order preserving maps
  $\sigma \colon \{0 < \cdots < n\} \to \{0 < \cdots < m\}$ with
  $\sigma(0) = 0$, $\sigma(n) = \sigma(m)$.  In particular $\sigma$ is
  injective and $n \leq m$.
\end{definition}

\begin{remark}
  There is an identification $\Delta_{\text{inj}}^\op \cong
  \Int$. Here $\Delta$ is the usual simplicial category of finite
  ordered sets of the form $\{ 0 < \cdots < n\}$ and
  $\Delta_{\text{inj}}$ is the subcategory obtained by restricting to
  injective maps.
\end{remark}

\begin{remark}\label{rem:evaluate-F-on-phi} Often we will identify $[0,n] \in \Int$ with
  the category associated to $[0,n]$ as a poset.  Note that a
  composable chain of morphisms
  \begin{equation*}
    A_0 \xleftarrow{\varphi_1} A_1 \xleftarrow{\varphi_2} \dots \xleftarrow{\varphi_n} A_n
  \end{equation*}
  in $\cala$ is the same thing as a functor $[0,n]^{\op} \to \cala$.
	
  Moreover, any finite linearly order set $I$ with at least two
  elements is canonically isomorphic to some $[0,n]$.  In particular,
  we can evaluate any homotopy coherent functor
  $F \colon \cala \to \ch \calb$ on any functor $I^{\op} \to \cala$.
\end{remark}

\begin{definition}  
  For $[0,n] \in \ob\Int$ we define the chain complex $C_*([0,n])$ as
  follows.  As an abelian group $C_*([0,n])$ is generated by pairs
  $(I,J)$ with $\{0 < n\} \subseteq J \subseteq I \subseteq [0,n]$.
  The degree of $(I,J)$ is the cardinality $|I \setminus J|$ of
  $I \setminus J$.  The boundary $d$ is determined as follows.  Write
  $I \setminus J = \{ i_1 < \cdots < i_l \}$. 
  Then
  \begin{equation*}
    (I,J) \mapsto \sum_{j=0}^{l} (-1)^j \Big( (I \setminus \{i_j\},J) - (I, J \cup \{i_j\}) \Big).
  \end{equation*} 
  For $\sigma \colon [0,n] \to [0,m]$ we define
  $\sigma_* \colon C_*([0,n]) \to C_*([0,m])$ by
  \begin{equation*}
    \sigma_*(I,J) := \begin{cases} (\sigma(I),\sigma(J)) & |I\setminus J| = |\sigma(I) \setminus \sigma(J)|
      \\ 0 & \text{else} \end{cases}. 
  \end{equation*}
  Altogether we obtain the functor\footnote{In the language
    of~\cite[end of Section~3]{Davis-Lueck(1998)} this is the
    covariant $\IZ\Int$-chain complex
    $C_*(E^{\operatorname{bar}}\Int)$ for $E^{\operatorname{bar}}\Int$
    the bar-model for the covariant classifying $\Int$-$CW$-complex of
    the category $\Int$.}
  \begin{equation*}
    C_* \colon \Int \to \ch(\MODcat{\IZ}).
  \end{equation*} 
\end{definition}

\begin{remark}
  The chain complex $C_*([0,n])$ can also be described as the cellular
  chain complex of the $(n\,\text{-})1$-dimensional cube.
\end{remark}

\begin{remark}\label{rem:C-is-free} 
  In the category of functors $\Int \to \MODcat{\IZ}$ the complex
  $C_*(-)$ is degreewise free; we have
  \begin{equation*}
    C_k(-) \cong \bigoplus_{i=1}^\infty
    \bigoplus_{\tiny \begin{array}{c} \{0 < i\} \subseteq  J \subseteq [0,i], \\ |[0,i] \setminus J|=k \end{array} } \IZ\Int(J,-).
  \end{equation*}
\end{remark}

\begin{remark}
  There is a concatenation map
  \begin{equation}\label{eq:concatenation-C} C_*([0,n]) \otimes
    C_*([0,m]) \to
    C_*([0,n+m])
  \end{equation}
  defined by
  $(I,J) \otimes (I',J') \mapsto (I \cup n+I',J \cup n+J')$.
\end{remark}

\begin{lemma}\label{lem:C-acyclic} The augmentation map
  $\epsilon \colon C_*([0,n]) \to \IZ$ with $\epsilon(I,I) = 1$ is a
  homology isomorphism.
\end{lemma}

\begin{proof}
  We proceed by induction on $n$.  For $n = 1$ the augmentation is an
  isomorphism.  For $n > 1$ we define $\tilde C_*([0,n])$ as the
  sub-complex of $C_*([0,n])$ spanned by all $(I,J)$ with
  $1 \notin I$.  Clearly $\tilde C_*([0,n]) \cong C_*([0,{n-1}])$.
  The map defined by
  \begin{equation*}
    (I,J) \mapsto \begin{cases} (I, J \cup \{1\}) & 1 \in I \setminus J \\ 0 & \text{else} \end{cases}
  \end{equation*}
  induces a contraction on the quotient
  $C_*([0,n]) / \tilde C_*([0,n])$.  Thus the inclusion
  $\tilde C_*([0,n]) \to C_*([0,{n}])$ induces an isomorphism in
  homology.
\end{proof}

In the following we abbreviate $\cala(A,A') := \mor_{\cala}(A,A')$.

\begin{definition}
  Let $\cala$ be a small $\IZ$-linear category.  For $[0,n] \in \Int$
  and $A, A' \in \ob \cala$ we define
  \begin{align*}
    \cala_{[0,n]}(A,A') := \bigoplus_{A_1,\dots,A_{n-1} \in \cala} \cala(A_1,A')  \otimes  \cala
    &(A_2,A_1) \otimes \dots  \otimes \cala(A,A_{n-1}). 
  \end{align*}  
  There is an evident concatenation map
  \begin{equation}\label{eq:concatenation-Hom} \cala_{[0,n]}(A',A'')
    \otimes
    \cala_{[0,m]}(A,A') \to \cala_{[0,n+m]}(A,A'').
  \end{equation}
  Using composition in $\cala$, this construction is functorial in
  $\Int$ and we obtain for fixed $A,A' \in \cala$ a contravariant
  functor
  \begin{equation*}
    \cala_{\--}(A,A') \colon \Int \to \MODcat{\IZ}; \quad [0,n] \mapsto \cala_{[0,n]}(A,A'). 
  \end{equation*}
\end{definition}

\begin{definition}
  Let $\cala$ be a small $\IZ$-linear category.  The $\dg$-category
  $\mathfrak{C}_\cala$ is defined as follows.  The objects of
  $\mathfrak{C}_\cala$ are the objects of $\cala$.  For objects
  $A,A' \in \cala$ the chain complex of morphisms is defined as
  \begin{equation*}
    \mathfrak{C}_\cala(A,A') := \cala_{\--}(A,A') \otimes_{\Int} C_*(\--).
  \end{equation*}
  Here $\cala_{[0,n]}(A,A')$ is viewed as a chain complex concentrated
  in degree $0$.  Composition is induced from the concatenation
  maps~\eqref{eq:concatenation-C} and~\eqref{eq:concatenation-Hom}.

\end{definition}

\begin{construction}
  We construct a functor $\comp \colon \mathfrak{C}_\cala \to \cala$
  as follows.  On objects the functor is the identity.  Let
  $A,A' \in \cala$.  Then
  \begin{equation*}
    \comp \colon \mathfrak{C}_\cala(A,A') \to \cala(A,A')
  \end{equation*}
  is defined as follows.  For
  $f \in \big( \mathfrak{C}_\cala(A,A') \big)_{>0}$ we set
  $\comp(\varphi) = 0$.  Let
  $f \in \big( \mathfrak{C}_\cala(A,A') \big)_{0}$.  We can write
  $f = (\varphi_1,\dots,\varphi_n) \otimes (I,I)$ with $n \in \IN$ and
  $\{0,1\} \subseteq I \subseteq [0,n]$.  Then
  $\comp(f) = \varphi_1 \circ \dots \circ \varphi_n$.
\end{construction}

\begin{lemma}\label{lem:comp-is-hmotopy-equiv} The functor
  $\comp \colon \mathfrak{C}_{\cala} \to \cala$ is a
  quasi-isomorphism.
\end{lemma}

\begin{proof}
  The objects of $\mathfrak{C}$ and $\cala$ coincide.  It remains to
  show that for all $A,A' \in \cala$
  \begin{equation*}
    \comp \colon \mathfrak{C}_{\cala}(A,A') \to \cala(A,A')
  \end{equation*}
  is a chain homotopy equivalence.  Here $\cala(A,A')$ is viewed as
  chain complex concentrated in degree $0$.
  	
  By Lemma~\ref{lem:C-acyclic} the augmentation
  $\epsilon \colon C_*([0,n]) \to \IZ$ is a homology isomorphism.  As
  $\Int([0,1],[0,n])$ consists of a single morphism we obtain a
  homology isomorphism
  \begin{equation*}
    C_*(-) \to \IZ\big[\Int([0,1],-)\big],
  \end{equation*}
  of chain complexes in the category of functors
  $\Int \to \MODcat{\IZ}$.  As both are free, see
  Remark~\ref{rem:C-is-free}, it is a chain homotopy equivalence.
  Thus
  \begin{align*}
    \mathfrak{C}_{\cala}(A,A') & = \cala_{[0,n]}(A,A') \otimes_{[0,n] \in \Int} C_*([0,n])
    \\ & \simeq \cala_{[0,n]}(A,A') \otimes_{[0,n] \in \Int} \IZ\big[\Int([0,1],[0,n])\big]
    \\ & = \cala_{[0,1]}(A,A') = \cala(A,A').
  \end{align*}
\end{proof}

\begin{construction}\label{cons:dg-functor-from-homotopy-coherent} Let
  $F \colon \cala \to \ch \calb$ be a homotopy coherent functor.  We
  define a functor
  $F_\mathfrak{C} \colon \mathfrak{C}_\cala \to \ch \calb$ as follows.
  On objects the functor is given by $A \mapsto F(A)$.  For
  $A,A' \in \cala$
  \begin{equation*}
    \mathfrak{C}_\cala(A,A') \to {\ch \calb}(F(A),F(A'))
  \end{equation*}
  is induced from the maps
  \begin{equation}\label{eq:what-we-want-to-define}
    \cala_{[0,n]}(A,A') \otimes C_*([0,n])
    \to {\ch \calb}\big(F(A),F(A')\big)
  \end{equation}
  that we define next.  Let
  $\varphi_1 \otimes \dots \otimes \varphi_n \otimes (I,J) \in
  \cala_{[0,n]}(A,A') \otimes C_*([0,n])$.  We obtain a functor
  $\phi \colon [0,n]^\op \to \cala$ that sends the map $l \to l+1$ in
  $[0,n]$ to $f_{l+1}$.  Write
  $J = \{ j_0= 0 < j_1 < \cdots < j_k =n \}$ and set
  $J_{r-1,r} := \{j_{r-1}< j_{r-1}+1 < \cdots < j_r\} \cap I$ for
  $r=1,\ldots,k$.  Recall that we can evaluate $F$ on any
  (contravariant) functor from a finite linear ordered set to $\cala$.
  Now~\eqref{eq:what-we-want-to-define} sends
  $\varphi_1 \otimes \dots \otimes \varphi_n \otimes (I,J)$ to
  \begin{equation*}
    F(\phi|_{J_{k-1,k}}) \circ F(\phi|_{J_{k-2,k-1}}) \circ \dots \circ F(\phi|_{J_{1,2}}) \circ  F(\phi|_{J_{0,1}}).  
  \end{equation*}
  Verifying that this yields a well-defined functor $F_\mathfrak{C}$
  is a lengthy but not difficult computation, that we omit.
\end{construction}

\begin{remark}
  The construction of $\comp \mathfrak{C}_\cala \to \cala$ is the
  special case of
  Construction~\ref{cons:dg-functor-from-homotopy-coherent} applied to
  $\id_\cala \colon \cala \to \cala$.
\end{remark}

\begin{remark}
  Let $F, G \colon \cala \to \ch \calb$ be homotopy coherent functors
  and $\tau \colon F \to G$ be a strict natural transformation.  Then
  $\tau$ also defines a natural transformation
  $F_\mathfrak{C} \to G_\mathfrak{C}$.
\end{remark}

\begin{remark}\label{rem:K-zig-zag-homotopy-coherent} Let
  $F \colon \cala \to \ch_{\fin} \calb$ be a homotopy coherent
  functor.  We can apply K-theory to
  \begin{equation*}
    \cala \xleftarrow{\comp} \mathfrak{C}_\cala
    \xrightarrow{F_{\mathfrak{C}}} \ch_{\fin} \calb \xleftarrow{i} \calb 
  \end{equation*}
  where $i$ is the inclusion.  Using Theorem~\ref{thm:dg-cat-K-theory}
  we obtain an induced zig-zag in K-theory
  \begin{equation}\label{eq:K-zig-zag-homotopy-coherent} \bfK \cala
    \xleftarrow{\sim}
    \bfK \mathfrak{C}_{\cala} \xrightarrow{\bfK(F_{\mathfrak{C}})} \bfK \ch_{\fin}
    \calb \xleftarrow{\sim} \bfK \calb.
  \end{equation}
\end{remark}

\begin{remark}\label{rem:K-zig-zag-strict-transf} Let
  $F, G \colon \cala \to \ch_{\fin} \calb$ be homotopy coherent
  functors and $\tau \colon F \to G$ be a strict natural
  transformation.  Then $\tau$ also defines a natural transformation
  $F_{\mathfrak{C}} \to G_{\mathfrak{C}}$.  Suppose that
  $\tau_A \colon F^0(A) \to G^0(A)$ is a weak equivalence for all
  $A \in \cala$.
  Corollary~\ref{cor:quasi-iso-between-dg-functors-homotopy-in-K}
  implies that then
  $\bfK (F_{\mathfrak{C}}) \simeq \bfK (G_{\mathfrak{C}})$.  In other
  words the zig-zags~\eqref{eq:K-zig-zag-homotopy-coherent} induced in
  K-theory by $F$ and $G$ agree up to homotopy.
\end{remark}

	
\section{Mapping $X$ to $|\bfJ_{\CVCYC}(G)|^\wedge$}\label{app:X-to-J}

Let $G$ be a reductive $p$-adic  group and let $X$ be the associated
 extended Bruhat-Tits building.  The extended Bruhat-Tits building $X$ can be given a
simplicial structure for which the action of $G$ is simplicial, proper
and smooth.  We will also use the existence of a $\CAT(0)$-metric
$d_X$ on $X$ that generates the topology of $X$ and is $G$-invariant.
We fix a base point $x_0 \in X$.  We write $B_R$ for the closed ball
of radius $R$ centered at $x_0$ and $\pi_R \colon X \to B_R$ for the
radial projection.
 
For a collection of subgroups $\calv \subseteq \CVCYC$ we will
consider the $\EPplus\All(G)$-simplicial complex $\bfJ_\calv(G)$ from
Example~\ref{ex:bfJ_F}.  For $N \in \IN$ we will also use
$\bfJ^N_\calv(G)$ from Example~\ref{ex:bfJ_F}.  Recall that the
simplicial complex $J^N_\calv(G)$ underlying $\bfJ^N_\calv(G)$ is a
finite complex if $\calv$ is finite.  We will also need the foliated
distance $\fold{\bfJ_\calv^N(G)}$ on $|\bfJ_\calv^N(G)|^\wedge$
(defined in Subsection~\ref{subsec:fol-dis-bfSigma-wedge}).  Recall
that it is compatible with restrictions to subcomplexes, i.e.,
$\fold{\bfJ_\calv^N(G)} =
\fold{\bfJ_{\calv'}^{N'}(G)}|_{|\bfJ_\calv^N(G)|^\wedge}$ for
$N \leq N'$, $\calv \subseteq \calv'$.  To ease notation we abbreviate
$\fold{\bfJ} := \fold{\bfJ_\calv^N(G)}$.

The following result is~\cite[Thm.~1.2]{Bartels-Lueck(2023almost)}
for td-groups with an
action on $\CAT(0)$-space satisfying a technical assumption.  This assumption is satisfied
for the action of a reductive $p$-adic  group on its extended Bruhat-Tits
building~\cite[Prop.~A.7]{Bartels-Lueck(2023almost)}.

\begin{theorem}[$X$ to $J$]\label{thm:X-to-J} 
  There is $N \in \IN$ such that for all $M \subseteq G$ compact and
  $\epsilon > 0$ there are $\beta > 0$ and $\calv \subseteq \CVCYC$
  finite with the following property.  For all $\eta > 0$ and all
  $L > 0$ we find $R > 0$ and a (not necessarily continuous) map
  $f \colon X \to |\bfJ_\calv^N(G)|^\wedge$ satisfying:
  \begin{enumerate}[label=(\alph*),leftmargin=*]
  \item\label{thm:X-to-J:G-equiv} for $x \in B_{R+L}$, $g \in M$ we
    have $\fold{\bfJ}(f(gx),gf(x)) < (\beta,\eta,\epsilon)$;
  \item\label{thm:X-to-J:pi_t} for $x \in B_{R+L}$, $R' \geq R$ we
    have $\fold{\bfJ}(f(x),f(\pi_{R'}(x))) < (\beta,\eta,\epsilon)$;
  \item\label{thm:X-to-J:rho-cont} there is $\rho > 0$ such that for
    all $x,x' \in X$ with $d_X(x,x') < \rho$ we have
    $\fold{\bfJ}(f(x),f(x')) < (\beta,\eta,\epsilon)$.
  \end{enumerate}
\end{theorem}

The three assertions appearing in Theorem~\ref{thm:X-to-J} correspond
to~\ref{assm:E_0:G},~\ref{assm:E_0:H} and~\ref{assm:E_0:eps} from
Assumption~\ref{assm:E_0}.

The goal in this section is to outline the proof Theorem~\ref{thm:X-to-J}.
 A detailed construction is
 given in~\cite[Thm.~1.2]{Bartels-Lueck(2023almost)}. 

\subsection{The flow space}\label{subsec:flow-space}
We briefly recall the flow space $\FS$ associated to $X$
from~\cite[Sec.~1]{Bartels-Lueck(2012CAT(0)flow)}.  It consists of all
\emph{generalized geodesic}, i.e., of continuous maps
$c \colon \IR \to X$ whose restriction to some close
interval\footnote{By this we mean a subset of the form $(-\infty, b]$,
  $[a,\infty)$, $(-\infty,\infty)$, or $[a,b]$.}  is an isometric
embedding and is locally constant on the complement of this interval.
The metric on $\FS$ is given by
\begin{equation*}
  d_\FS(c,c') := \int_\IR \frac{d_X(c(t),c'(t))}{2e^{|t|}} \, dt.
\end{equation*}
In this metric the distance between $c$ and $c'$ is small, iff the
restrictions $c|_{[-a,a]}$ and $c'|_{[-a,a]}$ are pointwise close for
large $a$.  We recall
from~\cite[Prop.~1.7]{Bartels-Lueck(2012CAT(0)flow)} that this metric
generates the topology of uniform convergence on compact subsets.  The
flow $\Phi$ on $\FS$ is defined by
\begin{equation*}
  (\Phi_\tau c)(t) := c(t+\tau).
\end{equation*}
For $c,c' \in \FS$, $\alpha,\delta > 0$ we write
\begin{equation*}
  \fold{\FS}(c,c') < (\alpha,\delta)
\end{equation*}
if there is $t \in [-\alpha,\alpha]$ with
$d_\FS(\Phi_t(c),c') < \delta$.

The construction is natural for the action of the isometry group of
$X$.  In particular, $G$ acts on $\FS$, the flow $\Phi_\tau$ is
$G$-equivariant, and $d_\FS$ and $\fold{\FS}$ are both $G$-invariant.


\subsection{$V$-foliated distance and the flow space}%
\label{subsec:V-fol-and-FS}
There is a close relation between the foliated distance $\fold{\FS}$
on the flow space and the $V$-foliated distance from
Subsection~\ref{subsec:foliated-distance-V}.  We discuss the case
$G = \SL_2(F)$.  We will not explicitly use elsewhere what follows,
but this was our motivation for the definition of $V$-foliated
distance.

The building $X$ for $\SL_2(F)$ is a simplicial tree, the Bass-Serre
tree.  We normalize the metric on $X$ such that each edge has length
$1$.  Let us say that a \emph{bi-infinite combinatorial geodesic} is
an isometric embedding $c \colon \IR \to X$ that sends
$\IZ \subseteq \IR$ into the $0$-skeleton $X^0$ of the Bass-Serre
tree\footnote{Bi-infinite combinatorial geodesics are uniquely
  determined by their restriction to $\IZ$ and any isometric embedding
  $\IZ \to X^0$ extends uniquely to a bi-infinite combinatorial
  geodesic.}.  The bi-infinite combinatorial geodesics form a closed
subspace $\FS^\sharp$ of $\FS$.  The flow on $\FS$ restricts to an
$\IZ$ action on $\FS^\sharp$.  Fix $c \in \FS^\sharp$ and let $K_c$ be
the (pointwise) stabilizer of $c$ and $V_c$ be the stabilizer of the
image of $c$ in the quotient of $\FS^\sharp$ by the $\IZ$-action,
compare~\ref{subsec:local-structure}.  There is a choice of $c$ such
that $K_c$ is the subgroup of diagonal matrices in $\SL_2(\calo)$,
where $\calo$ is the ring of integers of $F$ and $V_c$ is the subgroup
of diagonal matrices in $\SL_2(F)$, but this will not be important in
the following.  Define $\pi \colon G \to \FS^\sharp$ as
$g \mapsto gc$.  As the action of $\SL_2(F)$ on $X$ is strongly
transitive, the induced action on $\FS^\sharp$ is transitive.  Thus we
can identify $\FS^\sharp$ with $G/K_c$ via $\pi$.  The group
$V_c / K_c$ acts on $\IZ \cong c(\IZ) \subseteq X^0$ by translation.
This induces a group homomorphism $t \colon V_c \to \IZ$, that in turn
induces an isomorphism $\bar t \colon V_c / K_c \cong \IZ$.  The
quotient $V_c / K_c$ acts from the right on $G/K_c$ and under
$G/K_c \equiv \FS^\sharp$ and $V_c / K_c \cong \IZ$ this action is the
$\IZ$-action on $\FS^\sharp$ induced from the flow on $\FS$.  Recall
that $\fold{V_c}$ on $G$ depended on the choice of a left invariant
proper metric $d_G$ on $G$.  The metric $d_\FS$ restricts to a metric
on $\FS^\sharp$.  As $\pi$ is $G$-equivariant it is in particular
uniformly continuous.  As $d_G|_{V_c}$ is proper and $V_c$-invariant
and as $K_c$ is compact, $t \colon V_c \to \IZ$ is a coarse
equivalence.  This means that for all $A > 0$ there is $B > 0$ such
that for all $v,v' \in V_c$
\begin{equation*}
  \begin{array}{rcr}
    d_G(v,v') \leq A & \implies & |t(v) - t(v')| \leq B; \\
    |t(v) - t(v')| \leq A & \implies &  d_G(v,v') \leq B.	
  \end{array}
\end{equation*}
Combining all this we have the following: for all $\alpha > 0$ there
is $\beta > 0$, such that for $\epsilon > 0$ there is $\delta > 0$
such that for all $g,g' \in G$,
\begin{equation*}
  \begin{array}{rcr}
    \fold{V_c}(g,g') < (\alpha,\delta) & \implies & \fold{\FS}(\pi(g),\pi(g'))  < (\beta,\epsilon); \\
    \fold{\FS}(\pi(g),\pi(g')) < (\alpha,\delta) & \implies & \fold{V_c}(g,g')   < (\beta,\epsilon).	
  \end{array}
\end{equation*}
Thus we can translate between $\fold{\FS}$ and $\fold{V_c}$.


\subsection{Factorization over the flow space}\label{subsec:factor-over-FS}
The maps in Theorem~\ref{thm:X-to-J} are constructed in two steps as
compositions 
\begin{equation*}
  X \; \xrightarrow{f_0} \; \FS \; \xrightarrow{f_1} \; |\bfJ_\calv^N(G)|^\wedge;
\end{equation*}

\begin{theorem}[$X$ to $\FS$]\label{thm:X-to-FS}
  For all $M \subseteq G$ compact there is $\alpha > 0$ with the
  following property.  For all $\delta > 0$, $L > 0$ there exists
  $R > 0$ and a uniformly continuous map $f_0 \colon X \to \FS$ such
  that
  \begin{enumerate}[label=(\alph*),leftmargin=*]
  \item\label{thm:X-to-FS:G} for $x \in B_{R+L}$, $g \in M$ we have
    $\fold{\FS}\big(f_0(gx),gf_0(x)\big) < (\alpha,\delta)$;
  \item\label{thm:X-to-FS:pi} for $x \in {B}_{R+L}$, $R' \geq R$ we
    have
    $\fold{\FS}\big(f_0(x),f_0(\pi_{R'}(x))\big) < (\alpha,\delta)$,
    where $\pi_{R'}$ denotes the radial projection onto
    $\overline{B}_{R+L}$.
  \end{enumerate}
\end{theorem}

\begin{theorem}[$\FS$ to $J$]\label{thm:FS-to-J} 
  There is $N \in \IN$ such that for any $\alpha > 0$ and any
  $\epsilon > 0$ there are $\beta > 0$ and $\calv \subseteq \CVCYC$
  finite such that for any $\eta > 0$ there are $\delta > 0$,
  $f_1 \colon \FS \to |\bfJ^N_\calv|^\wedge$, satisfying the following
  properties.
  \begin{enumerate}[label=(\alph*),leftmargin=*]
  \item\label{thm:FS-to-J:alpha-delta-to-beta-eta-eps} For
    $c,c' \in \FS$ with $\fold{\FS}(c,c') < (\alpha,\delta)$ we have
    $\fold{\bfJ}\big(f_1(c),f_1(c')\big) < (\beta,\eta,\epsilon)$;
  \item\label{thm:FS-to-J:equiv-up-to-beta-eta-eps} For $c \in \FS$,
    $g \in G$ we have
    $\fold{\bfJ}\big(f_1(gc),gf_1(c)\big) < (\beta,\eta,\epsilon)$.
  \end{enumerate}
\end{theorem}

Theorem~\ref{thm:X-to-FS} is~\cite[Thm.~4.1]{Bartels-Lueck(2023almost)}.
Its proof uses the $\CAT(0)$-geometry of
$X$ and associated dynamic properties of the flow space and is
sketched in Subsection~\ref{subsec:dynamic}. 
Theorem~\ref{thm:FS-to-J} is~\cite[Thm.~4.3]{Bartels-Lueck(2023almost)}.
Its proof uses so called long and thin covers for the
flow space and is sketched in
Subsection~\ref{subsec:long-thin-covers}.

Theorem~\ref{thm:X-to-J} is a formal consequence of
Theorems~\ref{thm:X-to-FS} and~\ref{thm:FS-to-J}.  Formally this uses
that for any $\alpha > 0$, $\delta > 0$ there is $\rho > 0$ such that
$d_X(x,x') < \rho$ implies
$\fold{\FS}(f_1(x),f_1(x')) < (\alpha,\delta)$.  This statement
follows from uniform continuity of $f_1$ (even uniformly in
$\alpha$).


\subsection{Dynamic of the flow}\label{subsec:dynamic}

Theorem~\ref{thm:X-to-FS} is closely related to the results
from~\cite[Sec.~3]{Bartels-Lueck(2012CAT(0)flow)} 
and follows from
similar estimates\footnote{The present set-up is simpler and avoids
  the homotopy action in our estimates.}.  We sketch its proof here.

For $x,x' \in X$ we write $c_{x,x'} \in \FS$ for the generalized
geodesic from $x$ to $x'$, i.e., for the generalized geodesic
characterized by
\begin{eqnarray*}
  c_{x,x'}(t) & = & x \quad t \in (-\infty,0]; \\
  c_{x,x'}(t) & = & x \quad t \in [d(x,y),+\infty). 
\end{eqnarray*}
For $T \geq 0$ consider the map $f_1^T \colon X \to \FS$,
$x \mapsto \Phi_T(c_{x_0,x})$.  Recall that $x_0 \in X$ is our fixed
base point.  Both $x \mapsto c_{x_0,x}$ and $\Phi_T$ (for fixed $T$)
are uniformly continuous, in particular, $f_1^T$ is uniformly
continuous.

\begin{lemma}\label{lem:pi_R'-estimate}
  For all $\delta > 0$ there is $\Delta > 0$ such that for all $R',T$
  with $R' \geq T + \Delta$, $x \in X$ we have
  \begin{equation*}
    d_{\FS} \big(f_1^T(x),f_1^T(\pi_{R'}(x)) \big) < \delta.
  \end{equation*}  	
\end{lemma}

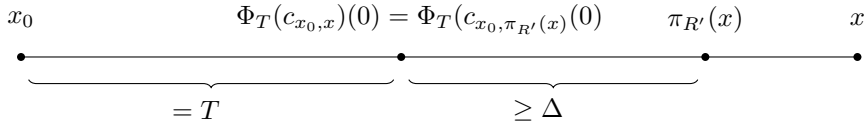
\begin{figure}[H]
  \begin{tikzpicture}
    \draw (0,0) -- (11,0); \fill [black,opacity=.5] (0,0) circle
    (1.5pt) (5,0) circle (1.5pt) (9,0) circle (1.5pt) (11,0) circle
    (1.5pt); \draw (0,.5) node {$x_0$} (5.24,.5) node
    {$\Phi_T(c_{x_0,x})(0) = \Phi_T(c_{x_0,\pi_{R'}(x)}(0)$} (9,.5)
    node {$\pi_{R'}(x)$} (11,.5) node {$x$};
    \draw[decorate,decoration={brace,amplitude=3pt,mirror}] (0.1,-.3)
    -- (4.9,-.3); \draw (2.3,-.7) node {$=T$};
    \draw[decorate,decoration={brace,amplitude=3pt,mirror}] (5.1,-.3)
    -- (8.9,-.3); \draw (6.8,-.7) node {$ \geq \Delta$};
  \end{tikzpicture}
  \caption{Schematic picture for
    Lemma~\ref{lem:pi_R'-estimate}}\label{fig:pi_R'-estimate}
\end{figure}

\begin{proof}[Sketch of proof for Lemma~\ref{lem:pi_R'-estimate}]
  For large $\Delta$ the generalized geodesics
  $f_1^T(x) = \Phi_T(c_{x_0,x})$ and
  $f_1^T(\pi_{R'}(x)) = \Phi_T(c_{x_0,\pi_{R'}(x)})$ agree on a large
  interval $[-a,a]$, see Figure~\ref{fig:pi_R'-estimate}, and are
  therefore close to each other in $\FS$.
  For more details see~\cite[Lem.~5.2]{Bartels-Lueck(2023almost)}.
\end{proof}

\begin{lemma}\label{lem:estimate-d_fol} For all $\alpha > 0$, $\Delta > 0$, $L > 0$, and
  $\delta > 0$ there are $R > 0$, $0 \leq T \leq R - \Delta$ such that
  for all $x \in \overline{B}_{R+L}(b)$ and $g \in G$ with
  $d_X(b,gb) \leq \alpha$, we have
  \begin{equation*}
    \fold{\FS}\big(gf_1^T(x), f_1^T(gx) \big) < (\alpha,\delta).   \end{equation*} 
\end{lemma}

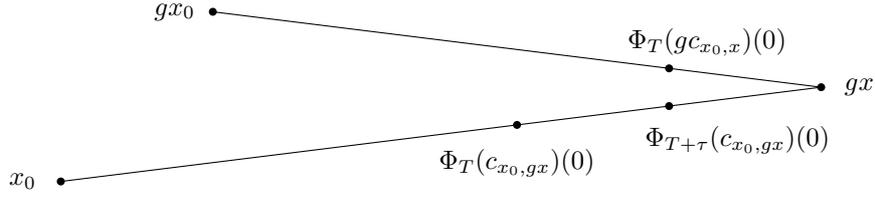
\begin{figure}[H]
  \begin{tikzpicture}
    \draw (2,1) -- (10,0) -- (0,-1.25); \fill [black,opacity=.5] (2,1)
    circle (1.5pt) (10,0) circle (1.5pt) (0,-1.25) circle (1.5pt)
    (8,.25) circle (1.5pt) (6,-.5) circle (1.5pt) (8,-.25) circle
    (1.5pt); \draw (1.5,1) node {$gx_0$} (-.5,-1.25) node {$x_0$}
    (10.5,0) node {$gx$} (8.5,.6) node {$\Phi_T(gc_{x_0,x})(0)$}
    (6,-1) node {$\Phi_T(c_{x_0,gx})(0)$} (8.9,-.7) node
    {$\Phi_{T+\tau}(c_{x_0,gx})(0)$};
  \end{tikzpicture}
  \caption{Schematic picture for
    Lemma~\ref{lem:estimate-d_fol}}\label{fig:estimate-d_fol}
\end{figure}

\begin{proof}[Sketch of proof for Lemma~\ref{lem:estimate-d_fol}]
  We have $gf_1^T(x) = \Phi_T (c_{gx_0,gx})$ and
  $f_1^T(gx) = \Phi_T (c_{gx_0,x})$.  Let
  $\tau := d_X(x_0,gx) - d_X(gx_0,gx)$.  By assumption
  $\tau \in [-\alpha,\alpha]$.  Now for $r >> 0$, $R >> 0$ and
  $T := R - r$, $\Phi_T (c_{gx_0,gx})$ and
  $\Phi_{T+\tau} (c_{gx_0,x})$ will be pointwise close on a large
  interval $[-a,a]$ by the $\CAT(0)$ inequality, see
  Figure~\ref{fig:estimate-d_fol}.
  For more details see~\cite[Lem.~5.4]{Bartels-Lueck(2023almost)}.
\end{proof}
 
\begin{proof}[Sketch of proof for Theorem~\ref{thm:X-to-FS}] 
  We set $\alpha := \max \{ d_X(gx_0,x_0) \mid g \in M \}$.  Given
  $\delta > 0$, $L > 0$ we first choose $\Delta$ as in
  Lemma~\ref{lem:pi_R'-estimate} and use then
  Lemma~\ref{lem:estimate-d_fol} to choose $R$ and $T$.  Then
  $f_1^T \colon X \to \FS$ has the desired properties:
  for~\ref{thm:X-to-FS:G} we use the estimate from
  Lemma~\ref{lem:estimate-d_fol} and for~\ref{thm:X-to-FS:pi} we use
  the estimate from Lemma~\ref{lem:pi_R'-estimate}.
\end{proof}
 

\subsection{Long thin covers}\label{subsec:long-thin-covers}

\begin{definition}[$\alpha$-long cover]\label{def:alpha_long_cover}
  A cover $\calu$ of the flow space $\FS$ by open subsets is said to
  be \emph{$\alpha$-long} if for any $c_0 \in \FS$ there exists
  $U \in \calu$ such that $\Phi_{[-\alpha,\alpha]}(c_0) \subseteq U$.
\end{definition}

Typically such covers are thin in directions transversal to the flow
and are often referred to as long thin covers.  The proof of
Theorem~\ref{thm:FS-to-J} uses the following three results for the
flow space.

\begin{proposition}[Partition of unity]\label{prop:partition} For all $\alpha > 0$,
  $\e > 0$, $N \in \IN$ there is $\alpha' > 0$ such that the following
  holds.  Let $\calu$ be a $\alpha'$-long cover of dimension $\leq N$
  by $G$-invariant open subsets of $\FS$.  Then there exists a
  partition of unity $\{ t_U \colon \FS \to [0,1] \mid U \in \calu \}$
  subordinate to $\calu$ and $\delta > 0$ such that
  \begin{enumerate}[label=(\alph*),leftmargin=*]
  \item for $U \in \calu$, $c,c' \in \FS$ with
        $\fold{\FS}(c,c') < (\alpha,\delta)$ we have
    \begin{equation*}
      |t_U(c) -t_U(c')| < \e;  
    \end{equation*}
  \item the $t_U$ are $G$-invariant.
  \end{enumerate}
\end{proposition}

Proposition~\ref{prop:partition}  is~\cite[Prop.~6.4]{Bartels-Lueck(2023almost)}.
Its proof is not complicated.  The
fact that the cover is $\alpha'$-long allows us to find a partition of
unity whose functions only vary very slowly along flow lines.

\begin{proposition}[Dimension of long thin covers]\label{prop:dimension}
  There is $N \in \IN$ such that for any $\alpha' > 0$ there is
  $\alpha''$ such that the following is true.  Let $\calw$ be an
  $\alpha''$-long cover of $\FS$ by $G$-invariant open subsets.  Then
  there exist collections $\calu_0,\ldots,\calu_N$ of open
  $G$-invariant subsets of $\FS$ such that
  \begin{enumerate}[label=(\alph*),leftmargin=*]
  \item $\calu := \calu_0 \sqcup \ldots \sqcup \calu_N$ is an
    $\alpha'$-long cover of $\FS$, in particular
    $\calu_i \cap \calu_j = \emptyset$ for $i \not= j$;
  \item for each $i$ the open sets in $\calu_i$ are pairwise disjoint;
  \item for each $U \in \calu = \calu_0 \sqcup \ldots \sqcup \calu_N$
    there is $W \in \calw$ with $U \subseteq W$.
  \end{enumerate}
\end{proposition}

Proposition~\ref{prop:dimension} is~\cite[Prop.~6.6]{Bartels-Lueck(2023almost)}.
It is closely related to other existence
results for long and thin covers in the literature, although it is not
stated in exactly this form elsewhere.  To prove it one can pass to
the quotient $G \backslash \FS$ and adapt the strategy
from~\cite{Kasprowski-Rueping(2017cov)}.  The main difference is that
here our constrain on $\calu$ is formulated in terms of the given
cover $\calw$, while in~\cite{Kasprowski-Rueping(2017cov)} the
constrain on the members of $\calu$ are formulated in terms of the
group action.

\begin{proposition}[Local structure]\label{prop:local}
  For all $\alpha'' > 0$ there are $\beta > 0$ and
  $\calv \subseteq \CVCYC$ finite with the following property.  For
  all $\eta > 0$ and all $c_0 \in \FS$ there exist $U \subseteq \FS$
  open, $h \colon U \to G$, $V \in \calv$ and $\delta'' > 0$ such that
  \begin{enumerate}[label=(\alph*),leftmargin=*]
  \item\label{prop:local:(1)} for some neighborhood $U_0$ of the orbit
    $Gc_0$ we have $\Phi_{[-\alpha'',\alpha'']}(U_0) \subseteq U$;
  \item\label{prop:local:(2)} $U$ is $G$-invariant;
  \item\label{prop:local:(3)} for $c,c' \in U$ we have
    \begin{equation*}
      \fold{\FS}(c,c') < (\alpha'',\delta'')  \implies \fold{V}(h(c),h(c')) < (\beta,\eta);
    \end{equation*} 
  \item\label{prop:local:(4)} for $c \in U$, $g \in G$ we have;
    \begin{equation*}
      \fold{V}(h(gc),gh(c)) < (\beta,\eta).
    \end{equation*} 
  \end{enumerate}
\end{proposition}

Proposition~\ref{prop:local} is~\cite[Prop.~6.8]{Bartels-Lueck(2023almost)}.
We will briefly discuss its proof  in
Subsection~\ref{subsec:local-structure}.

The map $f_1$ from Theorem~\ref{thm:FS-to-J} can locally be
constructed using Proposition~\ref{prop:local}.  Using the partition
of unity from Proposition~\ref{prop:partition} these can be patched
together to a map $\FS \to |\bfJ^N_\calv(G)|^\wedge$.  As noted
already in Subsection~\ref{subsec:comparison-discrete} this patching
procedure forces us to pass from orbits to products of orbits.  Some
care is needed to control the dimension of the image of theses maps;
this is where Proposition~\ref{prop:dimension} is needed.


\subsection{Local structure of the flow space}\label{subsec:local-structure}

For $c \in \FS(X)$ we set
\begin{eqnarray*}
  K_c & := & G_c = \{g \in G \mid gc=c \} 
             =   \{ g \in G \mid gc(t)=c(t) \; \text{for all} \; t \in \IR \}; \\
  V_c & := & \{ g \in G \mid \exists t \in \IR \, \text{such that} \, gc = \Phi_t(c) \}; \\ 
  \tau_c & := & \inf \{ t > 0 \mid \exists v \in V_c \setminus K_c,  \text{with} \; \Phi_t(c)=vc\}. 
\end{eqnarray*}
We use $\inf \emptyset = \infty$.  If $\tau_c < \infty$ then we say
that $c$ is \emph{periodic}.  We have $K_c \subseteq V_c$ as the flow
is $G$-equivariant.  For $\alpha > 0,\delta > 0$, $c \in \FS$ we set
\begin{equation}
  U^\fol_{\alpha,\delta}(c) := \{ c' \in \FS \mid \fold{\FS}(c,c') < (\alpha,\delta) \}. 
\label{U_upper_fol_(alpha,delta(c)}
\end{equation}
One may think of $U^\fol_{\alpha,\delta}(c)$ as an open ball around
$c$ with respect to $\fold{\FS}$.

\begin{proposition}\label{prop:constr-h} Let $\FS_0 \subseteq \FS$ be compact.  For all
  $\alpha > 0$ there is $\beta > 0$ such that the following is true:
  For all $\eta > 0$, $c_0 \in \FS_0$, there are $\delta > 0$ and a
  (not necessarily continuous) map
  $h \colon G \cdot U^{\fol}_{\alpha,\delta}(c_0) \to G$ satisfying
  \begin{enumerate}[label=(\alph*),leftmargin=*]
  \item\label{prop:constr-h:fol-to-fol} for
    $c,c' \in G \cdot U^{\fol}_{\alpha,\delta}(c_0)$ we have
    \begin{equation*}
      \fold{\FS}(c,c') < (\alpha,\delta) \quad \implies \quad \fold{V_{c_0}}(h(c),h(c')) < (\beta,\eta);
    \end{equation*}	
  \item\label{prop:constr-h:G} for $g \in G$,
    $c \in G \cdot U^{\fol}_{\alpha,\delta}(c_0)$ we have
    \begin{equation*}
      \fold{V_{c_0}}(h(gc),gh(c))< (\beta,\eta).
    \end{equation*}
  \end{enumerate}
\end{proposition}

\begin{proof}[Sketch of proof]
  Let $\alpha > 0$ be given.  Using compactness of $\FS_0$, it is not
  difficult to show the following: there is $\beta > 0$ such that for
  $g \in G$, $c \in \FS_0$ we have
  \begin{equation*}
    d_\FS(gc,c) < 3\alpha \quad \implies \quad d_G(g,e) < \beta.
  \end{equation*}    	
  Next let $\eta > 0$ and $c_0 \in \FS_0$ be given.  For $n \in \IN$
  choose $h_n \colon G \cdot U_{\alpha,1/n}^{\fol}(c_0) \to G$ such
  that $c \in h_n(c) \cdot U_{\alpha,1/n}^\fol(c_0)$ for all
  $c \in G \cdot U_{\alpha,\delta_n}^{\fol}(c_0)$.  It is not
  difficult to check that for all sufficiently large $n$ the map $h_n$
  satisfies~\ref{prop:constr-h:fol-to-fol} and~\ref{prop:constr-h:G}.
  For more details see~\cite[Prop.~9.2]{Bartels-Lueck(2023almost)}.
\end{proof}

\begin{remark}\label{rem:constr-h-K}
  In the assertion of Proposition~\ref{prop:constr-h} the estimates
  can be strengthened to use $\fold{K_{c_0}}$ in place of
  $\fold{V_{c_0}}$ provided $\tau_{c_0} > \ell$
  where $\ell$ is a constant only depending on $\alpha$.
\end{remark}

On its own Proposition~\ref{prop:constr-h} is not quite strong enough
to imply Proposition~\ref{prop:local}, because it is not quite clear
yet that we only need a finite set $\calv$ of subgroups.  To resolve
this one needs to understand how $V_{c_0}$ varies in $c_0$.  Ideally
we would like for $V_{c_0}$ not to increase in small neighborhoods of
$c_0$, at least up to conjugation.  While we do not know this, we have
the following result.
 
\begin{proposition}\label{prop:good-FS_0}
  There exists $\FS_0 \subseteq \FS$ compact such that
  \begin{enumerate}[label=(\alph*),leftmargin=*]
  \item\label{prop:good-FS_0:fund-domain} $G \cdot \FS_0 = \FS$;
  \item\label{prop:good-FS_0:V} for $\ell > 0$ and $c_0 \in \FS_0$
    there exists an open neighborhood $U$ of $c_0$ in $\FS_0$ such
    that for all $c \in U$ with $\tau_c \leq \ell$ we have
    $V_c \subseteq V_{c_0}$.
  \end{enumerate}
\end{proposition}

Proposition~\ref{prop:good-FS_0} is~\cite[Prop.~A.7]{Bartels-Lueck(2023almost)}.
We will not discuss its proof  in
detail.  But we want to point out that this proof uses the
combinatorial structure of the building $X$.  Here we fix an apartment
$A$ for $X$ and use that the $G$-action translates any geodesic to a
geodesic in $A$.  To study the groups $V_{c_0}$ for $c_0 \in \FS(A)$,
one can then use the combinatorial structure of $A$ as a Coxeter
complex.  Here $\FS(A)$ is the flow space for $A$.  This is the only
point where the proof of Theorem~\ref{thm:X-to-J} uses the
combinatorial structure of $X$.

To prove Proposition~\ref{prop:local} one combines
Proposition~\ref{prop:constr-h}, Remark~\ref{rem:constr-h-K} and
Proposition~\ref{prop:good-FS_0}.  
This concludes the discussion of Proposition~\ref{prop:local}.



\def\cprime{$'$} \def\polhk#1{\setbox0=\hbox{#1}{\ooalign{\hidewidth
  \lower1.5ex\hbox{`}\hidewidth\crcr\unhbox0}}}


\end{document}